\documentclass[12pt,a4paper]{amsart}
\usepackage{amsmath,amstext,amssymb,amsthm,amsfonts}
\usepackage[all]{xy}
\newcommand{\nc}{\newcommand}

\nc{\one}{\mbox{\bf 1}}
\nc{\invtensor}{\underset{\leftarrow}{\otimes}}
\nc{\const}{\operatorname{const}}

\nc{\ad}{\operatorname{ad}}
\nc{\tr}{\operatorname{tr}}

\nc{\tp}{\operatorname{top}}
\nc{\rank}{\operatorname{rank}}
\nc{\corank}{\operatorname{corank}}
\nc{\codim}{\operatorname{codim}}
\nc{\sdim}{\operatorname{sdim}}
\nc{\mult}{\operatorname{mult}}

\nc{\spn}{\operatorname{span}}
\nc{\Sym}{\operatorname{Sym}}
\nc{\sym}{\operatorname{sym}}
\nc{\id}{\operatorname{id}}
\nc{\Id}{\operatorname{Id}}
\nc{\Ree}{\operatorname{Re}}

\nc{\htt}{\operatorname{ht}}
\nc{\sch}{\operatorname{sch}}
\nc{\str}{\operatorname{str}}

\nc{\Ker}{\operatorname{Ker}}
\nc{\rker}{\operatorname{rKer}}
\nc{\im}{\operatorname{Im}}
\nc{\osp}{\mathfrak{osp}}
\nc{\sgn}{\operatorname{sgn}}
\nc{\F}{\operatorname{F}}
\nc{\Mod}{\operatorname{Mod}}
\nc{\Mat}{\operatorname{Mat}}
\nc{\Soc}{\operatorname{Soc}}
\nc{\Iso}{\operatorname{Iso}}
\nc{\Hom}{\operatorname{Hom}}
\nc{\End}{\operatorname{End}}
\nc{\supp}{\operatorname{supp}}
\nc{\Card}{\operatorname{Card}}
\nc{\Ann}{\operatorname{Ann}}
\nc{\Ind}{\operatorname{Ind}}
\nc{\Coind}{\operatorname{Coind}}
\nc{\wt}{\operatorname{wt}}
\nc{\ch}{\operatorname{ch}}
\nc{\Stab}{\operatorname{Stab}}
\nc{\Sch}{{\mathcal S}\mbox{\em ch}}
\nc{\Irr}{\operatorname{Irr}}
\nc{\Spec}{\operatorname{Spec}}
\nc{\Prim}{\operatorname{Prim}}
\nc{\Aut}{\operatorname{Aut}}
\nc{\Ext}{\operatorname{Ext}}

\nc{\Fract}{\operatorname{Fract}}
\nc{\gr}{\operatorname{gr}}
\nc{\deff}{\operatorname{def}}
\nc{\HC}{\operatorname{HC}}

\nc{\red}{\operatorname{red}}

\nc{\wdchi}{\widetilde{\chi}}
\nc{\wdH}{\widetilde{H}}
\nc{\wdN}{\widetilde{N}}
\nc{\wdM}{\widetilde{M}}
\nc{\wdO}{\widetilde{O}}
\nc{\wdR}{\widetilde{R}}
\nc{\wdS}{\widetilde{S}}
\nc{\wdV}{\widetilde{V}}

\nc{\wdC}{\widetilde{C}}

\nc{\Obj}{\operatorname{Obj}}
\nc{\Dglie}{\operatorname{{\mathcal D}glie}}
\nc{\Fin}{\operatorname{{\mathcal F}in}}

\nc{\Adm}{\operatorname{\mathcal{A}dm}}

\nc{\Sg}{{\cS(\fg)}}
\nc{\Shg}{{\cS(\fhg)}}
\nc{\Ug}{{\cU(\fg)}}
\nc{\Uhg}{{\cU(\fhg)}}
\nc{\Sh}{{\cS(\fh)}}
\nc{\Uh}{{\cU(\fh)}}
\nc{\Uhh}{{\cU(\fhh)}}
\nc{\Zg}{{{\mathcal{Z}}(\fg)}}

\nc{\Vir}{{\mathcal{V}ir}}
\nc{\NS}{{\mathcal{N}S}}

\nc{\tZg}{{\widetilde{\mathcal Z}({\mathfrak g})}}
\nc{\Zk}{{\mathcal Z}({\mathfrak k})}

\nc{\Up}{{\mathcal U}({\mathfrak p})}
\nc{\Ah}{{\mathcal A}({\mathfrak h})}
\nc{\Ag}{{\mathcal A}({\mathfrak g})}
\nc{\Ap}{{\mathcal A}({\mathfrak p})}
\nc{\Zp}{{\mathcal Z}({\mathfrak p})}
\nc{\cR}{\mathcal R}
\nc{\cS}{\mathcal S}
\nc{\cT}{\mathcal{T}}
\nc{\cY}{\mathcal Y}
\nc{\cA}{\mathcal A}
\nc{\cB}{\mathcal B}
\nc{\cD}{\mathcal D}
\nc{\cE}{\mathcal E}
\nc{\cU}{\mathcal U}
\nc{\cH}{\mathcal H}
\nc{\cM}{\mathcal M}
\nc{\cL}{\mathcal L}
\nc{\cV}{\mathcal V}
\nc{\cF}{\mathcal F}
\nc{\fg}{\mathfrak g}

\nc{\fo}{\mathfrak o}

\nc{\CO}{\mathcal O}
\nc{\CR}{\mathcal R}
\nc{\Cl}{\mathcal {C}\ell}

\nc{\cW}{\mathcal{W}}
\nc{\bM}{\mathbf{M}}
\nc{\bL}{\mathbf{L}}
\nc{\bN}{\mathbf{N}}

\nc{\zq}{\mathpzc q}

\nc{\fl}{\mathfrak l}
\nc{\fn}{\mathfrak n}
\nc{\fm}{\mathfrak m}
\nc{\fp}{\mathfrak p}
\nc{\fh}{\mathfrak h}
\nc{\ft}{\mathfrak t}
\nc{\fk}{\mathfrak k}
\nc{\fb}{\mathfrak b}
\nc{\fs}{\mathfrak s}
\nc{\fB}{\mathfrak B}

\nc{\vareps}{\varepsilon}
\nc{\varesp}{\varepsilon}
\nc{\veps}{\varepsilon}

\nc{\fsl}{\mathfrak{sl}}
\nc{\fgl}{\mathfrak{gl}}
\nc{\fso}{\mathfrak{so}}

\nc{\fpq}{\mathfrak{pq}}
\nc{\fq}{\mathfrak q}
\nc{\fsq}{\mathfrak{sq}}
\nc{\fpsq}{\mathfrak{psq}}

\nc{\fhg}{\hat{\fg}}
\nc{\fhn}{\hat{\fn}}
\nc{\fhh}{\hat{\fh}}
\nc{\fhb}{\hat{\fb}}
\nc{\hrho}{\hat{\rho}}

\nc{\hsl}{\hat{\fsl}}
\nc{\fpo}{\mathfrak{po}}
\nc{\dirlim}{\underset{\rightarrow}{\lim}\,}
\nc{\nen}{\newenvironment}
\nc{\ol}{\overline}
\nc{\ul}{\underline}
\nc{\ra}{\rightarrow}
\nc{\lra}{\longrightarrow}
\nc{\Lra}{\Longrightarrow}
\nc{\bo}{\bar{1}}
\nc{\Lla}{\Longleftarrow}

\nc{\Llra}{\Longleftrightarrow}

\nc{\thla}{\twoheadleftarrow}

\nc{\lang}{(}
\nc{\rang}{)}

\nc{\hra}{\hookrightarrow}

\nc{\iso}{\overset{\sim}{\lra}}

\nc{\ssubset}{\underset{\not=}{\subset}}

\nc{\vac}{|0\rang}

\nc{\Thm}[1]{Theorem~\ref{#1}}
\nc{\Prop}[1]{Proposition~\ref{#1}}
\nc{\Lem}[1]{Lemma~\ref{#1}}
\nc{\Cor}[1]{Corollary~\ref{#1}}
\nc{\Conj}[1]{Conjecture~\ref{#1}}
\nc{\Claim}[1]{Claim~\ref{#1}}
\nc{\Defn}[1]{Definition~\ref{#1}}
\nc{\Exa}[1]{Example~\ref{#1}}
\nc{\Rem}[1]{Remark~\ref{#1}}
\nc{\Note}[1]{Note~\ref{#1}}
\nc{\Quest}[1]{Question~\ref{#1}}
\nc{\Hyp}[1]{Hypoth\`ese~\ref{#1}}
\nen{thm}[1]{\label{#1}{\bf Theorem.\ } \em}{}
\nen{prop}[1]{\label{#1}{\bf Proposition.\ } \em}{}
\nen{lem}[1]{\label{#1}{\bf Lemma.\ } \em}{}
\nen{cor}[1]{\label{#1}{\bf Corollary.\ } \em}{}
\nen{conj}[1]{\label{#1}{\bf Conjecture.\ } \em}{}

\nen{claim}[1]{\label{#1}{\bf Claim.\ } \em}{}

\nen{defn}[1]{\label{#1}{\bf Definition.\ } }{}
\nen{exa}[1]{\label{#1}{\bf Example.\ } }{}
\nen{rem}[1]{\label{#1}{\em Remark.\ } }{}
\nen{note}[1]{\label{#1}{\em Note.\ } }{}
\nen{exer}[1]{\label{#1}{\em Exercise.\ } }{}
\nen{sket}[1]{\label{#1}{\em Sketch of proof.\ } }{}
\nen{quest}[1]{\label{#1}{\bf Question.\ } \em}{}

\nen{hyp}[1]{\label{#1}{\bf Hypoth\`ese.\ } \em}{}
\begin{document}

\setcounter{section}{0}
\setcounter{tocdepth}{1}

\title[Characters of (relatively) integrable modules over Lie superalgebras]
{Characters of (relatively) integrable modules over  affine Lie superalgebras}

\author{Maria Gorelik~$^\dag$ }

\address[]{Dept. of Mathematics, The Weizmann Institute of Science,
Rehovot 76100, Israel}
\email{maria.gorelik@weizmann.ac.il}
\thanks{$^\dag$
Supported in part by BSF Grant No. 711623.}

\author{ Victor Kac~$^\ddag$}

\address{Dept. of Mathematics, 2-178, Massachusetts Institute of Technology,
Cambridge, MA 02139-4307, USA}
\email{kac@math.mit.edu}
\thanks{$^\ddag$ Supported in part by  Simons fellowship.}

\maketitle

\begin{abstract}
In the paper we consider the problem of computation of characters 
of relatively integrable irreducible highest weight modules $L$
over  finite-dimensional basic Lie superalgebras and
over affine Lie superalgebras $\fg$. The problems consists of two parts.
First, it is the reduction of the problem to the $\ol{\fg}$-module
$F(L)$, where $\ol{\fg}$ is the associated to $L$ integral
Lie superalgebra and $F(L)$ is an integrable irreducible highest
weight $\ol{\fg}$-module. Second, it is the computation
of characters of integrable highest weight modules. There is a general conjecture concerning
the first part, which we check in many cases. As for the second part, we prove
in many cases the KW-character formula, provided that the KW-condition
holds, including almost all finite-dimensional $\fg$-modules when $\fg$ is basic, 
and all maximally atypical non-critical integrable $\fg$-modules 
when $\fg$ is affine with 
non-zero dual Coxeter number.
\end{abstract}

\tableofcontents

\section{Introduction}
\label{intro}
Given a symmetrizable  Kac-Moody  algebra $\fg$ with a Cartan subalgebra
$\fh$ and an irreducible non-critical highest weight $\fg$-module 
$L=L(\lambda)$, one constructs the associated integral Kac-Moody 
algebra $\ol{\fg}^{\lambda}$ as follows.
Let $\Delta_{re}\subset\fh^*$ be the set of real roots of $\fg$
and let 
$$\Delta_{re}(\lambda)=\{\alpha\in\Delta_{re}|\ 
2(\lambda+\rho,\alpha)/(\alpha,\alpha)\in\mathbb{Z}\}$$
be the set of integral real roots. Then $\Delta_{re}(\lambda)$
is the set of real roots of a Kac-Moody  algebra $\ol{\fg}^{\lambda}$
with the same Cartan subalgebra $\fh$. An important result
of representation theory is the following relation between the
highest weight $\fg$-module $L$ and the (non-critical) 
highest weight $\ol{\fg}^{\lambda}$-module 
$\ol{L}=\ol{L}(\lambda+\rho-\ol{\rho})$:
\begin{equation}\label{1}
Re^{\rho}\ch L(\lambda)=\ol{R} e^{\ol{\rho}}\ch \ol{L}(\lambda+\rho-\ol{\rho}),
\end{equation}
where $R$ and $\ol{R}$ denote the Weyl denominators, and $\rho$ and $\ol{\rho}$
denote the Weyl vectors (see~\cite{F},\cite{KT1},\cite{KT2} and references there).

In the case when the $\ol{\fg}^{\lambda}$-module $\ol{L}(\lambda+\rho-\ol{\rho})$
is integrable, its character is given by the Weyl-Kac character 
formula~\cite{K1.5}, hence~(\ref{1}) gives an explicit formula 
for $\ch L(\lambda)$.

A $\fg$-module is called {\em relatively integrable} if the 
 $\ol{\fg}^{\lambda}$-module $\ol{L}(\lambda+\rho-\ol{\rho})$
is integrable; and it is called {\em admissible} if, in addition,
the $\mathbb{Q}$-span of the set of roots of  $\ol{\fg}^{\lambda}$
coincides with $\mathbb{Q}\Delta$. In particular, if $\lambda=\ol{\rho}-\rho$,
we obtain from~(\ref{1}) that the character is given by a product:
\begin{equation}\label{2}
\ch L(\ol{\rho}-\rho)=e^{\ol{\rho}-\rho}R^{-1}\ol{R}.
\end{equation}
For example, if $\fg$ is an affine Lie algebra with symmetric Cartan matrix,
then there exist admissible $\lambda$ of rational level $k$,
 provided that 
$k+h^{\vee}\geq h^{\vee}/u$, where $h^{\vee}$ is the dual Coxeter number
and $u$ is the denominator of $k$~\cite{KW2}. In this case the
character $\ch L(\lambda)$, suitably normalized, is a ratio of theta functions,
which is a modular function.

The main problems discussed in this paper are whether for 
a finite-dimensional 
basic Lie superalgebra or an associated (untwisted or twisted) affine Lie superalgebra $\fg$,
similar results hold. This is a class of Lie superalgebras, which is 
the closest to symmetrizable Kac-Moody  algebras. Of course, 
there are also Kac-Moody superalgebras, associated to symmetrizable generalized Cartan matrix, 
which have no real isotropic roots. For them the Weyl-Kac character formula
is proved in the same way as in Lie algebra case, and the relation~(\ref{1}) can be derived
using Enright functors~\cite{IK}. Therefore we exclude these superalgebras from consideration.

Given an irreducible highest weight $\fg$-module $L=L(\lambda)$,
we construct in~\S\ref{DeltaL} a natural generalization of the set
of integral real roots for the Lie superalgebra $\fg$ and the 
corresponding integral Lie superalgebra 
$\ol{\fg}^{\lambda}$ (which is also basic or affine, or a sum of such 
superalgebras), and we prove formula~(\ref{1}) in some cases. 
In particular, we prove 
formula~(\ref{2}), see~\Cor{cordimFL=1}.
We believe that~(\ref{1}) holds for arbitrary $\lambda$,
but there are not enough techniques to prove this, mainly due to the lack
of translation functors (used in~\cite{F} in the Lie algebra case),
and the lack of Enright functors, associated to isotropic simple roots
(\cite{KT1} and~\cite{KT2} use them in the Lie algebra case,
where all simple roots are non-isotropic). So far, in
full generality formula~(\ref{1}) is proved only 
for finite-dimensional $\fg$ of type $A(m,n)$, see~\cite{CMW}.

It would be natural to call an irreducible highest weight 
module $L$ over the Lie 
superalgebra $\fg$ integrable if it is integrable as a 
$\fg_{\ol{0}}$-module.
For finite-dimensional $\fg$, this definition is adequate, because
it is equivalent to $\dim L<\infty$. However, for affine $\fg$, such
non one-dimensional integrable irreducible highest weight modules exist
only if the Dynkin diagram of $\fg_{\ol{0}}$ is connected,
see~\cite{KW4}. For that reason, in the affine case, it is natural
to study {\em $\pi$-integrable} modules, where $\pi$ is a subset
of the set of simple roots $\Pi_0$ of $\fg_{\ol{0}}$, 
namely the $\fg$-modules $L$ for which all root spaces $\fg_{-\alpha}, \alpha\in\pi$,
act locally nilpotently. Roughly speaking, a $\fg$-module $L$ is called
integrable if it is $\pi$-integrable for the "largest possible" $\pi\subset \Pi_0$.
 The definition of ($\pi-$)relative integrability and admissibility of 
$L$ is the same as in the Lie algebra case.

For example, if $\Pi_0$ is of finite type or is connected, a $\fg$-module $L$
is called integrable if it is $\Pi_0$-integrable. Another example
is the non-twisted affine Lie superalgebra $\fg$, associated to a simple
finite-dimensional Lie superalgebra $\dot{\fg}$ with a non-degenerate Killing form
 $\kappa$ (then $\dot{\fg}$ is automatically basic); a 
 $\fg$-module $L$ is called  integrable
if it $\pi$-integrable for
 $$\pi=\{\alpha\in\Pi_0|\ \kappa(\alpha,\alpha)>0\}\cup \{\text{ simple roots
 for }\dot{\fg}_{\ol{0}}\};$$
 these modules are called principal integrable in~\cite{KW4}.
(This coincides with the definition of integrability of
modules over affine Lie algebras.) Note that the set of roots
$\alpha$ of $\fg$ with the property $\kappa(\alpha,\alpha)>0$
corresponds to the "largest" simple component of $\fg_{\ol{0}}$.
The choice of $\pi$ in the definition of integrability for an arbitrary (possibly twisted) affine Lie superalgebra is explained in~\S~\ref{defintegr}.

Let $L(\lambda)$ be either an integrable module over
a finite-dimensional basic Lie superalgebra $\fg$, or
an integrable module of non-critical level over an affine Lie superalgebra
$\fg$, and let $\Delta$ be the set of roots of $\fg$.
 Let $\Delta^{\perp}_{\lambda+\rho}$ be the set of roots of $\fg$,
orthogonal to $\lambda+\rho$, and choose a maximal linearly independent subset $S$ 
in $\Delta^{\perp}_{\lambda+\rho}$, which spans an isotropic subspace.
Assume that $S$ satisfies the KW-condition, namely $S$ can be included in a set of simple roots of $\Delta$.
A natural analogue of the Weyl-Kac character formula for integrable
highest weight modules over Kac-Moody algebras is the following KW-formula,
proposed in~\cite{KW3}, Section 3, for basic $\fg$, and in~\cite{KW4}, Section 9, for affine $\fg$:

\begin{equation}\label{3}
j_{\lambda} Re^{\rho}\ch L(\lambda)=\sum_{w\in W'} sgn(w) w\bigl(\frac{e^{\lambda+\rho}}
{\prod_{\beta\in S}(1+e^{-\beta})}\bigr)
\end{equation}
for some positive integer $j_{\lambda}$, where $W'$ is a certain subgroup
of the Weyl group $W$.

In Section~\ref{sectKW} of the present paper we prove the KW-formula with $j_{\lambda}=1$ and $W'=W(\pi)$, 
in the case when $\mathbb{Q}S$ is a maximal isotropic
subspace in $\mathbb{Q}\Delta$, provided that either $\fg$ is basic,
or $\fg$ is affine with $h^{\vee}\not=0$ or is equal to 
$A(n,n)^{(1)}$ (in the case
$\lambda=0$ this was proved in~\cite{Gfin}, \cite{Gaff}, \cite{R}).
Recall that $h^{\vee}$ is  the half of the eigenvalue of the Casimir 
operator on basic $\fg$ (which is $0$ iff $\kappa=0$), and it is called 
the {\em dual Coxeter number} of any twisted affine superalgebra,
associated to $\fg$.

Incidentally, using a different character formula for level $1$
$\mathfrak{osp}(M,N)^{(1)}$-modules, obtained in~\cite{KW4},
we thereby derive in~\S~\ref{newid} an interesting identity for mock theta functions.

In Section~\ref{sectfindim} we prove the KW-formula for all irreducible finite-dimensional modules (satisfying the KW-condition) 
over a basic Lie superalgebra $\fg$, except for a few cases when $\fg$ is of type 
$D(m,n)$ \footnote{After this paper has been completed, we learned about the paper of
S.-J. Cheng and J.-H. Kwong "Kac-Wakimoto character formula for orthosymplectic Lie superalgebra", where  the KW formula is established by different methods for all
irreducible finite-dimensional $\mathfrak{osp}(m,n)$-modules, satisfying the KW-condition,
including the cases we were unable to settle.}. This formula has been previously verified only for 
$\fg$ of type $A(m,n)$,
see~\cite{CHR}, using the earlier work~\cite{S0},\cite{S05},
\cite{B},\cite{SZ} on computation of finite-dimensional characters of $\fgl(m,n)$. 
There have been a number of earlier papers, 
where the KW-formula was verified in the case $\# S=1$, see~\cite{BL},\cite{VJ1},\cite{VJ2},\cite{VJHKT},
\cite{KW3}.

In Section~\ref{sectKW0} we prove a similar to~(\ref{3}) 
result for vacuum modules over affine superalgebras, 
for which $h^{\vee}=0$, 
see~(\ref{vacform}).

In Section~\ref{sect10} we study another extremal case, - when $S$ is empty.
Such $\fg$-modules $L=L(\lambda)$ are called typical and it was proven in~\cite{K1.5}
that the usual Weyl character formula holds for them
if $\dim\fg<\infty$ and $\dim L(\lambda)<\infty$.
We prove that formula~(\ref{3}) with $S=\emptyset$ holds if we let $W'$ be 
the "integral" subgroup $W(L)$ of $W$, provided that $L$ is relatively 
integrable. In other words, we prove that in this case both conjectural
formulas~(\ref{1}) and~(\ref{3}) hold.
We also verify~(\ref{1}) in a few other instances of typical  and of 
relatively integrable modules. As a corollary, we obtain the character formula
for all relatively integrable modules over $A(0,n)^{(1)}$ and $C(n)^{(1)}$.

Note, however, that while we expect that~(\ref{1})
always holds, and that~(\ref{3}) holds for all irreducible finite-dimensional 
modules (satisfying the KW-condition) over  basic  $\fg$ ( cf. Section~\ref{sectfindim}), we do not expect~(\ref{3}) to hold
in full generality, except when $\mathbb{Q}S$ is  a maximal isotropic subspace of $\mathbb{Q}\Delta$.

We also prove formulas~(\ref{1}) and~(\ref{3}) for all admissible $\fg$-modules when 
$\ol{\fg}^{\lambda}$ is of small rank, see Sections~\ref{sect8},
\ref{sect10} and~\ref{sectexa}. In particular, we obtain the character formula
for all relatively integrable  $B(1,1)^{(1)}$-modules, and also for 
all admissible $A(1,1)^{(1)}$-modules, associated to integrable
vacuum $A(1,1)^{(1)}$-modules.

Our proofs use the ideas from \cite{KT1}, \cite{KT2}, and
\cite{Gfin}, \cite{Gaff}.

In the present paper we prove all character formulas, used 
in~\cite{KW5},\cite{KW6}  to show that for 
the affine Lie superalgebras $\fg=A(1,0)^{(1)}$, 
$A(1,1)^{(1)}$, and $B(1,1)^{(1)}$, the numerators of the characters
of the atypical admissible highest  weight $\fg$-modules $L$ are mock theta functions,
in the sense that their non-holomorphic
modifications in the spirit of Zwegers~\cite{Z} 
span an $SL(2,\mathbb{Z})$-invariant space 
(when $\dim \ol{L}=1$, they are ``honest'' theta functions).

The results of this paper were reported at the conferences 
in Uppsala in September 2012, in
Rome in December 2012, in Taipei in May 2013, and in Rio de Janeiro in June 2013.

{\em Acknowledgment.}
We are grateful to M.~Wakimoto and V.~Serganova for helpful discussions.
We would like to thank   S.-J.~Cheng and Sh.~Reif for the correspondence.
 A part of this work was done during  our stay at IHES.
We grateful to this institution for stimulating atmosphere and excellent
working conditions.

\section{Preliminaries}
Throughout the paper the base field is $\mathbb{C}$ and 
$\fg$ is either a  basic Lie superalgebra with
a non-degenerate invariant bilinear form $(-,-)$, or the associated
to it and its finite order automorphism, preserving $(-,-)$, symmetrizable
affine Lie superalgebra. Recall that a  basic Lie superalgebra $\fg$ is either
a simple finite-dimensional Lie algebra
or one of the simple finite-dimensional Lie superalgebras $\fsl(m,n) (m\not=n), 
\mathfrak{psl}(n,n) (n\geq 2),
\mathfrak{osp}(m,n), D(2,1,a), F(4), G(3)$ or $\fgl(m,n)$ \cite{K1},
and that the associated affine Lie superalgebras are constructed 
in the same way as in~\cite{K2}. Recall that the Killing form $\kappa$ of $\fg$
is non-degenerate iff $\fg=\mathfrak{sl}(m,n) (m\not=n), \mathfrak{osp}(m,n)$
($m$ is odd, or $m$ is even and $n\not=m-2\geq 2$), $F(4), G(3)$;
this is equivalent to the property that the dual Coxeter number ($=\frac{1}{2}$
eigenvalue of the Casimir operator on $\fg$) is non-zero. 
Recall that the dual Coxeter number associated to the Killing form is always a non-negative rational 
number, see~\cite{KW3}.
This number is also called the dual Coxeter number of the associated affine superalegbra.

It is  well known that for any affine Lie superalgebra the dual Coxeter number is
equal to $(\rho,\delta)$, where $\rho$ is the Weyl vector and $\delta$ is the primitive
imaginary root.

The invariant bilinear form
extends from the basic Lie superalgebra to the associated 
affine Lie superalgebra
and is denoted again by $(-,-)$.

Recall that one often uses the following notations: $A(m,n)=\fsl(m+1,n+1)$
or $\fgl(m+1,n+1)$ for $m\not=n$, $A(n,n)=\mathfrak{psl}(n,n)$ or
$\mathfrak{gl}(n,n)$, $B(m,n)=\mathfrak{osp}(2m+1,2n)$, 
$C(n)=\mathfrak{osp}(2,2n)$,$D(m,n)=\mathfrak{osp}(2m,2n) (m>1)$.
The associated affine Lie superalgebra, twisted by an automorphism
of $\fg$ of order $r$, is  denoted by  $\fg^{(r)}$.
We will often write $\Delta=A(m,n)$ to indicate that $\Delta$ is 
the root system of  $A(m,n)$, or 
$\Pi=A(m,n)$ to indicate that $\Pi$ is a subset of simple roots for
the root system of type $A(m,n)$.

Recall that we get all affine Lie superalgebras by picking an automorphism
in each connected component of the group of  automorphisms of $\fg$.
(The affine Lie superalgebra depends only on this connected component;
however, unlike in the Lie algebra case, some of the affine Lie superalgebras
corresponding to different connected components may be isomorphic.)

Let $\fh$ be a Cartan subalgebra of $\fg$.
As in the Lie algebra case, $\fg$ has the root space decomposition
with respect to $\fh$. Let $\Delta\subset\fh^*$ be the set of roots.
Denote by $\Delta_{\ol{0}}$ and $\Delta_{\ol{1}}$ the 
subsets of even and odd roots.  The restriction of $(-,-)$ to
$\fh$ is non-degenerate, hence it induces a bilinear form on $\fh^*$.
One can show that $\Delta_{\ol{0}}$ is a  union of 
a finite number of root systems  of affine Lie algebras
with the same primitive imaginary root $\delta$.

The Weyl group $W$ of $\Delta$ is the subgroup of $GL(\fh^*)$,
generated by reflections $r_{\alpha}$ in non-isotropic roots
$\alpha$, where 
$r_{\alpha}\lambda=\lambda-2(\lambda,\alpha)\alpha/(\alpha,\alpha)$.
One knows that $W$ coincides with the Weyl group of $\Delta_{\ol{0}}$
and that $W\Delta=\Delta$ (cf.~\cite{K2}, Chapter 3).

Let $Q:=\mathbb{Z}\Delta$ and  $Q_{\ol{0}}:=\mathbb{Z}\Delta_{\ol{0}}$
be the corresponding root lattices.

\subsection{Subsets of positive roots in $\Delta$}
Given a real valued additive function $\chi$ on $Q$, which
is positive on  $\delta$ and does not vanish on elements of $\Delta$,
we have the corresponding subsets of positive roots $\Delta^+$
and $\Delta_{\ol{0}}^+$ (on which $\chi$ is positive).

For different choices of $\chi$ the subsets of even 
positive roots may be different,
but they can be transformed to each other by the Weyl group.
Throughout the paper we will fix one of them, $\Delta_{\ol{0}}^+$,
and consider only the subsets of positive roots $\Delta^+$ in $\Delta$,
which contain $\Delta_{\ol{0}}^+$. This choice fixes a triangular decomposition 
of $\fg$, compatible with the triangular decomposition of $\fg_{\ol{0}}$, 
corresponding to $\Delta_{\ol{0}}^+$.

Recall that, given a subset of positive roots $\Delta^+$ 
(containing $\Delta_{\ol{0}}^+$)
and an odd simple root $\beta\in\Delta^+$ with $(\beta,\beta)=0$, we can construct a new subset of positive roots 
(containing $\Delta_{\ol{0}}^+$) by an {\em odd reflection} $r_{\beta}$:
\begin{equation}\label{Drbeta}
r_{\beta}(\Delta^+)=(\Delta^+\setminus\{\beta\})\cup\{-\beta\}.
\end{equation}

\subsubsection{Proposition~\cite{S}}\label{propS}

(a) Any two subsets of positive roots in $\Delta$ (containing $\Delta_{\ol{0}}^+$)
can be obtained from each other by a finite sequence of odd reflections.

(b) For a simple root $\alpha\in\Delta_{\ol{0}}^+$ there exists
a subset of positive roots, 
for which $\alpha$ or $\frac{\alpha}{2}$ is a simple root.

\subsubsection{}\label{choicerho}
Let $\Delta^+$ be a subset of positive roots in $\Delta$,
and denote by $\Pi$ the subset of its simple roots; we shall often write
$\Delta^+=\Delta^+(\Pi)$. One has $r_{\beta}\Delta^+=\Delta^+(\Pi')$, where
$$\Pi':=\{\alpha\in\Pi|\ \alpha\not=\beta, (\alpha,\beta)=0\}\cup\{\alpha+\beta|
\alpha\in\Pi, (\alpha,\beta)\not=0\}\cup\{-\beta\}.$$

Then we can choose a Weyl vector $\rho_{\Pi}\in\fh^*$,
such that the following two properties hold for each subset $\Pi$ of simple roots:

(i) $2(\rho_{\Pi},\alpha)=(\alpha,\alpha),\ \text{ if } \alpha\in\Pi$;

(ii) $\rho_{r_{\beta}\Pi}=\rho_{\Pi}+\beta,\ \text{ if }\beta\in\Pi, (\beta,\beta)=0$.

Indeed, choose a set of positive roots and  let $\Pi$ be its subset of simple roots;
 pick an arbitrary Weyl vector $\rho_{\Pi}$, satisfying (i). Define (cf.~\Prop{propS} (a)):
 
 $$\rho_{r_{\beta_1}\ldots r_{\beta_s}\Pi}:=\rho_{\Pi}+\beta_1+\ldots+\beta_s.$$
 
 Then (ii) obviously holds and (i) is straightforward to check. Finally, this
 is well defined since the equality
 
\begin{equation}\label{rst} 
 r_{\beta_1}\ldots r_{\beta_s}\Pi=r_{\gamma_1}\ldots r_{\gamma_t}\Pi
 \end{equation}
 forces the equality $\beta_1+\ldots+\beta_s=\gamma_1+\ldots\gamma_t$.
Indeed, let us show by induction that
$$\Delta^+(r_{\beta_1}\ldots r_{\beta_s}\Pi)=(\Delta^+(\Pi)\setminus S)\cup (-S),$$
where $S$ is obtained from a multiset $\{\beta_1,\ldots,\beta_s\}$
by removing all pairs of  opposite roots (i.e., the pairs of the form $(\beta,-\beta)$).
Moreover, if $r_{\beta_1}\ldots r_{\beta_s}\Pi$
is defined, then $S$ is a set (each element appears once) and $S\subset \Delta^+(\Pi)$.
This can be proven by induction on $s$.
For $s=1$ this follows from~(\ref{Drbeta}); if $r_{\beta_{j+1}}\ldots r_{\beta_1}\Pi$ is defined,
then $\beta_{j+1}$ lies in $r_{\beta_{j}}\ldots r_{\beta_1}\Pi$ and, in particular,
in $\Delta^+(r_{\beta_{j}}\ldots r_{\beta_1}\Pi)=(\Delta^+(\Pi)\setminus S)\cup (-S)$
by the induction hypothesis. This means that $\beta_{j+1}\not\in S$ and
$$\Delta^+(r_{\beta_{j+1}}\ldots r_{\beta_1}\Pi)=
(((\Delta^+(\Pi)\setminus S)\cup (-S))\setminus\{\beta_{j+1}\})\cup\{-\beta_{j+1}\}.$$
If $-\beta_{j+1}\not\in S$, then
$$\Delta^+(r_{\beta_{j+1}}\ldots r_{\beta_1}\Pi)=
(\Delta^+(\Pi)\setminus (S\cup\{\beta_{j+1}\}))\cup (-(S\cup\{\beta_{j+1}\})),$$
and, if $-\beta_{j+1}\in S$, then
$$\Delta^+(r_{\beta_{j+1}}\ldots r_{\beta_1}\Pi)=
(\Delta^+(\Pi)\setminus (S\setminus\{-\beta_{j+1}\}))\cup (-(S\setminus\{-\beta_{j+1}\})),$$
as required.

Now~(\ref{rst}) implies
$$(\Delta^+(\Pi)\setminus S)\cup\{-S\}=(\Delta^+(\Pi)\setminus T)\cup\{-T\},$$
where $S$ (resp., $T$)  is the set obtained from the set $\{\beta_1,\ldots,\beta_s\}$
(resp., $\{\gamma_1,\ldots \gamma_t\}$) by  removing  all pairs of opposite roots.
This gives $S=T$ so $\beta_1+\ldots+\beta_s=\gamma_1+\ldots+\gamma_t$, as required.

\subsubsection{}\label{ordering}
Let $Q^+:=\mathbb{Z}_{\geq 0}\Delta^+ (=\mathbb{Z}_{\geq 0}\Pi)$
and $Q^+_{\ol{0}}:=\mathbb{Z}_{\geq 0}\Delta^+_{\ol{0}}$,
where $\mathbb{Z}_{\geq 0} S$ denote the semigroup 
of linear combinations of elements from $S$ with coefficients from $\mathbb{Z}_{\geq 0}$.
The set $Q^+$ depends on the set $\Delta^+$ of positive roots,
and to emphasize this dependence we shall write $Q^+=Q^+(\Pi)$,
but the set $Q^+_{\ol{0}}$ does not (since we fixed $\Delta^+_{\ol{0}}$).

We consider the corresponding partial orderings on $\fh^*$. The first one is
$\nu\geq \mu$ if $\nu-\mu\in Q^+_{\ol{0}}$, and the second one is 
$\nu\geq_{\Pi} \mu$ if $\nu-\mu\in \mathbb{Z}_{\geq 0}\Pi$ (it depends on $\Pi$).
Given $S\subset\fh^*$, an element $\lambda\in S$ is called {\em maximal}
(resp., $\Pi$-maximal) if $\lambda\geq \nu$ (resp., $\lambda\geq_{\Pi}\nu$)
for all $\nu\in S$.

Note that the partial ordering $\geq_{\Pi}$ can be extended to a total ordering $\geq_{\Pi,tot}$ as follows.
Fix a total ordering on the basis $\Pi\cup\{\Lambda_0\}=\{\gamma_1,\ldots,\gamma_m\}$
of $\fh^*$ and the lexicographic order on $\mathbb{C}$. Then $\nu=\sum_i a_i\gamma_i\geq_{\Pi,tot} \mu=\sum_i b_i\gamma_i$
if $(a_1,\ldots,a_m)\geq (b_1,\ldots, b_m)$ in the lexicographic order.

\subsection{The algebra $\cR$}\label{cRPi}
We introduce the algebra $\cR=\cR(\Pi)$ as in~\cite{Gfin},\cite{Gaff}.
The new part is~\S~\ref{equivcV} and~\S~\ref{poles}--\ref{sectXX'}.

Let ${\cV}$ be the vector space over $\mathbb{Q}$ of all formal
sums (possibly infinite) $Y=\sum_{\nu\in\fh^*} b_{\nu} e^{\nu}, b_{\nu}\in\mathbb{Q}$,
and define the support of $Y$ by
$$\supp Y:=\{\nu|\ b_{\nu}\not=0\}.$$

Let $\cR(\Pi)$ be the subspace of $\cV$, consisting of finite linear
combinations of the elements of the form
$\sum_{\nu\in \mathbb{Z}_{\geq 0}\Pi} b_{\nu} e^{\lambda-\nu}$, where $\lambda\in\fh^*$. The space 
$\cR(\Pi)$ has an obvious  structure of a unital commutative algebra, induced by $e^{\mu}e^{\nu}=e^{\mu+\nu}, e^0=1$.
Moreover, $\cR(\Pi)$ is a domain. This is clear since for any $Y\in\cR(\Pi)$,
its support $\supp Y$ has a unique maximal element in the total ordering $\geq_{\Pi,tot}$
and  the maximal element in $\supp YY'$ is  equal to the sum of maximal elements in
$\supp Y$ and in $\supp Y'$.

For each $\Pi$ define a topology on $\cV$ by the set of open neighbourhoods 
$\cV_{\lambda}$ , consisting of $Y\in\cV$ such that $\supp Y\leq_{\Pi,tot}\lambda$.
This makes $\cR(\Pi)$ a topological algebra. 
We idenfity the convergent infinite sums of
elements of $\cR(\Pi)$ with their limits.

\subsubsection{}\label{equivcV}
Let 
$$\cV_{fin}:=\{Y\in \cV| \supp Y\ \text{ is finite}\}.$$
This is a subalgebra of all algebras $\cR(\Pi)$. Hence
\begin{equation}\label{eq2}
\cV_{fin}\cR(\Pi)\subset \cR(\Pi).
\end{equation}

Note also that $\cV$ is a $\cV_{fin}$-module (but not an algebra).
Introduce the equivalence relation $\sim $ on $\cV$
by: $X\sim X'$ if there exists $Y\in\cV_{fin}$ such that
$XY=X'Y$. Note that if $X\in\cR(\Pi), X'\in\cR(\Pi')$ and $Y\in\cV_{fin}$, then
$XY=X'Y\in \cR(\Pi)\cap\cR(\Pi')$ by~(\ref{eq2}).
Since $\cR(\Pi)$ is a domain, the equivalence of its two elements $X,X'\in\cR(\Pi)$
implies $X=X'$.

\subsubsection{Action of the Weyl group}\label{RW}

The Weyl group $W$ acts on $\cV$ in the obvious way:
$$w(\sum_{\nu} b_{\nu} e^{\nu}):=\sum_{\nu} b_{\nu} e^{w\nu}.$$
Obviously, $\cV_{fin}$ is $W$-invariant, but $\cR=\cR(\Pi)$ is not.
For a subgroup $W'\subset W$ introduce the following subalgebra of the algebra $\cR$
$$\cR_{W'}:=\{Y\in\cR|\ wY\in \cR \text{ for
each }w\in W'\}.$$

\subsubsection{Infinite products}\label{infprod}
A  product of the form 
\begin{equation}\label{eq3}
Y=\prod_{\alpha\in A}
(1+a_{\alpha}e^{-\alpha})^{d_{\alpha}},
\end{equation}
 where $A\subset \Delta$ is such that
the set $A\setminus\Delta^+(\Pi)$ is finite, and 
 $a_{\alpha}\in \mathbb{Q},\
\ d_{\alpha}\in\mathbb{Z}_{\geq 0}$, can be naturally viewed
as an element of $\cR$. Since $\Delta^+(\Pi)\setminus\Delta^+(\Pi')$ is a finite
set (by~\Prop{propS} (a)), the element $Y$ lies in all algebras $\cR(\Pi')$.
Hence the set $\cY$ of all such products is a multiplicative subset 
of each of the algebras $\cR(\Pi)$.

For any $w\in W$ the infinite product
$$wY:=\prod_{\alpha\in A}(1+a_{\alpha}e^{-w\alpha})^{d_{\alpha}},$$
is again an infinite product of the above form, since
 the set $w\Delta_+\setminus \Delta_+=-(w\Delta_-\cap\Delta_+)$
is finite. Hence $\cY$ is a $W$-invariant multiplicative subset 
of $\cR_W$ (for each $\Pi$).

Consider the localization $\cR_{W'}[\cY^{-1}]$ of the algebra $\cR_{W'}$
by the  multiplicative subset $\cY$. Let $\varphi_{\Pi}: \cR_{W'}[\cY^{-1}]\to\cR$
be an algebra homomorphism, defined by expanding in a geometric progression
for $\beta\in\Delta^+, a\in\mathbb{Q}\setminus\{0\}$:

$$\varphi_{\Pi}\bigl(\frac{e^{\lambda}}{1+ae^{-\beta}}\bigr)=e^{\lambda}(1-ae^{-\beta}+a^2e^{-2\beta}-\ldots);\ \ \ 
\varphi_{\Pi}\bigl(\frac{e^{\lambda}}{1+ae^{\beta}}\bigr)=\frac{1}{a}\varphi_{\Pi}\bigl(\frac{e^{\lambda-\beta}}{1+a^{-1}e^{-\beta}}\bigr).
$$
This homomorphism defines an embedding of $\cR_{W'}[\cY^{-1}]$ in $\cR$.

\subsubsection{}
We extend the action of $W'$
from $\cR_{W'}$ to $\cR_{W'}[\cY^{-1}]$ by setting $w(Y^{-1}X):=(wY)^{-1}(wX)$
for each $X\in\cR_{W'}, Y\in\cY$. Let $Y$ be as in~(\ref{eq3}). Then

\begin{equation}\label{expanwY}
\supp Y\subset\lambda'-Q^+,\ \text{ where }
\lambda':=-\sum_{\alpha\in A\setminus\Delta_+: a_{\alpha}\not=0}
d_{\alpha}\alpha.
\end{equation}

\subsubsection{}\label{compex}
Let $W'$ be a subgroup of $W$.
For $Y\in\cR[\cY^{-1}]$ we say that {\em $Y$ is $W'$-invariant
(resp., $W'$-skew-invariant)} if $wY=Y$
(resp., $wY=\sgn(w)Y$) for each $w\in W'$.

Note that $Y:=\sum_{\mu} a_{\mu}e^{\mu}\in\cR$ is a $W'$-skew-invariant  element of $\cR_{W'}$ if and only if
$a_{w\mu}=(-1)^{sgn(w)} a_{\mu}$. In particular, if $Y$ is a $W'$-skew-invariant  element of $\cR_{W'}$, then
$W'\supp (Y)=\supp(Y)$.

We will use the following fact: if $Y\in\cR_{W'}$ is  $W'$-skew-invariant
and $\rho+\supp Y$ consists of  non-critical weights, then $\supp Y$
is the union of regular $W'$-orbits, where regularity means that the 
elements of this orbit
have trivial stabilizers. This is an immediate corollary of the fact that for
an affine Lie algebra the $W$-orbit of each  weight of non-zero level contains
either maximal or minimal element and 
the stabilizer of this element in $W$ is generated by simple reflections; as a result the stabilizer of any  weight of non-zero level is generated by reflections. Since for such a reflection $r_{\alpha}$ we have $r_{\alpha}Y=-Y$,
we have $r_{\alpha}\lambda=\lambda\ \Longrightarrow\ \lambda\not\in\supp Y$.

Let $Y:=\sum_{\mu} a_{\mu}e^{\mu}$ be any element of $\cR_{W'}$. We claim that if 
$\sum_{w\in W'} sgn (w)\, w(Y)\in\cR$, then $\sum_{w\in W'} sgn (w)\, w(Y)$ is 
a $W'$-skew-invariant  element of $\cR_{W'}$. Indeed, $\sum_{w\in W'} sgn (w)\, w(Y)=\sum_{\nu} b_{\nu}e^{\nu}$,
where $b_{\nu}=\sum_{w\in W'}sgn(w) a_{w\nu}$, so $b_{w\nu}=sgn(w) b_{\nu}$, as required.

\subsubsection{}\label{RPi}
For each set of simple roots $\Pi'$ introduce  the following
infinite products 
$$R_{\ol{0}}:=\prod_{\alpha\in\Delta_{\ol{0}}^+}
(1-e^{-\alpha}),\ \ \ \
R(\Pi')_{\ol{1}}:=\prod_{\alpha\in\Delta_+(\Pi')\cap\Delta_{\ol{1}}}
(1+e^{-\alpha}).$$
One readily sees (by~\Prop{propS} (i))  
that $R_{\ol{0}}, R(\Pi')_{\ol{1}}\in\cY$.
 We view $R_{\ol{0}}, R(\Pi')_{\ol{1}}$ and
$$R(\Pi'):=\frac{R_{\ol{0}}}{R(\Pi')_{\ol{1}}}$$
as elements in $\cR(\Pi)$, as in~\S~\ref{infprod}. One readily sees
that $R(\Pi')e^{\rho_{\Pi'}}\in\cR(\Pi)$ does not depend on $\Pi'$,
so we write simply $Re^{\rho}$ (keeping in mind that this is an element of $\cR(\Pi)$
for particular $\Pi$). By~\S~\ref{infprod}, all these elements are equivalent 
(for different $\Pi$).  Since $R_{\ol{0}}, R(\Pi')_{\ol{1}}\in\cY\subset \cR_W$, 
the element $Re^{\rho}$ lies in $\cR_W[\cY^{-1}]$.
Clearly, $r_{\alpha}(R(\Pi')e^{\rho_{\Pi'}})=-R(\Pi')e^{\rho_{\Pi'}}$ for a non-isotropic root
$\alpha\in\Pi'$. 
From~\Prop{propS} (ii), we conclude that $Re^{\rho}$ is a $W$-skew-invariant element of $\cR_W[\cY^{-1}]$.

If $\Pi$ is fixed, we denote by 
$R_{\ol{1}}:=R(\Pi)_{\ol{1}},\ R:=R(\Pi)$
the corresponding elements in $\cR$.

If $\fg$ is finite-dimensional, or one of affine Lie superalgebras 
$A(0,n)^{(1)}, B(0,n)^{(1)},C(n)^{(1)}$, or $A(0,2n-1)^{(2)},
C(n+1)^{(2)}, A(0,2n)^{(4)}$,
then  we can introduce a Weyl vector $\rho_{\ol{0}}$ satisfying
$(\rho_{\ol{0}},\alpha^{\vee})=1$ for each $\alpha\in\Pi_0$. Then $R_{\ol{0}}e^{\rho_{\ol{0}}}$
is a  $W$-skew-invariant element of $\cR_W$, so $R_{\ol{1}}e^{\rho_{\ol{0}}-\rho}$
is a  $W$-invariant element of $\cR_W$. 

If $\fg$ is an affine Lie superalgebra and $\Delta_{\ol{0}}$ is not connected,
then the Weyl vector $\rho_{\ol{0}}$ does not exists. 
However, for each connected component $\pi\subsetneq \Pi_0$ there exists
a Weyl vector $\rho_{\pi}$ satisfying
$(\rho_{\pi},\alpha^{\vee})=1$ for each $\alpha\in\pi$; note that
$R_{\ol{0}}e^{\rho_{\pi}}$ is a  $W(\pi)$-skew-invariant element of $\cR_{W(\pi)}$
and  $R_{\ol{1}}e^{\rho_{\ol{0}}-\rho}$
is a  $W(\pi)$-invariant element of $\cR_{W(\pi)}$.

\subsubsection{Poles}
\label{poles}
For an odd isotropic root $\alpha\in\Pi$ we say that $X\in \cR(\Pi)$ has a {\em pole of order }$k$ at $\alpha$
if $k$ is minimal such that
$$(1+e^{-\alpha})^k X\in\cR(r_{\alpha}\Pi).$$

For example, $R$ has a pole of order $1$ at each odd isotropic root $\alpha\in\Pi$.  
Another important example appears in the next lemma.

\subsubsection{}\label{expA}
Consider $Y\in\cY$ (see~\S~\ref{infprod}) of the form  
$Y:=\prod_{\beta\in J} (1+e^{-\beta})$, where $J\subset\Delta$ is a finite set and,
for each $\beta\in\Delta_{\ol{1}}$, $J\cap\{\pm\beta\}$ contains at most one element.
Recall conventions of~\S~\ref{infprod} and view $Y^{-1}$ as an element in 
$\cR(\Pi)$,
which we denote by $Y^{-1}(\Pi)$.

Let $W'\subset W$ be a subgroup generated by simple reflections and
let $\lambda\in \fh^*\setminus\{0\}$ be such that 
 the orbit $W'\lambda$ has a unique maximal element. For each subset $W''\subset W'$
we introduce the following notation
$$\cF_{W''}\bigl(\frac{e^{\lambda}}{\prod_{\beta\in J} (1+e^{-\beta})}\bigr):=
\sum_{w\in W''} \sgn (w)\, e^{w\lambda}\prod_{\beta\in J} (1+e^{-w\beta})^{-1}.$$
From the lemma below it follows that $\cF_{W''}\bigl(\frac{e^{\lambda}}{\prod_{\beta\in J} (1+e^{-\beta})}\bigr)$
lies in $\cR(\Pi)$ (i.e., the corresponding partial sums converge in $\cR(\Pi)$,
cf.~\S~\ref{cRPi})
 and that these elements are equivalent for different choices of $\Pi$.

\begin{lem}{lemexpA}
Let $W'\subset W$ be a subgroup generated by simple reflections;
for each $w\in W'$ fix $x_w\in\mathbb{Q}$.  Write
$W'=W'_f\times W'_{aff}$, where
$W'_f$ is finite and $W'_{aff}$ is the product of affine  Weyl groups.
Let $\lambda\in \fh^*\setminus\{0\}$ be such that 
$(\lambda,\alpha^{\vee})\in\mathbb{Z}$
for each $\alpha\in\Pi_0$ such that $r_{\alpha}\in W'$, and that
$(\lambda,\delta)/(\alpha,\alpha)\geq 0$ for each 
$\alpha\in\Pi_0$ such that $r_{\alpha}\in W'_{aff}$.

(a) For each $\Pi$ the element 
$$X(\Pi):=\sum_{w\in W'} x_w e^{w\lambda}(wY)^{-1}(\Pi)$$
lies in $\cR(\Pi)$.

(b) All elements $X(\Pi)$ are equivalent (with respect to the relation introduced in~\S~\ref{equivcV}).

(c) For each odd isotropic root $\alpha\in\Pi$ the element $X(\Pi)$ has a pole of order
at most one at $\alpha$. Moreover, $X(\Pi)$ has a pole of order zero at $\alpha$ if
$ W'(J)\cap\{\pm\alpha\}=\emptyset$ .
\end{lem}
\begin{proof}
The assumptions on $\lambda$ imply that the orbit $W'\lambda$ contains a 
unique maximal element. We may (and will)  assume
that $\lambda$ is maximal in its orbit
that is $\lambda-w\lambda\in \mathbb{Z}_{\geq 0}\Pi_0$ for any $w\in W'$.

For a fixed set of simple roots $\Pi$, we denote by $\htt_{\Pi}\mu$,
the height of $\mu=\sum_{\alpha\in\Pi} k_{\alpha}\alpha$, the number $\htt_{\Pi}\mu=\sum_{\alpha\in\Pi} k_{\alpha}$.

Note that $X(\Pi)$ is an infinite sum of elements in $\cR(\Pi)$; 
for (a) we have to show that
the partial sums converge (cf.~\S~\ref{cRPi}). From~(\ref{expanwY}) we obtain
$$\supp \bigl(e^{w\lambda}(wY)^{-1}(\Pi)\bigr)
\subset w\lambda-\mathbb{Z}_{\geq 0}\Pi.$$
In order to prove that $X(\Pi)\in \cR(\Pi)$, it is enough to verify that
for each $r$ the set 
$$H_r(\lambda):=\{w\in W'| \htt(\lambda-w\lambda)\leq r\}$$ 
is finite. 

Recall that (see e.g.~\cite{K2}, Chapter 3) for an affine Lie algebra 
the stabilizer of any element $\nu$ which
is maximal  in its Weyl group orbit is generated by simple reflections; thus 
this stabilizer is either finite or coincides with $W$ itself 
(in this case $\nu=0$). Hence $\Stab_{W'}\lambda$ is finite.

Let $\alpha_1,\ldots,\alpha_r$ be the simple reflections ($\alpha_i\in\Pi_0$) which generate $W'$. 
Since $\lambda$ is  maximal in its
$W'$-orbit, the value $(\lambda,\alpha_i^{\vee})$ is a non-negative integer. An easy argument (see, for instance, Lemma 1.3.2 
in~\cite{Gaff}) shows that for each reduced expression $w=r_{\alpha_{i_1}}\ldots r_{\alpha_{i_r}}$
$$\htt (\lambda-w\lambda)\geq \#\{j|\  (\lambda,\alpha_{i_j}^{\vee})\not=0.\}$$
Now the fact that $H_r$ is finite follows as in Lemma 2.4.1 (i) in~\cite{Gaff}.

(b) Since any two subsets of positive roots  are connected by a finite
chain of odd reflections, it is enough to verify that 
\begin{equation}\label{equiv1}
X(\Pi)(1+e^{\gamma})=X(r_{\gamma}\Pi)(1+e^{\gamma})\in\cR(\Pi)\cap \cR(r_{\gamma}\Pi)
\end{equation}
Indeed,  by the assumption on $J$, the intersection $wJ\cap\{\pm\gamma\}$
contains at most one element. If the intersection is empty, then
$e^{w\lambda}(wY)^{-1}(\Pi)=e^{w\lambda}(wY)^{-1}(\Pi')\in\cR(\Pi)\cap\cR(\Pi')$, see~\S~\ref{infprod}.
If  the intersection $wJ\cap\{\pm\gamma\}$ is non-empty, then
$e^{w\lambda}(wY)^{-1}(\Pi)(1+e^{\gamma})=e^{\lambda'}(Y')^{-1}(\Pi)$, where
$$Y'=\prod_{\beta\in wJ\setminus\{\pm\gamma\}} (1+e^{-\beta})$$
and $\lambda'=w\lambda$ if $-\gamma\in wJ$, $\lambda'=w\lambda+\gamma$ if $\gamma\in wJ$.
Since $X(\Pi)\in\cR(\Pi), X(r_{\gamma}\Pi)\in \cR(r_{\gamma}\Pi)$,
the sum $X(\Pi)(1+e^{\gamma})$ (resp., $X(r_{\gamma}\Pi)$)
 is  a well-defined element in $\cR(\Pi)$ (resp., in $\cR(r_{\gamma}\Pi)$)
 and since all summands $(1+e^{\gamma})w\bigl(\frac{e^{\lambda}}{1+e^{-\beta}}\bigr)$
lie in $\cR(\Pi)\cap\cR(r_{\gamma}\Pi)$, we obtain~(\ref{equiv1}). This proves (b)
and (c).
\end{proof}

\begin{rem}{}
For $\lambda\not=0$ the conditions

(i) the orbit $W'\lambda$ has a unique maximal element
in the $\geq$-ordering; 

(ii) $\langle \lambda,\alpha^{\vee}\rangle\in\mathbb{Z}$ 
for each $\alpha\in\Pi_0$ such that $r_{\alpha}\in W'$;

are equivalent if $W'$ is finite. In the case when $W'$ is infinite,
 (i) is equivalent to (ii)+ (iii), where

(iii) $(\lambda',\delta)/(\alpha,\alpha)>0$ for each $\alpha\in\Pi_0$ such that
$r_{\alpha}\in W'_{aff}$.
\end{rem}

\subsubsection{}\label{sectXX'}
\begin{lem}{lemXX'}
Let $\alpha\in\Pi$ be an isotropic root. Assume that
$X=\sum x_{\mu}e^{\mu}\in\cR(\Pi), X'=\sum x'_{\mu}e^{\mu}\in\cR(r_{\alpha}\Pi)$ are equivalent and that
$X$ has a pole of order $\leq 1$ at $\alpha$.
Then for each $\mu\in\fh^*$ one has
\begin{equation}\label{eqXX'}
\forall k\in\mathbb{Z}\ \ \ x_{\mu+k\alpha}-x'_{\mu+k\alpha}=(-1)^k(x_{\mu}-x'_{\mu}).
\end{equation}
Moreover, $x_{\mu}=x'_{\mu}$ if $\bigl(\supp X\bigr)\cap\{\mu+\mathbb{Z}\alpha\}$ is finite.
\end{lem}
\begin{proof}
 Since $X$  has a pole of order $\leq 1$ at $\alpha$, one has
 $(1+e^{-\alpha})X\in\cR(r_{\alpha}\Pi)$. Since $X,X'$ are equivalent, 
 the elements $(1+e^{-\alpha})X,(1+e^{-\alpha})X'\in\cR(r_{\alpha}\Pi)$ are equivalent and so 
 $(1+e^{-\alpha})X=(1+e^{-\alpha})X'$. 

Recall that $\cV$ is a $\cV_{fin}$-module; for $Y=\sum y_{\mu}e^{\mu}\in\cV$ one has
$$(1+e^{-\alpha})Y=0\ \Longrightarrow\ y_{\mu}+y_{\mu-\alpha}=0 \text{ for all}\ \mu.$$
This gives~(\ref{eqXX'}).

Finally, note that $x'_{\mu-k\alpha}=0$ for $k>>0$, because $X'\in \cR(r_{\alpha}\Pi)$
and $-\alpha\in r_{\alpha}\Pi$.
If $\supp X\cap\{\mu+\mathbb{Z}\alpha\}$
is finite, then $x_{\mu-k\alpha}=0$ for $k>>0$ and thus
$x_{\mu}-x'_{\mu}=x_{\mu-k\alpha}-x'_{\mu-k\alpha}=0$,
as required.
\end{proof}

\section{Root systems of basic and affine Lie superalgebras}
\label{sect3}
In this section we give some (mostly known) properties of Dynkin diagrams of  basic and affine
Lie superalgebras which are used in the main text. We call a Dynkin diagram of an indecomposable 
affine (resp., basic) Lie superalgebra affine (resp., finite) type Dynkin diagram.
We identify a set of simple roots $\Pi$ with the vertices of its Dynkin diagram.

In this section $\fg$ is   an indecomposable affine or basic Lie superalgebra with a set of simple roots 
$\Pi$ and  a symmetrizable Cartan matrix $A$. We denote by $\Delta$  the root system of $\fg$.

Throughout Sections~\ref{sect3}--\ref{sectKW0}, unless otherwise stated,
 we use the following
normalization of the invariant bilinear form $(-,-)$. If
the dual Coxeter number is non-zero, we normalize the form by the condition
$h^{\vee}\in\mathbb{Q}_{>0}$. If $\fg=D(n+1,n), D(n+1,n)^{(1)}$ or
$D(2,1,a), D(2,1,a)^{(1)}, a\in\mathbb{Q}$,
we normalize the form by the condition $(\alpha,\alpha)\in\mathbb{Q}_{>0}$ for some $\alpha\in D_{n+1}$ or $\alpha\in D_2=A_1\times A_1$; 
if the dual Coxeter number is zero
and $\fg\not=D(n+1,n),D(n+1,n)^{(1)}, D(2,1,a), D(2,1,a)^{(1)}, a\in\mathbb{Q}$,
we normalize the form by the condition $(\alpha,\alpha)\in\mathbb{Q}_{>0}$ for some $\alpha\in \Delta$
(note that in this case  all connected components of
${\Pi}_{\ol{0}}$ have the same number of elements).

\subsection{Affine Lie superalgebras}
\subsubsection{}
\begin{lem}{lemdeltasum}
 Let $\Pi$ be a set of simple roots of  an indecomposable affine Lie superalgebra
 and let $\delta$ be the minimal imaginary root.
 Then $\delta=\sum_{\alpha\in\Pi} x_{\alpha}\alpha$, where
 each coefficient $x_{\alpha} \not=0$.
\end{lem}
\begin{proof}
Take $\alpha\in\Pi$ such that $x_{\alpha}=0$.  
Since $\Delta$ is affine, $\Delta+r\delta=\Delta$ for some $r>0$, so
$r\delta-\alpha\in\Delta$. One has
$$r\delta-\alpha=\sum_{\beta\in\Pi,\beta\not=\alpha}rx_{\beta}\beta-\alpha,$$
that is $rx_{\beta}\leq 0$ for each $\beta$. Then $r\delta\in-\Delta^+$, a contradiction.
\end{proof}

\subsubsection{Finite parts}\label{cordeltasum}
For each $\Pi'\subset \Pi$ the set $\mathbb{Z}\Pi'\cap\Delta$
is the set of roots of a Kac-Moody superalgebra with the Cartan matrix
$A'$, which is the submatrix of $A$, corresponding to $\Pi'$.
Using~\Lem{lemdeltasum}, we conclude that any proper subdiagram of 
a connected Dynkin diagram of affine type is of finite type, i.e., if
$\Pi$ is a set of simple roots of  an indecomposable affine Lie superalgebra,
then for any proper subset $\Pi'\subset\Pi$ the root system $\mathbb{Z}\Pi'\cap\Delta$
is finite (and is the root system of a certain basic Lie superalgebra).

Let $X$ be an affine Dynkin diagram. We call a connected
subdiagram $\dot{X}$, obtained from $X$ by removing one node,
its {\em finite part}.
By above, $\dot{X}$ is of finite type. We call a root subsystem $\dot{\Delta}$
a finite part of affine root system $\Delta$ if $\dot{\Delta}$
admits a set of simple roots $\dot{\Pi}$ which is finite part of a set of simple 
roots for $\Delta$.  The finite parts of affine root systems are described in~\S~\ref{app3}.

\subsubsection{Definitions}\label{defintegr}
Let $\fg$ be an  affine Lie superalgebra with the  root system $\Delta$.
Let $\dot{\Delta}$ be a finite part of $\Delta$
(see~\S~\ref{cordeltasum}). An irreducible {\em vacuum module} is a module $L(\lambda)$ such that
$(\lambda,\dot{\Delta})=0$. 
Note that if $\Pi$ is a set of simple roots of $\Delta$
and $\dot{\Pi}$ is a finite part of $\Pi$, and
$(\lambda,\dot{\Pi})=0$, then $L(\lambda)$ is a vacuum module.

Let $\fg$ be a basic or affine Lie superalgebra. For each subset $\pi\subset\Pi_0$ 
we say that a $\fg$-module $N$ is {\em $\pi$-integrable} if $\fh$ acts diagonally on $N$
 and for each $\alpha\in\pi$
 the root spaces $\fg_{\pm\alpha}$ act locally nilpotently on $N$.
 
Note that if $N$ is $\pi$-integrable, then for each $w\in W(\pi)$ one has $\dim N_{\nu}=\dim N_{w\nu}$,
 so $\ch N$ is a $W(\pi)$-invariant element of $\cV$, see~\S~\ref{cRPi} for notation.
 In particular, if $N$ is a $\pi$-integrable irreducible
 highest weight module, then $\ch N$ is  a $W(\pi)$-invariant element of 
 $\cR_{W(\pi)}$,  see~\S~\ref{RW}.

Let $\fg$ be an affine Lie superalgebra, let $\dot{\fg}$ be its finite part,
and let $\dot{\Pi}_0$ be the subset of simple roots for $\dot{\fg}_{\ol{0}}=\fg_{\ol{0}}
\cap\dot{\fg}$. We say that  a $\fg$-module $N$ is {\em 
 integrable} if $N$ is $\pi$-integrable for 
$\pi=\dot{\Pi}_0\cup
\{\alpha\in\Pi_0|\ (\alpha,\alpha)\in\mathbb{Q}_{>0}\}$.

Note that $\pi$ is independent of our normalization of $(-,-)$ if $h^{\vee}\not=0$,
but $\pi$ changes if we change the sign of $(-,-)$ if  $h^{\vee}=0$.
In all cases, except for $D(2,1,a)^{(1)}$, $\pi$ is a connected component of $\Pi_0$.

\subsection{}
The sets of simple roots of basic Lie superalgebras which consist of isotropic 
roots are the following:
\begin{equation}\label{isofinite}
\begin{array}{ll}
A(n,n) & \{\vareps_1-\delta_1,\delta_1-\vareps_2,\ldots,\delta_{n-1}-\vareps_n,
\vareps_n-\delta_n\},\\
A(n+1,n) & \{\vareps_1-\delta_1,\delta_1-\vareps_2,\ldots,\vareps_n-\delta_n,
\delta_n-\vareps_{n+1}\},\\
D(n,n) & \{\delta_1-\vareps_1,\vareps_1-\delta_2,\ldots,\delta_n-\vareps_n,
\delta_n+\vareps_n\},\\
D(n+1,n) & \{\vareps_1-\delta_1,\delta_1-\vareps_2,\ldots,\vareps_n-\delta_n,\delta_n\pm\vareps_{n+1}\},
\end{array}\end{equation}
and for $D(2,1,a)$ it is as for $D(2,1)$. The invariant bilinear form (satisfying~\S~\ref{defintegr})
can be chosen in such a way that the vectors $\vareps_i,\delta_j$ are mutually orthogonal and 
$1=||\vareps_i||^2=-||\delta_j||^2$, except for the case $D(n,n)$,
where $1=-||\vareps_i||^2=||\delta_j||^2$.

We claim that the sets of simple roots of indecomposable
affine Lie superalgebras which consist of isotropic 
roots are the following:
\begin{equation}\label{isoaffine}
\begin{array}{ll}
A(n,n)^{(1)} & \{\delta-\vareps_1+\delta_n,
\vareps_1-\delta_1,\delta_1-\vareps_2,\ldots,\delta_{n-1}-\vareps_n,\vareps_n-\delta_n\},\\
D(n+1,n)^{(1)}  & \{\delta-\vareps_1+\delta_1,
\vareps_1-\delta_1,\delta_1-\vareps_2,\ldots,\vareps_n-\delta_n,\delta_n\pm\vareps_{n+1}\},\\
A(2n-1,2n-1)^{(2)} & \{\delta-\vareps_1+\delta_1,\vareps_1-\delta_1,\delta_1-\vareps_2,\ldots,\vareps_n-\delta_n,\vareps_n+\delta_n\}
\end{array}\end{equation}
and for $D(2,1,a)^{(1)}$ it is as for $D(2,1)^{(1)}$.

This can be explained as follows. 
If  $\Pi$ consists of isotropic roots and $\Delta(\Pi)$ is affine,
 then $\fg$ has zero dual Coxeter number and $\fg$ has a finite part which
appears in~(\ref{isofinite}). Using the tables in~\S~\ref{app3}
we conclude that this holds only for $\fg$ listed in~(\ref{isoaffine}). It is easy to see
that for these algebras all Dynkin diagrams consisting of isotropic roots are
as in~(\ref{isoaffine}).

Another result that we are going to use is the following.
\subsubsection{}
\begin{lem}{lembud}
Let $\beta,\beta'\in\Pi$ be isotropic roots with 
$(\beta,\beta')\not=0$. If $(\beta,\alpha)\not=0$ for some
non-isotropic $\alpha\in\Pi$, then
$(\beta,\beta')/(\alpha,\alpha)^2\in\mathbb{Q}_{>0}$.
\end{lem}
\begin{proof}
Since $\beta+\beta'\in r_{\beta}\Pi$, we have $\beta+\beta'\in\Pi_0$.
Normalize the form in such a way that $||\alpha||^2=2$.
Then $(\alpha,\beta')\in\mathbb{Z}_{\leq 0}, (\alpha,\beta)\in\mathbb{Z}_{<0}$, 
so $(\alpha,\beta+\beta')\in\mathbb{Z}_{<0}$. Since $\beta+\beta'\in\Pi_0$,
we have $(\alpha, (\beta+\beta')^{\vee})=2(\alpha,\beta+\beta')/||\beta+\beta'||^2\in\mathbb{Z}_{<0}$, 
so  $||\beta+\beta'||^2=2(\beta,\beta')\in\mathbb{Q}_{>0}$.
\end{proof}

\subsubsection{}\label{corbud}
Let $\Pi$ be a connected Dynkin diagram which contains a non-isotropic node,
let  $Iso$ be its subdiagram consisting
of isotropic nodes, and let $\Pi'$ be a connected component of $Iso$. Note that
$\Pi'$ appears in~(\ref{isofinite}). Since $\Pi$ is connected, $\Pi'$ contains a node 
$\beta\in \Pi'$ which is connected to a node in $\Pi\setminus \Pi'$.
By~\Lem{lembud}, $\beta$ can be described as follows.  For $A(n+1,n), A(n,n)$, $\beta$
is one of the ending nodes; for  $D(n+1,n), D(n,n), n>1$
$\beta$ is the first node.  For $D(2,1,a)$,
$\beta\in \Pi$ is such that 
$(\beta,\beta_1)/(\beta,\beta_2)\in\mathbb{Q}_{>0}$, where $\Pi=\{\beta,\beta_1,\beta_2\}$;
such $\beta$ is unique and exists only if $a\in\mathbb{Q}$.
In particular,  $\Pi'$ of type $D(2,1,a)$ with $a\not\in\mathbb{Q}$ can not be a connected component
of $Iso$.

\subsection{Choice of  $\Pi, S,\pi$}\label{PiSpi}
Let $\Pi$ be a set of simple roots for $\Delta$ which satisfies the following 
property: $||\alpha||^2\in\mathbb{Q}_{\geq 0}$ for each $\alpha\in\Pi$.
Recall that such $\Pi$ exists for all root systems except for $A(2k,2k)^{(4)}, D(k+1,k)^{(2)}$:
for each non-twisted (resp., twisted) affine $\Delta$ the example of such 
$\Pi$ appears in the end of~\cite{Gaff}, (resp.,~\cite{R}), where such $\Pi$ was used for a proof of the denominator identity.

We consider $\fg\not=D(2,1,a), D(2,1,a)^{(1)}$ with $a\not\in\mathbb{Q}$. Set 
$$\pi:=\{\alpha\in\Pi_0|\ ||\alpha||^2>0\}.$$

Recall that the {\em defect} of a finite type root system $\Delta$ is the dimension
of maximal isotropic subspace in $\mathbb{Q}\Delta$; for $A(m-1,n-1), B(m,n), D(m,n)$ the defect is
equal to $min (m,n)$; for other cases of non Lie algebras it is one.
It is well-known that
it is equal to the maximal number of mutually orthogonal isotropic simple roots for 
for some choice of $\Pi$, a set of simple roots of $\Delta$.

From~\S~\ref{app3} it follows that for affine root system $\Delta$ 
all its finite parts have the same defect and that
the maximal number of mutually orthogonal isotropic simple roots for $\Delta$
is equal to the defect of its finite part.  We call this number the {\em defect} of the affine superalgebra.
A subset $S\subset \Pi$ is called {\em maximal isotropic} 
if $\mathbb{Q}S$ is  isotropic and $\dim \mathbb{Q} S$
is equal to the defect.

\subsubsection{}\label{choiceS} 
We say that an isotropic node $\beta\in\Pi$ is {\em "branching"} if the connected component of $Iso$ 
which contains $\beta$ is of the type $D(k+1,k)$ and $\beta$ is the branching node in 
this connected component. Note that $\Pi$ contains at most two "branching" nodes.

If $\Pi$ is not the diagram of $D(2,1,a)$,
consisting of isotropic nodes, we let
 $S\subset\Pi$
be a  maximal subset of mutually orthogonal isotropic simple roots,
which contains all "branching" isotropic nodes.
For example, if $\Pi=D(n+1,n)$ consists of isotropic roots (see~(\ref{isofinite})),
then $S=\{\vareps_i-\delta_i\}_{i=1}^n$.
If  $\Pi=D(2,1,a)$, consisting of isotropic nodes, we take
$S=\{\beta\}$, where
$\beta\in\Pi$ is such that $(\beta,\alpha^{\vee})>0$ for $\alpha\in\pi$ (such $\beta$ is unique, since
$\pi=A_1\times A_1$).


\subsubsection{Remark}\label{legitimate}
Let $\Pi'$ be a connected component of $Iso$. Using~\S~\ref{corbud}
one easily sees that 
$\pi':=\pi\cap\Delta(\Pi')$  and $S':=S\cap \Pi'$
is a "correct" choice of $\pi,S$ in the sense 
of~\S~\ref{choiceS} for $\Pi'$.

\subsubsection{Example}
For example, for $A(m,n)^{(1)}, m>n$, we have 
$\pi=\{\delta-\vareps_1+\vareps_m,\vareps_1-\vareps_2,\ldots,\vareps_{m-1}-\vareps_m\}$
and we can take
$$\Pi=\{\delta-\vareps_1+\vareps_m,\vareps_1-\delta_1,\ldots,
\delta_n-\vareps_{n+1},\vareps_{n+1}-\vareps_{n+2},
\ldots,\vareps_{m-1}-\vareps_m\};$$
and for $A(n,n)^{(1)}$ we can take
$$\Pi:=\{\delta-\vareps_1+\delta_n,\vareps_1-\delta_1,\delta_1-\vareps_2,\ldots,
\vareps_n-\delta_n\},\ \ \pi=\{\delta-\vareps_1+\vareps_n,\vareps_1-\vareps_2,\ldots,\vareps_{n-1}-\vareps_n\},$$
with $S=\{\vareps_i-\delta_i\}_{i=1}^n$ in both cases.

\subsubsection{}
\begin{lem}{lemap11}
Let $\Pi, S,\pi$ be as above.

(i) For each $\alpha\in\Pi$ with $||\alpha||^2=0$ one has $\alpha\in S$ or
$\alpha+\beta\in\pi$ for some $\beta\in S$. 

(ii) If $w\in W(\pi)$ is such that $w\rho=\rho$ and $wS\subset \Delta^+$, then $w= Id$.
\end{lem}
\begin{proof}
(i) In the light of~\S~\ref{legitimate}, it is enough to verify (i) for the case, when $\Pi$
contains only isotropic roots, that is $\Pi$ is either $A(n,n)^{(1)}$ or 
appears in~(\ref{isofinite}). It is easy to check that the claim holds in each case.

(ii) Consider first the case when $\Pi$ contains only isotropic roots. 

Recall that for $S=\{\vareps_i-\delta_i\}_{i=1}^n$ for $D(n+1,n)$.
For $A(n,n), D(n,n), A(n,n)^{(1)}$ the choice of $S$ is unique up to an automorphism of the Dynkin diagram $\Pi$,
so we  take $S=\{\vareps_i-\delta_i\}_{i=1}^n$ for $A(n,n), D(n+1,n), A(n,n)^{(1)}$ 
and  $S=\{\delta_i-\vareps_i\}_{i=1}^n$ for $D(n,n)$.

For $A(n,n)$,  $A(n,n)^{(1)}$ the group $W$ stabilizes $\sum_{\beta\in S}\beta$.
Therefore $wS\subset \Delta^+$ forces $w\beta=\beta$ for each $\beta\in S$, so
$w\vareps_i=\vareps_i, w\delta_i=\delta_i$ for each $i=1,\ldots,n$. This gives $w=Id$.

For $D(n,n)$, the group $W(\pi)$ acts by signed permutations on $\{\delta_i\}_{i=1}^{n}$,
and the condition $wS\subset\Delta^+$ gives for each $i=1,\ldots,n$ that
$w\delta_i=\delta_j$ for $j\leq i$. Thus $w=Id$.

For $D(n+1,n)$, 
the group $W(\pi)$ acts by signed permutations with even number of sign changes 
on the set $\{\vareps_i\}_{i=1}^{n+1}$,
and the condition $wS\subset\Delta^+$ gives for each $i=1,\ldots,n$ that
$w\vareps_i=\vareps_j$ for $j\leq i$. Thus $w=Id$.

For $A(n+1,n)$, the group $W(\pi)$ permutes $\{\vareps_i\}_{i=1}^{n+1}$
and $S$ is of the form $\{\vareps_i-\delta_i\}_{i=1}^k\cup\{\delta_i-\vareps_{i+1}\}_{i=k+1}^n$ for some
$k=0,\ldots,n$. The condition $wS\subset\Delta^+$ gives for each $i=1,\ldots,k$ (resp., $i=k+1,\ldots,n+1$) that
$w\vareps_i=\vareps_j$ for $j\leq i$ (resp., $j\geq i$). This implies $w=Id$.

Finally consider the case $D(2,1,a)$. Write $\Pi=\{\beta,\beta_1,\beta_2\}$ 
with $S=\{\beta\}$. Recall that $\pi=\{\beta+\beta_1,\beta+\beta_2\}$ and
$(\beta+\beta_1,\beta+\beta_2)=0$. Since $r_{\beta}\Pi=\{-\beta,\beta+\beta_1,\beta+\beta_2\}$,
for each $w\in W(\pi)$ one has $w(-\beta)=-\beta+a_1(\beta+\beta_1)+
a_2(\beta+\beta_2)$ for some $a_1,a_2\in\mathbb{Z}_{\geq 0}$ and $a_1=a_2=0$ only for $w=Id$.
Hence $w\beta\in\Delta^+$ forces $w=Id$, as required.

Now consider the general case. Let $X$ be the set of connected components
of $Iso$; for each connected component $\Pi'$ choose $\pi':=\Delta(\Pi')\cap\pi$,
see~\S~\ref{legitimate}. We claim that 
\begin{equation}\label{stab}
Stab_{W(\pi)}\rho=\prod_{\pi'\in X} W(\pi').
\end{equation}
Indeed, since $||\alpha||^2\geq 0$ for each $\alpha\in\Pi$,
one has $(\rho,\alpha)\geq 0$ for each $\alpha\in\Delta^+$. 
This gives
$(\rho,\alpha^{\vee})\geq 0$ for each $\alpha\in\pi$, so the stabilizer
of $\rho$ in $W(\pi)$ is generated by the reflections $\{r_{\alpha}|\ \alpha\in\pi,
(\rho,\alpha)=0\}$. Clearly, $(\rho,\alpha)=0$ for $\alpha\in\Delta^+$ means
that $\alpha\in\mathbb{Z}Iso$. Since $\alpha\in\pi$, we obtain $\alpha\in\pi'$ for some
$\Pi'\in X$. This establishes~(\ref{stab}).

Now take $w\in Stab_{W(\pi)}\rho$. If $w\not=Id$, then there exists $\Pi'\in X$ such that
the projection of $w$ to $W(\pi')$ is not $Id$. Denote this projection by $w'$.
Recall that  the set $S\cap \Pi'$ is a maximal isotropic set
in $\Pi'$. By above,  there exists
$\beta\in (S\cap \Pi')$ such that $w'\beta\in -\Delta^+$. Clearly, $w\beta=w'\beta$,
so $w\beta\in-\Delta^+$ for some $\beta\in S$. This proves (ii).
\end{proof}

\section{Proof of the KW-formula in the case when $\#S=defect (\fg)$ and $h^{\vee}\not=0$}
\label{sectKW}
Let $\fg$ be either a basic Lie superalgebra, except for $D(2,1,a)$ with $a\not\in\mathbb{Q}$, or
an affine Lie superalgebra with  non-zero dual Coxeter number, or $A(n,n)^{(1)}$,
with a subset of simple roots $\Pi$ such
that  $||\alpha||^2\in\mathbb{Q}_{\geq 0}$ for each $\alpha\in\Pi$
(see Section~\ref{sect3} for the normalization of $(-,-)$).
 Set $\pi:=\{\alpha\in\Pi_0|\ (\alpha,\alpha)>0\}$.

In this section we prove the  KW-formula for some maximally atypical 
(i.e. $\#S=defect (\fg)$) $\pi$-integrable 
$\fg$-modules, which include, in particular, 
the  integrable vacuum modules over
the affine Lie superalgebras with non-zero dual Coxeter number
and over  $A(n,n)^{(1)}$.

\subsection{Main result}\label{sectKWass}
If $\Pi$ is not the set of simple roots of $D(2,1,a)$,
consisting of isotropic roots, we let
 $S\subset\Pi$
be a  maximal set of mutually isotropic orthogonal simple roots,
which contains all "branching" isotropic nodes
(recall that an isotropic node $\beta\in\Pi$ is "branching" if the connected component of $Iso$ 
which contains $\beta$ is of the type $D(k+1,k)$ and $\beta$ is the branching node in 
this connected component, see~\S~\ref{choiceS}).
If  $\Pi=D(2,1,a)$ consists of isotropic roots, we take
$S=\{\beta\}$, where
$\beta\in\Pi$ is such that $(\beta,\alpha^{\vee})>0$ for $\alpha\in\pi$ (such $\beta$ is unique, since
$\pi=A_1\times A_1$).

Let $L(\lambda)$ be a non-critical $\pi$-integrable $\fg$-module 
with the property $(\lambda,\beta)=0$
for each $\beta\in S$. We claim that $\ch L$ is given by the KW-formula:

\begin{equation}\label{KWcharfor}
Re^{\rho}\ch L(\lambda)=\sum_{w\in W(\pi)}\sgn (w)\,\frac{e^{w(\lambda+\rho)}}{\prod_{\beta\in S} (1+e^{-w\beta})}.
\end{equation}

The condition that $L(\lambda)$ is $\pi$-integrable implies $(\lambda,\alpha)\geq 0$ for each $\alpha\in\pi$.
Combining with $(\lambda,\beta)=0$
for $\beta\in S$, we obtain, using~\Lem{lemap11} (i), that $(\lambda,\alpha)\geq 0$ for each $\alpha\in\Pi$.
In particular,  if $\fg\not=A(n,n)^{(1)}$, the condition that $\lambda$ is non-critical is superfluous.

If $\fg=A(m,n)^{(1)}$ with $m\not=n$, the conditions on $\lambda$ are equivalent to 
$(\lambda,\alpha)\in\mathbb{Z}_{\geq 0}$ for each $\alpha\in\Pi$, $(\lambda,\beta)=0$ for each $\beta\in S$.
For $\fg=A(n,n)^{(1)}$ the additional condition is $\lambda\not\in\mathbb{Z}\delta$.

If $h^{\vee}\not=0$ and 
$L=L(\lambda,\Pi')$ is a vacuum $\pi$-integrable module (i.e., $(\lambda,\dot{\Delta})=0$),
 then, by~\S~\ref{app3}, using the odd reflections from $\dot{\Delta}$ we can transform $\Pi'$ to
$\Pi$ with the above properties. Then $L=L(\lambda,\Pi)$.
Since the defect of $\dot{\Delta}$ is equal to the maximal number of
pairwise orthogonal isotropic roots in $\Pi$, we can choose $S$ in $\dot{\Delta}\cap \Pi$.
Hence $\# S= defect (\fg)$, and formula~(\ref{KWcharfor}) applies.

\subsection{Proof}\label{proofI}
Note that $S,\pi$ are  as in~\S~\ref{PiSpi}.

We rewrite formula~(\ref{KWcharfor}) in the form
$$Re^{\rho}\ch L(\lambda)=\sum_{w\in W(\pi)} sgn(w)\, Y_w,\ \text{ where } 
Y_w:=w\bigl(\frac{e^{\lambda+\rho}}{\prod_{\beta\in S} (1+e^{-\beta})}\bigr).$$
By~\Lem{lemexpA}, $\sum_{w\in W(\pi)} sgn(w)\, Y_w\in\cR$. One has
\begin{equation}\label{suppYw}
\supp  Y_w\subset
w(\lambda+\rho)+(\sum_{\beta\in S: w\beta\in\Delta^-} w\beta)-\mathbb{Z}_{\geq 0}\Pi.
\end{equation}

Since $L(\lambda)$ is $\pi$-integrable, $(\lambda,\alpha^{\vee})\in\mathbb{Z}_{\geq 0}$ for each $\alpha\in\pi$. It is easy to deduce from~\Prop{propS} that $r_{\alpha}\rho\in\rho-\Delta$ for each $\alpha\in\pi$.
Note that $(\rho,\alpha)\geq 0$ for each $\alpha\in\Delta^+$ (because $||\alpha||^2\geq 0$ 
for each $\alpha\in \Pi$), so  $r_{\alpha}\rho\in\rho-\Delta^+$ for each $\alpha\in\pi$.
Therefore $w(\lambda+\rho)\in \lambda+\rho-\mathbb{Z}_{\geq 0}\Pi$, 
so 
$$\supp(\sum_{w\in W(\pi)} sgn(w) Y_w)\subset\lambda+\rho-\mathbb{Z}_{\geq 0}\Pi.$$
Clearly, $\supp (Re^{\rho}\ch L(\lambda))\subset\lambda+\rho-\mathbb{Z}_{\geq 0}\Pi$.

Let us show that
\begin{equation}\label{uki}
(\lambda+\rho-\mathbb{Z}_{\geq 0} S)\cap \supp ( Y_w)=\emptyset\ \ 
\text{ for }\ \ w\in W(\pi), w\not=Id.
\end{equation}
Indeed, assume that the intersection is non-empty. Then, by~(\ref{suppYw}),
 $w\in Stab_{W(\pi)}(\lambda+\rho)$ and for each $\beta\in S$ one has either $w\beta\in\Delta^+$ or 
$-w\beta\in \mathbb{Z}_{\geq 0} S$. If $-w\beta \in \mathbb{Z}_{\geq 0} S$, then $\beta-w\beta\in\mathbb{Z}S$,
so $||\beta-w\beta||^2=0$; but $\beta-w\beta\in\mathbb{Z}\pi$, thus $\beta=w\beta$.
Therefore $w\in Stab_{W(\pi)}(\lambda+\rho)$ and $wS\subset \Delta^+$. 
By~\Lem{lemap11} (ii), $w=Id$,
as required.

We conclude that the coefficient of $e^{\lambda+\rho}$ in $\sum_{w\in W(\pi)} sgn(w)\, Y_w$ is $1$;
clearly, the coefficient of $e^{\lambda+\rho}$ in $Re^{\rho}\ch L(\lambda)$ is also $1$.
Let
$$Z:=Re^{\rho}\ch L(\lambda)-\sum_{w\in W(\pi)} sgn(w) w\bigl(\frac{e^{\lambda+\rho}}{\prod_{\beta\in S} (1+e^{-\beta})}\bigr).$$
Suppose that $Z\not=0$. By above,
$$\supp (Z)\subset \lambda+\rho-\mathbb{Z}_{\geq 0}\Pi,\ \lambda+\rho\not\in\supp (Z).$$
The action of the Casimir element gives  $||\nu||^2=||\lambda+\rho||^2$
for each $\nu\in\supp (Z)$.
Let $\lambda+\rho-\mu$ be a maximal element in $\supp (Z)$ with respect to the order
$\geq_{\Pi}$ (see~\S~\ref{ordering}).
We have
$$||\lambda+\rho-\mu||^2=||\lambda+\rho||^2,\ \mu\in\mathbb{Z}_{\geq 0}\Pi,\ \mu\not=0.$$

\subsubsection{}\label{maxinorbit}
By~\S~\ref{RPi}, $R_{\ol{0}}e^{\rho_{\pi}}$ is a $W(\pi)$-skew-invariant
element of $\cR_{W(\pi)}$. Since  $L(\lambda)$ is $\pi$-integrable, $\ch L(\lambda)$ is a $W(\pi)$-invariant element of $\cR_{W(\pi)}$, see~\S~\ref{defintegr}. Thus
$$R_{\ol{0}}e^{\rho_{\pi}}\ch L(\lambda)=R_{\ol{1}}e^{\rho_{\pi}-\rho}( Re^{\rho}\ch L(\lambda))$$
 is a $W(\pi)$-skew-invariant element of $\cR_{W(\pi)}$.
By~\S~\ref{RPi}, $R_{\ol{1}}e^{\rho_{\pi}-\rho}$ is a $W(\pi)$-invariant
element of $\cR_{W(\pi)}$, so 
$$R_{\ol{1}}e^{\rho_{\pi}-\rho}\sum_{w\in W(\pi)}\sgn (w)\,
\frac{e^{w(\lambda+\rho)}}{\prod_{\beta\in S} (1+e^{-w\beta})}=
\sum_{w\in W(\pi)}\sgn (w)\,
e^{w(\lambda+\rho_{\pi})}\prod_{\beta\in \Delta^+_{\ol{1}}\setminus S} (1+e^{-w\beta})$$
By~\S~\ref{infprod}, 
$e^{\lambda+\rho_{\pi}}\prod_{\beta\in \Delta^+_{\ol{1}}\setminus S} (1+e^{-\beta})\in\cR_{W(\pi)}$,
so, by~\S~\ref{compex}, the sum in the right-hand side is a $W(\pi)$-skew-invariant
element of $\cR_{W(\pi)}$. 

We conclude that 
$R_{\ol{1}}e^{\rho_{\pi}-\rho}Z$ is a non-zero $W(\pi)$-skew-invariant
element of $\cR_{W(\pi)}$.  Clearly, $\lambda-\mu+\rho_{\pi}$ is a maximal element
in its support. Hence  $(\lambda-\mu+\rho_{\pi},\alpha^{\vee})$
is a positive integer for each $\alpha\in\pi$ 
(positive, since $\lambda-\mu+\rho_{\pi}$ is a maximal element
in the support of a skew-invariant element of $\cR_{W(\pi)}$, and integer, since
$(\lambda,\alpha^{\vee})\in\mathbb{Z}_{\geq 0}$). Therefore
$$(\lambda-\mu,\alpha)\geq 0\ \text{ for each }\alpha\in\pi.$$

\subsubsection{}\label{proof11}
By~\Lem{lemap11}, for any $\gamma\in\Pi$  one has $\gamma\in\pi\cup S$, or $2\gamma\in\pi$,
or 
$\gamma=\alpha-\beta, \alpha\in\pi,\beta\in S$. Define a linear map
$p:\mathbb{Z}_{\geq 0}\Pi\to\frac{1}{2}\mathbb{Z}_{\geq 0}\pi$ which is zero on $S$ and 
the identity on $\pi$. Recall that $(\lambda+\rho, S)=(S,S)=0$.
We have $2(\lambda+\rho,\mu)=(\mu,\mu)$, so $2(\lambda+\rho, p(\mu))=(2\mu-p(\mu),p(\mu))$,
which implies
$$2(\lambda+\rho-\mu,p(\mu))=-||p(\mu)||^2.$$
Since $(\rho,\alpha),(\lambda-\mu,\alpha)\geq 0$ for each $\alpha\in\pi$,
the left-hand side is non-negative. Hence $||p(\mu)||^2\leq 0$. Since
$p(\mu)\in \frac{1}{2}\mathbb{Z}_{\geq 0}\pi$, we obtain $p(\mu)=0$ if $\Delta$ is finite, and
$p(\mu)=s\delta$ if $\Delta$ is affine. In the latter case
$2(\lambda+\rho, p(\mu))=(2\mu-p(\mu),p(\mu))$ gives $2(\lambda+\rho,s\delta)=0$.
Since $\lambda$ is non-critical, $s=0$, so $p(\mu)=0$. Hence
$\mu\in\mathbb{Z} S\setminus\{0\}$.

\subsubsection{}\label{ukiki}
By~(\ref{uki}),  for $\mu\in\mathbb{Z}S$,
the coefficient of $e^{\lambda+\rho-\mu}$ in $Y_w$ is equal to zero if $w\not=Id$.
One readily sees that $L(\lambda)_{\lambda-\nu}=0$ for each $\nu\in\mathbb{Z}S,\nu\not=0$.
Thus the coefficient of $e^{\lambda+\rho-\mu}$  in $Re^{\rho}\ch L$ 
is equal to the coefficient of $e^{\lambda+\rho-\mu}$ in 
$e^{\lambda+\rho}\prod_{\beta\in S}(1+e^{-\beta})^{-1}=Y_{Id}$. 
Hence the coefficient of $e^{\lambda+\rho-\mu}$ in $Z$ is equal to zero, a contradiction. 
This gives $Z=0$ and establishes the KW-formula.\qed

\subsection{Remark}\label{remint}
Let $\fg$ be of the type $A(0,n)^{(1)}$ or $C(n)^{(1)}$. 
Note that $\Delta_{\ol{0}}$
is indecomposable, so $\pi=\Pi_0, W(\pi)=W$. Let $L=L(\lambda,\Pi)$ be a $\pi$-integrable module.
If $L$ is typical, i.e., $(\lambda+\rho,\beta)\not=0$ for all $\beta\in\Delta_{\ol{1}}$,
then $\ch L$ is given by the Weyl-Kac character formula
$Re^{\rho}\ch L=\sum_{w\in W} sgn(w) e^{w(\lambda+\rho)}$.
If $L$ is atypical, then there exists $\Pi'$ such that $L=L(\lambda',\Pi')$ and
$(\lambda'+\rho',\beta')=0$ for some $\beta'\in\Pi'$, see the next paragraph.
By above, $\ch L$ is given by the KW-character formula 
$Re^{\rho}\ch L=\sum_{w\in W} sgn(w) e^{w(\lambda'+\rho')}/(1+e^{-w\beta'})$. This formula was proven earlier
in~\cite{S}.

Let us show that $(\lambda+\rho,\beta)=0$ for some $\beta\in\Delta_{\ol{1}}$ 
forces the existence of $\Pi'$ such that $L=L(\lambda',\Pi')$ and
$(\lambda'+\rho',\beta')=0$ for some $\beta'\in\Pi'$.  Indeed, it is easy to show  that 
for this $\fg$ for any odd root $\beta$ there exists $\Pi''$  such that 
$\beta\in\Pi''$. Recall that $\Pi''$ can be obtained from $\Pi$ by a chain
of odd reflections $r_{\beta_1},\ldots,r_{\beta_s}$. 
Writing $\Pi^0:=\Pi,\lambda^0:=\lambda^i$
and $\Pi^j=r_{\beta_j}\Pi^{j-1}$ (so $\Pi''=\Pi^{s+1}$), we introduce $\lambda^j$ by 
$L=L(\lambda^j,\Pi^j)$. Then $\lambda^j+\rho^j=\lambda^{j-1}+\rho^{j-1}$, 
(where $\rho^j$ is the Weyl vector for 
$\Pi^j$) if $(\lambda^{j-1}+\rho^{j-1},\beta_{j})\not=0$.
Therefore, if $(\lambda^{j-1}+\rho^{j-1},\beta_{j})\not=0$ for $j=1,\ldots,s$,
then $\lambda^{s+1}+\rho^{s+1}=\lambda+\rho$, and taking $\Pi':=\Pi''=\Pi^{s+1}$
we obtain $(\lambda'+\rho',\beta)=0$ and $\beta\in\Pi''$.
If $(\lambda^{j-1}+\rho^{j-1},\beta_{j})=0$ for some $j$, then
we take $\Pi':=\Pi^{j-1}$ and $\beta':=\beta_{j}$ ($\beta_j\in\Pi^{j-1}$
by the definition of an odd reflection).

\subsection{A new identity}\label{newid}
In~\cite{KW4}  a product character formula was obtained for
the  $\mathfrak{osp}(m,n)^{(1)}$-module $V_1:=L(\Lambda_0)$.
The bilinear form in this case is 
normalized there by $||\vareps_i||^2=1=-||\delta_j||^2$.
If $M\geq N+2$, this normalization coincides with
our normalization. On the other hand, we have established
in this section the KW-character formula for $V_1$.
Comparing these formulas, we obtain a new identity, see below.

\subsubsection{}
Let $\fg=\mathfrak{osp}(m,n)^{(1)}, M\geq N+2$.
Consider a set of simple roots 
$$\Pi=\{\delta-(\vareps_1+\delta_1),\vareps_1-\delta_1,\delta_1-\vareps_2,
\ldots,\vareps_n-\delta_n,\delta_n-\vareps_{n+1},\vareps_{n+1}-\vareps_{n+2},
\ldots\}.$$
The set $\Pi$ has an involution $\iota$ which exchanges the first two roots
(via $\vareps_1\mapsto \delta-\vareps_1$) and fixes the rest.
This involution induces an involution $\iota$ of $\fg$.
Consider the vacuum $\fg$-module $V_1$ and its twisted by $\iota$ 
 module $V_1^{\iota}$. By~\cite{KW4}, these are all 
integrable (i.e., $\pi$-integrable)  $\fg$-modules of level $1$.

One has $V_1=L(\Lambda_0)$, where $(\Lambda_0,\delta)=1, 
(\Lambda_0,\vareps_i)=(\Lambda_0,\delta_j)=0$, and 
$V_1^{\iota}=L(\Lambda_0+\vareps_1)$.
Formula~(\ref{KWcharfor}) gives
$$\begin{array}{l}
Re^{\rho}\ch V_1=\sum_{w\in W(\pi)}\sgn (w)\,
\frac{e^{w(\Lambda_0+\rho)}}{\prod_{i=1}^n
(1+e^{-w(\vareps_i+\delta_i)})}\\
Re^{\rho}\ch V_1^{\iota}=\sum_{w\in W(\pi)}\sgn (w)\,\frac{e^{w(\Lambda_0+\vareps_1+\rho)}}
{(1+e^{-w(\delta+\vareps_1+\delta_1)})\prod_{i=2}^n
(1+e^{-w(\vareps_i+\delta_i)})},
\end{array}$$
where $W(\pi)$ is the Weyl group of $\mathfrak{so}_M$ ($B_m$ if $M=2m+1$,
and $D_m$ if $M=2m$).

On the other hand, formula (7.5) from~\cite{KW4} gives
$$Re^{\rho}(\ch V_1\pm \ch V_1^{\iota})=Re^{\Lambda_0+\rho}\prod_{k=1}^{\infty}
\frac{(1\pm q^{k-1/2})^{p(M)}\prod_{i=1}^m(1\pm e^{\vareps_i}q^{k-1/2})
(1\pm e^{-\vareps_i}q^{k-1/2})}
{\prod_{j=1}^n (1\mp e^{\delta_j}q^{k-1/2})(1\mp e^{-\delta_j}q^{k-1/2})},$$
where  $q=e^{-\delta}$, $N=2n$, and either 
$p(M)=0, M=2m, \mathfrak{osp}(m,n)=D(m,n)$, or
$p(M)=1, M=2m+1, \mathfrak{osp}(m,n)=B(m,n)$.

Comparing the last character formula with the previous two,
gives a product formula for some mock theta functions (as defined in~\cite{KW5}).


\section{KW-character formula for finite-dimensional modules}\label{sectfindim}
We say that a highest weight module $L$ satisfies the KW-condition, if
 for some set of simple roots $\Pi$ one has $L=L(\lambda,\Pi)$
and $\Pi$ contains a subset $S$, which spans a maximal isotropic subspace in $\mathbb{Q}\Delta^{\perp}_{\lambda+\rho}$ (where $\Delta^{\perp}_{\lambda+\rho}=\{\alpha\in\Delta|\ 
(\lambda+\rho,\alpha)=0\}$). Sometimes we say that $L$ satisfies the
KW-condition for $\Pi$ (or for $(\Pi,S)$).

Note that $\dim\mathbb{Q}\Delta^{\perp}_{\lambda+\rho}$
is the invariant of $L$ (if $L=L(\lambda,\Pi)=L(\lambda',\Pi')$,
then $\dim\mathbb{Q}\Delta^{\perp}_{\lambda+\rho}=
\dim\mathbb{Q}\Delta^{\perp}_{\lambda'+\rho'}$), see~\cite{KW3}, Cor. 3.1.

Throughout this section $\fg$ is a basic Lie superalgebra $\fg$ and
$L$ is a finite-dimensional irreducible (hence highest weight) $\fg$-module, which 
satisfies the KW-condition for some $(\Pi,S)$, and
$r:=\dim\mathbb{Q}\Delta^{\perp}_{\lambda+\rho}$. Note that $\# S=r\leq \, defect (\fg)$.
We assume that $r>0$ (otherwise
$L$ is typical and $Re^{\rho}\ch L=\sum_{w\in W} \sgn (w)\, e^{w(\lambda+\rho)}$, by~\cite{K1.5},\cite{KW3}).

Recall that KW-formula~(\ref{3}) has the form 
\begin{equation}\label{333}
j_{\lambda}Re^{\rho}\ch L=
\sum_{w\in W} \sgn (w)\, \frac{e^{w(\lambda+\rho)}}{\prod_{\beta\in S}(1+e^{-w\beta})},
\end{equation}
where $j_{\lambda}\not=0$.
In this section we prove this formula for finite-dimensional modules,
satisfying the KW-condition, in all cases except for $\fg=D(m,n)$ with $S=\{\delta_k-\vareps_m\}$
or $S=\{\vareps_m-\delta_k\}$ with $k<n$. For $C(n)$ the formula was proved in Section~\ref{sectKW}.

The coefficient $j_{\lambda}$ is  equal to 
$r!$ for $A(m,n)$,
$2^r\,r!$ for $B(m,n)$,  to $1$ for $C(n)$,
to $2^r\,r!$ or $2^{r-1}\,r!$ 
for $D(m,n)$, and to $2$ for the exceptional Lie superalgebras,
cf.~(\ref{jlambda})  and~\S~\ref{521} below.

\subsection{Outline of the proof}
Let us explain the outline of the proof.
In~\S~\ref{findimmaxatyp} we deduce~(\ref{333}) from~(\ref{KWcharfor})
for $(\Pi,S)$ satisfying the assumptions
of~\S~\ref{sectKWass}, and then, using~\Lem{lemchPIS}, we deduce~(\ref{333}) 
for any $(\Pi,S)$ for the exceptional Lie superalegbras.
In~\S~\ref{ADr>1} we establish~(\ref{333}) under the assumption that
$(\rho,\alpha^{\vee})\geq 0$ for all $\alpha\in\Pi_0$ (this assumption
holds for some subsets of simple roots,  if $\fg$ is
$A(m-1,n-1)$ or $D(m,n)$). For $B(m,n)$ we establish~(\ref{333}) for  some special
subsets of simple roots in~\S~\ref{Bmnspecial}. Then, 
in~\S~\ref{changePiS}-\ref{chPIS}, we explain why for $A(m-1,n-1)$
(resp., $B(m,n)$)
the KW-formula for any $(\Pi',S')$ is equivalent to the KW-formula for
$(\Pi,S)$ as in~\S~\ref{ADr>1} (resp., in~\S~\ref{Bmnspecial}),
and why this holds for 
$D(m,n)$ except for the cases, when $S=\{\delta_k-\vareps_m\}$
or $S=\{\vareps_m-\delta_k\}$ with $k<n$.

For $D(m,n)$ one of the simplest cases, when we have not established the KW-formula,
 is  $D(3,2)$ with $\Pi=\{\vareps_1-\vareps_2,\vareps_2-\vareps_3,\vareps_3-\delta_1,
\delta_1-\delta_2,2\delta_2\}$ and $\lambda=4\vareps_1+4\vareps_2-\vareps_3+\delta_1$.

\subsection{Case of maximal atypicality}\label{findimmaxatyp}
Let $r$($=\# S$) be equal to the defect of $\fg$. Assume that $\Pi,S$ satisfy the assumptions
of~\S~\ref{sectKWass}. Take $\pi=\{\alpha\in\Pi_0|\ ||\alpha||^2>0\}$
for $\Delta\not=D(n+1,n), D(2,1,a)$, $\pi=D_{n+1}$ for $D(n+1,n)$, 
and $\pi=D_2=A_1\times A_1$ for $D(2,1,a)$.
Then, by~(\ref{KWcharfor}),
$$Re^{\rho}\ch L=
\cF_{W(\pi)}\bigl(\frac{e^{\lambda+\rho}}{\prod_{\beta\in S} (1+e^{-\beta})}\bigr),$$
see~\S~\ref{expA} for notation. Write $W=W(\pi)\times W(\Pi_0\setminus\pi)$. 
Since $L$ is finite-dimensional, the left-hand side of this formula 
is $W$-skew-invariant. Then
\begin{equation}\label{denomfin}
\begin{array}{ll}
|W(\Pi_0\setminus\pi)|Re^{\rho}\ch L &=\cF_{W(\Pi_0\setminus\pi)}(Re^{\rho}\ch L)=
\cF_{W(\Pi_0\setminus\pi)}
\cF_{W(\pi)}\bigl(\frac{e^{\lambda+\rho}}{\prod_{\beta\in S} (1+e^{-\beta})}\bigr)\\
&=\cF_W\bigl(\frac{e^{\lambda+\rho}}{\prod_{\beta\in S} (1+e^{-\beta})}\bigr).
\end{array}
\end{equation}

This establishes KW-formula~(\ref{333}) for this case with
 \begin{equation}\label{jlambda}
j_{\lambda}=|W(\Pi_0\setminus\pi)|;
\end{equation}
note that  $W(\Pi_0\setminus\pi)$ is the smallest factor
in the presentation of $W(\Pi_0)$ as the direct product
of Coxeter groups.
 
In particular, for $\lambda=0$
we obtain the denominator identity (for such $\Pi, S$).

Note that $\pi$ is the "largest part" of $\Pi_0$ in the following sense:
$\Pi_0\setminus\pi$ is a connected component of $\Pi_0$ with the property
$|W(\pi)|\geq |W(\Pi_0\setminus\pi)|$ (for $A(n,n), B(n,n), D(2,1,a)$
the choice of $\pi$ is not unique).

\subsubsection{}\label{521}
The coefficient $j_{\lambda}$  for a non-exceptional Lie superalgebra 
can be obtained as follows: it is not hard to show (see~\S~\ref{KWfin})
that $\Pi$ contains a connected subdiagram $\Pi'$ of defect $r$ with the property
$(\lambda,\alpha)=0$ for $\alpha\in\Pi'$; moreover,
$\Pi'$ is of "the same type"
as $\Pi$ (if $\Pi=A(m,n)$, then $\Pi'=A(m',n')$, etc.). Write
 $W(\Pi'_0)=W_1\times W_2$, where $W_1,W_2$ are the Weyl groups
of the components of $\Pi'_0$ (connected components if $\Pi'\not=D(2,1)$).
If we choose $W_2$ such that $|W_1|\geq |W_2|$, then $j_{\lambda}=|W_2|$.
If $\Pi=A(m,n)$ (resp., $B(m,n), D(m,n)$),  
then $\Pi'=A(m',n')$ (resp., $B(m',n'), D(m',n')$), 
with $r=\min (m',n')$. Therefore for $A(m,n)$ (resp., for $B(m,n)$)
one has $j_{\lambda}=|W(A_r)|=r!$ (resp., $j_{\lambda}=2^r\, r!$). For $D(m,n)$ one has 
either $\Pi'=D(m',r)$ for $m'>r$ and  $j_{\lambda}=|W(C_r)|=2^r\,r!$, or
$\Pi'=D(r,n')$ for $n'\geq r$ and  $j_{\lambda}=|W(D_r)|=2^{r-1}\,r!$.

\subsubsection{}
Now let $\fg$ be an exceptional Lie superalgebra and let 
$L=L(\lambda,\Pi)$ be a finite-dimensional $\fg$-module,
satisfying the KW-condition
for $(\Pi,S)$. We claim that  the KW-formula holds and $j_{\lambda}=2$.
Note that $\fg$ has defect one, so $r=1$, that is $S=\{\beta\}$ for some $\beta\in\Pi$.

Indeed, by above, if $\Pi,S$ satisfy the assumptions
of~\S~\ref{sectKWass}, then the KW-formula holds
(this was also proved previously, see~\cite{KW2}) and 
$j_{\lambda}=|W(\Pi_0\setminus\pi)|=2$.
Assume that $\fg\not=D(2,1,a)$ with $a\not\in\mathbb{Q}$
and $\Pi,S$ do not satisfy the assumptions
of~\S~\ref{sectKWass}.
It is easy to check that in this case $\beta$ is the only isotropic root
in $\Pi$ and that $(r_{\beta}\Pi, S=\{-\beta\})$  satisfy the assumptions
of~\S~\ref{sectKWass}. In particular, the KW-formula holds for $(r_{\beta}\Pi, S'=\{-\beta\})$.
 By~\Lem{lemchPIS}, 
this implies the KW-formula for $(\Pi,\{\beta\})$.

Now let $\fg= D(2,1,a)$ for irrational $a$ (this case is not covered by Section~\ref{sectKW}).
It is easy to see that, in this case,
the trivial module is the only finite-dimensional atypical module.
KW-formula for the trivial module is the denominator identity, which, clearly,
 does not depend on $a$; it holds for rational $a$, hence it holds in general.

This establishes KW-formula~(\ref{333}) with $j_{\lambda}=2$
for the exceptional Lie superalgebras.

\subsection{}\label{}
Denote by $j_{\lambda}$ the coefficient of $e^{\lambda+\rho}$ in 
$\cF_W (\frac{e^{\lambda+\rho}}{\prod_{\beta\in S}
(1+e^{-\beta})})$. Set
$$Z:=j_{\lambda} Re^{\rho}\ch L-\cF_W(\frac{e^{\lambda+\rho}}{\prod_{\beta\in S} (1+e^{-\beta})}).$$
The KW-formula is equivalent to $j_{\lambda}\not=0$ and $Z=0$.

If $Z\not=0$, we denote by $\lambda'$ a maximal element in $\supp Z$. 
The arguments of~\S~\ref{maxinorbit} show that
\begin{equation}\label{maxfin}
(\lambda'-\rho,\alpha^{\vee})\geq 0\ \text{ for each }\alpha\in\Pi_0.
\end{equation}

\subsubsection{}
\begin{lem}{lem0KW}
$\supp Z\subset W(\lambda+\rho-\mathbb{Z}S).$
\end{lem}
\begin{proof}
Clearly,
$\supp \bigl(\frac{e^{w(\lambda+\rho)}}{\prod_{\beta\in S}(1+e^{-\beta})}\bigr)\subset 
 W(\lambda+\rho-\mathbb{Z}S).$

Let us show that
$\supp Re^{\rho}\ch L\subset W(\lambda+\rho-\mathbb{Z}S)$.

In the light of~\Prop{propKK} (in Section~\ref{sectDeltaL}), it is enough to verify that
for $y\in W, \mu\in\mathbb{Z}S$ and an isotropic root $\beta$, 
if $(y(\lambda+\rho-\mu),\beta)=0$, then $y(\lambda+\rho-\mu)-\beta=y'(\lambda+\rho-\mu')$
for some $y'\in W, \mu'\in\mathbb{Z}S$.

We start with the case $y=Id$, so $(\lambda+\rho-\mu,\beta)=0$.
If $(S,\beta)=0$, then
 $(\lambda+\rho,\beta)=0$ and, by the KW-condition, $\beta\in\mathbb{Z}S$.
 Therefore the claim holds for $y'=Id, \mu'=\mu+\beta$.
Now let  $(S,\beta)\not=0$, that is 
$(\beta,\beta')\not=0$ for some $\beta'\in S$.
Since $\fg_{\pm{\beta}}$ generate a copy of $\fsl(1,1)\subset\fg$,
either $\beta+\beta'$ or $\beta-\beta'$ is a root, i.e.
$\alpha:=\beta-x\beta'\in \Delta_{\ol{0}}$ for $x=1$ or $x=-1$.
Since $\lambda+\rho-\mu$ is orthogonal to $S$ and to $\beta'$,
it is orthogonal to $\alpha$. One has $r_{\alpha}\beta=x\beta'$, so
$$\lambda+\rho-\mu-\beta=\lambda+\rho-\mu-xr_{\alpha}\beta'=r_{\alpha}(\lambda+\rho-\mu-x\beta').$$
Thus the claim holds for $y'=r_{\alpha}, \mu'=\mu+x\beta$.

Now take an arbitrary $y\in W$.
Then $(y(\lambda+\rho-\mu),\beta)=0$ implies $(\lambda+\rho-\mu,y^{-1}\beta)=0$.
By above, $\lambda+\rho-\mu-y^{-1}\beta=w(\lambda+\rho-\nu)$ for some
$w\in W, \mu'\in\mathbb{Z}S$. Then
$y(\lambda+\rho-\mu)-\beta=yw(\lambda+\rho-\mu')$, as required.
\end{proof}

\subsection{Case $A(m-1,n-1)$ and $D(m,n), r>1$, or $D(m,n), S=\{\beta\}$,
where $\beta=\pm(\vareps_m-\delta_n)$}
\label{ADr>1}

In~\Cor{coreqAD} below we show that the KW-formula for any $(\Pi',S')$ is equivalent to the KW-formula for
$(\Pi,S)$ such that $(\rho,\alpha^{\vee})\geq 0$ for each $\alpha\in\Pi_0$. 
For such $(\Pi,S)$ we prove the formula below, proving thereby the KW-formula for $(\Pi',S')$.

\subsubsection{}\label{chundr}
Let $\fg$ be $A(m-1,n-1)$ or $D(m,n)$. We shall assume that
\begin{equation}\label{assmrho}
\forall\alpha\in\Pi_0\ \ 
(\rho,\alpha^{\vee})\geq 0.
\end{equation}
We will prove the KW-formula under this assumption, by showing that $j_{\lambda}\not=0$ and $Z=0$.

\subsubsection{}\label{Pi'0}
Since $L$ is finite-dimensional, $(\lambda,\alpha^{\vee})\geq 0$ for each $\alpha\in\Pi_0$.
Assumption~(\ref{assmrho}) implies that $\lambda+\rho$ is maximal in its $W$-orbit and
$\Stab_W(\lambda+\rho)=W(\Pi'_0)$, where
$$\Pi'_0:=\{\alpha\in\Pi_0|\ (\lambda+\rho,\alpha)=0\}=\{\alpha\in\Pi_0|\ (\lambda,\alpha)=(\rho,\alpha)=0\}.$$

Take $\alpha\in\Pi_0'$. Since $\Delta$ is not exceptional,
 $(\rho,\alpha)=0$ implies $\alpha=\beta+\beta'$, where
$\beta,\beta'\in\Pi$ are isotropic, see~\Lem{lemuse}.
 In this case, $(\beta,\alpha^{\vee})=(\beta,\alpha^{\vee})=1$.
If $(\lambda,\beta)\not=0$,
then $L=L(\lambda,\Pi)=L(\lambda-\beta,r_{\beta}\Pi)$ and $(\lambda,\alpha)=0$ gives
$(\lambda-\beta,\alpha^{\vee})=-1$,
which is impossible, since $L$ is finite-dimensional. Therefore $(\lambda,\beta)=0$;
similarly, $(\lambda,\beta')=0$.
We conclude that $\Pi'_0$ is spanned by $\Pi'$, where
$$\Pi':=\{\alpha\in\Pi|\ (\lambda,\alpha)=(\rho,\alpha)=0\}$$
(more precisely, $\Pi'_0$ is a set of simple roots for $\Delta(\Pi')_{\ol{0}}$).

\subsubsection{}
Since $\lambda+\rho$ is maximal in its $W$-orbit, using~(\ref{suppYw}) we obtain
$$\supp Z\subset \lambda+\rho-\mathbb{Z}_{\geq 0}\Pi.$$

Suppose that $Z\not=0$ and let $\lambda'$ be a maximal element in $\supp Z$.
Then $\lambda'=\lambda+\rho-\nu'$  for some
$\nu'\in \mathbb{Z}_{\geq 0}\Pi$.

Let us show that $\nu'\in\mathbb{Z}S$.
By~\Lem{lem0KW}, $\lambda':=w(\lambda+\rho-\mu)$ with $w\in W,\mu\in\mathbb{Z}S$.
Combining~(\ref{maxfin}) and~(\ref{assmrho}), we get
$(\lambda',\alpha^{\vee})\geq 0$ for each $\alpha\in\Pi_0$. Thus
$\lambda'=w(\lambda+\rho-\mu)$ is maximal in its $W$-orbit. Since
$\lambda+\rho-\mu$ lies in this orbit, 
we have $\lambda+\rho-\mu=\lambda'-\nu$ for some $\nu\in\mathbb{Z}_{\geq 0}\Pi_0$.
Thus $\mu=\nu+\nu'$, where $\mu\in\mathbb{Z}S$, $\nu\in\mathbb{Z}_{\geq 0}\Pi_0$ and
$\nu'\in \mathbb{Z}_{\geq 0}\Pi$. Since $S$ is a set of mutually orthogonal
isotropic roots, $\nu=0$ and
$\nu'=\mu$, as required.

\subsubsection{}\label{chundra}
Denote by $P:\cV\to \cV$ the projection
sending $\sum_{\nu\in\fh^*} a_{\nu} e^{\lambda+\rho-\nu}$ to
$\sum_{\nu\in\mathbb{Z}S} a_{\nu} e^{\lambda+\rho-\nu}$.
Since $\lambda'\in\lambda+\rho-\mathbb{Z}S$, it
 is enough to verify that $P(Z)=0$.

If $w\not\in W(\Pi'_0)$, then $w(\lambda+\rho)=(\lambda+\rho)-\gamma$, where
$\gamma\in \mathbb{Z}_{\geq 0}\Pi_0$, $\gamma\not=0$. By~(\ref{suppYw}),
this implies $P(Y_w)=0$. Hence
$$P\bigl(\cF_W(\frac{e^{\lambda+\rho}}{\prod_{\beta\in S} (1+e^{-\beta})})\bigr)=
P\bigl(\cF_{W(\Pi'_0)}(\frac{e^{\lambda+\rho}}{\prod_{\beta\in S} (1+e^{-\beta})})\bigr).$$

Since $(\rho,\alpha)=$ for each $\alpha\in\Pi'$, $\Pi'$ consists of isotropic roots.
Clearly, $S$ is a maximal set of  mutually orthogonal isotropic roots in $\Pi'$
(otherwise, $(S,\beta)=0$ for some $\beta\in\Pi'$, which
contradicts the KW-condition). In particular, from the denominator identity
(~(\ref{denomfin}) for $\lambda=0$) for $\Pi'$ 
one has
$$\cF_{W(\Pi'_0)}(\frac{e^{\lambda+\rho}}{\prod_{\beta\in S} (1+e^{-\beta})})\bigr)=
j(\Pi')R(\Pi') e^{\lambda+\rho},$$
where $j(\Pi')\not=0$.

Since $\Pi'\subset\Pi$ is orthogonal to $\lambda$, the highest weight vector
in $L(\lambda)$ generates the trivial module over the corresponding to $\Pi'$ subalgebra,
and so $L(\lambda)_{\lambda-\nu}=0$ for each $\nu\in\mathbb{Z}\Pi',\nu\not=0$.
Therefore $P(e^{\rho}\ch L)=1$. Since $S\subset \Pi$ we have 
$$P(Re^{\rho}\ch L)=P(R(\Pi') e^{\lambda+\rho}).$$

Then
$$P(Z)=P\bigl(j_{\lambda}Re^{\rho}\ch L-
\cF_W(\frac{e^{\lambda+\rho}}{\prod_{\beta\in S} (1+e^{-\beta})})\bigr)
=(j_{\lambda}-j(\Pi'))P(R(\Pi') e^{\lambda+\rho}).$$

The coefficient of $e^{\lambda+\rho}$ in the left-hand side is zero
and in the right-hand side is $j_{\lambda}-j(\Pi')$. Hence $j_{\lambda}=j_{\Pi'}\not=0$ and
$P(Z)=0$, as required. This completes the proof of the  KW-formula under the assumption~(\ref{assmrho}).

\subsection{Case $B(m,n)$}\label{Bmnspecial}
In this case there is no $\Pi$ such that
$(\rho,\alpha^{\vee})\geq 0$ for each $\alpha\in\Pi_0$.

We will show below that the KW-formula for any $(\Pi',S')$ is equivalent to the KW-formula for
$(\Pi,S)$ such that $(\rho,\alpha^{\vee})\geq 0$ for 
 all except one root in $\Pi_0$ (which is either 
 $\vareps_m$ or $2\delta_n$).
Moreover, by~(\ref{B1}), 
$\{\alpha\in\Delta|\ (\lambda+\rho,\alpha)=0\}$ is spanned
by
$$\Pi^0:=\{\alpha\in\Pi|\ (\lambda+\rho,\alpha)=0\}.$$
Since $(\lambda,\alpha^{\vee})\geq 0$ for each $\alpha\in\Pi_0$, $\Pi^0$ consists
of isotropic roots.
 
\subsubsection{}
Consider the case when $\alpha_{m+n}=\vareps_m$ (for $\alpha_{m+n}=\delta_n$
we interchange $\vareps$'s and $\delta$'s). Then $\alpha_{m+n-1}=
\delta_n-\vareps_m$.

Normalize $(-,-)$ by $||\vareps_i||^2=1$.
Then $(\lambda,\vareps_i)\geq 0\geq (\lambda,\delta_j)$ for each $i,j$. 
Since $r>1$, the KW-condition implies that 
$(\lambda,\vareps_n)=(\lambda,\delta_m)=0$ and
that $\Pi^0$ is connected (and contains $\delta_n-\vareps_m$).
Then
$$\Pi':=\Pi^0\cup \{\alpha_{m+n}\}$$
 is a connected subdiagram  containing $\alpha_{m+n}$.
Recall that $S$ is a maximal set of mutually orthogonal isotropic roots in $\Pi^0$,
so $\Pi'$ is $B(t,r)$ for $t=r$ or $t=r+1$, and
$(\lambda,\vareps_{m-i})=(\lambda,\delta_{n-j})=0$
for $0\leq i\leq r-1$.

Write 
$W(\Pi'_0)=W_+\times W_-$, where $W_-$ is the group of signed permutations of 
$\{\delta_{n-i}\}_{i=0}^{r-1}$ and $W_+$ is the group of signed permutations of 
$\{\vareps_{m-i}\}_{i=0}^{t-1}$.
We denote by $S_t, S_r$ the subgroups of unsigned permutations
in the corresponding groups, and by $w_-$ the longest element in
$W_-$ ($w_-\delta_i=-\delta_i$ for $i >n-r$).

\subsubsection{}
Since $(\lambda+\rho,\alpha^{\vee})\geq 0$ for $\alpha\in\Pi_0\setminus\{2\delta_n\}$,
$\lambda+\rho$ is maximal in $W(B_m)\times S_n$-orbit
(where $S_n$ is the subgroup of unsigned permutations in $W(C_n)\subset W$).

One has
$$(\lambda+\rho,\vareps_{m-i})=(\lambda+\rho,\delta_{n-j})=\frac{1}{2},$$
so $w_-(\lambda+\rho)$ is maximal in its $W$-orbit.
Since $\Delta^0$ is spanned by $\Pi^0$, we have
\begin{equation}\label{Bdd}
(\lambda+\rho,\delta_i)=\frac{1}{2}\ \text{ for } n-r<i\leq n;\ \ 
(\lambda+\rho,\delta_j)<-\frac{1}{2}\ \text{ for } 1\leq i\leq n-r.
\end{equation}
In particular,  $\Stab_W (\lambda+\rho)$ is $S_t\times S_r\subset W(\Pi'_0)$.

\subsubsection{}
Let $Z:=2^r r!Re^{\rho}\ch L-\cF_W\bigl(\frac{e^{\lambda+\rho}}{\prod_{\beta\in S}(1+e^{-\beta})}\bigr)$. 

Let us show that
\begin{equation}\label{nuka}
\supp Z\subset \lambda+\rho-(\mathbb{Z}_{\geq 0}\Pi\setminus\{0\}).
\end{equation}

Let $\rho'$ be the Weyl vector of $\Pi'$. Note that $\Pi',S$ satisfy the assumptions of~\S~\ref{sectKWass}, 
so~(\ref{denomfin}) for $B(t,r)$ gives
$$\cF_{W(\Pi'_0)} \bigl(\frac{e^{\rho'}}{\prod_{\beta\in S}(1+e^{-\beta})}\bigr)=
2^r r! \,\cF_{W_+} \bigl(\frac{e^{\rho'}}{\prod_{\beta\in S}(1+e^{-\beta})}\bigr).$$
The term in the left-hand side is  obviously $W(\Pi'_0)$-skew-invariant, so
$$\cF_{W(\Pi'_0)} \bigl(\frac{e^{\rho'}}{\prod_{\beta\in S}(1+e^{-\beta})}\bigr)=\sgn (w_-)
2^r r! \cF_{W_+}\, \bigl(w_-(\frac{e^{\rho'}}{\prod_{\beta\in S}(1+e^{-\beta})})\bigr).$$

Recall that $\sgn (w_-)=1$.
Since $\lambda$ and $\rho-\rho'$ are $W(\Pi'_0)$-invariant, 
using the denominator identity (see~(\ref{denomfin})) for $B(t,r)$, we obtain
\begin{equation}\label{BFW}
\begin{array}{l}
\cF_W\bigl(\frac{e^{\lambda+\rho}}{\prod_{\beta\in S}(1+e^{-\beta})}\bigr)=
\cF_{W/W(\Pi'_0)}\bigl(e^{\lambda+\rho-\rho'}
\cF_{W(\Pi'_0)}(\frac{e^{\rho'}}{\prod_{\beta\in S}(1+e^{-\beta})})\bigr)
\\=
2^r r! \,\cF_{W/W(\Pi'_0)}\bigl(e^{(\lambda+\rho-\rho'}
\cF_{W_+} \bigl(w_-(\frac{e^{\rho'}}{\prod_{\beta\in S}(1+e^{-\beta})})\bigr)\bigr)\\
=2^r r! \,\cF_{(W/W(\Pi'_0))\times W_+}\bigl(
w_-(\frac{e^{\lambda+\rho}}{\prod_{\beta\in S}(1+e^{-\beta})})\bigr),
\end{array}\end{equation}
where $W/W(\Pi'_0)$ is a set of coset representatives.

Hence
$$\cF_W\bigl(\frac{e^{\lambda+\rho}}{\prod_{\beta\in S}(1+e^{-\beta})}\bigr)=2^r r!\,
\sum_{w\in  (W/W(\Pi'_0))\times W_+} sgn (ww_-)\, Y_{ww_-},$$
where $Y_w:=\frac{e^{w(\lambda+\rho)}}{\prod_{\beta\in S}(1+e^{-w\beta})}$.

Write $W=W(B_m)\times W(C_n)$. Then $(W/W(\Pi'_0))\times W_+=(W(C_n)/W_-)\times W(B_m)$.
Consider the Bruhat order $\geq$ on $W(C_n)$. The following claim can be easily proven
by induction on the Bruhat order:
$$\forall y\in W(C_n)\ \ \exists z\in W_-\ \text{ such that } \  yz\geq w_-.$$
Using this claim we choose the set of representatives in $W(C_n)/W_-$ consisting of the 
elements which are larger than $w_-$ (with respect to the Bruhat order). Since
$w_-(\lambda+\rho)$ is maximal in its $W$-orbit, we get
$ww_-(\lambda+\rho)\leq \lambda+\rho$ for each $w\in (W(C_n)/W_-)\times W(B_m)$.

Take $w\in (W(C_n)/W_-)\times W(B_m)$.
Using~(\ref{suppYw}) we obtain 
$$\supp Y_{ww_-}\subset \lambda+\rho-\mathbb{Z}_{\geq 0}\Pi,$$
so $\supp Z\subset \lambda+\rho-\mathbb{Z}_{\geq 0}\Pi$.
Moreover, since $\Stab_W (\lambda+\rho)$ is equal
to $S_t\times S_r$ , we conclude that 
  $\lambda+\rho\in\supp Y_w$ forces  $w=w_-x$ for
$x\in W_+$. Combining~(\ref{BFW}) and~(\ref{denomfin})
 for $B(t,r)$, we 
conclude that the coefficient of $e^{\lambda+\rho}$ in
$\cF_W\bigl({e^{\lambda+\rho}}{\prod_{\beta\in S}(1+e^{-\beta})}\bigr)$ is equal to $2^r\,r!$.
Hence $\lambda+\rho\not\in\supp Z$. This establishes~(\ref{nuka}).
 
\subsubsection{}
Suppose that $Z\not=0$. Let $\lambda':=\lambda+\rho-\nu'$
be maximal in $\supp Z$. By~(\ref{nuka}), 
$\nu'\in\mathbb{Z}_{\geq 0}\Pi$. By~\Lem{lem0KW},
$\lambda'=w(\lambda+\rho-\mu)$
for $w\in W,\mu\in\mathbb{Z}S$.

Recall that $(\lambda'-\rho,\alpha^{\vee})\geq 0$ for $\alpha\in\Pi_0$.
One has $||\delta_i||^2=-1$ and
$(\rho,\delta_n)=\frac{1}{2}$; therefore $(\lambda',\delta_i)\leq \frac{1}{2}$ for each $i$
and the maximal element in $W$-orbit of $\lambda'$ is of the form
$$\lambda'':=\lambda'+\sum_{i|\  (\lambda'-\rho,\delta_i)=0}\delta_i.$$

Since  $(\lambda'-\rho,\delta_j-\delta_{j+1})\geq 0$, the set
$D:=\{\delta_i: (\lambda'-\rho,\delta_i)=0\}$ is either empty or of the form 
$\{\delta_{j},\delta_{j+1},\ldots,\delta_n\}$.
Assume that $\delta_{n-r}\in D$. Then $\delta_{i}\in D$ for $n-r\leq i\leq n$;
for such $i$ one has $(\rho,\delta_{i})=\pm\frac{1}{2}$, so
$(\lambda',\delta_{i})=\pm\frac{1}{2}$. 
Therefore $(\lambda+\rho-\mu, w\delta_{i})=\pm\frac{1}{2}$.
Recall that  for $j\leq n-r$ one has $(\mu,\delta_j)=0$
and $(\lambda+\rho,\delta_j)<-\frac{1}{2}$ by~(\ref{Bdd}).
Hence  for each $i=n-r,\ldots, n$ one has $w\delta_i=\pm\delta_j$, where
$n-r<j\leq n$, a contradiction. We conclude that
$$\lambda''=\lambda'+\sum_{i=s}^n \delta_i,\ \text{ where }n-r<s\leq n.$$

Since $\lambda''$ is maximal in $W\lambda'=W(\lambda+\rho-\mu)$, we have
$$\lambda+\rho-\mu=\lambda''-\nu=\lambda'+\sum_{i=s}^n \delta_i=\lambda+\rho-\nu'+\sum_{i=s}^n \delta_i,$$
for some $\nu\in\mathbb{Z}_{\geq 0}\Pi$.
Then
$$\nu+\nu'=\mu+\sum_{i=j}^n \delta_i\in \mathbb{Z}_{\geq 0}\Pi'$$
and $\nu,\nu'\in \mathbb{Z}_{\geq 0}\Pi$. Hence $\nu,\nu'\in\mathbb{Z}_{\geq 0}\Pi'$, that is
$\lambda'-\rho=\lambda-\nu'$ for $\nu'\in\mathbb{Z}_{\geq 0}\Pi'$.

Recall that $\Pi'$ is $B(t,r)$. Therefore $\mathbb{Q}\Pi'=\mathbb{Q} \Pi'_0$.
For each $\alpha\in\Pi'_0$ one has $(\lambda'-\rho,\alpha^{\vee})\geq 0$;
since $(\lambda,\alpha)=0$  we obtain $(\nu',\alpha^{\vee})\leq 0$. 
By Thm. 4.3 in~\cite{K2},
if $A$ is a Cartan matrix of a semisimple Lie algebra and $v$ is a vector
with rational coordinates, then $Av\geq 0$ implies $v>0$ or $v=0$.
Therefore 
$\nu'\in -\mathbb{Q}_{\geq 0}\Delta^+(\Pi')_{\ol{0}}$.
Since $\nu'\in\mathbb{Z}_{\geq 0}\Pi'$, $\nu'=0$, which contradicts~(\ref{nuka}).

This proves the KW-formula if $\Pi$ is such that $(\rho,\alpha^{\vee})\geq 0$ for all except
one root in $\Pi_0$.

\subsection{Dots and crosses diagrams}\label{dotcross}
Let $\fg$ be $A(m-1,n-1)=\fgl(m,n),\ B(m,n)$ or $D(m,n)$.

For a set of simple roots $\Pi$ we denote by $Iso$ the subdiagram of 
Dynkin diagram consisting of isotropic nodes. We call a set of $r$ mutually orthogonal isotropic 
simple roots {\em dense}
if $S$ is contained in a connected subdiagram consisting of $2r-1$ isotropic roots.

We will  show that if $L$ satisfies the KW-condition for some pair, 
then $L=L(\lambda,\Pi)$ satisfies the KW-condition 
for a pair $(\Pi,S)$, such that $Iso$  is connected and $S$ is dense
(plus some additional conditions for $B(m,n)$ and $D(m,n)$),
and the KW-formula for the former pair is equivalent to the KW-formula
for $(\Pi,S)$.

\subsubsection{}
In this section we encode subsets of simple roots for the root system
$\Delta(\fg)$ by diagrams described in~\cite{GKMP}. 
We recall this construction below.

Recall that the standard basis of $\fh^*$ consists of $\vareps_i$ with $i=1,\ldots,m$
and  $\delta_j$ with $j=1,\ldots,n$, which are mutually orthogonal. We can normalize
the form $(-,-)$ in such a way that 
$||\vareps_i||^2=1$ for each $i$ and $||\delta_j||^2=-1$ for each $j$.

Set 
$$\cE:=\{\vareps_i\}_{i=1}^m,\ \ \cD:=\{\delta_j\}_{j=1}^n,\ \ 
\cB=\cE\cup\cD.$$
We call two elements $v_1,v_2\in \cB$
{\em elements of the same type} if $||v_1||^2=||v_2||^2$ (i.e.,
$\{v_1,v_2\}\subset \cE$   or $\{v_1,v_2\}\subset\cD$)
and {\em elements of different types} otherwise.

Fix a total order $>$ on $\mathcal B=\{\xi_1>\ldots>\xi_{m+n}\}$ and define the corresponding set of simple roots $\Pi(\mathcal B,>)$
as follows

$$\begin{tabular}{l|l}
$\fg$&$\Pi(\mathcal B,>)$\\
\hline\\
$A(m-1,n-1)$&$\{\xi_i-\xi_{i+1}\}_{i=1}^{m+n-1}$\\
$B(m,n)$&$\{\xi_i-\xi_{i+1}\}_{i=1}^{m+n-1}\cup\{\xi_{m+n}\}$\\
$D(m,n)$&$\{\xi_i-\xi_{i+1}\}_{i=1}^{m+n-1}\cup\{2\xi_{m+n}\}$ if $\xi_{m+n}\in\cD$\\
 & $\{\xi_i-\xi_{i+1}\}_{i=1}^{m+n-1}\cup\{\xi_{m+n-1}+\xi_{m+n}\}$ if $\xi_{m+n}\in\cE$.
\end{tabular}
$$
We  encode a subset $\Pi$ of simple roots for the root system
$\Delta(\fg)$  by the ordered set $\mathcal B$, 
which is  pictorially represented by an ordered sequence of
dots and crosses, the former corresponding to vertices in $\cE$ and the latter to vertices in $\cD$.

For instance, the sequence
$\ \ \cdot \ \ \cdot \ \ \times \ \ \times\ \ $
encodes $\{\vareps_1-\vareps_2,\varesp_2-\delta_1,\delta_1-\delta_2\}$ for $A(1,1)$,
$\{\vareps_1-\vareps_2,\varesp_2-\delta_1,\delta_1-\delta_2,\delta_2\}$ for $B(2,2)$
and $\{\vareps_1-\vareps_2,\varesp_2-\delta_1,\delta_1-\delta_2,2\delta_2\}$ for $D(2,2)$.

For each $u,v\in\cB$, $u-v$ is a root. We call $u,v$ the {\em ends} of $\alpha=u-v$.
The root $u-v$ is isotropic if $u,v$ are  of different types and is simple if
$u,v$ are neighbors.

A simple odd reflection
$r_{v-w}$ with $v,w\in\mathcal B$ corresponds to  the switch of  consecutive
vertices $v,w$ in the ordered sequence ($v$ and $w$ should be of different types).

 For $v,w\in\mathcal B$ denote by $[v,w]$ the (ordered) set
of elements of $\mathcal B$ lying between $v$ and $w$, namely,  if $v>w$, then
$[v,w]=\{u\in \mathcal B|\ v\geq u\geq w\}$ and by $]v,w[$ the set $[v,w]\setminus\{v,w\}$.

Let $[v,w]$ be any interval and $|[v,w]\cap\cE|=k,\ |[v,w]\cap\cD|=l$.
We denote by $\Pi([u,v])$ the corresponding set of simple roots of $A(k,l)$-type:
if $[v,w]=\{u_1>u_2>\ldots>u_s\}$, then $\Pi([u,v]):=\{u_1-u_2,\ldots,u_{s-1}-u_s\}$.
Note that each permutation of dots and crosses in $[v,w]$ correspond to a sequence of odd reflections
and thus to a choice of another set of simple roots in $A(k,l)$.

Let $W_{[v,w]}$ be the Weyl group of $\Pi([u,v])$
(the subgroup of $W$ consisting of (non-signed) permutations of $[v,w]\cap\cE$ and of $[v,w]\cap\cD$), 
so $W_{[v,w]}\cong S_k\times S_l$.

We say that $[v,w]$ is {\em balanced} if $\Pi([v,w])$ has a maximal possible number
of mutually orthogonal isotropic roots; in other words, if $[v,w]$ consists of $k_1$ dots
and $k_2$ crosses, then $[v,w]$ is balanced if it contains $\min (k_1,k_2)$ disjoint pairs
consisting of neighboring vertices of different types. 

Any $\nu\in\fh^*$ can be written in the form $\nu=\sum_{u\in\cB} y_u u$ for some scalars $y_u$;
we define the {\em restriction of } $\nu$ to $[v,w]$ by the formula
$$\nu_{[u,v]}:=\sum_{u\in [v,w]} y_u u.$$
We say that $\nu_{[u,v]}$ is {\em trivial } if $(\nu,\alpha)=0$ for
all $\alpha\in \Pi([v,w])$; it means that $A(k,l)$-module of highest weight
$\nu_{[u,v]}$ is one-dimensional (where $\Pi([v,w])$ 
is a set of simple roots for $A(k,l)$).

 A set $S$ of mutually orthogonal simple 
isotropic roots is represented by a set of disjoint pairs
consisting of neighboring vertices of different types. It can be encoded by the set of 
these vertices, which we denote by $\supp S$. 
Note that $S$ is dense if
$\supp S\subset\cB$ form an interval consisting of the elements of alternating types.

\subsubsection{}
For the cases $A(m-1,n-1)$ and $B(m,n)$,
using the description of Borel subalgebras in~\cite{K1}, 
it is not difficult to show that
any set of positive  roots for $\fg$ (satisfying $\Delta^+(\Pi)\cap \Delta_{\ol{0}}=\Delta^+_{\ol{0}}$)
is of the form  $\Pi(\mathcal B,>)$  for  some total order  $>$ on $\mathcal B$.

Consider the case $D(m,n)$.
Recall that $D(m,n)$ has an autormorphism $\iota_D$ (induced by a
Dynkin diagram automorphism of $D_m$) which preserves $\fh$ and
satisfies $\iota_D(\delta_i)=\delta_i$ for $1\leq i\leq n$,
 $\iota_D(\vareps_i)=\vareps_i$ for $1\leq i\leq m-1$ and
  $\iota_D(\vareps_m)=-\vareps_m$.
Any set of positive  roots for $D(m,n)$ is of the form 
$\Pi(\mathcal B,>)$ or $\iota_D (\Pi(\mathcal B,>))$ 
 for  some total order  $>$ on $\mathcal B$.
If $\delta_n+\vareps_m\in S$, then $\delta_n\pm\vareps_m\in\Pi$ and
$\iota_D(\Pi)=\Pi$. Thus either $(\Pi,S)$ or $(\iota(\Pi),\iota(S))$
are such that $\Pi=\Pi(\mathcal B,>)$ and $S$ is of the form
$S=\{u_i-v_i\}_{i=1}^r$, $u_i,v_i\in\cB$.
Therefore, using $\iota_D$, we can (and will ) assume that $\Pi=\Pi(\mathcal B,>)$ and
$S$ is of this  form.

\subsubsection{}
Summarizing, we consider  $\Pi=\Pi(\mathcal B,>)$, which is encoded by 
a sequence of $m$ dots and $n$ crosses, and $S\subset \Pi$ is encoded by $\supp S$,
which is a subset of $\cB$, 
consisting of $r$ pairs of neihboring dots and  crosses.

\subsubsection{}
\begin{lem}{lemuse}
If $\alpha\in\Pi_0$ is such that $(\rho,\alpha)=0$, then $\alpha=\beta+\beta'$, where $\beta,\beta'\in\Pi$
are isotropic.
\end{lem}
\begin{proof}
Mark each element of $u\in \cB$ by the number $(\rho,u)$.
If $u>v$ are neighboring in $\cB$, then $(\rho,u)=(\rho,v)$ if $u,v$
are of different types and $(\rho,u)-(\rho,v)=||u||^2$ otherwise.

If $u-v\in\Pi_0$ for $u,v\in\cB$, then $u,v$ are of the same type and there are no other elements
of these type between them. This means that $(\rho,u)-(\rho,v)=(1-t)||u||^2$,
where $t$ is the number of elements in $]u,v[$.

Take $\alpha\in\Pi_0$ such that $(\rho,\alpha)=0$. 

If $\alpha=u-v$, then, by above,
 $]u,v[$ consists of one element, say $w$, and $u,w$ are of different types.
 Hence $u-w,v-w$ are isotropic simple roots and $u-v=(u-w)+(w-v)$
as required. This establishes the claim for $A(m,n)$.

For $B(m,n)$ the last mark is equal to $\pm 1/2$, so all marks are not integral.
If $\alpha\not=u-v$ for $u,v\in\cB$, then $\alpha=u$ or $\alpha=2u$ for some $u\in\cB$, and thus
$(\rho,\alpha)\not=0$.

For $D(m,n)$ the last mark is zero if the last element is a dot
and is $-1$ otherwise. If $\alpha\not=u-v$ for $u,v\in\cB$, then $\alpha=2\delta_n$,
 or $\alpha=\vareps_{m-1}+\vareps_m$. If $\alpha=2\delta_n$, then (since $\delta_n$
 is represented by the last cross), $(\rho,\delta_n)=t'1$, where
 $t'$ is the number of dots after the last cross. Hence $(\rho,\alpha)=0$ forces
 $t'=1$, that is $\delta_n\pm\vareps_m\in\Pi$, which gives
 $2\delta_n=(\delta_n-\vareps_m)+(\delta_n+\vareps_m)$, as required.
Consider the remaining case  $\alpha=\vareps_{m-1}+\vareps_m$.
By above,  $(\rho,\vareps_{m-1}-\vareps_m)=1-t$, where $t$ is the number
of elements in $]\vareps_{m-1},\vareps_m[$. Note that $\vareps_m$ is the last dot in $\cB$, so
$(\rho,\vareps_m)=-t'$, where
$t'$ is the number of crosses after $\vareps_m$. 
Then $0=(\rho,\vareps_{m-1}+\vareps_m)=1-t-2t'$, so $t'=0$ and $t=1$.
This means that $\vareps_{m-1}-\delta_n,\delta_n\pm\vareps_m$ are isotropic simple roots, and
$\vareps_{m-1}+\vareps_m=(\vareps_{m-1}-\delta_n)+(\delta_n+\vareps_m)$
is the required presentation.
\end{proof}

\subsection{Equivalence of the KW-formulas}\label{changePiS}
Let $L=L(\lambda,\Pi)$  be a finite-dimensional
$\fg$-module, satisfying the  KW-condition for $(\Pi,S)$.

In~\Lem{lemchPIS} we show that if $\beta\in\Pi$ is isotropic and
 $(\beta,S)=0$, then $L$ satisfies the KW-condition for
$(r_{\beta}\Pi,r_{\beta}S)$ and the corresponding KW-formulas
are equivalent (one has
$r_{\beta}S=S$ if $\beta\not\in S$ and
$r_{\beta}S=(S\setminus\{\beta\})\cup\{-\beta\}$ if $\beta\in S$).
In particular, we can permute dots and crosses in any interval
$[u,v]$ such that $[u,v]\cap \supp S=\emptyset$; then
$L$ satisfies the KW-condition for  the resulting $\Pi'$ and the original $S$,
and the corresponding KW-formulas are equivalent.

In~\Lem{lemmay2} we assume that $(\cB,>)$ contains an interval $[u,v]$
 such that $\supp S\subset [u,v]$ and
 $(\lambda,u-u')=0$ for each $u'\in [u,v]$. Let 
 the ordered set $(\cB,'>)$ be obtained from $(\cB,>)$ by some  permutation of dots and crosses
  in $[u,v]$ in such a way that with the new ordering this interval is balanced
 (with a  maximal set of mutually orthogonal isotropic roots $S'$). 
 In this case $L=L(\lambda,\Pi')$ and $L$ satisfies the KW-condition for $(\Pi',S')$, where $\Pi'=\Pi(\cB,'>)$; moreover, the corresponding KW-formulas are equivalent.

\subsubsection{}
\begin{lem}{lemchPIS}
Let $L=L(\lambda,\Pi)$ be a finite-dimensional $\fg$-module which satisfies the KW-condition for
$(\Pi,S)$. Let $\beta\in\Pi$ be an isotropic root orthogonal to all elements of $S$.

(i)  The KW-condition  holds  for $(\Pi',S'):=(r_{\beta}\Pi,r_{\beta}S)$.
 
(ii) Let $L=L(\lambda',\Pi')$ and $\rho'$ be the Weyl vector for $\Pi'$.
Then
$$\cF_{W} \bigl(\frac{e^{\lambda+\rho}}{\prod_{\beta\in S}(1+e^{-\beta})}\bigr)\in \cR(\Pi),\ \ 
\cF_{W} \bigl(\frac{e^{\lambda'+\rho'}}{\prod_{\beta\in S'}(1+e^{-\beta})}\bigr)\in\cR(\Pi')$$
are equivalent. 

In particular, formulas~(\ref{333}) for $(S,\Pi)$  and $(S',\Pi')$ 
are equivalent.
\end{lem}
\begin{proof}
The proof of (i) is straightforward. For (ii) set
$$X:=\cF_{W} \bigl(\frac{e^{\lambda+\rho}}{\prod_{\beta\in S}(1+e^{-\beta})}\bigr)\in \cR(\Pi),\ \ 
X':=\cF_{W} \bigl(\frac{e^{\lambda'+\rho'}}{\prod_{\beta\in S'}(1+e^{-\beta})}\bigr)\in\cR(\Pi').$$

Note that $X,X'$ are finite sums; for each isotropic root
$\beta_1\in\Pi$ (resp., $\beta_1\in\Pi'$) $X$ (resp., $X'$) has a pole of order $\leq 1$
at $\beta_1$.  

If $\beta\not\in S$, then $S=S'$; moreover, the
KW-condition implies that $(\lambda+\rho,\beta)\not=0$,
so $\lambda+\rho=\lambda'+\rho'$. Hence  $X,X'$
are the expansions of the same element in $\cR(\Pi)$ and in $\cR(\Pi')$,
so they are equivalent.

If $\beta\in S$, then $\lambda=\lambda'$ and so
$$\frac{e^{\lambda'+\rho'}}{\prod_{\beta\in S'}(1+e^{-\beta})}=
\frac{e^{\lambda+\rho}}{\prod_{\beta\in S'}(1+e^{-\beta})}.$$
Then again $X,X'$
are the expansions of the same element and they are equivalent.
\end{proof}

\subsubsection{}
\begin{lem}{lemmay2}
Let $L=L(\lambda,\Pi)$ satisfy the KW-condition for $(\Pi,S)$ and 
let $\ \Pi=\Pi(\cB,>)$. Assume that  $[u,v]\in (\cB, >)$ 
is such that $\supp S\subset [u,v]$ and
$\lambda_{[u,v]}$ is trivial. Let the ordered set
 $(\cB,'>)$ be obtained from $(\cB,>)$ by permuting 
 some dots and crosses in $[u,v]$:
we denote the resulting interval in $(\cB,'>)$ by $[u',v']$ 
($[u,v]=[u',v']$ as non-ordered
sets). Then
 
(i) $L=L(\lambda,\Pi')$, where $\Pi':=\Pi(\cB,'>)$.

(ii)   If $[u',v']$ is balanced, then $L$ satisfies the KW-condition for $(\Pi',S')$,
where $S'$ is a maximal set of mutually orthogonal isotropic roots in $\Pi([u',v'])$.

Moreover, the KW-formula for $(\Pi,S)$ is equivalent to the KW-formula for $(\Pi',S')$.
\end{lem}
\begin{proof}
Let $\Pi([u,v])=A(k,l)$, where $k\leq l$. Then
(i) follows from the fact the one-dimensional $A(k,l)$-module
has the same highest weight  for any choice
of simple roots. 

For (ii) assume that  $[u',v']$ is balanced and $S'$ is a maximal set of mutually orthogonal isotropic roots in $\Pi([u',v'])$.  Then $S'$ contains $k$ elements. 
Since $L$ satisfies the KW-condition,
$[u,v]$ is balanced and $S$ contains $k$ elements. Since
$(\lambda,\alpha)=0$ for each $\alpha\in\Pi([u,v])$, one has $(\lambda,\alpha)=0$ for 
 each $\alpha\in S'$. Hence $L$ satisfies the KW-condition for $(\Pi',S')$.

It remains to show that the KW-formulas are equivalent.

Denote by $W/W_{[u,v]}$  a set of coset representatives.
Note that $\lambda-\lambda_{[u,v]}$ and $\rho-\rho_{[u,v]}$ 
are $W([u,v])$-stable.
Since the denominator identity for $A(k,l)$ holds for $(\Pi([u,v]),S)$ and for $(\Pi([u',v']), S')$,
we have
$$\begin{array}{l}
\cF_W \bigl( \frac{e^{\lambda+\rho}}{\prod_{\beta\in S}(1+e^{-\beta})}\bigr)=
\cF_{W/W_{[u,v]}}\cF_{W_{[u,v]}} \bigl( \frac{e^{\lambda+\rho}}{\prod_{\beta\in S}(1+e^{-\beta})}\bigr)\\
=\cF_{W/W_{[u,v]}} \bigl(e^{\lambda-\lambda_{[u,v]}+\rho-\rho_{[u,v]}}
\cF_{W_{[u,v]}} 
\bigl( \frac{e^{\lambda_{[u,v]}+\rho_{[u,v]}}}{\prod_{\beta\in S}(1+e^{-\beta})}\bigr)\bigr)\\=
\cF_{W/W_{[u,v]}} \bigl(e^{\lambda-\lambda_{[u,v]}+\rho-\rho_{[u,v]}}
\cF_{W_{[u,v]}} 
\bigl( \frac{e^{\lambda_{[u,v]}+\rho"}}{\prod_{\beta\in S'}(1+e^{-\beta})}\bigr)\bigr)\\=
\cF_W \bigl( \frac{e^{\lambda+\rho-\rho_{[u,v]}+\rho"}}{\prod_{\beta\in S'}(1+e^{-\beta})}\bigr),
\end{array}$$
where $\rho''$ is  the Weyl vector for $\Pi([u',v'])$. 
Recall that $\Pi([u',v'])$ can be obtained from
$\Pi([u,v])$  by a sequence of odd reflections $r_{\beta_1}\ldots r_{\beta_t}$. Then
 $\rho''=\rho_{[u,v]}+\sum_{i=1}^t \beta_i$. 
Since $\Pi'$ is obtained from $\Pi$ by the same sequence of odd reflections,  we have
$\rho'=\rho-\rho_{[u,v]}+\rho"$ is the Weyl vector for $\Pi'$.

We conclude that $\cF_W \bigl( \frac{e^{\lambda+\rho}}{\prod_{\beta\in S}(1+e^{-\beta})}\bigr)\in\cR(\Pi)$
and $\cF_W \bigl( \frac{e^{\lambda+\rho'}}{\prod_{\beta\in S'}(1+e^{-\beta})}\bigr)\in\cR(\Pi')$
are equivalent elements. The assertion follows.
\end{proof}

\subsection{Properties of $\Pi,S$}
\label{KWfin}
For each $u\in\cB$ set $y_u:=(\lambda,u)$.

Let $u>v\in\cB$ and $||u||^2=||v||^2$. Then $u-v\in\Pi_0$ and,
since $L(\lambda,\Pi)$ is finite-dimensional, $(\lambda,(u-v)^{\vee})=(\lambda,u-v)||u||^2\geq 0$. This gives
$y_u ||u||^2\geq y_v||v||^2$. Moreover, since the irreducible $A(1,0)$-module of the highest weight $a\vareps_1-b\delta_1+a\vareps_2$
is finite-dimensional only for $b=a$,  we obtain
that $y_u=y_v$ forces $y_w=y_u$ for each $w\in [u,v]$
($y_w$ is constant for  $w\in [u, v]$).

Let $u_S$ (resp., $v_S$) be the smallest (resp., largest) element in $\supp S$. Take $u',v'$ such that
$u_S-u',v_S-v'\in S$.
Recall that $u_S,u'$ and $v_S,v'$ are neighbors of different types and $y_{u_S}=y_{u'}, y_{v_S}=y_{v'}$.  
By above,  $y_w$ is constant for  $w\in [u_S, v_S]$. 
Since $[u_S,v_S]$ contains elements of different types, the set
$\{w\in \cB|\ y_w=y_{u_S}\}$ is an interval (containing $[u_S,v_S]$).
Denote this interval by $[u,v]$ ($u\geq u_S>v_S\geq v$). Then
$(\lambda,\alpha)=0$
for each $\alpha\in\Pi([u,v])$. In particular, $[u,v]$ is balanced and
$S\subset \Pi([u,v])$ is a maximal set of mutually orthogonal isotropic roots.

For $B(m,n)$, finite-dimensionality of $L$ implies $y_w ||w||^2\geq 0$ for each $w\in \cB$.
This gives $y_u=0$ and thus $y_w=0$ for each $w<u$.  Hence 
$v$ is the minimal element in $\cB$ and $\lambda_{[u,v]}=0$.

\subsubsection{}\label{assmD}
For  $D(m,n)$, finite-dimensionality of $L$ implies $y_w ||w||^2\geq 0$ for each $w\in \cB\setminus \{\vareps_{m}\}$
and $|y_{\vareps_i}|\geq |y_{\vareps_m}|$ for $i<m$.
In particular, if $(\lambda,u_S)\not=0$, then $r=1$
$S=\{\delta_k-\vareps_m\}$ or $S=\{\vareps_m-\delta_k\}$.
If $(\lambda,u_S)=0$, then 
 $(\lambda,u)=0$ for each $u\in\supp S$,
and, as for $B(m,n)$, we obtain  that $v$ is the minimal element in $\cB$
and $\lambda_{[u,v]}=0$. Note that $[u,v]$ contains at least $r$ elements
of each type, so $\vareps_m,\delta_n\in [u,v]$.

We assume for the rest of Section~\ref{sectfindim} that {\em either
$(\lambda,u_S)=0$ or $S=\{a(\delta_n-\vareps_m)\}$, where $a=\pm 1$}.
Thus we exclude the case $(\lambda,u_S)\not=0$ and
$S=\{a(\delta_k-\vareps_m)\}$ with
$k<n$ .

\subsubsection{}
\begin{lem}{lemlambda+}
Let the $\fg$-module $L=L(\lambda,\Pi)$ satisfy the KW-condition for $(\Pi,S)$, 
where $Iso$
is connected and $S$ is dense. Set $\lambda^0:=(\lambda+\rho,u_S)$.

(i) One has $(\lambda+\rho,\alpha^{\vee})\geq 0$ if $\alpha\in\Pi_0$ and 
$\alpha=u-v$ for $u,v\in\cB$;

(ii) The multiset $\{(\lambda+\rho,u)\}_{u\in\cB}$ contains $2r$ or $2r+1$ copies of
$\lambda^0$ and all other elements are distinct.
The set $\{x\in\cB|\ \ (\lambda+\rho,x)=\lambda^0\}$ form an interval $[u_0,v_0]$,
which contains $[u_S,v_S]$.
\end{lem}

\begin{rem}{}
Note that  $[u_S,v_S]$ contains $2r$ elements and $[u_0,v_0]$ contains $2r$ or $2r+1$ elements.
\end{rem}

\begin{proof}
Take $u\in [u_S,v_S]$. From~\S~\ref{KWfin} one has
$(\lambda,u)=(\lambda,u_S)$; since $S$ is dense
one  has $(\rho,u)=(\rho,u_S)$. Therefore 
$(\lambda+\rho,u)=(\lambda+\rho,u_S)$.

Let $u,v\in\cB$ be of the same type and $u>v$.
Since $Iso$ is connected,
$(\rho,(u-v)^{\vee})\geq 0$ and if  $(\rho,(u-v)^{\vee})=0$,
then $(\rho,w)$ is constant for $w\in [u,v]$ (so, $[u,v]$ 
consists of the elements of alternating types).
 By~\S~\ref{KWfin}
the same holds for $\lambda$. This proves (i) and, moreover, shows that $(\lambda+\rho,u)=(\lambda+\rho,v)$
implies  that $\lambda$ and $\rho$ are constant on $[u,v]$.
If $[u,v]\setminus [u_S,v_S]$ contains more than one element, then
it contains two neghboring elements $u',v'$ of different types.
However, $(\lambda+\rho,u'-v')=0$ and $(u'-v',S)=0$, which contradicts
the KW-condition.
Hence $[u,v]\setminus [u_S,v_S]$ contains at most one element
(in particular, $u$ or $v$ lies in $[u_S,v_S]$). This proves (ii).
\end{proof}

\begin{cor}{coreqAD}
Let $L=L(\lambda,\Pi)$ satisfy the KW-condition for $(\Pi,S)$, where $Iso$
is connected and $S$ is dense. If $\Pi$ is 
$A(m-1,n-1)$, or if $\Pi$ is $D(m,n)$ with $\delta_n-\vareps_m\in\Pi$, then
$$
(\rho,\alpha^{\vee})\geq 0\ 
\text{ for each }\alpha\in\Pi_0.$$
\end{cor}

\subsection{Choice of $(\Pi,S)$}\label{chPIS}
Finally, we show that for $A(m,n), B(m,n)$ and $D(m,n)$, 
if a finite-dimensional  $\fg$-module $L$ 
satisfies the KW-condition for $(\tilde{\Pi},\tilde{S})$,
and, for $D(m,n)$ the assumption of~\S~\ref{assmD} is fulfilled,
 then the KW-formula is equivalent to the KW-formula
for $(\Pi,S)$, where
$Iso$ is connected, $S$ is dense, and, for $D(m,n)$, $\delta_n-\vareps_m\in \Pi$.

\subsubsection{}
Let $[u,v]$ be the interval constructed in the second paragraph in~\S~\ref{KWfin}.
Using~\Lem{lemmay2} we can rearrange dots and crosses in 
$[u,v]$ such that the resulting interval $[u',v'], u'>v'$, 
is balanced. We do this in such a way that
the interval has first a segment of the elements of same type and then a segment of elements of alternating types;
for $D(m,n)$ we choose the minimal element to be  $\vareps_m$. We choose
$S'$ such that $\supp S'$ consists of the last $2r$ elements
in $[u',v']$; then $S'$ is dense and $v'\in\supp S'$.

Using~\Lem{lemchPIS} we now permute dots and crosses in the intervals
$\{w\in\cB|\ w>u''\}$ and $\{w\in\cB|\ w<v'\}$; the second interval is empty for $B(m,n)$ and 
for $D(m,n)$ (if the assumption in~\S~\ref{assmD} holds).
We do this in such a way that in the resulting order we have $\cB=[u_0,u_+]\cup [u_+,v_+]\cup [v_+,v_0]$,
where 
$u_0\geq u_+\geq u''>v'\geq v_+\geq v_0$, the interval $[u_0,u_+]$ (reps., $[v_+,v_0]$) consists of the elements
of the same type, and the interval $[u_+,v_+]$ consists of the elements of alternating types.

\subsubsection{}
Consider the resulting ordering $(\cB,>)$ and set $\Pi:=\Pi(\cB,>)$, $S:=S'$.
Let $\rho$ be the corresponding Weyl vector and $\lambda$ be the highest weight of $L$.

From Lemmas~\ref{lemmay2}, \ref{lemchPIS} we conclude that the  KW-formula
is equivalent to the KW-formula for $(\Pi,S)$.

Since $\cB=[u_0,u_+]\cup [u_+,v_+]\cup [v_+,v_0]$ as above, we have obtained
$\Pi, S$ for which $Iso$ is connected and $S$ is dense, which completes
the proof for $A(m,n)$.

In addition, 
for $D(m,n)$ we have obtained that $\vareps_m$ is minimal in $\cB$ and $\delta_n-\vareps_m\in S$. Therefore $\Pi$ contains $\delta_n\pm\vareps_m$ and
 $S=\{\delta_{n-i}-\vareps_{m-i}\}_{i=0}^{r-1}$. This completes the proof for $D(m,n)$.

\subsubsection{}\label{lambda+B}
For $B(m,n)$ we have obtained
$S=\{\delta_{n-i}-\vareps_{m-i}\}_{i=0}^{r-1}$
or $S=\{\vareps_{m-i}-\delta_{n-i}\}_{i=0}^{r-1}$.
Retain notations of~\Lem{lemlambda+}.
Recall that $v_S=v_0$ is minimal in $\cB$
and $(\lambda,v)=0$ for $v\in [u_S,v_S]$.
We will show that $(\Pi,S)$ can be chosen in such a way that
$Iso$ is connected, $S$ is dense and for $u,v\in\cB$
we have
\begin{equation}\label{B1}
|(\lambda+\rho,u)|=|(\lambda+\rho,v)|\ \ \Longrightarrow\ (\lambda+\rho,u)
=(\lambda+\rho,v)=(\lambda+\rho,\alpha_{m+n}).
\end{equation}

Since $v_S$ is minimal, $\cB$ is of the following form: 
$\xi_1>\ldots >\xi_k$ are of the same type and
$\xi_k>\xi_{k+1}>\ldots >\xi_{m+n}=v_S$ are of alternating types
(i.e., $Iso$ is connected and contains $\alpha_{m+n-1}$).
This implies $(\rho,u-v)=0$ if $||u||^2=||v||^2=-||\xi_1||^2$.

Set $x_u:=(\lambda+\rho,u)$.
Normalize $(-,-)$ by the condition $||v_S||^2=1$.
Then $\lambda^0=\frac{1}{2}$.

Let us show that $|x_u|=|x_v|$ forces $u$ or $v$ in $\supp S$.

Indeed, if $x_u=x_v$, then, by~\Lem{lemlambda+}, $u$ or $v$ is in $\supp S$.
If  $|x_u|=|x_v|$
and $u,v$ are of different types, 
then KW-condition forces $u$ or $v$ in $\supp S$.
Consider the remaining case, when $x_u=-x_v$, $u,v$ are of the same type
and $u\not\in \supp S$.
 If $||u||^2=||v||^2=1$, then 
$x_u,x_v\geq \frac{1}{2}$, so $x_u\not=-x_v$.
If  $||u||^2=||v||^2=-1$, then 
$x_u,x_v\leq \frac{1}{2}$, so $x_u=-x_v$ forces
$|x_u|=|x_v|=\frac{1}{2}$. Then $x_u$ or $x_v$ is  $\frac{1}{2}$.
By~\Lem{lemlambda+},
 $u$ or $v$ is in $[u_0,v_0]$ and
$[u_0,v_0]\setminus [u_S,v_S]$ is either empty, or is 
$\{u_0\}$ and $||u_0||^2=-||u_S||^2=||v_S||=1$. Thus $u$ or $v$ is in $[u_S,v_S]=\supp S$,
as required.

We conclude that the multiset $\{|x_u|\}$ contains $t$ copies of $\frac{1}{2}$, where
$t=2r$ or $t=2r+1$, and all other elements are distinct. Set $B:=\{u\in \cB| \ |x_u|=\frac{1}{2}\}$.
Clearly, $[u_0,v_0]\subset B$ and~(\ref{B1})
 holds if $[u_0,v_0]=B$. In particular,~(\ref{B1}) holds if $[u_0,v_0]$ contains $2r+1$ elements.

Consider the case when~(\ref{B1}) does not hold.
Since $[u_0,v_0]$ contains $[u_S,v_S]$ which has $2r$ elements,
this means that $[u_0,v_0]=[u_S,v_S]$  and $B=[u_S,v_S]\coprod\{u\}$, where
$x_u=-\frac{1}{2}$. Then $||u||^2=-1$, so 
 $u,u_S$ are of the same type. 
If $w\in [u,u_S]$ is of the same type as $u$,  then $x_w\in [-\frac{1}{2},\frac{1}{2}]$
(since $x_{u_S}=\frac{1}{2}=-x_u$),
so $x_w$ is $\pm \frac{1}{2}$, that is $w\in B$. This means
that either $w=u$ or $w=u_S$. Hence
$]u,u_S[$ does not contain elements of the same type as $u,u_S$;
that is either $]u,u_S[=\emptyset$ or $]u,u_S[=\{v\}$ with
$||u||^2=-||v||^2$.
 If $]u,u_S[=\{v\}$, we make
the reflection with respect to $u-v$ and obtain a new ordered set
$(\cB,'>)$; one has $L=L(\lambda',\Pi')$ ($\Pi'=\Pi(\cB,'>)$), where $\lambda'$ is
such that $\lambda'+\rho'=\lambda+\rho$, so $(\lambda'+\rho',u)=-\frac{1}{2}$
and in $\cB'$ one has $]u,u_S[=\emptyset$.
Thus, in both cases (for $\cB$ if $]u,u_S[=\emptyset$, and for $\cB'$
otherwise)  we have $]u,u_S[=\emptyset$. Then $[u,v_0]$ contains
$2r+1$ elements ($r+1$ elements of type $u_S$ and $r$ elements of type $v_0$).
Note that the restriction of $\lambda$ (resp., $\lambda'$) to $[u,v_0]$
is zero weight, since $(\rho,u-u_S)=-1$. 

Using~\Lem{lemmay2} we can rearrange dots and crosses in $[u,v_0]$ 
in an alternating way;
the resulting interval is $[u,v']$, where $||v'||^2=||u||^2=-1$
and $v'$ is minimal in the new order on $\cB$. Since 
$\lambda_{[u,v_0]}=0$, this rearrangement
preserves $\lambda$ (i.e., $\lambda$ is the highest weight of $L$
with respect to the new set of simple roots). 
Using~\Lem{lemchPIS},
we rearrange dots and crosses in the rest of $\cB$ (in $[u_0,u[$)
in such a way that
in the resulting order we, again, have first several elements
of the same type and then a segment of elements of alternating types.
Let $\Pi''$ be the new set of simple roots, $\rho''$ be the new Weyl vector
and $\lambda''$ is such that $L=L(\lambda'',\Pi'')$. Then
$\lambda''_{[u,v']}=0$, so $(\lambda''+\rho'',w)=-\frac{1}{2}$
for each $w\in [u,v']$. Note that $Iso$ in $\Pi''$ is connected;
take $S''$ such that $\supp S''$ consists of last $2r$ elements
(i.e., $\supp S'=]u,v']$). Then $S''$ is dense.
Since $[u,v']$ contains $2r+1$ elements and $(\lambda''+\rho'',w)=-\frac{1}{2}$
for each $w\in [u,v_0]$,~(\ref{B1}) holds for $(\Pi'',S'')$.

\section{KW-character formula for integrable (near) vacuum modules and arbitrary $h^{\vee}$}
\label{sectKW0}

In this section $\fg$ is a symmetrizable affine Lie superalgebra (with arbitrary $h^{\vee}$). Let $\Delta$ be  the root system of $\fg$ and
let $\dot{\Delta}$ be a finite part of $\Delta$.

We say that a subset of simple  roots $\Pi$ for $\Delta$ is compatible
with $\dot{\Delta}$ if $\dot{\Pi}=\dot{\Delta}\cap\Pi$ is 
a subset of simple  roots for $\dot{\Delta}$ (in other words, 
$\Pi\setminus \dot{\Delta}$ contains only one root).

Let $L=L(\lambda,\Pi)$ be a  non-critical  integrable $\fg$-module
of maximal atypicality,
where $\Pi$ is compatible with  $\dot{\Delta}$ and $L$ satisfies
the KW-condition for $\Pi,S$ with $S\subset \dot{\Pi}$.

We prove the following KW-character formula:

\begin{equation}\label{vacform}
Re^{\rho}\ch L(\lambda)=j_{\lambda}^{-1}\sum_{w\in W(\dot{\Pi}_0\cup\pi)} sgn(w)\ w\bigl(\frac{e^{\lambda+\rho}}{\prod_{\beta\in S} (1+e^{-\beta})}\bigr),
\end{equation}
where $j_{\lambda}$ is the number of elements in the "smallest"
factor of $W(\dot{\Pi})$, see~(\ref{jlambda}).

We also prove this formula for non-critical $C_n^{(1)}$-integrable vacuum $D(n+1,n)^{(1)}$-modules.

\subsection{Vacuum modules over $\fg=D(2,1,a)^{(1)}$}\label{D21a}
Recall that $a\not=0,-1$. One has 
$D(2,1,a)_{\ol{0}}=A_1\times A_1\times A_1$; if we denote the root
in $i$th copy of $A_1$ by $2\vareps_i$, then
$||2\vareps_1||^2:||2\vareps_2||^2:||2\vareps_3||^2=1:a:(-a-1)$.

Let $L(\lambda)$ be a $\pi$-integrable vacuum module 
for some $\pi\subset \Pi_0$ and $k:=(\lambda,\delta)$.
If $\pi\setminus\dot{\Pi_0}$ contains one root, then
$\pi=\{\delta-2\vareps_r,\vareps_r\}$ and $L(\lambda)$ is $\pi$-integrable
 if and only if $2k/||2\vareps_r||^2\in\mathbb{Z}_{\geq 0}$. If $\pi\setminus\dot{\Pi_0}$ contains two roots, then
$\pi=\{\delta-2\vareps_r,\delta-2\vareps_q, 2\vareps_r, 2\vareps_q\}$, and 
 $L(\lambda)$ is  $\pi$-integrable  if and only if
$2k/||2\vareps_r||^2,\ 2k/||2\vareps_q||^2\in\mathbb{Z}_{\geq 0}$; in particular, if $k\not=0$,
then $||2\vareps_r||^2/||2\vareps_q||^2\in\mathbb{Q}_{>0}$, so $a\in\mathbb{Q}$. If  $\pi\setminus\dot{\Pi_0}$ contains three roots,
then $\pi=\Pi_0$ and $k=0$.

We consider a non-critical module $L(\lambda)$, that is $k\not=0$. 
We see that $L(\lambda)$ can be $A_1^{(1)}$-integrable for any copy 
$A_1^{(1)}$ in $\Pi_0$,
but it is $A_1^{(1)}\times A_1^{(1)}$-integrable  only if $a\in\mathbb{Q}$
and the roots of $\pi$ have positive integral square length for some normalization of $(-,-)$.

Let $\pi=\{\alpha\in\Pi_0|\ ||\alpha||^2\in\mathbb{Q}_{>0}\}$
for some normalization of $(-,-)$. 
If $a\not\in\mathbb{Q}$, then $\pi$ can be any copy of  $A_1^{(1)}$.
If  $a\in\mathbb{Q}$, then either ${\pi}=A_1^{(1)}$, 
which corresponds
to the longest root (the absolute value of $||2\vareps_i||^2$ is maximal; this
is $2\vareps_1$ if $-1<a<0$),
or ${\pi}=A_1^{(1)}\times A_1^{(1)}$ 
($\dot\pi=\{2\vareps_2,2\vareps_3\}$ if $-1<a<0$).

We fix $\Pi$ which consists of isotropic roots:
$$\Pi=\{\delta-\vareps_1-\vareps_2-\vareps_3,
-\vareps_1+\vareps_2+\vareps_3, \vareps_1+\vareps_2-\vareps_3,
\vareps_1-\vareps_2+\vareps_3\}.$$

\subsubsection{}
Recall that $j_{\lambda}=2$ and set
$$Z:=j_{\lambda}Re^{\rho}\ch L-\sum_{w\in W(\pi\cup\dot{\Pi}_0)} sgn(w)\, w\bigl(\frac{e^{\lambda+\rho}}
{\prod_{\beta\in S} (1+e^{-\beta})}\bigr).$$
Suppose that $Z\not=0$.

The $\pi\cup\dot{\Pi}_0$-integrability of $L(\lambda)$
gives $(\lambda,\alpha^{\vee})\geq 0$ for each $\alpha\in \pi\cup\dot{\Pi}_0$.
Since $\rho=0$, $\lambda+\rho=\lambda$ is maximal in its  $W(\pi\cup\dot{\Pi}_0)$-orbit,
so $\supp Z\subset \lambda+\rho-\mathbb{Z}_{\geq 0}\Pi$.
Let $\lambda-\mu$ be a maximal element in $\supp Z$ ($\mu\in \mathbb{Z}_{\geq 0}\Pi$).
The arguments of~\S~\ref{maxinorbit} show that
\begin{equation}\label{maxfinn}
2(\lambda+\rho,\mu)=(\mu,\mu) \ \text{ and }
(\lambda-\mu,\alpha^{\vee})\geq 0\ \text{ for each }\alpha\in\pi\cup\dot{\Pi}_0.
\end{equation}
The coefficient of $e^{\lambda}$ in $Z$ is equal to the coefficient of
$e^{\lambda}$ in 
$$j_{\lambda}e^{\lambda}-\sum_{w\in W(\pi\cup\dot{\Pi}_0): w\lambda=\lambda} sgn (w)\, w\bigl(\frac{e^{\lambda}}
{\prod_{\beta\in S} (1+e^{-\beta})}\bigr).$$
Since  $\lambda$ is maximal in its  $W(\pi\cup\dot{\Pi}_0)$-orbit,
the stabilizer of $\lambda$  in  $W(\pi\cup\dot{\Pi}_0)$ is equal to $W(\dot{\Pi}_0)$.
The KW-formula for $D(2,1,a)$ implies that the coefficient of $e^{\lambda}$ is zero.
Hence $\mu\not=0$.

Since $(\lambda,\dot{\Pi}_0)=0$, (\ref{maxfinn}) gives
$(\mu,\alpha^{\vee})\geq 0$ for each $\alpha\in\dot{\Pi}_0$.
Therefore $\mu=j\delta-\sum_{i=1}^3 e_i\vareps_i$, where 
$e_i\geq 0$ for each $i$. One readily  sees that $\mu\in\Pi$ forces 
$2j-e_i-e_s\geq 0$ for each $\{i,s\}\subset \{1,2,3\}$.

\subsubsection{}
Consider the case $a\not\in\mathbb{Q}$; without loss of generality
we assume $\pi=\{\delta-2\vareps_1,2\vareps_1\}$ and normalize the form
$(-,-)$ by $||\vareps_1||^2=1$. 

Since  $2(\lambda,\mu)=(\mu,\mu)$, we have
$e_2=e_3=0$, $2jk=e_1^2$. Moreover, $(\lambda-\mu,(\delta-2\vareps_1)^{\vee})\geq 0$
gives $k\geq 2e_1$. Since
$2j\geq e_1\geq 0$, we obtain $j=e_1=0$, that is $\mu=0$, a contradiction.

\subsubsection{}
Assume that $a\in\mathbb{Q}$. 

For the case  $\pi=A_1^{(1)}$, without loss of generality
we assume $\pi=\{\delta-2\vareps_1,2\vareps_1\}$ and  write
$\mu=j(\delta-2\vareps_1)+(2j-e_1)\vareps_1-e_2\vareps_2-e_3\vareps_3$.

For the case  $\pi=A_1^{(1)}\times A_1^{(1)}$, without loss of generality
we assume $\pi=\{\delta-2\vareps_i,2\vareps_i\}_{i=1,2}$ and write
$\mu=e_1/2(\delta-2\vareps_1)+(j-e_1/2)(\delta-2\vareps_2)
+(2j-e_1-e_2)\vareps_2-e_3\vareps_3$.

In both cases we obtain
$\mu=\sum_{\alpha\in X} x_{\alpha}\alpha$,
where $X\subset \pi\cup\dot{\Pi}_0$ and 
$x_{\alpha}||\alpha||^2\geq 0$ for each $\alpha\in X$
(recall that $||\alpha||^2>0$ for $\alpha\in\pi$). 
Since $2(\lambda,\mu)=(\mu,\mu)$ we have
$(\lambda,\mu)+(\lambda-\mu,\mu)=0$, that is
$$\sum_{\alpha\in X} x_{\alpha}\bigl((\lambda,\alpha)+
(\lambda-\mu,\alpha)\bigr)=0.$$
For each $\alpha\in \dot{\Pi}_0\cup \pi$ one has
 $(\lambda,\alpha)||\alpha||^2\geq 0$ and 
$(\lambda-\mu,\alpha)||\alpha||^2\geq 0$ (by~(\ref{maxfinn})). Therefore
for each $\alpha\in X$ we have $x_{\alpha}(\lambda,\alpha),
x_{\alpha}(\lambda-\mu,\alpha)\geq 0$. Hence
for each $\alpha\in X$ we obtain
$x_{\alpha}(\lambda,\alpha)=x_{\alpha}(\lambda-\mu,\alpha)=0$,
that is $x_{\alpha}(\mu,\alpha)=0$. 

Write $\mu=\mu'+\mu''$, where $\mu'\in\mathbb{Q}\pi$ and
$\mu''\in\mathbb{Q}(\dot{\Pi}_0\setminus\pi)$.
One has $||\mu'||^2=(\mu,\mu')=\sum_{\alpha\in X\cap \pi} x_{\alpha}
(\mu,\alpha)=0$. Similarly, $||\mu''||^2=0$.
Since $(-,-)$ is non-negatively (resp., negatively) definite on
$\mathbb{Q}\pi$ (resp., on $ \mathbb{Q}(\dot{\Pi}_0\setminus\pi)$),
we get $\mu''=0$ and $\mu=\mu'\in \mathbb{Q}\delta$. Then 
$2(\lambda,\mu)=(\mu,\mu)$ gives $\mu=0$, a contradiction.

\subsubsection{Remark}
The following denominator identity for $D(2,1,a)^{(1)}$ was proven in~\cite{GR}:
$$Re^{\rho}=\prod_{n=1}^{\infty}(1-e^{-n\delta})\sum_{t\in T} t(\dot{R}e^{\rho}),$$
where $T$ is the translation group of $A_1^{(1)}\subset 
D(2,1,a)^{(1)}_{\ol{0}}=
A_1^{(1)}\times A_1^{(1)}\times A_1^{(1)}$.
The corresponding embedding $A_1\subset D(2,1,a)_{\ol{0}}=
A_1\times A_1\times A_1$ is not
specified  in~\cite{GR} and we take an opportunity to correct this. 
This embedding is the same as described above, namely,
 any embedding if $a\not\in\mathbb{Q}$, and 
the copy with the maximal absolute value
of the square length of the root if $a\in\mathbb{Q}$.  
This choice is necessary for the proof
 of  Prop. 2.3.2~\cite{GR}, where it is used that  a  non-zero 
linear combination of the two remaining even roots has non-zero square length.

\subsection{Other forms of~(\ref{3}) for $\dot{\Delta}\not=D(2,1,a)$}
Let $L=L(\lambda,\Pi)$ be a  non-critical  integrable $\fg$-module
of maximal atypicality,
where $\Pi$ is compatible with  $\dot{\Delta}$ and $L$ satisfies
the KW-condition for $\Pi,S$ with $S\subset \dot{\Pi}$,
or let $L$ be a non-critical $C_n^{(1)}$-integrable vacuum $D(n+1,n)^{(1)}$-module.
In the first case we normalize $(-,-)$ as in Section~\ref{sect3};
in the second case
we normalize the form on $D(n+1,n)^{(1)}$ in such a way that
$||\alpha||^2=2$ for some $\alpha\in C_n^{(1)}$.
Then 
$$\pi:=\{\alpha\in\Pi_0|\ ||\alpha||^2>0\}$$
is a connected component of $\Pi_0$, and $L$ is $\pi$-integrable.

Let $\dot{\fg}$ be the subalgebra of $\fg$ with the root system
$\dot{\Delta}$ and the Cartan algebra $\fh$ (i.e., $\dot{\fg}=(\sum_{\alpha\in\dot{\Delta}} \fg_{\alpha})+\fh$)
and $\dot{L}=\dot{L}(\lambda,\dot{\Pi})$ be the irreducible
$\dot{\fg}$-module of the highest weight $\lambda$. Clearly, $\dot{L}$
is a finite-dimensional module satisfying the KW-condition for $\dot{\Pi}, S$.

By~\S~\ref{app3}, $\dot{\pi}=\pi\cap\dot{\Pi}_0$ is a finite part of $\pi$, i.e., 
$\dot{\pi}$ is connected and 
$\pi\setminus\dot{\pi}$ consists of one root, which we denote by
$\alpha^{\sharp}$ (we exclude the case $\Delta=G(3)^{(1)}, \dot{\Delta}=D(2,1,-3/4)$).

From~\cite{K2}, 6.5, it follows
that in the case when
 $\alpha^{\sharp}=j\delta-b\theta$, where $\theta\in{\Delta}(\dot{\pi}), b\in\mathbb{Q}$,  one has
$W(\pi)=W(\dot{\pi})\ltimes T$, where $T$ is a free abelian subgroup of $W(\pi)$.
Recall that $\pi$ is one of the root systems $A_n^{(r)}, B_n^{(1)}, C_n^{(1)}, D_n^{(r)}, G_2^{(1)}$
with $r=1,2$).
The condition $\alpha^{\sharp}=j\delta-\theta$ holds  for all pairs $(\pi,\dot{\pi})$ 
(where $\pi$ is as above and $\dot{\pi}$ is a finite part of $\pi$), 
except for $(B_n^{(1)}, D_n), (A_{2n-1}^{(2)},D_n)$
and $(G_2^{(1)},A_2)$. The last case is not possible, since $\pi=G_2^{(1)}$ 
appears only for $\Delta=G(3)^{(1)}$ and in this case $\dot{\Delta}=G(3), \dot{\pi}=G_2$, by above.
 
Assume that the pair $(\pi,\dot{\pi})$ is not $(B_n^{(1)}, D_n)$ or $(A_{2n-1}^{(2)}, D_n)$.
Then, by above, $W(\pi)=W(\dot{\pi})\ltimes T$ and so $W(\pi')=W(\dot{\Pi}_0)\ltimes T$,
since $\pi'\setminus\pi$ is a connected component of $\dot{\Pi}_0$. 
Let $\dot{\rho}$ be the Weyl vector
for $\dot{\Delta}=\Delta(\dot{\Pi})$. Notice that $(\rho-\dot{\rho},\dot{\Pi})=(\lambda,\dot{\Pi})=0$,
so $\lambda+\rho-\dot{\rho}$ is $W(\dot{\Pi}_0)$-invariant.
Then~(\ref{3}) can be rewritten as
$$Re^{\rho}\ch L(\lambda)=\sum_{t\in T}t\bigl(e^{\lambda+\rho-\dot{\rho}} j_{\lambda}^{-1}\,
\sum_{w\in W(\dot{\Pi}_0)} sgn(w)\ w\bigl(\frac{e^{\dot{\rho}}}{\prod_{\beta\in S} (1+e^{-\beta})}\bigr)\bigr).$$

Using the KW-formula for $\dot{L}$, we get
\begin{equation}\label{KWcharvac}
Re^{\rho}\ch L(\lambda)=
\sum_{t\in T}t\bigl(\dot{R}e^{\rho}\ch\dot{L}(\lambda)\bigr),
\end{equation}
where $\dot{R}$ is the Weyl denominator for   $\dot\Pi$. 

For  the cases $(B_n^{(1)}, D_n), (A_{2n-1}^{(2)}, D_n)$
we can extend $W(\pi)$ to $W(C_n^{(1)})$ and present $W(C_n^{(1)})=W(C_n)\ltimes T$
as in~\cite{R}; then we obtain~(\ref{KWcharvac}) for $T\subset W(C_n^{(1)})$.

Note that in~(\ref{KWcharvac}) there is no $S$; we do not assume that $\dot{\Pi}$ contains 
a  maximal isotropic subset. More precisely,
if the KW-formula holds for some $\Pi, S$ with $S\subset\dot{\Pi}$, then 
(\ref{KWcharvac}) holds for each set of simple roots $\Pi'$ compatible
with $\dot{\Delta}$. In particular, if $L$ satisfies 
KW-conditions for $\Pi,S$ and $\Pi',S'$, where
$\Pi,\Pi'$ are compatible with $\dot{\Delta}$ and 
$S, S'\subset \dot{\Delta}$, then the KW-formulas for $\Pi,S$ and $\Pi',S'$
are equivalent.

Note that $\dot{\pi}$ is the "largest part" of $\dot{\Pi}_0$
in the sense of~\S~\ref{findimmaxatyp},
except for the following cases: $\Delta=G(3)^{(1)}$ with $\dot{\Delta}=D(2,1,-3/4)$,
$\Delta=D(2,1,a)^{(1)}$ with $\pi=A_1^{(1)}$,
$\Delta=D(n+1,n)^{(1)}$ with  $\pi=C_n^{(1)}$, and $\Delta=A(2n-1,2n-1)^{(2)}$
with $\dot{\pi}=D_n$ (in the case $\Delta=A(2n-1,2n-1)^{(2)}$ one has
$\pi=A_{2n-1}^{(2)}$, $\dot{\Delta}=D(n,n)$,
and  $\dot{\pi}$ can be $D_n$ or $C_n$).
If $\dot{\pi}$ is the "largest part" of $\dot{\Pi}_0$, 
then, using~(\ref{denomfin})  we can rewrite~(\ref{vacform}) as
\begin{equation}\label{vacpi}
Re^{\rho}\ch L(\lambda)=\sum_{w\in W(\pi)} sgn(w)\ w\bigl(\frac{e^{\lambda+\rho}}{\prod_{\beta\in S} (1+e^{-\beta})}\bigr),
\end{equation}
cf.~(\ref{KWcharfor}).

\subsection{Case $h^{\vee}\not=0$ or $\Delta=A(n,n)^{(1)}$}

\subsubsection{Cases $G(3)^{(1)}, F(4)^{(1)}$}
One readily sees that if $L$ satisfies the KW-condition for $\Pi, S=\{\beta\}$,
then $L$ satisfies the KW-condition for $r_{\beta}\Pi, S'=\{-\beta\}$
and the KW-formulas are equivalent (cf.~\Lem{lemchPIS} for finite case).
For $G(3)^{(1)}, F(4)^{(1)}$ the assumptions of Section~\ref{sectKW}
are equivalent to $||\alpha||^2\geq 0$ for $\alpha\in\Pi$.  
For $G(3)^{(1)}, F(4)^{(1)}$ each $\Pi$ either satisfies this property
or contains a unique isotropic root $\beta$ and $r_{\beta}\Pi$
satisfies this property. This establishes the KW-formula
for these cases, due to the results of Section~\ref{sectKW}.

\subsubsection{}
Let $\fg\not=G(3)^{(1)}, F(4)^{(1)}$, i.e., 
$\dot{\Delta}$ is not exceptional. In this case 
$\dot{\pi}$ is the "largest part" of $\dot{\Pi}_0$.
Since formulas~(\ref{vacpi}) and~(\ref{KWcharfor}) are the same,
 it is enough to show that $L$ satisfies the KW-condition for some
$\Pi',S'$, where $\Pi'$ is compatible with $\dot{\Delta}$,
$S'\subset \dot{\Delta}$, and $\Pi',S'$
satisfy the assumptions of Section~\ref{sectKW}.

In the light of~\S~\ref{KWfin}, one has $(\lambda,\alpha)=0$ for each $\alpha\in\Pi^0$,
where $\Pi^0$ is a  Dynkin subdiagram of $\Pi$ and
$\Pi^0=A(k,l)$ (resp., $B(k,l), D(k,l)$) for $\dot{\Pi}=A(m,n)$ (resp., $B(m,n), D(m,n)$) and $\min (k,l)=\# S$ (and $=\min (m,n)$, since $\lambda$ has maximal
atypicality). Let $\cB^0\subset \cB$ be the ordered subset corresponding
to $\Pi^0$ (i.e., $\Pi^0=(\cB^0,>)$). We can rearrange dots and crosses in $\cB^0$ 
in such a way that  the last $2\min (k,l)$ elements are of alternating types, and
the last element in $\cB^0$ is of positive square length (resp.,
is $\vareps_m$) if  $\dot{\Delta}=B(m,n)$ (resp., if $\dot{\Delta}=D(m,n)$).

This rearrangement corresponds to a certain sequence of odd reflections
(with respect to roots in $\Delta(\Pi^0)\subset \dot{\Delta}$);
let $\Pi'$ be the subset of simple roots obtained from $\Pi$ by this
sequence of odd reflections. Clearly, 
$\dot{\Pi}':=\dot{\Delta}\cap \Pi'$ is a subset of simple roots for $\dot{\Delta}$
and the corresponding dot-cross diagram contains $\cB^0$ with the new order.
Thus the dot-cross diagram for $\dot{\Pi}$ contains a segment of
$2\min (m,n)$ elements of alternating types; moreover, if  
$\dot{\Delta}\not=A(m,n)$, then these are the last $2\min (m,n)$ elements
 and the last element 
is of positive square length for $B(m,n)$ and is $\vareps_m$ 
for $D(m,n)$. Since $\dot{\pi}$ is the "largest part" of $\dot{\Pi}_0$,
$||\alpha||^2\geq 0$ for each $\alpha\in\dot{\Pi}'$. 

One has $L=L(\lambda,\Pi)=L(\lambda,\Pi')$ and so $L$ satisfies the
KW-condition for $(\Pi',S')$, where $S'$ is any subset of $\dot{\Pi}'$
which contains $\min (m,n)$ mutually orthogonal isotropic roots. 
Thus the assumptions of Section~\ref{sectKW} are reduced to the condition
$||\alpha||^2\geq 0$ for each $\alpha\in {\Pi}'$.
Recall that $||\alpha||^2\geq 0$ for each $\alpha\in\dot{\Pi}'$
and that $\Pi'\setminus\dot{\Pi}'$ consists of one root, which we denote
by $\alpha_0$. Hence it remains to verify that $||\alpha_0||^2\geq 0$.

For ${\Delta}=A(m,n)^{(1)}$, any subset of simple roots
is naturally  encoded by a cyclic dot-cross
diagram, which contains  $m$ dots and $n$ crosses. Let $m\geq n$.
 Since the diagram for $\Pi'$ contains
$2n$ elements of alternating types, it does not contain
two neighboring crosses, so $||\alpha||^2\geq 0$ for each $\alpha\in {\Pi}'$.

Suppose $||\alpha_0||^2<0$ and  ${\Delta}\not=A(m,n)^{(1)}$, that is
$\dot{\Delta}\not=A(m,n)$.
One has $(\alpha_0,\alpha_1)\not=0$ or $(\alpha_0,\alpha_2)\not=0$,
where  $\alpha_1,\alpha_2$ are the first two roots in $\dot{\Pi}'$.
Thus $||\alpha_i||^2=0$ for $i=1$ or $i=2$.
By the construction of  $\dot{\Pi}'$,  the pair $\alpha_1,\alpha_2$ can be written as
$\vareps_1-\vareps_2, \vareps_2-\delta_1$, or
$\vareps_1-\delta_1, \delta_1-\vareps_2$, or 
$\vareps_1-\delta_1, \delta_1$ (case $B(1,1)$), 
or $\delta_1-\vareps_1,\vareps_1-\delta_2$,
where $(\vareps_i,\vareps_j)=-(\delta_i,\delta_j)=
\delta_{ij}$ 
and $(\delta_i,\vareps_j)=0$ for $i,j=1,2$. Since $||\alpha_0||^2<0$, one has
 $\alpha_0\in \Pi_0$ or $2\alpha_0\in\Pi_0$, so $\alpha_0$ or $2\alpha_0$
 is a root $\Pi_0\setminus \{\pi\cup \dot{\Pi}_0)$.
Thus $\alpha_0=\delta-x\delta_1$ for $x\in \{1,2\}$
or $\alpha_0=\delta-(\delta_1+\delta_2)$.
Then $\alpha_1,\alpha_2$ is the pair $\delta_1-\vareps_1,\vareps_1-\delta_2$
and $\alpha_0=\delta-x\delta_1$.
By the construction of  $\dot{\Pi}'$,  $||\alpha_1||^2=0$
forces $||\alpha||^2=0$ for each $\alpha\in\dot{\Pi}$ (resp., for
each $\alpha\in\dot{\Pi}\setminus\{\vareps_m\}$)
if $\dot{\Delta}=D(m,n)$ (resp., if $\dot{\Delta}=B(m,n)$).
If $||\alpha||^2=0$ for each $\alpha\in\dot{\Pi}$, then,
since $(\rho,\delta)=h^{\vee}\geq 0$, we get $(\rho,\alpha_0)=||\alpha_0||^2/2\geq 0$,
a contradiction.  Finally,
for $\dot{\Delta}=B(m,n)$ we obtain $\delta=\alpha_0+x(\sum_{\alpha\in \dot{\Pi}'}
\alpha)$, so $(\rho,\delta)=-x^2/2+x^2/2=0$, a contradiction.

\subsection{Case $h^{\vee}=0$}
The remaining cases are $D(n+1,n)^{(r)}, (n>1, r=1,2)$,
$A(2n-1,2n-1)^{(2)}$ and  $A(2n,2n)^{(4)}$. 

Note that if 
$L$ is of maximal atypicality, then $\dot{L}$
is a finite-dimensional $\dot{\fg}$-module of maximal atypicality, and
 \S~\ref{KWfin} implies $(\lambda,\dot{\Delta})=0$, except
 for the case $\Delta=D(n+1,n)^{(1)}$ 
 (since for other cases  $\dot{\Delta}=D(n,n)$ or $B(n,n)$).
Thus if $\Delta\not=D(n+1,n)^{(1)}$, then $L(\lambda,\Pi)$
is a non-critical integrable vacuum module. 

We set $\pi'':=\{\alpha\in \dot{\Pi}_0|\ ||\alpha||^2<0\}$.
Then $\pi\cup\dot{\Pi}=\pi\coprod\pi''$. One has
 $(\lambda,\pi'')=0$ (it is obvious if $L$ is a vacuum module; 
 otherwise  $\Delta=D(n+1,n)^{(1)}, \dot{\Delta}=D(n+1,n),\pi''=C_n$ and, by~\S~\ref{KWfin},
 $(\lambda,\alpha)=0$ for each $\alpha\in D(n,n)\subset D(n+1,n)$, in particular,
 for $\alpha\in \pi''$).

We introduce 
$$Z:=j_{\lambda}Re^{\rho}\ch L-\sum_{w\in W(\pi\cup\dot{\Pi})} sgn(w)\, w\bigl(\frac{e^{\lambda+\rho}}
{\prod_{\beta\in S} (1+e^{-\beta})}\bigr).$$
Suppose that $Z\not=0$. Let $\lambda'$ be a maximal element in $\supp Z$; we
write $\lambda'-\rho=\lambda-\mu$ for $\mu\in \mathbb{Z}\Pi$.
The arguments of~\S~\ref{maxinorbit} show that~(\ref{maxfinn}) holds.

\subsubsection{Cases $D(n+1|n)^{(1)}, n>1$ and $A(2n-1,2n-1)^{(2)}$}\label{A2n-12n-1} 
In these cases we choose $\Pi$, which consists
of isotropic roots, see~(\ref{isoaffine}).
 The proof is similar to the one in~\S~\ref{D21a}.

For $D(n+1,n)^{(1)}$ one has $\Pi_0=D_{n+1}^{(1)}\times C_n^{(1)},\ \dot{\Delta}=D(n+1,n)$. If $L(\lambda)$ is integrable, then
$\pi=D_{n+1}^{(1)}$ and $\pi''=C_n$; if $L(\lambda)$
is a $C_n^{(1)}$-integrable vacuum module, then
$\pi=C_{n}^{(1)}$ and $\pi''=D_{n+1}$. 

For $A(2n-1,2n-1)^{(2)}$ one has $\Pi_0=A_{2n-1}^{(2)}\times A_{2n-1}^{(2)},\ \dot{\Delta}=D(n,n)$. Note that
$\dot{\pi}=\pi\cap\dot{\Delta}$ can be $D_n$ or $C_n$ and $\pi''$
is $C_n$ or $D_n$ respectively.

Recall that $\rho=0$.
Since $L=L(\lambda,\Pi)$ is $\pi\cup\pi''$-integrable,
$\lambda=\lambda+\rho$ is maximal in its $W(\pi)\times W(\pi'')$-orbit, so 
$\supp Z\subset \lambda+\rho-\mathbb{Z}_{\geq 0}\Pi$. Thus 
$\mu\in \mathbb{Z}_{\geq 0}\Pi$. Observe that $\mathbb{Q}\Pi=\mathbb{Q}(\pi\cup\pi'')$.
Write 
$\mu=\mu'-\mu''$, where $\mu'\in \mathbb{Q}\pi, \mu''\in\mathbb{Q}\pi''$. 
We claim that 
\begin{equation}\label{mu'mu''}
\mu'\in \mathbb{Q}_{\geq 0}\pi,\ \  \mu''\in\mathbb{Q}_{\geq 0}\pi''.
\end{equation}
Indeed, by above, $(\lambda,\pi'')=0$. Using~(\ref{maxfinn}),
 we get $(\mu'',\alpha^{\vee})\leq 0$ for each $\alpha\in\pi''$;
 Thm. 4.3 in~\cite{K2} gives $\mu''\in \mathbb{Q}_{\geq 0}\pi''$.
This implies $\mu''\in \mathbb{Q}_{\geq 0}\Pi$, so 
$\mu'=\mu+\mu''\in \mathbb{Q}_{\geq 0}\Pi$.  Thus $\mu'\in \mathbb{Q}\pi\cap
\mathbb{Q}_{\geq 0}\Pi$. It remains to verify that 
\begin{equation}\label{QpiPi}
(\mathbb{Q}_{\geq 0}\Pi\cap \mathbb{Q}\pi)\subset \mathbb{Q}_{\geq 0}\pi.
\end{equation}
Observe that each $\alpha\in\pi$ is a sum of two
simple roots $\alpha=\beta(\alpha)+\beta'(\alpha)$ ($\beta(\alpha),\beta'(\alpha)\in\Pi$) and we can choose $\beta(\alpha)$ in such a way that
$\alpha\mapsto \beta(\alpha)$ is an injective map from $\pi$ to $\Pi$
(for instance, for $D(n+1,n)^{(1)}$ with $\pi=D_{n+1}^{(1)}$, one has
$\delta-\vareps_1-\vareps_2=(\delta-\vareps_1-\delta_1)+(\delta_1-\vareps_2)$,
$\vareps_i-\vareps_{i+1}=(\vareps_i-\delta_i)+(\delta_i-\vareps_{i+1})$,
$\vareps_n+\vareps_{n+1}=(\vareps_n-\delta_n)+(\delta_n-\vareps_{n+1})$,
so we can take $\beta(\delta-\vareps_1-\vareps_2)=\delta-\vareps_1-\delta_1$,
$\beta(\vareps_i-\vareps_{i+1})=\delta_i-\vareps_{i+1}$,
$\beta(\vareps_n+\vareps_{n+1})=\delta_n+\vareps_{n+1}$), 
cf. (B) in~\S~\ref{appendix2}.
Then $\sum_{\alpha\in \pi} a_{\alpha}\alpha=\sum_{\beta\in \Pi} b_{\beta}\beta$,
where $b_{\beta(\alpha)}=a_{\alpha}$; this establishes~(\ref{QpiPi})
and~(\ref{mu'mu''}).

In the light of~(\ref{mu'mu''}) we have
 $$\mu=\mu'-\mu''=\sum_{\alpha\in \pi\cup\pi''} x_{\alpha}\alpha,
\ \text{ where } x_{\alpha} ||\alpha||^2\geq 0.$$

The formula $2(\lambda+\rho,\mu)=(\mu,\mu)$ gives
\begin{equation}\label{xalp}
0=(\lambda,\mu)+(\lambda-\mu,\mu)=
\sum_{\alpha\in \pi\cup\pi''} x_{\alpha} (\lambda,\alpha)+ 
x_{\alpha}(\lambda-\mu,\alpha).\end{equation}

Take $\alpha\in \pi\cup\pi''$. The 
$\pi\cup\pi''$-integrability of $L(\lambda)$ gives
$(\lambda,\alpha)||\alpha||^2\geq 0$. Moreover, 
$(\lambda-\mu,\alpha)||\alpha||^2\geq 0$, by~(\ref{maxfinn}).
Thus $x_{\alpha} (\lambda,\alpha), x_{\alpha}(\lambda-\mu,\alpha)\geq 0$.
Using~(\ref{xalp}) we obtain 
$x_{\alpha}(\lambda,\alpha)=x_{\alpha}(\lambda-\mu,\alpha)=0$, that is
$x_{\alpha}(\mu,\alpha)=0$. One has 
$$(\mu',\mu')=(\mu,\mu')=\sum_{\alpha\in \pi} x_{\alpha} (\mu,\alpha)=0;\ \ \
(\mu'',\mu'')=(\mu,\mu'')=\sum_{\alpha\in \pi''} x_{\alpha} (\mu,\alpha)=0.$$
Since $(-,-)$ is negatively definite on $\mathbb{Q}\pi''$ and 
non-negatively definite on $\mathbb{Q}\pi'$, we obtain
$\mu''=0,\mu'\in\mathbb{Z}\delta$, that is $\mu=s\delta$.
Since $L$ is non-critical, $(\lambda+\rho,\delta)\not=0$, so
the formula $2(\lambda+\rho,\mu)=(\mu,\mu)$ gives  $\mu=0$,
that is $\lambda-\mu=\lambda\in\supp Z$.

It remains to verify that $\lambda\not\in\supp Z$.
Since $(\lambda,\alpha^{\vee})\geq 0$ for $\alpha\in \pi\cup \pi''$,
the coefficient of $e^{\lambda}$ in $Z$ is equal to the coefficient of
$e^{\lambda}$ in 
$$j_{\lambda}e^{\lambda}-\sum_{w\in W(\pi\cup\dot{\Pi}_0): w\lambda=\lambda} sgn (w)\, w\bigl(\frac{e^{\lambda}}
{\prod_{\beta\in S} (1+e^{-\beta})}\bigr).$$
If $\Stab_{W(\pi\cup\pi'')}\lambda\subset
W(\dot{\Pi}_0)$, then the KW-formula for $\dot{L}$ implies $\lambda\not\in\supp Z$
as required. Otherwise, $(\lambda,\alpha^{\sharp})=0$, where  $\alpha^{\sharp}$
is the "affine root" in $\pi$, i.e. $\pi=\dot{\pi}\cup \{\alpha^{\sharp}\}$.
Since $(\lambda,\alpha^{\sharp})=0$, $L$ 
 is not a vacuum module (since $L$ is non-critical), so
 $\fg=D(n+1,n)^{(1)}$ and  $(\lambda,\alpha)=0$ for each $\alpha\in D(n,n)\subset D(n+1,n)=\dot{\Delta}$. Let $\alpha_1$ be the first root in $\dot{\Pi}$ (i.e., $\alpha_1\in\Pi$
and $\alpha_1\not\in D(n,n)$). Then $(\lambda,\alpha_1)\not=0$. 
Note that 
$\Pi$  admits an involution $\sigma$, which interchanges $\alpha_0$ and $\alpha_1$
 and stabilizes  all other simple roots. 
 This involution preserves $\pi,\pi''$ and $S$.
 One has $(\sigma(\lambda),\alpha_1)=(\lambda,\alpha_0)=0$,
 so  $L(\sigma(\lambda),\Pi)$ is a vacuum module and, by above,
 its character satisfies the KW-formula. This implies
  the KW-formula for $L(\lambda,\Pi)$.
This completes the proof for  $D(n+1|n)^{(1)}, A(2n-1,2n-1)^{(2)}$.

\subsubsection{Cases $D(n+1,n)^{(2)}, A(2n,2n)^{(4)}$}\label{A2n4}
In these cases 
$\dot\Delta=B(n,n)$ and we choose $\dot\Pi$ in such a way
that $||\alpha||^2\geq 0$ for $\alpha\in\dot{\Pi}$ (recall that 
$||\alpha||^2>0$ for $\alpha\in\pi$). The Dynkin diagram is 
$$\odot-\otimes-\otimes-...-\otimes-\odot$$
where both ends are non-isotropic roots; for $A(2n,2n)^{(4)}$ the ends have the same parity and for $D(n+1,n)^{(2)}$ the ends have different parity. One has
$$\dot\Pi=\{\delta_1-\vareps_1,\vareps_1-\delta_2,\ldots,\delta_n-\vareps_n,
\vareps_n\},\ \alpha_0=\delta-\delta_1;$$
we can (and will) normalize the form $(-,-)$ by  $||\vareps_i||^2=1=-||\delta_i||^2$.
Note that $\dot{\Pi}$ contains 
$S=\{\delta_i-\vareps_i\}_{i=1}^n$. One has 
$$\pi=\{a(\delta-\vareps_1), \vareps_1-\vareps_2,
\ldots,\vareps_{n-1}-\vareps_n, a'\vareps_n\},$$
where  for $\fg=D(n+1,n)^{(2)}$ one has 
$a=a'=1$ (resp., $a=a'=2$) if  $\pi=D(n+1)^{(2)}$ (resp., if $\pi=C_n^{(1)}$),
and for $\fg=A(2n,2n)^{(4)}$ one has 
$a=1,a'=2$ or $a=2,a'=1$. Observe that
$\Delta$ contains $\delta-\vareps_1$ and
$\vareps_n\in\Delta$ (they are of the same parity for $D(n+1,n)^{(2)}$
and of different parity for $A(2n,2n)^{(4)}$).

Let $k=(\lambda,\delta)$.
Recall that $L(\lambda)$ is a non-critical vacuum module, so $k\not=0$.
 If $\delta-\vareps_1$ is even (resp., odd), then
$L(\lambda)$ is $\pi$-integrable if and only if 
$2k\in\mathbb{Z}_{\geq 0}$ (resp.,  $k\in\mathbb{Z}_{\geq 0}$).

One has 
$2\rho=\sum_{i=1}^n (\vareps_i-\delta_i)$.

Since $\dot{\pi}$ is the "largest part" of $\dot{\Pi}_0$
in the sense of~\S~\ref{findimmaxatyp}, (\ref{vacform}) can be rewritten as~(\ref{vacpi}). We have $(\rho,\alpha)\geq 0$ for $\alpha\in \dot{\pi}$,
$(\rho,\delta-\vareps_1)=-1/2$. Since $(\lambda,\delta-\vareps_1)=k\geq \frac{1}{2}$,
one has $(\lambda+\rho,\alpha)\geq 0$ for each $\alpha\in\pi$. Therefore 
$\lambda+\rho$ is maximal in its $W(\pi)$-orbit, so
$\supp Z\subset \lambda+\rho-\mathbb{Z}_{\geq 0}\Pi$, that is
$\mu \in\mathbb{Z}_{\geq 0}\Pi$.
Moreover, from~\Lem{lemap11} (ii) 
we obtain $\mu\not=0$. Write $\mu=j\delta-\sum_{i=1}^n (e_i\vareps_i+d_i\delta_i)$.
From~(\ref{maxfinn}) we obtain
$$\begin{array}{l}
2kj-\sum_{i=1}^n (e_i+d_i)=\sum_{i=1}^n (e_i^2-d_i^2),\\
0\leq e_n\leq e_{n-1}\leq\ldots\leq e_1\leq k;\ \ 0\leq d_n\leq d_{n-1}\leq\ldots\leq d_1.
\end{array}$$
Since $\mu\in\mathbb{Z}_{\geq 0}\Pi$ we have
$e_i,d_i\in\mathbb{Z}$ and $\sum_{i=1}^n (e_i+d_i)\leq j$.
Since $e_i,d_i\geq 0$, the equality $j=0$ forces $\mu=0$, which contradicts the above. 
Let us show that $j=0$.
Since $\sum_{i=1}^n (d_i-d_i^2)\leq 0$, we have $2kj\leq \sum_{i=1}^n (e_i^2+e_i)$.
If $k=1/2$, then $e_i=0$ for each $i$, 
so $j=0$. If $k>1$, then $\sum_{i=1}^n e_i\leq j$ and $e_i\leq k$
imply $2kj\leq kj+j$, that is $j=0$, as required.

Finally, for $k=1$ we have $e_i\in\{0,1\}$ for each $i$, which implies
$2j=\sum_{i=1}^n (2e_i+d_i-d_i^2)$. Combining with  $\sum_{i=1}^n (e_i+d_i)\leq j$,
we get $d_i=0$ for each $i$. This gives $\mu=j\delta-\sum_{i=1}^j\vareps_j$ and
 $j\leq n$.
Set $\alpha:=\delta_j-\vareps_j$. In the light of~\S~\ref{RPi} and~\Lem{lemexpA},
 $Z=Z(\Pi)$ is equivalent to $Z(r_{\alpha}\Pi)$ (where $Z(\Pi')$ stands for an element $Z$ defined
 for $\Pi'$ and viewed as an element of $\cR(\Pi')$). Clearly,
 $$\supp (Z(r_{\alpha}\Pi))\subset \lambda+\rho_{r_{\alpha}\Pi}-\mathbb{Z}_{\geq 0}(r_{\alpha}\Pi)=
\lambda+\rho+\alpha -\mathbb{Z}_{\geq 0}(r_{\alpha}\Pi).$$
 
Observe that $(\lambda+\rho-\mu,\alpha)\not=0$, so
 $\supp (Z)\cap \{\lambda+\rho-\mu+\mathbb{Z}\alpha\}=
 \{\lambda+\rho-\mu\}$ (since $||\lambda+\rho-\nu||^2=||\lambda+\rho||^2$ 
 for each $\nu\in \supp (Z)$). By~\Lem{lemexpA}, $Z$ has a pole of order $\leq 1$ at
 $\alpha$, and so, by~\Lem{lemXX'}, $\lambda+\rho-\mu \in \supp (Z(r_{\alpha}\Pi))$, that is
 $$\lambda+\rho-\mu \in\lambda+\rho+\alpha -\mathbb{Z}_{\geq 0}(r_{\alpha}\Pi),$$
which gives $\mu+\alpha\in \mathbb{Z}_{\geq 0}(r_{\alpha}\Pi)$;
one readily sees that this does not hold for
$\mu=j\delta-\sum_{i=1}^j\vareps_j$ (if $\mu\not=0$), a contradiction.
This completes the proof.

\section{The root system $\Delta(L)$}\label{sectDeltaL}
In this section we exclude $\fg$ of types $D(2,1,a)$ and 
$D(2,1,a)^{(1)}$ with $a\not\in\mathbb{Q}$ from consideration.
Then we can (and will) choose a normalization of the bilinear form $(-,-)$,
such that $(\alpha,\beta)\in \mathbb{Z}$ for each 
pair $\alpha,\beta\in\Delta$. Consequently, for each set of simple roots 
$\Pi$ one has
 $2(\rho_{\Pi},\alpha)\in\mathbb{Z}$ for each $\alpha\in\Pi$
and thence for each $\alpha\in\Delta$.

For the Lie superalgebras $D(2,1,a),\ D(2,1,a)^{(1)}$ with non-rational $a$  
all the results of this section
remain valid if we fix a standard symmetric Cartan matrix for 
$D(2,1,a)$ as in~\cite{K1} and replace $\mathbb{Z}$ by $\mathbb{Z}+\mathbb{Z}a$
in the construction of $\Delta(L)$.

\subsection{Notation}\label{normal}
We define $\alpha^{\vee}=2\alpha/(\alpha,\alpha)$ 
if $\alpha\in\Delta$ is a non-isotropic root;
for isotropic $\alpha\in\Delta$ we set $\alpha^{\vee}=\alpha$. Notice that
$\lang\mu,\alpha^{\vee}\rang$  does not depend on the normalization of
$(-,-)$ if $\alpha$ is non-isotropic.

\subsubsection{}
We will use the following fact.

\begin{prop}{prop311}
If $\gamma$ is a non-isotropic root and $\alpha$ is a root, 
then $(\alpha,\gamma^{\vee})$ is an integer (resp., even integer) 
if $\gamma$ is even (resp., odd). 
\end{prop}
\begin{proof}
Since $(\gamma,\gamma)\not=0$, $||\beta\pm N\gamma||^2\to \infty$ as 
$N\to\infty$, hence $\fg_{\pm \gamma}$ act locally nilpotently on $\fg$,
and thus $(\alpha,\gamma^{\vee})$ is an integer (resp., an even integer)
if $\gamma$ is even (resp., odd), using representation theory of
$A_1$ (resp., $B(0,1)$).
\end{proof}

\subsection{Definition of $\Delta(L)$}\label{DeltaL}
Let $\Delta^+$ be a subset of positive roots in $\Delta$. 
Consider an irreducible highest
weight module $L=L(\lambda,\Delta^+)$ over $\fg$, associated with $\Delta^+$.
If $\beta$ is a simple isotropic root, then  $L$ is again an irreducible highest
weight module, but associated with the subset of positive roots $r_{\beta}(\Delta^+)$.
Indeed, if $v_{\lambda}\in L$ is a highest weight vector for $\Delta^+$,
then $v_{\lambda}$ (resp., $e_{-\beta}v_{\lambda}$)
 is a highest weight vector for  $r_{\beta}(\Delta^+)$
if $(\lambda,\beta)=0$ (resp., if  $(\lambda,\beta)\not=0$).
Since $(\rho,\beta)=0$, we obtain $L(\lambda,\Delta^+)=L(\lambda',r_{\beta}\Delta^+)$, where
the highest weight $\lambda'$ of the module $L(\lambda',r_{\beta}\Delta^+)$ is given by
 \begin{equation}\label{lambdalambda'}\lambda'=\left\{
\begin{array}{ll}
\lambda-\beta &\text{ for }(\lambda+\rho,\beta)\not=0,\\
\lambda &\text{ for }(\lambda+\rho,\beta)=0.
\end{array}
\right.
\end{equation}

Thus the notion of an irreducible highest weight 
module is independent of the choice of $\Delta^+$ (by~\Prop{propS} (a)).

In this paper we consider only non-critical irreducible highest weight 
modules $L=L(\lambda,\Delta^+)$, i.e.  we assume that the highest weight $\lambda$ satisfies
$(\lambda+\rho_{\Pi},\delta)\not=0$. This property is independent
of the choice of $\Delta^+$, since, by~\S~\ref{choicerho}, one has
$\rho_{\Pi}-\rho_{\Pi'}\in\mathbb{Z}\Delta$ and $L(\lambda,\Pi)=L(\lambda',\Pi')$
forces $\lambda'-\lambda\in \mathbb{Z}\Delta$.

\subsubsection{}\label{DeltaLdef}
For each  $\lambda\in\fh^*$ we introduce the sets
$$\begin{array}{l}
D(\lambda)_{iso}=\{\alpha\in\Delta|\ (\alpha,\alpha)=0,\ \ (\lambda+\rho,\alpha)=0\},\\
D(\lambda)_{o}:=\{\alpha\in\Delta_{\ol{1}}|\ (\alpha,\alpha)\not=0, \lang\lambda+\rho,\alpha^{\vee}\rang\in 2\mathbb{Z}+1\},\\
D(\lambda)_{e}:=\{\alpha\in\Delta_{\ol{0}}|\ (\alpha,\alpha)\not=0, \frac{\alpha}{2}\not\in\Delta_{\ol{1}},\ \lang\lambda+\rho,\alpha^{\vee}\rang\in \mathbb{Z}\}.
\end{array}
$$

Let $W_{ess,\lambda}$ be the subgroup of $W$ generated by the reflections 
$\{r_{\alpha}|\ \alpha\in D(\lambda)_o\cup D(\lambda)_e\}$. 
From~\Prop{prop311}, it follows that
\begin{equation}\label{eq311}
D(\lambda)_o=D(\lambda+\nu)_o,\ D(\lambda)_e=D(\lambda+\nu)_e,\ 
W_{ess,\lambda}=W_{ess,\lambda+\nu} \text{ for each } \nu\in Q.
\end{equation}

We introduce the following subset of $\Delta$
$$\Delta_{ess}(\lambda)=D(\lambda)_o\cup D(\lambda)_e\cup 
\{2\alpha|\ \alpha\in D(\lambda)_o\}\cup
W_{ess,\lambda}D(\lambda)_{iso}.$$
Note that the even roots in
$\Delta_{ess}(\lambda)$ are $D(\lambda)_e\cup
 \{2\alpha|\ \alpha\in D(\lambda)_o\}$ and that
$D(\lambda)_o$ (resp., $W_{ess,\lambda}D(\lambda)_{iso}$) is the set of 
non-isotropic 
(resp., isotropic) odd roots in $\Delta_{ess}(\lambda)$. 
The main motivation for this definition is~\Prop{propKK}.

For an irreducible highest weight module $L=L(\lambda,\Pi)$ we set 
$\Delta_{ess}(L):=\Delta_{ess}(\lambda+\rho),
W_{ess}(L):=W_{ess;\lambda+\rho}$.
\Prop{propDeltaLwell} shows that $\Delta_{ess}(L)$ is well defined.

\subsubsection{}
\begin{lem}{lembetagamma}
For any $\beta\in D(\lambda+\rho)_{iso}$ and 
$\gamma\in D(\lambda+\rho+\beta)_{iso}$,
one has $\gamma\in \Delta_{ess}(\lambda+\rho)$.
\end{lem}
\begin{proof}
If $(\beta,\gamma)=0$, then $\gamma\in D(\lambda+\rho)_{iso}$.
Assume that $(\beta,\gamma)\not=0$. Then $\beta+\gamma$ or $\beta-\gamma$ is an 
even root (this is proven in~\cite{VGRS} for the finite root systems;
the non-twisted affine case follows immediately); we denote this root by $\alpha$.
One has
$$(\lambda+\rho,(\beta\pm\gamma)^{\vee})=
\frac{2(\lambda+\rho,\gamma\pm\beta)}{(\gamma\pm\beta,\gamma\pm\beta)}=
\frac{2(\lambda+\rho,\gamma)}{\pm 2(\gamma,\beta)}=
\frac{2(-\beta,\gamma)}{\pm 2(\gamma,\beta)}
=\pm 1,$$
so $\alpha\in D(\lambda+\rho)_e$ if $\frac{\alpha}{2}\not\in\Delta_{\ol{1}}$.

Consider the case when $\frac{\alpha}{2}\in\Delta_{\ol{1}}$; then 
$\frac{\alpha}{2}$ is a non-isotropic odd root, so $\fg$ is of the types 
$B(m,n), G(3)$
or their affinizations. In these cases the roots $\beta,\gamma$ are of the form 
$k_1\delta\pm (\vareps_i+\delta_j),\ k_2\delta\pm (\vareps_i-\delta_j)$,
where $k_1,k_2\in\mathbb{Z}$ and $k_1+k_2$ is even. 
(Here and further we use the description of
root systems in~\cite{K1}.)
Then either $\alpha=\beta+\gamma$ and 
$\frac{\beta-\gamma}{2}\in\Delta_{\ol{0}}$, 
or $\alpha=\beta-\gamma$ and $\frac{\beta+\gamma}{2}\in\Delta_{\ol{0}}$;
observe that $\frac{\beta\pm\gamma}{4}\not\in\Delta$.
By above,
$$(\lambda+\rho,(\frac{\beta\pm\gamma}{2})^{\vee})=\pm 2$$
so $D(\lambda+\rho)_e$ contains $\frac{\beta-\gamma}{2}$ or $\frac{\beta+\gamma}{2}$.

Since $\gamma=r_{\beta-\gamma}\beta=-r_{\beta+\gamma}\beta$, we obtain $\gamma\in\Delta_{ess}(\lambda+\rho)$, as required.
\end{proof}

\subsubsection{}
\begin{prop}{propDeltaLwell}
If $\Pi$ and $\Pi'$ are two sets of simple roots and $L(\lambda,\Pi)=L(\lambda',\Pi')$, then
$\Delta_{ess}(\lambda+\rho)=\Delta_{ess}(\lambda'+\rho')$.
\end{prop}
\begin{proof}
Since any two sets of simple roots are connected by a chain of odd reflections and
each odd reflection is invertible,
it is enough to show that $\Delta_{ess}(\lambda'+\rho')\subset\Delta_{ess}(\lambda+\rho)$ if
$\Pi'=r_{\beta}\Pi$, where $\beta$ is an odd isotropic root. 

If $\lambda+\rho=\lambda'+\rho'$,
then, obviously, $\Delta_{ess}(\lambda+\rho)=\Delta_{ess}(\lambda'+\rho')$. 
By~(\ref{lambdalambda'}), 
if $\lambda+\rho\not=\lambda'+\rho'$, then 
$\lambda'+\rho'=\lambda+\rho+\beta$ and $(\lambda+\rho,\beta)=0$.
From~\Prop{prop311} it follows that $D(\lambda'+\rho')_e=D(\lambda+\rho)_e$ and
$D(\lambda'+\rho')_o=D(\lambda+\rho)_o$; by~\Lem{lembetagamma},
$D(\lambda'+\rho')_{iso}\subset \Delta_{ess}(\lambda+\rho)$. 
The assertion follows.
\end{proof}

\subsubsection{}
\begin{prop}{propDeltaL}
One has $W_{ess}(L)\Delta_{ess}(L)=\Delta_{ess}(L)$.
\end{prop}
\begin{proof}
It is enough to verify that for $\alpha,\gamma\in D(\lambda)_o\cup D(\lambda)_e$
one has $r_{\alpha}\gamma\in D(\lambda)_o\cup D(\lambda)_e$. We have
$$(\lambda,(r_{\alpha}\gamma)^{\vee})=(\lambda,\gamma^{\vee})-(\lambda,\alpha^{\vee})
(\alpha,\gamma^{\vee}).$$
By~\Prop{prop311}, $(\alpha,\gamma^{\vee})$ is an integer (resp., even integer) if $\gamma$ is even
(resp., odd). Thus $r_{\alpha}\gamma\in D(\lambda)_o\cup D(\lambda)_e$, as required.
\end{proof}

\subsection{The properties of $\supp (Re^{\rho}\ch L)$.}
Fix a choice of $\Delta^+$, and let $\lambda\in\fh^*$ be a non-critical weight.
Recall that $\lambda$ is called {\em typical} if $(\lambda+\rho,\beta)\not=0$ for all odd
isotropic roots $\beta$, and {\em atypical} otherwise.

The following proposition is proven in~\cite{KK} in the Lie algebra case and in~\cite{GK}
in the general Lie superalgebra case.

\subsubsection{}
\begin{prop}{propKK}
If $\mu\in \supp (Re^{\rho}\ch L(\lambda))$, then
there exists a chain $\mu=\mu_r<\mu_{r-1}<\ldots<\mu_0=\lambda+\rho$, where
either $\mu_k=r_{\gamma}\mu_{k-1}$ for 
$\gamma\in (D(\mu_{k-1}+\rho)_o\cup D(\mu_{k-1}+\rho)_e)\cap\Delta^+$
or
$\mu_k=\mu_{k-1}-\gamma$ for $\gamma\in D(\mu_{k-1}+\rho)_{iso}\cap\Delta^+$.
\end{prop}

\subsubsection{}
Taking into account~(\ref{eq311}) and~\Lem{lembetagamma},
we obtain by induction on $k$ that $D(\mu_k+\rho)_e=D(\lambda+\rho)_e, D(\mu_k+\rho)_o=D(\lambda+\rho)_o$
and $\Delta_{ess}(\mu_k+\rho)=\Delta_{ess}(\lambda+\rho)$.

This leads to the following useful properties of $\supp (Re^{\rho}\ch L(\lambda))$.

\subsubsection{}
\begin{cor}{corsuA}
Let $L=L(\lambda)$ be a non-critical irreducible highest weight module. For each  
$\nu\in\supp (Re^{\rho}\ch L(\lambda))$ one has

(i) $\nu\in\lambda+\rho-\mathbb{Z}_{\geq 0}(\Delta_{ess}(\lambda)\cap\Delta^+)$, 
$||\lambda+\rho-\nu||^2=||\lambda+\rho||^2$.

(ii) If $L(\lambda)$ is typical, then $\nu\in W_{ess}(L)(\lambda+\rho)$.

If $L(\lambda)$ is atypical, then $(\nu,\beta)=0$ for some
odd isotropic root $\beta$.

(iii) $\Delta_{ess}(\nu)=\Delta_{ess}(\lambda+\rho)$.
\end{cor}

\subsubsection{}
\begin{cor}{corsuA2}
Let $L=L(\lambda)$ be a non-critical irreducible highest weight module and
$\Delta_{ess}$ is not irreducible, that is
$$\Delta_{ess}=\coprod_{i\in X} \Delta_{ess}^{i},\ \text{ where } (\Delta^i_{ess},\Delta^j_{ess})=0\ 
\text{ for all } i\not=j.$$

If $\lambda+\rho-\mu\in\supp (Re^{\rho}\ch L(\lambda))$, then $\mu$ can be written as 
$\mu=\sum_{i\in X} \mu_i$
with the property that for each $i\in X$ there exists a chain 
$\lambda-\mu^i=\mu_r<\mu_{r-1}<\ldots<\mu_0=\lambda+\rho$, where
either $\mu_k=r_{\gamma}\mu_{k-1}$ for 
$\gamma\in (D(\mu_{k-1}+\rho)_o\cup D(\mu_{k-1}+\rho)_e)\cap\Delta^+\cap\Delta_{ess}^i$
or
$\mu_k=\mu_{k-1}-\gamma$ for $\gamma\in D(\mu_{k-1}+\rho)_{iso}\cap\Delta^+\cap\Delta_{ess}^i$.

In particular, $\mu^i\in \mathbb{Z}_{\geq 0}(\Delta^i_{ess}(\lambda)\cap\Delta^+)$
and $||\lambda+\rho-\mu^i||^2=||\lambda+\rho||^2$.
\end{cor}

\subsubsection{}\label{WL}
A subset $E$ of the set of real roots 
$\Delta_{re}=\Delta\setminus\mathbb{Z}\delta$ 
is called a {\em root subsystem} if the following properties hold:

(i) if $\alpha\in E$, then $-\alpha\in E$;

(ii) if $\alpha\in E$ is not isotropic, then $r_{\alpha}E=E$;

(iii) if $\alpha\in E$ is isotropic and $\beta\in E$ is such that 
$(\alpha,\beta)\not=0$,
then either $\beta+\alpha$ or $\beta-\alpha$ lies in $E$.

Note that any subset $E'$ of $\Delta_{re}$ is contained in a unique minimal root subsystem
of $\Delta_{re}$. Indeed, in order to construct a minimal root subsystem containing $E'$,
we should add to $E'$ an element $-\alpha$ if $\alpha\in E'$,
$r_{\alpha}\beta$ if $\alpha,\beta\in E'$ and $\alpha$ is non-isotropic, and
one of the elements $\beta\pm\alpha$ which is in $\Delta_{re}$ if
$\alpha$ is isotropic with $(\alpha,\beta)\not=0$ (exactly one of
them is in $\Delta_{re}$), and repeat this procedure. 
It follows from the results of~\cite{VGRS}
that if this minimal root subsystem is finite, then it is a root system 
of a finite-dimensional basic simple Lie superalgebra. 
If this minimal root subsystem is infinite, then it is not hard to show 
that it is the set of real roots
of an affine Lie superalgebra. 

If $E'=E'_1\coprod E'_2$ such that $(E_1,E_2)=0$, then 
the minimal root subsystem of $\Delta_{re}$ containing $E'$ is  of the form
$E_1\coprod E_2$, where $E_1,E_2$  minimal root subsystems containing 
$E_1', E_2'$ respectively.

Let $\Delta(L)$ (resp., $\Delta(\nu)$) 
be the minimal root subsystem of $\Delta_{re}$ containing $\Delta_{ess}(L)$ 
(resp., $\Delta_{ess}(\nu)$). Denote by $W(L)$ the Weyl group of $\Delta(L)$, 
i.e. the subgroup of $W$ generated by reflections in non-isotropic
roots from $\Delta(L)$.

All results~\Lem{lembetagamma}---\Cor{corsuA}
remain valid if we replace $\Delta_{ess}(L)$ by $\Delta(L)$.

\subsection{Examples}\label{exaDeltaess}
Assume that $L:=L(\lambda)$ is non-critical.

If $L(\lambda)$ is typical, then $\Delta_{ess}(L)=\Delta(L)$
consists of non-isotropic roots. 

If $L(\lambda)$ is finite-dimensional, then
$\Delta(L)=\Delta_{ess}(L)=\Delta_{re}$ if $L$ is atypical, and 
$\Delta(L)=\Delta_{ess}(L)=\{\alpha\in\Delta|\ (\alpha,\alpha)\not=0\}$ if $L$ is typical.

If $\fg=A(m,n)$ or $A(m,n)^{(1)}$, then $\Delta_{ess}(L)=\Delta(L)$.
However it is not true in general, as the following example shows.

\subsubsection{}
Let $\fg=D(m,n)$  and 
let $\lambda$ be such that $(\lambda,\vareps_i)=(\lambda,\delta_i)=\frac{1}{2}$.
Then $\Delta_{ess}(\lambda)_{\ol{0}}=
\{\pm\vareps_i\pm\vareps_j;\pm\delta_i\pm\delta_j:
i\not=j\}$ and $\Delta_{ess}(\lambda)_{\ol{1}}=\Delta_{\ol{1}}$,
so $\Delta_{ess}(\lambda)$ is not a root system;
in this case $\Delta(\lambda)=\Delta$.

\subsection{The sets $\Pi(L)$}\label{PiL}
Fix a set of positive roots $\Delta^+$, and let $\Pi$ be the subset of simple roots.
Let $\Delta(L)^+:=\Delta(L)\cap\Delta^+$ and
denote the corresponding set of simple roots  by $\Pi(L)$.
Denote by $\Pi_0(L)$ the set of simple roots of $\Delta(L)\cap\Delta_{\ol{0}}^+$
(it does not depend on the choice of $\Pi$). Note that
both $\Pi(L)$ and $\Pi_0(L)$ are linearly dependent, if $\Delta(L)$ has more
than one affine component.

\subsubsection{}
\begin{lem}{lemsuA2}
Let $\alpha\in\Pi(L(\lambda))$ be a non-isotropic root.  
For each $\nu\in supp(R e^{\rho}\ch L)$
one has $r_{\alpha}\nu\in\lambda+\rho-\mathbb{Z}_{\geq 0}\Pi(L(\lambda)))$.
\end{lem}
\begin{proof}
We prove the assertion by induction on the length of the chain in~\Prop{propKK}.
One has $\nu=\mu_r=\mu_{r-1}-a\gamma$ for $\gamma\in\Delta_{ess}(L(\lambda))^+$,
$a\in\mathbb{Z}_{>0}$. If $\gamma=\alpha$, then $\mu_r=r_{\alpha}\mu_{r-1}$
and $r_{\alpha}\nu=\mu_{r-1}\in\lambda+\rho-\mathbb{Z}_{\geq 0}\Pi(L(\lambda)))$ by~\Cor{corsuA} (i).
If $\gamma\not=\alpha$, then $r_{\alpha}\gamma\in\Delta(L(\lambda))^+$ and thus
$r_{\alpha}\mu_r=r_{\alpha}\mu_{r-1}-ar_{\alpha}\gamma$. By the induction hypothesis,
$r_{\alpha}\mu_{r-1}\in\lambda+\rho-\mathbb{Z}_{\geq 0}\Pi(L(\lambda)))$, so
$r_{\alpha}\mu_{r}\in\lambda+\rho-\mathbb{Z}_{\geq 0}\Pi(L(\lambda)))$, as required.
\end{proof}

\subsubsection{}
Denote by $R_L$ the analogue of $R$:
$$R_L=\frac{\prod_{\alpha\in\Delta_{\ol{0}}^+(L)}(1-e^{-\alpha})}
{\prod_{\alpha\in\Delta_{\ol{1}}^+(L)}(1+e^{-\alpha})}.$$
Note that $R_L\in\cR(\Pi)$.

Fix $\rho_{L}\in\fh^*$ such that $2(\rho_L,\alpha)=(\alpha,\alpha)$ for all
$\alpha\in\Pi(L)$.

\subsubsection{Remark}
Since $R_L,\rho_L$ depend on $\Pi$,
we will sometimes write $R_{L,\Pi},\rho_{L,\Pi}$ to prevent a confusion.
We choose $\rho_{L,\Pi}$ compatible for all subsets  of simple roots
$\Pi$, proceeding as in~\S~\ref{choicerho}.
The elements $R_{L,\Pi}e^{\rho_{L,\Pi}}$
are equivalent for all $\Pi$.

\subsubsection{}\label{Pilambda}
Let $\beta\in\Pi$ be an odd simple root and let $r_{\beta}$ 
be the corresponding odd reflection. Then
$$r_{\beta}\Delta^+\cap\Delta(L)=\left\{\begin{array}{ll}
       \Delta(L)^+& \text{ if }\beta\not\in\Pi(L)\\
         \Delta(L)^+\setminus\{\beta\}\cup\{-\beta\}  & 
\text{ if }\beta\in\Pi(L).
                                      \end{array}
\right.$$

 Recall that $r_{\beta}\Delta^+$ has
the set of simple roots $r_{\beta}\Pi:=\{r_{\beta}\alpha|\ \alpha\in\Pi\}$,
where
$$r_{\beta}\alpha=\left\{\begin{array}{ll}
-\beta &\text{ if }\alpha=\beta,\\
\alpha+\beta&\text{ if }(\alpha,\beta)\not=0,\\
\alpha &\text{ otherwise.}
\end{array}\right.$$

As a result, the set of simple roots for $r_{\beta}\Delta^+\cap\Delta(L)$ 
coincides
with $\Pi(L)$ if $\beta\not\in\Pi(L)$ and is equal to $r_{\beta}\Pi(L)$, defined
by the above formulas, if $\beta\in\Pi(L)$.

\subsection{Character formulas for different choices of $\Delta^+(\Pi)$}
Let $L$ be a irreducible highest weight module. 
Then $\ch L\in\cR(\Pi)$ for each $\Pi$.
Suppose that $\Delta(L)=\Delta$ and that for some $\Pi$ the following formula holds in
$\cR(\Pi)$:
$$Re^{\rho}\ch L=\sum_{w\in W'} x_w w\bigl(\frac{e^{\nu}}
{\prod_{\beta\in J}(1+e^{-\beta})}\bigr),$$
where $\nu\in \fh^*\setminus\{0\}, J\subset\Delta, x_w\in\mathbb{Q}$, 
and $W'$ satisfy the conditions 
of~\S~\ref{expA}. Combining~\S~\ref{RPi} and~\S~\ref{expA}, 
we conclude that the above formula holds
in $\cR(\Pi')$ for any $\Pi'$.

An important  special case is when $\nu=\lambda+\rho_{\Pi}, L=L(\lambda,\Pi), J=J(\Pi)$.
In this case we can obtain a formula of  similar
form  for certain other subsets of positive roots $\Delta^+(\Pi')$ as follows.
Let $\Pi'=r_{\beta}\Pi$ for an odd isotropic root $\beta\in\Pi$.
Then $L(\lambda,\Pi)=L(\lambda',\Pi')$, where $\lambda'$ is given by~(\ref{lambdalambda'}).
This implies the formula
$$Re^{\rho}\ch L(\lambda',\Pi')=\sum_{w\in W'} x_w 
w\bigl(\frac{e^{\lambda'+\rho_{\Pi'}}}
{\prod_{\beta\in J(\Pi')}(1+e^{-\beta})}\bigr),$$
where $J(\Pi')=J(\Pi)$ if $(\lambda+\rho,\beta)\not=0$ and
$J(\Pi')=(J(\Pi)\setminus\{\beta\})\cup\{-\beta\}$ if $(\lambda+\rho,\beta)=0$
and $\beta\in J(\Pi)$ (this method does not work if $(\lambda+\rho,\beta)=0$ and
$\beta\not\in J(\Pi)$).

\subsubsection{Remark}
If $L=L(\lambda,\Pi)$ with $\lambda\not=-\rho$, similar results hold if 
we substitute $W$ by the integral Weyl group $W(L)$, see~\S~\ref{WL}: 
in this case we take $W'$ generated
by reflections $r_{\alpha}, \alpha\in I$, where $I\subset \Pi_0(L)$.

\subsection{The map $F:\ L\mapsto \ol{L}$}\label{ollambda}
Fix $\Delta^+(\Pi)$ and a non-critical irreducible highest weight module 
$L=L(\lambda,\Pi)$. Set
$$\ol{\lambda}=\lambda+\rho-\rho_L.$$

Let $\fg'$ be the Kac-Moody  superalgebra with the set of real roots
$\Delta(L)$ and the set of simple roots $\Pi(L)$. This Kac-Moody superalgebra
is of finite or affine type (with a symmetrizable
Cartan matrix-- its symmetrization is given by the restriction of $(-,-)$
to $\Pi(L)$). Let $\fh'$ be a Cartan subalgebra of $\fg'$ and let $\fh''$
be a commutative Lie algebra of dimension $\dim\fh-\dim\fh'$.
Consider the Lie superalgebra $\ol{\fg}^{\lambda}:=\fg'\times\fh''$ and identify its Cartan subalgebra with $\fh$
in such a way that the simple roots of $\fg'$ are identified with $\Pi(L)$.
We denote by $\ol{L}(\ol{\lambda})$
an irreducible highest weight module
with highest weight $\ol{\lambda}$ over the Lie superalgebra $\ol{\fg}^{\lambda}$.

Recall that $L$ is non-critical and $\fg'$ is  either basic finite-dimensional (if $\Delta(L)$ is finite) or affine. 
It is easy to see that in the second case $\ol{L}(\ol{\lambda})$ 
is also non-critical. 
Indeed, let $\delta'\in\Delta(L)$ be the primitive imaginary root of $\fg'$. 
Then $||\delta'||^2=0$, so
$\delta'$ is an imaginary root in $\Delta_{\ol{0}}$, hence 
$\delta'=k\delta$, where $k\not=0$ and 
$\delta\in\Delta_{\ol{0}}^+$
is a primitive imaginary root. But 
$(\ol{\lambda}+\rho_L,\delta')=k(\lambda+\rho,\delta)\not=0$,
 so $\ol{L}(\ol{\lambda})$ is non-critical.

The following lemma shows that the map $F:\ L\to \ol{L}$ is a well-defined
map of  non-critical irreducible highest weight modules.

\subsubsection{}
\begin{lem}{FLproperty}
Let $\Pi,\Pi'$ be two sets of simple roots.
If $L=L(\lambda,\Pi)=L(\lambda',\Pi')$, then
$$F\bigl(L(\lambda,\Pi)\bigr)\cong F\bigl(L(\lambda',\Pi')\bigr).$$
\end{lem}
\begin{proof}
Recall that $\Pi'$ can be obtained from $\Pi$ by a sequence of odd reflections.
Let $\beta\in\Pi$ be an isotropic root and $\Pi'=r_{\beta}\Pi$. Denote by $\Pi(L)$
(resp., by $\Pi'(L)$) the set of simple roots for $\Delta^+\cap \Delta(L)$
(resp., for $(r_{\beta}\Delta^+)\cap \Delta(L)$) and choose $\rho_{L}$;
we may (and will) choose $\rho'_L:=\rho_L+\beta$.

Consider the case when $(\lambda+\rho,\beta)\not=0$. Then
$\lambda'+\rho_{r_{\beta}\Pi}=\lambda+\rho$ so
$\ol{\lambda}+\rho_{L}=\ol{\lambda}'+\rho'_{L}$.

If $\beta\not\in\Delta(L)$, then, by Section~\ref{Pilambda},
 $\Pi(L)=\Pi'(L)$
and thus $\ol{\lambda'}=\ol{\lambda}$ and
$$F(L(\lambda',r_{\beta}\Pi))=\ol{L}(\ol{\lambda}, \Pi(L))=
F(L(\lambda),\Pi)).$$

If $\beta\in\Delta(L)$, then, by Section~\ref{Pilambda}, $\Pi(L)=r_{\beta}\Pi'(L)$
and, in particular, we choose $\rho_L'=\rho_L+\beta$. Thus  $\ol{\lambda'}=
\ol{\lambda}-\beta$ and
$$F\bigl(L(\lambda',r_{\beta}\Pi)\bigr)=\ol{L}(\ol{\lambda}-\beta, r_{\beta}\Pi(L))\cong
\ol{L}(\ol{\lambda},\Pi(L))=
F\bigl(L(\lambda,\Pi)\bigr).$$

Consider the  case when $(\lambda+\rho,\beta)=0$. Then 
$\beta\in\Pi(L)$ and $\rho'=\rho+\beta,\
\rho_L'=\rho_L+\beta$. One has $\lambda'=\lambda$, so 
$\lambda'+\rho'=\lambda+\rho+\beta$ that is
$\ol{\lambda'}+\rho'_L=\ol{\lambda}+\rho_L+\beta$ which gives
$\ol{\lambda'}=\ol{\lambda}$. Therefore
$F\bigl(L(\lambda',r_{\beta}\Pi)\bigr)=\ol{L}(\ol{\lambda'}, 
r_{\beta}\Pi(L))\cong
\ol{L}(\ol{\lambda}, \Pi(L))=
F\bigl(L(\lambda,\Pi)\bigr)$.
\end{proof}

\subsubsection{}\label{KRWconjecture}
In~\cite{KRW} it was conjectured that  the characters of an admissible $\fg$-module 
$L:=L(\lambda,\Pi)$ and the $\ol{\fg}^{\lambda}$-module $F(L)$ are related by the formula (cf.~(\ref{1}))

\begin{equation}\label{KRWconj}
  Re^{\rho}\ch L=R_Le^{\rho_L}\ch F(L).
\end{equation}

Since $\Delta(L)^+\subset \Delta^+(\Pi)$ for each $\Pi$, the algebra
$\cR(\Pi(L))$ can be naturally embedded in $\cR(\Pi)$, so 
we consider the above formula as an equality in $\cR(\Pi)$.
By~\S~\ref{RPi} and~\Lem{FLproperty} the elements in the left-hand side (resp., the elements in the right-hand side) are equivalent for all subsets of 
simple roots.
Therefore if the formula holds for some $\Pi$, it holds for all other choices
of $\Pi$.

In the next sections we verify formula~(\ref{KRWconj}) for certain cases.

\section{Linkage $L\sim L'$}\label{enr}
Let $\fg$ be a basic or  affine Lie superalgebra 
and $\Delta^+(\Pi)$ be a subset of positive roots in the set of roots $\Delta$.
Let $\Pi_0$ be the set of simple roots for $\Delta^+_{\ol{0}}$. Define
the standard dot action
of the Weyl group by $w.\lambda:=w(\lambda+\rho)-\rho$.

\subsection{Enright functor}
Fix $\alpha\in\Pi_0$ and $a\in\mathbb{C}/\mathbb{Z}$.
Let $\mathbb{M}(\fg,a)$ (resp., $\mathbb{M}(\fg_{\ol{0}},a)$)
be the category of $\fg$-modules (resp., $\fg_{\ol{0}}$-modules) 
$M$ with a locally
nilpotent action of  a root vector
$e_{\alpha}$, a diagonal action of $\fh$,
i.e., $M=\oplus_{\mu\in\fh^*} M_{\mu}$, and such that
$(\mu,\alpha^{\vee})\equiv a\ \mod\mathbb{Z}$ if $M_{\mu}\not=0$.

For each $\lambda\in\fh^*$ denote by $\CO_{\lambda}$ the full subcategory of 
the category $\CO$, whose objects are $\fg$-modules $N$ satisfying $N_{\mu}=0$ if 
$\mu-\lambda\not\in\mathbb{Z}\Delta$.
Note that  $\CO_{r_{\alpha}.\lambda}=\CO_{r_{\alpha}\lambda}$.

We denote by $\breve{M}(\nu)$ (resp., $\breve{L}(\nu)$) 
the Verma (resp., irreducible) $\fg_{\ol{0}}$-module with the highest weight $\nu$ 
(since $\Pi_0$ is fixed, these modules do not depend on the choice of $\Pi$).

We will use the following theorem which can be easily deduced from~\cite{KT2},\cite{IK}.

\subsubsection{}
\begin{thm}{thmEnr}
For each $a\in\mathbb{C}/\mathbb{Z}$
there exists a left exact
functor $T_{\alpha}(a): \mathbb{M}(\fg_{\ol{0}},a)\to \mathbb{M}(\fg_{\ol{0}},a)$ ({\em Enright functor})
which induces  a left exact functor 
$T_{\alpha}(a): \mathbb{M}(\fg,a)\to \mathbb{M}(\fg,-a)$ with the following properties.

(a) Assume that $\alpha\in\Pi$ or $\frac{\alpha}{2}\in\Pi$, $M(\lambda,\Pi)\in \mathbb{M}(\fg,a)$ and 
$(\lambda,\alpha^{\vee})\not\in\mathbb{Z}$. One has

$$\begin{array}{l}
T_{\alpha}(a)(M(\lambda,\Pi))=M(r_{\alpha}.\lambda,\Pi),\ \ 
T_{\alpha}(a)(L(\lambda,\Pi))=L(r_{\alpha}.\lambda,\Pi),\\
T_{\alpha}(a)(\breve{M}(\lambda,\Pi))=\left\{
\begin{array}{ll}
\breve{M}(r_{\alpha}.\lambda,\Pi)\ & \text{ if } \alpha\in\Pi,\\
\breve{M}(r_{\alpha}.\lambda-\frac{\alpha}{2},\Pi)\ & \text{ if } \alpha/2\in\Pi.\\
\end{array}
\right. \\
T_{\alpha}(a)(\breve{L}(\lambda,\Pi))=\left\{
\begin{array}{ll}
\breve{L}(r_{\alpha}.\lambda,\Pi)\ & \text{ if } \alpha\in\Pi,\\
\breve{L}(r_{\alpha}.\lambda-\frac{\alpha}{2},\Pi)\ & \text{ if } \alpha/2\in\Pi.\\
\end{array}
\right.
\end{array}$$

(b) If  $0\not=a\in\mathbb{C}/\mathbb{Z}$, then $T_{\alpha}(a)$ is an equivalence of categories 
$\mathbb{M}(\fg,a)\to \mathbb{M}(\fg,-a)$ and the inverse is given by
$T_{\alpha}(-a)$.
 
(c) If  $\alpha\not\in\Delta(\lambda)$ and  $a\equiv (\lambda+\rho,\alpha^{\vee})$, 
then $T_{\alpha}(a)$ provides an equivalence of categories 
$\CO_{\lambda}\iso \CO_{r_{\alpha}\lambda}$.
\end{thm}

\subsection{The linkage $\sim$}
\label{equivcat}
Let $\Theta$ be the set of functors $T: \CO_{\lambda}\to\CO_{\lambda'}$
which can be presented as compositions of  Enright functors $T_{\alpha}(a):\CO(\lambda)\iso \CO_{r_{\alpha}\lambda}$
with $\alpha\not\in\Delta(\lambda),\ a=(\lambda,\alpha^{\vee})$. 
By~\Thm{thmEnr} (c),
each $T\in\Theta$ is an equivalence of categories.

We will say that the highest weight irreducible modules $L$ and $L'$ are {\em linked}
if $L'=T(L)$ for $T\in\Theta$, and denote it by $L\sim L'$.

\subsubsection{}\label{simseq}
Let $L\sim L'$ be two linked highest weight irreducible modules 
and $\Pi,\Pi'$ be two subsets of simple roots. Then there exists a finite chain
$$L=L^1=L(\lambda^1,\Pi^1), 
L^2=L(\lambda^2,\Pi^2),\ldots, L^t=L(\lambda^t,\Pi^t)=L',$$
where $\Pi^1=\Pi, \Pi^t=\Pi'$, and
for each $i$ we have
$$\begin{array}{l}
L^{i+1}=L^i, \Pi^{i+1}=r_{\beta}\Pi^i\ \text{ for some isotropic } \beta\in\Pi^i,\ \ \ \ \ \text{ or }\\
\Pi^{i+1}=\Pi^i \text{ and }L^{i+1}=T(L_{i}) \text{ for some } T=T_{\alpha}(a)\in\Theta:
\alpha\in\Pi^i\text{ or }\frac{\alpha}{2}\in\Pi^i.
\end{array}$$

Note that in the first case $\lambda^i+\rho^i=\lambda^{i+1}+\rho^{i+1}$
if $(\lambda^i+\rho^i,\beta)\not=0$ and $\lambda^i+\rho^i+\beta=\lambda^{i+1}+\rho^{i+1}$ if  $(\lambda^i+\rho^i,\beta)=0$,
see~(\ref{lambdalambda'}); in particular, one has
$||\lambda^i+\rho^i||^2=||\lambda^{i+1}+\rho^{i+1}||^2$.
In the second case
($\Pi^{i+1}=\Pi^i$) one has $\alpha\not\in\Delta(\lambda^i)$ and 
$\lambda^{i+1}=r_{\alpha}(\lambda^i+\rho^i)-\rho^i$,
where $\rho^i$ is the Weyl vector for $\Pi^i$. This implies the following useful property of the linkage:
$$L(\lambda,\Pi)\sim L(\lambda',\Pi')\ \ \Longrightarrow\ \
||\lambda+\rho||^2=||\lambda'+\rho'||^2.$$

\subsubsection{}\label{FLFL'}
Fix $\Pi$. Let $L'=T_{\alpha}(a)(L)$ for some $\alpha\not\in\Delta(L)$ such that
$\alpha$ or $\frac{\alpha}{2}$ lies in $\Pi$. Write $L=L(\lambda,\Pi), 
L'=L(\lambda',\Pi)$
and recall that $\lambda=r_{\alpha}.\lambda'$. Then 
$\Delta(L')=r_{\alpha}(\Delta(L))$.
The conditions on $\alpha$ imply 
$r_{\alpha}(\Delta(L))\cap\Delta^+=\Delta(L)\cap\Delta^+$
so $\Pi(L')=r_{\alpha}(\Pi(L))$. 
 
Recall the algebra $\ol{\fg}^{\lambda}$ and the map $L\to F(L)$ introduced in~\S~\ref{ollambda}.
Since $\Pi(\lambda')=r_{\alpha}(\Pi(\lambda))$, there exists
a natural isomorphism of the Kac-Moody superalgebras
$\iota: [\ol{\fg}^{\lambda},\ol{\fg}^{\lambda}]\iso [\ol{\fg}^{\lambda},\ol{\fg}^{\lambda}]$ with the property
$\iota: \ol{\fg}^{\lambda}_{\beta}\to\ol{\fg}^{\lambda}_{r_{\alpha}\beta}$
for each $\beta\in\Pi(L)$. This isomorphism can be extended to
the isomorphism $\iota:\ol{\fg}^{\lambda}\iso\ol{\fg}^{\lambda}$ such that
$\iota(h)=r_{\alpha}(h)$ for each $h\in\fh$. Then 
$\ol{\lambda}(h)=\ol{\lambda'}(\iota(h))$
(see~\S~\ref{ollambda} for notation), 
so there exists an isomorphism $F(L)\iso F(L')$ 
 compatible with $\iota$
(i.e. $\iota'(av)=\iota(a)v$ for all $a\in \ol{\fg}^{\lambda}, v\in F(L)$).

\subsubsection{}
\begin{cor}{corFLsim}
If $L(\lambda,\Pi)\sim L(\lambda',\Pi')$, then  there exists
an isomorphism $\iota:\ol{\fg}^{\lambda}\iso \ol{\fg}^{\lambda}$
and an isomorphism $F: F(L)\iso F(L')$ compatible with $\iota$.
\end{cor}

\subsubsection{}
\begin{prop}{propsimch}
If $L\sim L'$ and~(\ref{KRWconj}) holds for $L$, then~(\ref{KRWconj}) holds for $L'$.
\end{prop}
\begin{proof}
Recall~\S~\ref{KRWconjecture} that if~(\ref{KRWconj})  
holds in $\cR(\Pi)$, it holds in $\cR(\Pi')$ for each $\Pi'$.

In the light of~\S~\ref{simseq} it is enough to consider the case when 
$L'=T_{\alpha}(a)(L)$ for some $\alpha\not\in\Delta(L)$ such that 
$\alpha\in\Pi\text{ or }\frac{\alpha}{2}\in\Pi$. By~\S~\ref{FLFL'}, in this case
$$\ch F(L')=r_{\alpha}\bigl(\ch F(L)\bigr)$$
in $\cR(\Pi)$ 
(both elements lie in $\cR(\Pi)$ since $\Pi(L), \Pi(L')=r_{\alpha}(\Pi(L))$
lie in $\Delta^+$). One has $R_{L'}e^{\rho_{L'}}=r_{\alpha}\bigl(R_{L}e^{\rho_L}\bigr)$ so
$$R_{L'}e^{\rho_{L'}}\ch F(L')=r_{\alpha}\bigl(R_{L}e^{\rho_{L}}\ch F(L)\bigr)$$
in $\cR(\Pi)$.

Write $L=L(\lambda,\Pi)$. By~\Thm{thmEnr}, $T_{\alpha}(a)$ is an equivalence of categories
$\CO_{\lambda}\iso\CO_{\lambda'}$ and $T_{\alpha}(a)(M(\nu))=M(r_{\alpha}.\nu)$
for each $\nu\in \CO_{\lambda}$. Since $L'=T_{\alpha}(a)(L)$, we obtain
$$Re^{\rho}\ch L'=r_{\alpha}\bigl(Re^{\rho}\ch L\bigr)$$
in $\cR(\Pi)$. This completes the proof.
\end{proof}

\subsection{}\label{WbetainDeltaL}
We will also use the following simple fact.

\begin{lem}{}
For each $w\in W$ there exists $L'\sim L$ such that
$\Delta(L')=w\Delta(L)$.
\end{lem}
\begin{proof}
Take $\alpha\in\Pi_0$.  If $\alpha\in\Delta(L)$, then 
$r_{\alpha}\Delta(L)=\Delta(L)$; if $\alpha\not\in\Delta(L)$, then 
$\Delta(L')=r_{\alpha}\Delta(L)$ for $L'=T_{\alpha}(a)(L)$.
\end{proof}

Note that for $\Delta=A(m,n), C(n), A(m,n)^{(1)}, C(n)^{(1)}$ any odd root
lies in a set of simple roots. Hence, in these cases, for any $\beta\in \Delta(L)$
we can choose $\Pi$ such that $\beta\in\Pi$.

If $\Delta\not=A(m,n), C(n), A(m,n)^{(1)}, C(n)^{(1)}$, then 
$W$ acts transitively on the set of odd isotropic roots, so
if $\Delta(L)$ contains odd isotropic roots, then for any
odd isotropic root $\beta\in\Delta$ there exists $L'\sim L$ such that
$\beta\in \Delta(L')$.

\section{Typical case for $\ol{\fg}^{\lambda}$ with even root subsystem of rank $\leq 2$}
\label{sect8}
Recall that for an affine Weyl group every orbit on  non-zero level has 
a unique maximal element in the order $\geq$ (if the level is not a 
negative rational number)
or a unique minimal element (if the level is not a positive rational number).
We call $\lambda\in\fh^*$ {\em extremal}
if for each connected component $\pi$ of $\Pi_0(\lambda)$, $\lambda$
is maximal or minimal in its $W(\pi)$-orbit.

Fix a non-critical level $k$ and denote by $W_+$ (resp., $W_-$)
the subgroup of $W(\lambda)$ generated by $r_{\alpha}$
such that $k(\alpha,\alpha)>0$  (resp., $k(\alpha,\alpha)<0$).
We introduce a partial  order on the Weyl group $W(L)$ as follows:

for $x_+,y_+\in W_+, x_-,y_-\in W_-$ let $x_+x_-\leq_k y_+y_-$ 
if $x_+\leq y_+, x_-\geq y_-$ in the Bruhat order on $W$.

\subsection{}
\begin{prop}{propextr}
If $\lambda$ is extremal, then for each $y\leq_k w$ there exists  an embedding
$M(w.\lambda)\to M(y.\lambda)$ and $\Hom (M(w.\lambda), M(y.\lambda))=1$.
\end{prop}
\begin{proof}

Recall that the Bruhat order is the unique order satisfying 
$e\geq w$ forces $w=e$,  and for each $\alpha\in\Pi_0$ one has:

$l(r_{\alpha}w)<l(w)$ forces $w\geq r_{\alpha}w$;

$w\geq w'$ forces  $r_{\alpha}w\geq r_{\alpha}w'$ or $w\geq r_{\alpha}w'$;

$w\geq w'$ forces $r_{\alpha}w\geq r_{\alpha}w'$ or $r_{\alpha}w\geq w'$.

We claim that any order $\leq '$ with the properties 

$l(r_{\alpha}w)<l(w)$ forces $ r_{\alpha}w\leq ' w$, and
$y\leq' w$ forces  $r_{\alpha}y\leq' r_{\alpha}w$ or $r_{\alpha}y\leq ' w$,

satisfies $y\leq w\ \Longrightarrow\ y\leq ' w$.

Indeed, the first property implies that $e\leq' x$ for all $x$.
Let us prove the assertion by induction on $l(w)$.
If $l(w)=0$, then $w=y=e$ and thus $y\leq' w$. 
Now take any $w$ with $l(w)>0$ and $\alpha\in\Pi_0$ such that 
$l(r_{\alpha}w)<l(w)$.
Then the  property implies $r_{\alpha}w\leq 'w$.
One has $y\leq r_{\alpha}w$  or $ r_{\alpha}y \leq r_{\alpha}w$. 
In the first case,
the induction hypothesis gives $y\leq' r_{\alpha}w$, so $r_{\alpha}w\leq 'w$
implies $y\leq ' w$, as required. In the case $r_{\alpha}y \leq r_{\alpha}w$
the induction hypothesis gives $r_{\alpha}y \leq' r_{\alpha}w$ and 
the second property gives $y\leq ' w$ or $y\leq' r_{\alpha}w\leq ' w$.
The claim follows.

\subsubsection{}\label{typ1}
Next, let us show that  for any typical weight $\lambda$ ,
$M(r_{\alpha}.\lambda)$ is a submodule of $M(\lambda)$ and $\dim\Hom(M(r_{\alpha}.\lambda), M(\lambda))=1$
if $r_{\alpha}.\lambda<\lambda$ and $\alpha\in\Pi_0(\lambda)$.

The proof is by induction on $\alpha$ (with respect to the order given by $\Pi_0$). 
Indeed, if $\alpha\in\Pi_0$ this immediately follows from typicality of $M(\lambda)$.
Take $\gamma\in\Pi_0$ such that
$r_{\gamma}\alpha<\alpha$. Since $\alpha\in\Pi_0(\lambda)$ one has $\gamma\not\in\Pi_0(\lambda)$
so the Enright functor $T_{\gamma}$ is an equivalence of categories. By induction,
$M(r_{r_{\gamma}\alpha}.(r_{\gamma}.\lambda))$ is a submodule of $M(r_{\gamma}.\lambda)$
so $M(r_{\alpha}.\lambda)$ is a submodule of $M(\lambda)$
(since $r_{\alpha}=r_{\gamma}r_{r_{\gamma}\alpha}r_{\gamma}$).
This proves that $M(r_{\alpha}.\lambda)$ is a submodule of $M(\lambda)$
for any $\alpha\in\Pi_0(\lambda)$.  Moreover, by induction
\begin{equation}\label{dimhom1}
\dim Hom(M(r_{r_{\gamma}\alpha}.(r_{\gamma}.\lambda)),M(r_{\gamma}.\lambda))=1.
\end{equation}
If $\dim\Hom(M(r_{\alpha}.\lambda), M(\lambda))>1$, then 
there exists an exact sequence
$$0\to N\to M(r_{\alpha}.\lambda)\oplus M(r_{\alpha}.\lambda)\to M(\lambda)$$
and $N_{r_{\alpha}.\lambda}=0$ that is
$\Hom(M(r_{\alpha}.\lambda), N)=0$. Using the Enright functor we obtain
the exact sequence
$$0\to N'\to M(r_{\gamma}r_{\alpha}.\lambda)\oplus 
M(r_{\gamma}r_{\alpha}.\lambda)\to M(r_{\gamma}.\lambda)$$
with $\Hom(N',M(r_{\gamma}r_{\alpha}.\lambda))=0$.
Since $r_{\gamma}r_{\alpha}=r_{r_{\gamma}\alpha}r_{\gamma}$, this contradicts~(\ref{dimhom1}).

\subsubsection{}
Denote $M(w.\lambda)$ as $M(w)$. Let us show the existence of the embedding $M(w)\subset M(y)$
for $y\leq _k w$.

By~\S~\ref{typ1}, $M(r_{\alpha}x)$ is a submodule of $M(x)$  if $r_{\alpha}x>_k x$ for $\alpha\in\Pi_0(\lambda)$.

It remains to show that for $\alpha\in\Pi_0(\lambda)$
 $M(w)\subset M(y)$ the module $M( r_{\alpha}y)$ contains
  $M(r_{\alpha}w)$ or $M(w)$.

Indeed, using the Enright functors as in~\S~\ref{typ1} 
we can reduce the question to the case when $\alpha\in\Pi_0$.
If $M(r_{\alpha}y)$ contains $M(y)$ it contains $M(w)$ as well.
Otherwise, $M(r_{\alpha}y)$ is a submodule of  $M(y)$
and $f_{\alpha}$ acts locally nilpotently on $M(y)/M(r_{\alpha}y)$.
If $M(r_{\alpha}w)$ contains $M(w)$, then $f_{\alpha}$ acts
and injectively on $L(w.\lambda)$ so
$\Hom(L(w.\lambda), M(y)/M(r_{\alpha}y)=0$
and thus  $\Hom (M(w),M(r_{\alpha}y))=\Hom (M(w),M(y))$, as required.
If $M(r_{\alpha}w)$  is a submodule of $M(w)$, it is also a submodule of $M(y)$
and, by above, $\Hom (M(r_{\alpha}w),M(r_{\alpha}y))=\Hom (M(r_{\alpha}w),M(y))$.

\subsubsection{}
It remains to verify that $\dim\Hom(M(w),M(y))\leq 1$ for each $y,w$.

Write $w=w_-w_+, y=y_-y_+$ with $y_-,w_-\in W_-, y_+,w_+\in W_+$. We proceed by induction on 
$l(w_+)+l(y_-)$.

Assume that $l(w_+)\not=0$. Then there exists 
$\alpha\in\Pi(\lambda)$ such that $r_{\alpha}\in W_+$ and $M(w)\subset M(r_{\alpha}w)$.
Using the Enright functors as in~\S~\ref{typ1} 
we can reduce the question to the case when $\alpha\in\Pi_0$.
Then, by~\cite{IK} Cor. 4.1, $T_{\alpha}(M(w))=M(r_{\alpha}w)$
and $\dim \Hom (M(w),M(y))\leq \dim \Hom (T_{\alpha}(M(w), T_{\alpha}(M(y))$.
Since $T_{\alpha}(M(y))$ is either $M(y)$ or $M(r_{\alpha}y)$,
we obtain $\dim \Hom (M(w),M(y))\leq \dim \Hom (M(r_{\alpha}w),M(y'))$ for some $y'$.
Arguing like this we obtain
$\dim \Hom (M(w),M(y))\leq \dim \Hom (M(w_-),M(y'))$.

It remains to verify that $\dim \Hom (M(w),M(y))\leq 1$
for $w\in W_-$.  We prove this by induction on $l(w)$.

If $l(w)=0$, then $w=e$. If $y\not\in W_+$, 
there exists $\alpha\in\Pi(\lambda)$ 
such that $r_{\alpha}\in W_-$ and $M(y)\subset M(r_{\alpha}y)$.
Then $f_{\alpha}$ acts locally nilpotently on $M(y)/M(r_{\alpha}y)$
and injectively on $M(e)$. Hence 
$\dim \Hom (M(e),M(y))\leq \dim \Hom (M(e),M(r_{\alpha}y)$
and therefore $\dim \Hom (M(e),M(y))=\dim \Hom (M(e),M(y_+))$.
However $y_+\lambda\leq\lambda$, so that we have $\dim \Hom (M(e),M(y_+))\leq 1$.

Assume that $\dim \Hom (M(w),M(y))\leq 1$ for all $w\in W_-$ with $l(w)\leq r$.
Take $w\in W_-$ with $l(w)=r+1$. Take $\alpha\in\Pi(\lambda)$ 
such that $r_{\alpha}\in W_-$ and $l(r_{\alpha}w)=r$. Then 
$M(r_{\alpha}w)\subset M(w)$ and $\dim  \Hom (M(r_{\alpha}w),M(y))\leq 1$.
As before, using Enright functors we can assume that $\alpha\in\Pi_0$.
 But then 
the embedding of $M(r_{\alpha}w)$ in $M(w)$ is given by the multiplication 
of the highest weight vector to $f_{\alpha}$ which is non-zero divisor
so $\dim  \Hom (M(w),M(y))\leq 1\leq \dim  \Hom (M(r_{\alpha}w),M(y))\leq 1$,
as required. This completes the proof.
\end{proof}

\subsection{}
Now let $L=L(\lambda,\Pi)$ be a $\fg$-module.
Assume that $\ol{\fg}^{\lambda}$ is one of the  affine Lie superalgebra with 
the set of even roots 
which is the union of affine or finite root systems  of rank at most two.

Notice that $W(L)$ is the direct product of several ($1,2$ or $3$) copies of 
the infinite (or finite, if $\dim\ol{\fg}^{\lambda}<\infty$) dihedral groups
(generated by $s_0,s_1$ subject to the relations $s_0^2=s_1^2$
and the braid relations if $\dim\ol{\fg}^{\lambda}$ is finite),
so all Kazhdan-Lusztig polynomials are trivial, namely are either $0$
or $1$.

\subsubsection{}
Let $\lambda$ be a non-critical extremal typical weight. By~\Prop{propextr}, 
 for $y\leq _k z, y,z\in W(\lambda)$ 
the module $M(y.\lambda)$ contains a unique singular vector $v(z)$
of weight $z.\lambda$ and this vector gives rise to an embedding
$M(z.\lambda)\subset M(y.\lambda)$. 

Consider a Kac-Moody algebra $\tilde{\fg}$ with  the set of 
simple roots $\Pi_0(L)$ (and the set of real roots
$\Delta_{\ol{0}}(L)\cap \Delta_{re}$); this algebra  coincides
with $\fg_{\lambda,\ol{0}}$ iff $\Delta_{\ol{0}}(L)$ is indecomposable.
Let $\tilde{\fh}$ be the Cartan subalgebra of this Kac-Moody algebra and
let $\rho_{\ol{0}}\in \tilde{\fh}^*$ be a Weyl vector.
We introduce the $\cdot$-action of the Weyl group $W$ on  $\tilde{\fh}^*$
by the usual formula $w\cdot\mu:=w(\mu+\rho_{\ol{0}})-\rho_{\ol{0}}$.
Consider $\lambda_0\in \tilde{\fh}^*$ satisfying
$$(\lambda_0+\rho_{\ol{0}},\alpha)=(\lambda+\rho,\alpha)\ \text{ for each } \alpha\in \Pi_0$$
(if $\Delta_{\ol{0}}$ is indecomposable, then $\lambda_0+\rho_{\ol{0}}=\lambda+\rho$).

Denote by $\breve{M}(w)$
the $\tilde{\fg}$-Verma module with the highest weight $w\cdot \lambda_0$.
For $y\leq _k w$ the module $\breve{M}(w)$ contains a unique singular vector $v_0(z)$
of weight $z\cdot\lambda_0$ and this vector gives rise to an embedding
$M(z\cdot\lambda_0)\subset M(w\cdot\lambda_0)$.

\subsubsection{}
Since all Kazhdan-Lusztig polynomials for the Weyl group $W$ in question are $0$ or $1$,
any submodule of $\breve{M}(w)$ is generated by the singular vectors 
(and is a sum of the submodules $\breve{M}(y)$).
Define a map $\Psi$ from the set of submodules of $\breve{M}(w)$ to the set of submodules
of $M(w)$ given by $\breve{M}(y)\mapsto M(y)$. It is easy to see that this map is compatible 
with inclusions.  

Let $\breve{M}'(w)$ be the maximal proper submodule of $\breve{M}(w)$. Then $\Psi(\breve{M}')$
is a proper submodule of $M(w)$.

\subsubsection{}
Suppose that  for some $w$ the module $M(w)$ has a submodule which is not generated by singular vectors. 
Then for some $w'\geq_k w$ we have
$[M(w): L(w')]>1$. Let $w,w'$ be such a pair with the minimal value of $w.\lambda-w'.\lambda$.
Then,  if $N$ is a submodule of $\Psi(\breve{M}')$ and
$\nu\leq w.\lambda-w'.\lambda$, one has 
$$p_{w.\lambda-\nu}(\ch \Psi (N))=p_{w\cdot\lambda_0-\nu}(R_{\ol{1}}\ch N),$$
where
 $p_{\nu}(\sum a_{\mu}e^{\mu}):=a_{\nu}$ and $R_{\ol{1}}:=\prod_{\beta\in\Delta_{\ol{1}}^+}(1+e^{-\beta})$.

Let $\{M^i(w)\}_{i\geq 0}, \{\breve{M}^i(w)\}_{i\geq 0},$ be the Jantzen filtrations of $M(w)$,
$\breve{M}(w)$ respectively. 
Recall that for each $i$ the modules $M^i(w)/M^{i+1}(w), \breve{M}^i(w)/\breve{M}^{i+1}(w)$ are semisimple.
This implies $M(z)\subset M^i(w),\ \breve{M}(z)\subset \breve{M}^i(w)$ if $z\geq_k w$ and $l(z^{-1}w)\leq i$.

It is easy to see that 
$$\breve{M}^i(w)=\sum_{z\geq_k w:\ l(z^{-1}w)\leq i} \breve{M}(z).$$ 
Therefore $\Psi(\breve{M}^i(w))\subset M^i(w)$. Clearly, $\breve{M}^1=\breve{M}'$. Therefore
for $\nu=w.\lambda-w'.\lambda$ one has 
$$\dim M^i(w)_{w.\lambda-\nu}\geq \dim (\Psi (M^i))_{w.\lambda-\nu}=
p_{w\cdot\lambda_0-\nu}(R_{\ol{1}}\ch \breve{M}^i)$$
for each $i>0$. Since $[M(w): L(w')]>1$ one has
$$\dim M^1(w)_{w.\lambda-\nu}>\dim (\Psi (M^1))_{w.\lambda-\nu}=
p_{w\cdot\lambda_0-\nu}(R_{\ol{1}}\ch \breve{M}^1).$$

However, the Jantzen sum formula  implies
\begin{equation}\label{Jantzen}
e^{-w.\lambda}\sum_{i=1}^{\infty} \ch M^i(w)=e^{-w\cdot\lambda_0}R_{\ol{1}}\sum_{i=1}^{\infty} \ch \breve{M}^i(w)
\end{equation}
so
$$\sum_{i=1}^{\infty}\dim M^i(w)_{w.\lambda-\nu}=
p_{w\cdot\lambda_0-\nu}(R_{\ol{1}}\sum_{i=1}^{\infty} \ch \breve{M}^i),$$
a contradiction.

We conclude that all submodules of $M(w)$ are generated by  singular vectors
(and lie in the image of $\Psi$). Combining the inclusions
$\Psi(\breve{M}^i(w))\subset M^i(w)$ and~(\ref{Jantzen}) we obtain
$M^i(w)=\Psi(\breve{M}^i(w))$. This gives
$$M^i(w)=\Psi(\breve{M}^i(w))=\sum_{z\geq_k w:\ l(z^{-1}w)\leq i} {M}(z).$$ 

\subsubsection{}
Denote by $\breve{L}(\nu)$ (resp., $\breve{M}(\nu)$)
the irreducible (resp., Verma) $\tilde{\fg}$-module with 
the highest weight $\nu$.
One has
$$\ch \breve{L}(w\cdot\lambda_0)=\sum_{y\in C} sgn(yw^{-1}) \ch \breve{M}(y\cdot\lambda_0),$$
where 
$C:=\{y\in W(\lambda_0)/\Stab_{W(\lambda_0)}(\lambda_0+\rho_{\ol{0}})|\ 
y\cdot\lambda_0\leq w\cdot\lambda_0\}$.

This implies
$$\ch L(w.\lambda)=\sum_{y\in C} sgn(yw^{-1}) \ch M(y.\lambda).$$

From the construction of $\lambda_0$, one readily sees that
$C=\{y\in W(\lambda)/\Stab_{W(\lambda)}(\lambda+\rho)|\ \ 
y.\lambda\leq w.\lambda\}$.

 \subsubsection{}
 \begin{cor}{}
Let $\ol{\fg}^{\lambda}$ be one of the  affine Lie superalgebra with 
the set of even roots 
which is the union of affine or finite root systems  of rank at most two and
 let $\ol{\lambda}$ be a non-critical extremal typical weight. For
 each $w\in W(L)$ the Jantzen filtration
 $M^i(w.\lambda)$ is given by
$M^i(w.\lambda)=\sum_{z\in W(\lambda): z\geq_k w,\ l(z^{-1}w)\leq i} {M}(z)$
and
$$\ch L(w.\lambda)=\sum_{y\in C} sgn(yw^{-1}) \ch M(y.\lambda),$$
where $C=\{y\in W(\lambda)/\Stab_{W(\lambda)}(\lambda+\rho)|\ \ 
y.\lambda\leq w.\lambda\}$.
\end{cor}

\section{$\pi$-Relatively integrable modules}\label{sectrelint}
This section is continuation of Section~\ref{enr}. Throughout the
section, $L$ is an irreducible highest weight $\fg$-module.

\subsection{Definition of $\pi$-relative integrability}
We retain notations of~\S~\ref{defintegr}. Recall the definition of $\Pi_0(L)$ 
from~\S~\ref{PiL}.

\subsubsection{}
\begin{defn}{}
For a subset $\pi\subset \Pi_0(L)$ we
 call $L$ {\em $\pi$-relatively integrable} if $F(L)$
 is $\pi$-integrable. We call $L$ {\em relatively integrable} if 
$F(L)$ is integrable.
\end{defn}

\subsubsection{}\label{relint0}
In the light of~\Lem{FLproperty}, the above notion does not depend on the choice of $\Pi$
(if $L(\lambda,\Pi)\cong L(\lambda',\Pi')$, then $L(\lambda,\Pi)$
is  $\pi$-relatively integrable
if and only if $L(\lambda',\Pi')$ is $\pi$-relatively integrable).
Moreover, by~\Cor{corFLsim}, the linkage $\sim$
preserves relative integrability.

For each $\pi\subset \Pi_0(L)$ we  denote by $W(\pi)$ the subgroup of $W(L)$ generated by 
$r_{\alpha},\alpha\in\pi$.

\subsubsection{}
In the affine Lie algebra case, relative integrability of $L$ implies
admissibility in the sense of~\cite{KW2}.
Any boundary admissible module in the sense of~\cite{KW4} 
is $\Pi_0(L)$-relatively integrable (these are the modules $L$,
such that $\dim F(L)=1$).

\subsection{Properties}
In this section we will prove several useful properties of the characters
of relatively integrable modules. 

Note that the term $R_{\ol{0}}\ch {L}$
does not depend on the choice of $\Pi$ (since $\Delta_{\ol{0}}^+$
is fixed) and lies in $\cR(\Pi)$ for each $\Pi$.

\subsubsection{}
\begin{lem}{lemFenright}
Take  $\gamma\in\Pi_0(L)$.

(i) There exists $L'\sim L$ such that the root corresponding to 
$\gamma$ in $\Pi_0(L')$ in the sense of~\Cor{corFLsim} lies in $\Pi_0$.

(ii) Assume  that $f_{\gamma}$ acts locally nilpotently
on $F(L)$.  Denote by $\Delta^{\#}$ the connected component of 
$\Delta_{\ol{0}}$ which contains $\gamma$ and by $\rho^{\#}$
the corresponding Weyl vector (i.e., the Weyl vector for  $\Delta^{\#}\cap \Delta_{\ol{0}}^+$).
The element $e^{\rho^{\#}}R_{\ol{0}}\ch {L}$ is a
$W(\gamma)$-skew-invariant element of $\cR_{W(\gamma)}$ for
$W(\gamma)=\{r_{\gamma}, Id\}$.
\end{lem}
\begin{proof}
Let us show that (ii) holds for $\gamma\in\Pi_0$.
Choose $\Pi$ such that $\gamma\in\Pi$ 
or $\gamma/2\in\Pi$ and denote by $\lambda$ the highest weight of $L$: 
$L=L(\lambda,\Pi)$. By~\Lem{FLproperty} one has $F(L)=\ol{L}(\ol{\lambda},\Pi(L))$.
Since $f_{\gamma}$ acts locally nilpotently on $F(L)$
one has
$$(\lambda+\rho_{\Pi},\gamma^{\vee})=(\ol{\lambda}+\rho_{L,\Pi},\gamma^{\vee})
\in\mathbb{Z}_{>0}.$$
This means that $f_{\gamma}$ acts locally nilpotently on $L$,
so $L$ can be decomposed as a direct sum of $\mathfrak{sl}_2(\gamma)$-modules.
Therefore $(1-e^{-\gamma})e^{\gamma/2}\ch L$ is a 
$W(\gamma)$-skew-invariant element of $\cR_{W(\gamma)}$.
Since $\rho^{\#}-\gamma/2$ and $\Delta_{\ol{0}}\setminus\{\gamma\}$
are $r_{\gamma}$-invariant, $e^{\rho^{\#}}R_{\ol{0}}\ch {L}$ is also a
$W(\gamma)$-skew-invariant element of $\cR_{W(\gamma)}$.

Now we prove (i) and (ii)  by induction on $\gamma$.
If $\gamma\not\in\Pi_0$, there exists $\alpha\in\Pi_0$ such that
$r_{\alpha}\gamma<\gamma$. Since $\gamma\in\Pi_0(L)$, one has
$r_{\gamma'}\gamma\geq \gamma$ for each $\gamma'\in\Pi_0(L)$, so
$\alpha\not\in\Pi_0(L)$ and thus
$\alpha\not\in\Delta(L)$ (because $\alpha\in\Pi_0$).
 Using the Enright functor $T_{\alpha}(a)$  
we obtain  $L'=T_{\alpha}(a)(L)\sim L$ and $r_{\alpha}\gamma<\gamma$ is the root
corresponding to $\gamma$ in $\Pi_0(L')$. This proves (i).

For (ii) choose $\Pi$ which contains $\alpha$ or $\alpha/2$.
Recall that, by~\Cor{corFLsim}, 
$F(T_{\alpha}(a)(L)\cong F(L')$
under the identification of Kac-Moody superalgebras with the sets of 
simple roots $\Pi(L)$  and
 $r_{\alpha}\Pi(L)$; under this identification $\mathbb{C}f_{\gamma}$
is identified with $\mathbb{C}f_{r_{\alpha}\gamma}$, so
$f_{r_{\alpha}\gamma}$ acts  locally nilpotently
on $F(L')$. 

Since $T_{\alpha}(a)$ provides the equivalence of categories,
$R e^{\rho}\ch L'\in \cR_{W(\alpha)}$ and
$Re^{\rho}\ch L=r_{\alpha}(R e^{\rho}\ch L')$.
Therefore
$$e^{\rho^{\#}}R_{\ol{0}}\ch {L}=R_{\ol{1}}e^{\rho^{\#}-\rho}Re^{\rho}\ch L
=R_{\ol{1}}e^{\rho^{\#}-\rho}r_{\alpha}(R e^{\rho}\ch L').$$
Note that
$R_{\ol{1}}e^{\rho^{\#}-\rho}$ is $r_{\alpha}$-invariant element of $\cR_{W(\alpha)}$
(if $\alpha\in\Pi$, then $\Delta_{\ol{1}}$ and $\rho^{\#}-\rho$
are $r_{\alpha}$-invariant; if $\alpha/2\in\Pi$, then $\Delta_{\ol{1}}\setminus\{\alpha/2\}$ 
and $e^{\rho^{\#}-\rho}(1+e^{-\alpha/2})$  is $r_{\alpha}$-invariant).
Hence $e^{\rho^{\#}}R_{\ol{0}}\ch {L}'\in \cR_{W(\alpha)}$ and
$$e^{\rho^{\#}}R_{\ol{0}}\ch {L}=r_{\alpha}(e^{\rho^{\#}}R_{\ol{0}}\ch {L}').$$
By induction hypothesis, $e^{\rho^{\#}}R_{\ol{0}}\ch {L'}$ is a
$W(r_{\alpha}\gamma)$-skew-invariant element of $\cR_{W(r_{\alpha}\gamma)}$,
so $e^{\rho^{\#}}R_{\ol{0}}\ch {L}$ is a
$W(\gamma)$-skew-invariant element of $\cR_{W(\gamma)}$.
\end{proof}

\subsubsection{}
\begin{cor}{propreint}
Let $L$ be $\pi$-relatively integrable ($\pi\subset\Pi_0(L)$).
Assume that $\pi$ lies in a connected component
$\Delta^{\#}$  of $\Delta_{\ol{0}}$; let $\rho^{\#}$ be
the corresponding Weyl vector.

Then the element $R_{\ol{0}}e^{\rho^{\#}}\ch L$ is a
$W(\pi)$-skew-invariant element of $\cR_{W(\pi)}$ and
the element $Re^{\rho}\ch L$ is a $W(\pi)$-skew-invariant element of
$\cR_{W(\pi)}[\cY^{-1}]$ (see~\S~\ref{infprod} for notation).
\end{cor}
\begin{proof}
From~\Lem{lemFenright} it follows that $R_{\ol{0}}e^{\rho^{\#}}\ch L$ is a
$W(\pi)$-skew-invariant element of $\cR_{W(\pi)}$.
Recall that $R_{\ol{1}}\in \cY$, so $Re^{\rho}\ch L\in\cR_{W(\pi)}[\cY^{-1}]$.
It remains to verify that $Re^{\rho}\ch L$ is $W(\pi)$-skew-invariant.
Since $R_{\ol{0}}e^{\rho^{\#}}$ is  $W(\Delta^{\#})$-skew-invariant 
and $Re^{\rho}$ is $W$-invariant,
their ratio $R_{\ol{1}}e^{\rho^{\#}-\rho}$ is  $W(\Delta^{\#})$-skew-invariant,
and, in particular, is $W(\pi)$-skew-invariant.
Hence $Re^{\rho}\ch L=R_{\ol{1}}e^{\rho^{\#}-\rho}\cdot 
R_{\ol{0}}e^{\rho^{\#}}\ch L$  is  $W(\pi)$-skew-invariant.
\end{proof}

\subsubsection{}
Denote by $R_{L,\ol{0}}, R_{L,\ol{1}}(\Pi)$ the following analogues of $R_{\ol{0}},R_{\ol{1}}(\Pi)$:
$$\begin{array}{ll}
   R_{L,\ol{0}}:=\prod_{\alpha\in\Delta^+(L)_{{\ol{0}}}}(1-e^{-\alpha}), & 
R_{L,\ol{1}}(\Pi):=\prod_{\alpha\in\Delta^+(L)_{{\ol{1}}}}(1+e^{-\alpha})
  \end{array}
$$
Note that these elements lie in $\cY$ (and $\cY\subset \cR(\Pi')$ for each $\Pi'$)
and $R_L(\Pi)=R_{L,\ol{1}}^{-1}(\Pi)R_{L,\ol{0}}$.

Recall (see~\S~\ref{RPi})  that the elements $R(\Pi)e^{\rho_{\Pi}}\in\cR(\Pi)$ are equivalent for all $\Pi$:
the expansion of $R(\Pi)e^{\rho_{\Pi}}$ in $\cR(\Pi)$ does not belong to $\cR(\Pi')$,
however, the expansion of $R(\Pi)e^{\rho_{\Pi}}$ in $\cR(\Pi')$
coincides with the expansion of $R(\Pi')e^{\rho_{\Pi'}}$ in $\cR(\Pi')$.
 Similarly, $R_{L}(\Pi)e^{\rho_{L,\Pi}}$  are equivalent for all $\Pi$.
Hence the elements $R_{\ol{1}}(\Pi)e^{-\rho_{\Pi}},\ 
R_{\ol{1},L}(\Pi)e^{-\rho_{L,\Pi}}$  
are equivalent for all $\Pi$; since these elements lie in $\cY$, the expansion
of $R_{\ol{1}}(\Pi)e^{-\rho_{\Pi}}$ (resp., of $R_{\ol{1},L}(\Pi)e^{-\rho_{L,\Pi}}$)
in $\cR(\Pi')$ is equal to the expansion of
$R_{\ol{1}}(\Pi")e^{-\rho_{\Pi"}}$ (resp., of $R_{\ol{1},L}(\Pi")e^{-\rho_{L,\Pi"}}$)
in $\cR(\Pi''')$ for any sets of simple roots
$\Pi,\Pi',\Pi",\Pi'''$. This allows
us to use the notation $R_{\ol{1},L}e^{-\rho_{L}}$ for
 $R_{\ol{1},L}(\Pi)e^{-\rho_{L,\Pi}}$ .

\subsubsection{}
\begin{lem}{lemRRchL}
For a  set of simple roots $\Pi$ let
$X(\Pi)$ be the expansion of $R_{\ol{1},L}Re^{\rho-\rho_{L}}\ch L$ in $\cR(\Pi)$.
For any $\Pi$ and $\Pi'$ one has $X(\Pi)=X(\Pi')$ 
(in particular, $X(\Pi)\in\cR(\Pi')$). 
\end{lem}
\begin{proof}
By above, the elements  $R_{\ol{1},L}Re^{\rho-\rho_{L}}\ch L$ are equivalent
for all $\Pi$, so $X(\Pi)$ is equivalent to $X(\Pi')$.
It remains to show that $X(\Pi)\in\cR(\Pi')$.

Since any two sets of simple roots are connected by a chain of odd reflections,
it is enough to consider the case $\Pi'=r_{\beta}\Pi$, where $\beta\in\Pi$ is
an odd isotropic root.
We denote by $R_{\ol{1},L},R,\rho,\rho_{L}$ the corresponding elements for  $\Pi$
and set $\cR:=\cR(\Pi), \cR':=\cR(r_{\beta}\Pi)$.
Observe that
$\ch L\in\cR\cap\cR'$, since $L$ is an irreducible highest weight module.

If $\beta\in\Delta(L)$, then  the element $R_{\ol{1},L}Re^{\rho-\rho_{L}}$
has the same expansion in $\cR$ and in $\cR'$, since
$$R_{\ol{1},L}Re^{\rho-\rho_{\ol{1},L}}=\prod_{\alpha\in\Delta_{\ol{0}}^+}(1-e^{-\alpha})\cdot
\prod_{\alpha\in\Delta_{\ol{1}}^+(\Pi)\setminus\Delta_{\ol{1}}(L)}(1+e^{-\alpha})^{-1}e^{\rho-\rho_{\ol{1},L}}$$
and 
$\Delta_{\ol{1}}^+(\Pi)\setminus\Delta_{\ol{1}}(L)$ lies in $\Delta^+(r_{\beta}\Pi)
\cap\Delta^+(\Pi)$.

Consider the case $\beta\not\in\Delta(L)$. Let us show that the expansion of
$Re^{\rho}\ch L$ in $\cR$ lies in $\cR'$. Indeed, denote this expansion by $Y$.
Since $\ch L\in\cR\cap\cR'$, $Y$
has a pole of order $\leq 1$ at $\beta$.
By~\Lem{lemXX'}, in order to show that $Y$ does not have a pole, i.e. that
$Y\in\cR'$, it is enough
to verify that for each $\mu\in\fh^*$ the set $\supp Y\cap\{\mu+\mathbb{Z}\beta\}$ is finite.
In fact this set contains at most one element, since
 the action of the Casimir element
gives  $||\mu||^2=||\mu+r\beta||^2$ if $\mu,\mu+r\beta\in\supp (Y)$,  so $(\mu,r\beta)=0$. 
However, for $\mu\in\supp Y$ one has $\Delta(\mu)=\Delta(L)$,
so $\beta\not\in\Delta(\mu)$, that is $r=0$. Hence $Y\in\cR'$. 
Since the infinite product
$R_{\ol{1},L}e^{-\rho_{L}}$ lies
in $\cR$ and in $\cR'$, we conclude that $X(\Pi)=R_{\ol{1},L}e^{-\rho_{L}}Y\in\cR(\Pi')$. 
The assertion follows.
\end{proof}

\subsubsection{}
By above, $X(\Pi)$ does not depend on $\Pi$ (for fixed $L$). 
The next lemma shows that the equivalence relation $\sim$ preserves
$X(\Pi)$.

\begin{lem}{lemRRch2}
Let $X(L)$ be the expansion of $R_{\ol{1},L}Re^{\rho-\rho_{L}}\ch L$ in $\cR(\Pi)$.
Recall that for $L\sim L'$ one has $F(L)\cong F(L')$ via the natural identification
$\Delta(L)$ and $\Delta(L')$. Under this identification $X(L)=X(L')$.
\end{lem}
\begin{proof}
It is enough to verify the assertion for $L'=T_{\alpha}(a)(L)$, where
$\alpha\in\Pi_0\setminus\Delta(L)$. Choose $\Pi$ 
which contains $\alpha$ or $\frac{\alpha}{2}$.
One has 
$$Re^{\rho}\ch L(\lambda,\Pi')=\sum a_{\nu}e^{\nu},\ \ 
Re^{\rho}\ch L'=r_{\alpha}\bigl(Re^{\rho}\ch L\bigr).$$

By~\S~\ref{FLFL'}, $\Delta(L')=r_{\alpha}\Delta(L)$ and
$\Pi(L')=r_{\alpha}\Pi(L)$. Thus
$$r_{\alpha}(R_{\ol{1},L})=R_{\ol{1},L'}, \ \ r_{\alpha}(\rho_L)=\rho_{L'}.$$
Hence $r_{\alpha} X(\Pi)$ is the expansion 
$R_{\ol{1},L'}Re^{\rho-\rho_{L'}}\ch L'$ in $\cR(\Pi)$.
\end{proof}

\subsubsection{}
Let $L$ be $\pi$-relatively integrable ($\pi\subset \Pi_0(L)$).
Assume that $\pi$ admits a Weyl vector $\rho_{\pi}$
($(\rho_{\pi},\alpha)=(\alpha,\alpha)/2$ for each $\alpha\in\pi$).
By above, the element $R_{\ol{1},L}Re^{\rho+\rho_{\pi}-\rho_{L}}\ch L$ 
does not depend on $\Pi$ (and have the same expansions in all algebras
$\cR(\Pi')$).

\begin{prop}{RRchLrel}
The element $R_{\ol{1},L}Re^{\rho+\rho_{\pi}-\rho_{L}}\ch L$ is a 
$W(\pi)$-skew-invariant element of $\cR_{W(\pi)}$.
\end{prop}
\begin{proof}
Denote by $X(L)$ the expansion of $R_{\ol{1},L}Re^{\rho+\rho_{\pi}-\rho_{L}}\ch L$
in $\cR(\Pi)$ (this does not depend on $\Pi$ by~\Lem{lemRRchL}).
It is enough to verify that $X(L)e^{\rho_{\pi}}$ is 
$W(\gamma)$-skew-invariant element of $\cR_{W(\gamma)}$ for each $\gamma$.
We prove this by induction on $\gamma\in\Delta_{\ol{0}}^+$.

Assume first that $\gamma\in\Pi_0$.   Take $\Pi$ which contains 
$\gamma$ or $\frac{\gamma}{2}$
and set $\cR:=\cR(\Pi)$. By~\Cor{propreint}, $R_{\ol{0}}e^{\rho^{\#}}\ch L$
is a $W'$-skew-invariant element of $\cR_{W'}$. 
Clearly, $\Delta_{\ol{1}}^+\setminus \Delta_{\ol{1}}(L)^+$
is $r_{\gamma}$-invariant, so
$R_{\ol{1},L} R_{\ol{1}}^{-1}$  is a $W(\gamma)$-invariant element of $\cR_{W(\gamma)}$. 
Since $\gamma$ or $\frac{\gamma}{2}$
lies in $\Pi$, $\rho-\rho_L$ is $r_{\gamma}$-invariant;
since $\gamma\in\Pi_0$, $\rho^{\#}-\rho_{\pi}$ is $r_{\gamma}$-invariant.
Hence $e^{\rho+\rho_{\pi}-\rho_L-\rho^{\#}}$
is a $W(\gamma)$-invariant element of $\cR_{W(\gamma)}$, so
$X(L)e^{\rho_{\pi}}$ is a $W(\gamma)$-skew-invariant element of $\cR_{W(\gamma)}$ as well.

Now take $\gamma\not\in\Pi_0$ and $\alpha\in\Pi_0$ such that $\gamma':=r_{\alpha}\gamma<\gamma$.
Then $\alpha\not\in\Delta(L)$ (see the proof of~\Lem{lemFenright}).
Let $L':=T_{\alpha}(a)(L)$. By~\S~\ref{relint0}, we conclude that $L'$ is $r_{\alpha}\pi$-relatively integrable. By induction $X(L')e^{\rho_{r_{\alpha}\pi}}$
is a $W(\gamma')$-skew-invariant element of $\cR_{W(\gamma')}$.
Clearly, $\rho_{r_{\alpha}\pi}$
can be chosen equal to $r_{\alpha}\rho_{\pi}$. Moreover, by~\Lem{lemRRch2}, $X(L')=r_{\alpha}(X(L))$. Then $X(L)e^{\rho_{\pi}}=r_{\alpha}(X(L')e^{\rho_{r_{\alpha}\pi}})$
and thus $X(L)e^{\rho_{\pi}}$ is a $W(\gamma)$-skew-invariant element of $\cR_{W(\gamma)}$,
as required.
\end{proof}

\section{Character formulas  for some typical and relatively integrable modules}
\label{sect10}
In this section we prove formula~(\ref{KRWconj}) from~\S~\ref{KRWconjecture} for some cases. 

\subsection{}
Recall that a module $L(\lambda,\Pi)$ is typical if 
$(\lambda+\rho,\beta)\not=0$ for all 
isotropic $\beta\in\Delta_{\ol{1}}$.

Recall (see Section~\ref{sect8}) that we call $\lambda\in\fh^*$ {\em extremal}
if for each connected component $\pi$ of $\Pi_0(\lambda)$, $\lambda$
is maximal or minimal in its $W(\pi)$-orbit. We say that $\lambda$ is regular
if $Stab_{W}\lambda=\{Id\}$.

\subsubsection{}
\begin{thm}{thmint12}
If $L=L(\lambda,\Pi)$ is typical and $\lambda+\rho$ is a regular extremal weight, then
$$Re^{\rho}\ch L=\sum_{w\in W(L): w(\lambda+\rho)\leq \lambda+\rho} \sgn(w) e^{w(\lambda+\rho)}$$
and~(\ref{KRWconj}) holds.
\end{thm}

\subsubsection{}
Recall that $\Pi_0(L)$ is the set of simple roots
of $\Delta(L)\cap \Delta^+_{\ol{0}}$.

\begin{cor}{corint2}
 If $L=L(\lambda,\Pi)$ is a typical  module and $F(L)$ is $\Pi_0(L)$-integrable, then
$$Re^{\rho}\ch L=\sum_{w\in W(L)} \sgn(w) e^{w(\lambda+\rho)}$$
and~(\ref{KRWconj}) holds.
\end{cor}

The above theorem admits the following generalization.

\subsubsection{}
\begin{thm}{thmadm23}
Let $L=L(\lambda,\Pi)$ is a typical module. Set 
$$\pi:=\{\alpha\in\Pi_0(L)|\ (\lambda+\rho,\alpha^{\vee})>0\}$$ 
(that is $\pi\subset \Pi_0(L)$ is maximal such that $F(L)$ is $\pi$-integrable).
Write $\pi=\pi_f\coprod \pi_{aff}$, where
$\pi_f$ (resp., $\pi_{aff}$) is the union of connected finite (resp., affine) 
type diagrams in $\pi$. 
Assume that for each $\alpha\in\Pi_0(L)\setminus\pi$ one has 
$$(\lambda+\rho,w_0\alpha^{\vee})<0,$$
where $w_0$ is the product of the longest elements in $W({\pi}_f)$.
Then
$$Re^{\rho}\ch L=\sum_{w\in W(\pi)} \sgn(w) e^{w(\lambda+\rho)}.$$ 
\end{thm}

\subsubsection{Remark}
We do not expect~\Thm{thmadm23} to hold in general.
Namely, the coefficients of the character formula may involve
non-trivial Kazhdan-Lusztig polynomials.

\subsection{Proof of Theorems~\ref{thmint12},\ref{thmadm23} and~\Cor{corint2}}
First note that~\Thm{thmint12} is a particular case of~\Thm{thmadm23}:
if $\lambda$ is extremal, then $\pi$ is a union of connected components of $\Pi_0(L)$
and each root $\alpha\in\Pi_0(L)\setminus\pi$ lies in a connected component
 $\tilde{\pi}\subset\Pi_0(L)$ such that $\lambda$ is a minimal element
 in its $W(\tilde{\pi})$-orbit. One has 
 $(\lambda+\rho,w\alpha^{\vee})=(\lambda+\rho,\alpha^{\vee})$ for each $w\in W(\pi)$;
 if $\lambda$ is regular, then $(\lambda+\rho,\alpha^{\vee})<0$.
 Thus a regular typical extremal weight satisfies the assumptions
 of~\Thm{thmadm23}.

 Now let us deduce~\Cor{corint2}  from~\Thm{thmint12}. 
 First, consider the case when $L'=L(\lambda',\Pi)$ is a typical  module and $\fg_{\pm\alpha}$ acts
nilpotently on $L'$ for some $\alpha\in\Pi_0$. We claim that
$(\lambda'+\rho, \alpha^{\vee})\in\mathbb{Z}_{>0}$. Indeed,
choosing $\Pi'$ which contains
$\alpha$ or $\alpha/2$ we obtain $L'=L(\lambda'+\rho-\rho',\Pi')$ and
$(\lambda'+\rho-\rho',\alpha^{\vee})\geq 0$; since $(\rho',\alpha^{\vee})>0$, we get
$(\lambda'+\rho,\alpha^{\vee})>0$. Since  $\fg_{\pm\alpha}$ acts locally
nilpotently on $L'$, $\alpha\in \Delta(L)$, so $(\lambda+\rho, \alpha^{\vee})\in\mathbb{Z}_{>0}$,
as required.

Next let  $L=L(\lambda,\Pi)$ be a typical  module such that $F(L)$ is $\Pi_0(L)$-integrable.
By above, $(\lambda+\rho,\alpha^{\vee})=(\ol{\lambda}+\rho_L,\alpha^{\vee})\in\mathbb{Z}_{>0}$
for each $\alpha\in\Pi_0(L)$. Hence $\lambda+\rho$ is regular and extremal.
Hence~\Cor{corint2} follows from~\Thm{thmint12}.

\subsubsection{Proof of~\Thm{thmadm23}}
We claim that 
$$\supp (Re^{\rho}\ch L)\subset W(\pi)(\lambda+\rho).$$

Indeed, by~\Prop{propKK} it is enough to verify that
for each $w\in W(\pi)$ if
 $(w.\lambda+\rho,\gamma^{\vee})>0$ for some non-isotropic even positive root $\gamma$,
 then $\gamma\in \Delta(\pi)$. One has
$(w.\lambda+\rho,\gamma^{\vee})=(w_0\lambda+\rho, (w_0w^{-1}\gamma)^{\vee})$.
Since $\gamma\not\in\Delta(\pi)$,  the root $w_0w^{-1}\gamma$ lies in 
$\Delta^+\setminus\Delta(\pi)$. 
By the assumptions, $(w_0.\lambda+\rho,\alpha^{\vee})<0$ for $\alpha\in\Pi_0(L)\setminus\pi$
and for $\alpha\in\pi_f$; therefore this inequality holds for
$\alpha\in \Pi_0(L)\setminus\pi_{aff}$ and thus for a positive non-isotropic even root $\alpha$
which does not lie in $\Delta(\pi_{aff})$. Hence $(w.\lambda+\rho,\gamma^{\vee})<0$,
as required.

By Prop. 3.12 in~\cite{K2},  
$\lambda+\rho$ is a unique $\pi$-maximal element in its $W(\pi)$-orbit (by definition 
of $\pi$).

We claim that $F(L)$ is $\pi$-integrable. Indeed, 
take $\alpha\in\pi$ and let $\Pi'$ be a subset of simple roots for $\Delta(L)$ 
such that $\alpha$ or $\alpha/2$ lies in $\Pi'$.
Since $F(L)$ is typical, $F(L)=L(\lambda',\Pi')$, where $\lambda+\rho=\lambda'+\rho'$,
so $(\lambda'+\rho',\alpha^{\vee})>0$. Since $\alpha\in\Delta(L)$, we conclude that
$\fg_{\lambda,\pm\alpha}$ acts locally nilpotently on $F(L)$. Hence 
$F(L)$ is $\pi$-integrable. 

Using~\Prop{propreint}, we conclude that 
$Re^{\rho}\ch L$ is $W(\pi)$-skew-invariant.
Therefore $Re^{\rho}\ch L\in\cR_{W(\pi)}$ and so
$$Re^{\rho}\ch L=c\cF_{W(L)}  e^{(\lambda+\rho)}.$$
Since the coefficient of $e^{\lambda+\rho}$ in $Re^{\rho}\ch L$ is $1$, 
$c=1$, as required.
\qed


\subsubsection{}
\begin{thm}{thmadm1}
Let $L$ be a non-critical module such that $(F(L),\Pi(L))$ satisfies the conditions 
of Section~\ref{sectKW} 
and $S\subset\Pi$. Then~(\ref{KRWconj}) holds.
\end{thm}

\subsubsection{}
\begin{thm}{thmadmvac}
Let $L$ be a non-critical vacuum module such that $\Delta(L)$ is affine and 
$F(L)$ is integrable. Then~(\ref{KRWconj}) holds.
\end{thm}

\subsubsection{}
\begin{cor}{corPi0(L)con}
If $\Delta(L)_{\ol{0}}$ is connected (i.e., $\Delta(L)=A(0,m), B(0,n), C(n)$,
their untwisted affinizations, or $A(0,2n-1)^{(2)},
C(n+1)^{(2)}, A(0,2n)^{(4)}$) and $F(L)$ is $\Pi_0(L)$-integrable, 
then~(\ref{KRWconj}) holds.
\end{cor}

In particular, by the results of Section~\ref{sectKW}, we obtain
the character formula for all admissible modules over  $A(0,n)^{(1)},
C(n)^{(1)}$.

\subsubsection{}
\begin{thm}{dimFL=1}
Let $L$ be a non-critical $\fg$-module such that  for each connected component $\Delta^i$
of $\Delta(L)$ one has either $(F(L),\Delta^i)=0$ or $F(L)$ is $\Delta^i_{\ol{0}}$-integrable and
$\Delta^i$-typical (i.e., $(F(L),\beta)\not=0$ for each $\beta\in\Delta^i$).
Then~(\ref{KRWconj}) holds.
\end{thm}

\subsubsection{}
\begin{cor}{cordimFL=1}
Let $L$ be a non-critical $\fg$-module such that  $\dim F(L)=1$.
Then~(\ref{KRWconj}) holds.
\end{cor}

\subsection{The case when $F(L)$ is a vacuum module}
Consider the case when $F(L)$ is a vacuum module (i.e. $F(L)=L(\ol{\lambda},\Pi(L))$, where
$(\ol{\lambda},\dot{\Pi}(L))=0$ for some finite part $\dot{\Pi}(L)$ of $\Pi(L)$), which is
 integrable  (i.e. is $\pi$-integrable for 
a connected component $\pi$ of $\Pi_0(L)$, see~\S~\ref{defintegr}).
We denote by $\alpha_0$ the affine root in $\Pi(L)$, i.e.
$$\Pi(L)=\dot{\Pi}(L)\cup\{\alpha_0\}.$$
Normalize the bilinear form in such a way that $||\alpha||^2\in\mathbb{Q}_{>0}$ for 
$\alpha\in\pi$.

The following theorems improve the result of~\Thm{thmadm1} in the case
when $F(L)$ is  an integrable vacuum module with $\Delta(L)\not=A(n,n)^{(1)}$.

\subsubsection{}
\begin{thm}{thmadm2}
Let $F(L)$ be an integrable vacuum module such that the dual Coxeter number of $\Delta(L)$
is non-zero.

Assume that $\Pi(L)$ is such that
 $||\alpha||^2\geq 0$ for each $\alpha\in\Pi(L)$;
if $\Pi(L)=C(m)$ or $\Pi(L)=A(m,n)^{(1)}, m>n$, assume, 
in addition, that the affine root in $\Pi(L)$ is not isotropic.
Then~(\ref{KRWconj}) holds.
\end{thm}

\subsubsection{}
\begin{thm}{thmadm4}
Let $F(L)$ be a non-critical integrable vacuum module.

If $\Delta(L)=D(2,1,a)^{(1)}$ and $a\not=-\frac{1}{2},-2$, then (\ref{KRWconj}) holds
for any $\Pi(L)$; for $\Delta(L)=D(2,1,-\frac{1}{2})^{(1)}=D(2,1,-2)^{(1)}$, (\ref{KRWconj}) holds
if $||\alpha_0||^2\geq 0$.

If $\Delta(L)=A(2n-1,2n-1)^{(2)}, D(n+1,n)^{(1)}$ and $\Pi(L)$
is as in~\S~\ref{A2n-12n-1}, then~(\ref{KRWconj}) holds.

If $\Delta(L)=A(2n,2n)^{(4)}, D(n+1,n)^{(2)}$ with $\Pi(L)$
is as in~\S~\ref{A2n4} and the level of $F(L)$ is not $1$, then~(\ref{KRWconj}) holds.

If $\Delta(L)=A(2n,2n)^{(4)}, D(n+1,n)^{(2)}$ with $\Pi(L)$
is as in~\S~\ref{A2n4} and the level of $F(L)$ is $1$, then~(\ref{KRWconj}) holds if
$\dot{\Pi}(L)\subset \Pi$. 
\end{thm}

\subsubsection{}
For  each $\alpha\in\dot{\pi}$ we write 
$\alpha=\sum_{\beta\in\Pi(L)} x_{\alpha,\beta} \beta$, and let
$\supp(\alpha):=\{\beta\in\Pi|\ x_{\alpha,\beta}\not=0\}$.
Consider the following conditions on a set of simple roots $\Pi(L)$:

(A)  $||\alpha_0||^2\geq 0$;

(B) for each $\alpha\in\pi$ there exists $\beta\in\supp(\alpha)$
such that $\beta\not\in\supp(\alpha')$ for each $\alpha'\in\pi$; 
this $\beta$ is denoted by $b(\alpha)$;

(C) $\rho_L\in X_1-X_2$, where
$$X_1:=\{\mu\in\fh^*|
(\mu,\alpha)\in \mathbb{Q}_{\geq 0} \ \text{ for all }\alpha\in \Pi(L)\},\ \ 
X_2:=\sum_{\alpha\in\dot{\Pi}(L)_0} \mathbb{Q}_{\geq 0}\alpha^{\vee}.$$

It is easy to see that (B) holds for each $\Pi$ if $\Delta\not=F(4)^{(1)}$, 
see~\S~\ref{appendix2}.
Note that (C)  holds if $\rho_L=0$ and if $||\alpha||^2\geq 0$ for each $\alpha\in\Pi(L)$
(in this case $\rho_L\in X_1$). We give examples of sets of simple roots $\Pi(L)$
satisfying (A),(B), (C) in~\S~\ref{appendix2}.

\subsubsection{}
\begin{thm}{thmadm3}
Let $F(L)$ be a non-critical integrable vacuum module such that $\Delta(L)\not=A(m,n)^{(1)}, C(n)^{(1)}$.
If $\Pi(L)$ satisfies the conditions (A)-(C), then~(\ref{KRWconj}) holds.
\end{thm}

\subsection{Proofs: Step 1}\label{proofs1}
Set
$$Z:=Re^{\rho}\ch L-R_Le^{\rho_L}\ch F(L)=0,\ \ \ Z':=R_{\ol{1}, L} e^{-\rho_{L}}Z.$$

We have to prove that $Z'=0$. Suppose that $Z'\not=0$.

Denote by $\lambda_{\Pi}$ (resp., $\ol{\lambda}_{\Pi}$)
 the highest weight of $L$ (resp., of $\ol{L}$); recall that 
 $\lambda_{\Pi}+\rho_{\Pi}=\ol{\lambda}_{\Pi}+\rho_{L,\Pi}$. 
By~\Cor{corsuA2}, $\supp(Z)\subset \ol{\lambda}_{\Pi}+\rho_{L,\Pi}-
\mathbb{Z}_{\geq 0}\Delta^+(L)$,
so for each $\Pi$
\begin{equation}\label{suppZ'}
\supp(Z')\subset \ol{\lambda}_{\Pi}-\mathbb{Z}_{\geq 0}\Pi(L).
\end{equation}
Clearly, $\lambda_{\Pi}+\rho_{\Pi}\not\in\supp(Z)$, so
$\ol{\lambda}_{\Pi}\not\in\supp(Z')$.

Let $F(L)$ be $\pi$-integrable, where $\pi\subset\Pi_0(L)$ is connected.
Since $\pi$ is connected, it admits a Weyl vector $\rho_{\pi}\in \fh^*$
($(\rho_{\pi},\alpha^{\vee})=1$ for each $\alpha\in\pi$).
By~\Prop{RRchLrel}, for each $\pi\subset\Pi_0(L)$ the element
 $e^{\rho_{\pi}}Z'$ is a 
$W(\pi)$-skew-invariant element of $\cR_{W(\pi)}$.
 Therefore $\supp(e^{\rho_{\pi}}Z')$ is a union of regular $W(\pi)$-orbits 
(the regularity means that the stabilizer of each element is trivial). From~(\ref{suppZ'})
it follows  that each orbit has a $\pi$-maximal element (i.e.,
maximal with respect to the following order:
$\nu'\geq_{\pi} \nu"$ if $\nu'-\nu"\in\mathbb{Z}_{\geq 0}\pi$).
Let $\mu=\ol{\lambda}_{\Pi}-\nu$ be a $\pi$-maximal element in its orbit.
Using the regularity of the orbit we obtain
$(\ol{\lambda}_{\Pi}-\nu,\alpha^{\vee})\geq 0$ for each $\alpha\in\pi$.

Combining maximality of $\nu$ and regularity of the orbit we obtain
\begin{equation}\label{C331}
(\ol{\lambda}_{\Pi}-\nu,\alpha^{\vee})\geq 0\ \text{ for each }\alpha\in\pi.
\end{equation}

Now  let $\ol{\lambda}_{\Pi}+\rho_L-\nu$ 
be a maximal element in $\supp (Z)$ with respect to the order
$\nu'\geq \nu"$ if $\nu'-\nu"\in\mathbb{Z}_{\geq 0}\Delta^+$.  Clearly, $\nu\not=0$.
Then $\ol{\lambda}_{\Pi}-\nu$  is a maximal element in $\supp (Z')$
with respect to the same order and so this is a $\Pi_0(L)$-maximal element in $\supp(Z')$.
Since $\ol{\lambda}_{\Pi}+\rho_L-\nu\in \supp(Z)$ we have 
$2(\ol{\lambda}_{\Pi}+\rho_L,\nu)=(\nu,\nu)$. Combining with~(\ref{suppZ'}) we get

\begin{equation}\label{C111}
2(\ol{\lambda}_{\Pi}+\rho_L,\nu)=(\nu,\nu),\ \nu\in\mathbb{Z}_{\geq 0}\Pi(L),\ \nu\not=0.
\end{equation}

We will show that~(\ref{C331}) contradicts~(\ref{C111}).

\subsection{Proofs of Theorems~\ref{thmadm1}, \ref{thmadmvac} and~\Cor{corPi0(L)con}}
\subsubsection{Proof of~\Thm{thmadm1}}\label{proofadm1}
Arguing as in~\S~\ref{proof11} we deduce from~(\ref{C331}) and~(\ref{C111})
 that  $\nu\in\mathbb{Z}_{\geq 0} S$.
Write $\nu=\sum_{\beta\in S} x_{\beta}\beta$, $x_{\beta}\geq 0$.
Let $\beta$ be such that $x_{\beta}\not=0$.
By~\Lem{lemRRchL},   $Z'$ does not depend on the choice of $\Pi$, so
$\supp(Z')\subset \ol{\lambda}_{\Pi'}-\mathbb{Z}_{\geq 0}\Pi'(L)$ for any $\Pi'$.
In particular, $\ol{\lambda}_{\Pi}-\nu\in \ol{\lambda}_{\Pi'}-\mathbb{Z}_{\geq 0}\Pi'(L)$
for any $\Pi'$.
For $\Pi'=r_{\beta}\Pi$ we have $\ol{\lambda}_{\Pi'}=\ol{\lambda}_{\Pi}$, so
$\nu\in \mathbb{Z}_{\geq 0}\Pi'(L)$. However,
$-\beta\in\Pi'(L)$ and  $S\setminus\{\beta\}\in \Pi'(L)$,  a contradiction.
\qed

\subsubsection{Proof of~\Thm{thmadmvac}}\label{proofadmvac} 
Recall that $L$ is a vacuum module means that $L=L(\lambda,\Pi)$, where
$(\lambda,\dot{\Delta})=0$ for some finite part $\dot{\Delta}$ of $\Delta$.
In particular, $\dot{\Delta}\subset \Delta(L)$ and the inclusion is strict, since 
$\Delta(L)$ is affine. Note that the elements of $\dot{\Pi}$ are indecomposable
in $\Delta(L)^+=\Delta^+\cap\Delta(L)$, so $\dot{\Pi}\subset\Pi(L)$.
Hence $\dot{\Pi}$ (resp., $\dot{\Delta}$) is a finite part of $\Pi(L)$ (resp., of $\Delta(L)$).

Consider $\Delta':=\Delta(L)$ with a set of simple roots $\Pi':=\Pi(L)$ and a finite part
$\dot{\Delta}$.

If the dual Coxeter number of $\Delta'$ is non-zero or 
$\Delta'=A(n,n)^{(1)}$,  then, by~\S~\ref{app3}, 
there exists a chain of odd reflections with respect to 
the roots in $\dot{\Delta}$ which  transform $\Pi'$ to a set of simple roots $\Pi''$
with the following property: for each $\alpha\in\Pi''$ one has $||\alpha||^2\geq 0$
(where $(-,-)$ is normalized in such a way that $(\rho',\delta)\geq 0$).
Acting by the same chain of odd reflections of $\Pi$, we obtain a set of simple roots
$\tilde{\Pi}$ for $\Delta$ such that $\tilde{\Pi}(L)=\Pi''$. 
Clearly, $L=L(\tilde{\lambda},\tilde{\Pi})$, where $(\tilde{\lambda},\dot{\Delta})=0$. Hence
$(F(L),\tilde{\Pi}(L))$ satisfies the conditions of Section~\ref{sectKW} 
and $S\subset\Pi$. By~\Thm{thmadm1} (\ref{KRWconj}) holds.

Similarly, if the dual Coxeter number of $\Delta'$ is zero and
$\Delta'\not=A(n,n)^{(1)}$,  then, by~\S~\ref{app3}, 
there exists a chain of odd reflections with respect to 
the roots in $\dot{\Delta}$ which  transform $\Pi'$ to a set of simple roots $\Pi''$
given in~\S~\ref{D21a},
\S~\ref{A2n-12n-1},\S~\ref{A2n4},  respectively. 
Acting by the same chain of odd reflections of $\Pi$, we obtain a set of simple roots
$\tilde{\Pi}$ for $\Delta$ such that $\tilde{\Pi}(L)=\Pi''$. 
One has $L=L(\tilde{\lambda},\tilde{\Pi})$, where $(\tilde{\lambda},\dot{\Delta})=0$ and $F(L)$
is integrable.  In this case the statement follows from~\Thm{thmadm4}.
\qed

\subsubsection{Proof of~\Cor{corPi0(L)con}}
If $L$ is typical, the assertion follows from~\Cor{corint2}. Assume that $L$ is not typical.
 
In the light of~\Thm{thmadm1} it is enough to verify that there exists $L'\sim L$ and
$\Pi$ such that $(\lambda'_{\Pi}+\rho_{\Pi},\beta)=0$
for some $\beta\in\Pi$, where $L'=L(\lambda'_{\Pi},\Pi)$.
Assume that this is not the case. 
Then for any $L'$ one has
$\lambda'_{\Pi}+\rho_{\Pi}=\lambda'_{\Pi'}+\rho_{\Pi'}$ for all $\Pi,\Pi'$.
Fix any $\Pi$ and let $\beta\in\Delta(L)$ be such that  $(\lambda_{\Pi}+\rho_{\Pi},\beta)=0$.

First, consider the case when $\beta\in\Pi'$ for some $\Pi'$.
Then $\lambda_{\Pi}+\rho_{\Pi}=\lambda_{\Pi'}+\rho_{\Pi'}$ implies that
$(\lambda_{\Pi'}+\rho_{\Pi'},\beta)=0$, so the assertion holds for  $L'=L$
and $\beta\in\Pi'$.

Assume that $\beta\not\in\Pi'$ for any $\Pi'$. Then 
$\Delta\not= A(m,n), C(n), A(m,n)^{(1)}, C(n)^{(1)}$. In particular, $\Pi_0$
is not connected. Let $\pi$ be a connected component of $\Pi_0$ such that
$\Delta(\pi)\cap \Delta(L)=\emptyset$. There exists $w\in W(\pi)$ and $\Pi'$ 
such that $w\beta\in\Pi'$ (by \Lem{lemapp1} the stronger assertion holds).
Write $w$ as a product of simple reflections
and let $T$ be the product of the corresponding Enright functors.
Then $w\beta\in \Delta(T(L))$
and  $T(L)=L(w(\lambda_{\Pi}+\rho_{\Pi})-\rho_{\Pi}, \Pi)$, so
the assertion holds for $L'=T(L)$ and $w\beta\in\Pi'$.
\qed

\subsection{Proofs for vacuum cases}\label{FinAff}
Since $F(L)$ is a vacuum module, $F(L)$ is not only $\pi$-integrable, but
also $\dot{\Pi}_0(L)$-integrable,
so $(\ol{\lambda}_{\Pi}-\nu,\alpha^{\vee})\geq 0$ for each $\alpha\in\dot{\Pi}_0(L)$.
For $\alpha\in\dot{\Pi}(L)$ one has $(\ol{\lambda}_{\Pi},\alpha^{\vee})=0$, so 
we obtain 
\begin{equation}\label{c333}
(\nu,\alpha^{\vee})\leq 0\ \text{ for each }\alpha\in\dot{\Pi}_0(L).
\end{equation}

We will frequently use the following statement, which is a part of Thm. 4.3 in~\cite{K2}:

(Fin) If $A$ is a Cartan matrix of a semisimple Lie algebra and $v$ is a vector
with rational coordinates, then $Av\geq 0$ implies $v>0$ or $v=0$;

(Aff)  If $A$ is a Cartan matrix of an affine Lie algebra and $v$ is a vector
with rational coordinates, then $Av\geq 0$ 
implies $v\in\mathbb{Q}\delta$.

\subsubsection{Proof of~\Thm{thmadm2}}
Arguing as in~\S\ref{proofadm1}
we obtain that for any $S\subset\dot{\Pi}(L)$ satisfying the conditions in~\S~\ref{PiSpi}
one has $\nu\in\mathbb{Z}_{\geq 0} S$, $\nu\not=0$. In particular,
$\nu\in \mathbb{Z}_{\geq 0}\dot{\Pi}$.

Consider the case $\dot{\Delta}(L)\not=A(m,n), C(n)$, i.e.
$\mathbb{Q}\dot{\Delta}_{\ol{0}}=\mathbb{Q}\dot{\Delta}$. Combining~(\ref{c333}) and
(Fin)  we obtain $\nu\in- \mathbb{Q}_{\geq 0} \dot{\Delta}_{\ol{0}}^+$.
Since $\nu\in \mathbb{Z}_{\geq 0}\Delta^+$ we get $\nu=0$, a contradiction.

If $\Pi(L)=C(m)$ or $\Pi(L)=A(m,n)^{(1)}, m>n$, then using the condition
that affine root is not isotropic, we obtain that the set of isotropic roots
$Iso\subset\Pi(L)$ lies in $\dot{\Pi}(L)$. It is easy to see that
the condition $||\alpha||^2\geq 0$ implies that each connected component
of $Iso$ is of type $A(n+1,n)$ (and not of type $A(n,n)$).
By above, it is enough to verify that for some $S$ 
$$\mathbb{Z}_{\geq 0} S\cap\{\mu|\ \forall \alpha\in\dot{\Pi}_0(L)\ 
(\mu,\alpha^{\vee})\leq 0\}=\{0\}.$$
Clearly, it is  enough to check the assertion for 
each connected component of $Iso\subset \dot{\Pi}(L)$.
Retain notation of~\S~\ref{choiceS} and
choose $S=\{\vareps_i-\delta_i\}$ in each component of $Iso$.
Write $\nu=\sum k_i (\vareps_i-\delta_i)$.
Taking $\alpha=\vareps_i-\vareps_{i+1},\delta_i-\delta_{i+1}\in \Pi_0$, we
obtain $k_i=k_{i+1}$ for each $i$, that is $\nu=k\sum_{i=1}^n (\vareps_i-\delta_i)$ for some 
$k\geq 0$. Then $(\nu,(\vareps_n-\vareps_{n+1})^{\vee})=k$, so $k\leq 0$, that is
$\nu=0$, a contradiction.
\qed

\subsubsection{Proof of~\Thm{thmadm4}}
$$ $$
Combining~(\ref{C111}) and~(\ref{c333}), we obtain that $\nu$ satisfies
the formulas~(\ref{maxfinn}).
Then arguing as in~\S~\ref{A2n-12n-1} (resp., \S~\ref{A2n4}, \S~\ref{D21a})
we get the assertion for $A(2n-1,2n-1)^{(2)}, D(n+1,n)^{(1)}$ (resp., $A(2n,2n)^{(4)}, D(n+1,n)^{(2)}$
and $D(2,1,a)^{(1)}$); the restriction
that $\dot{\Pi}(L)\subset \Pi$ for $\Delta(L)=A(2n,2n)^{(4)}, D(n+1,n)^{(2)}$
with level $1$, comes from the use of odd reflections in the proof of this case.

Now consider the remaining case $\Delta(L)=\fg=D(2,1,a)^{(1)}$.
Recall that $a\not=0,-1$ and $D(2,1,a)\cong D(2,1,a^{-1})\cong D(2,1,-1-a)$;
if $a\in\mathbb{Q}$, we assume (without loss of generality) that $-1<a<0$.
 
One has $D(2,1,a)_{\ol{0}}=A_1\times A_1\times A_1$; if we denote the root
in $i$th copy of $A_1$ by $2\vareps_i$, then
$||2\vareps_1||^2:||2\vareps_2||^2:||2\vareps_3||^2=1:a:(-a-1)$.

Let $L(\lambda)$ be a $\pi$-integrable vacuum module 
for some $\pi\subset \Pi_0$. If $\pi\setminus\dot{\Pi_0}$ contains one root, then
$\pi=\{\delta-2\vareps_r,\vareps_r\}$ and $L(\lambda)$ is $\pi$-integrable
 if and only if $2k/||2\vareps_r||^2\in\mathbb{Z}_{\geq 0}$. If $\pi\setminus\dot{\Pi_0}$ contains two roots, then
$\pi=\{\delta-2\vareps_r,\delta-2\vareps_q, 2\vareps_r, 2\vareps_q\}$, and, by above,
 $L(\lambda)$ is  $\pi$-integrable  if and only if
$2k/||2\vareps_r||^2,\ 2k/||2\vareps_q||^2\in\mathbb{Z}_{\geq 0}$; in particular, if $k\not=0$,
then $||2\vareps_r||^2/||2\vareps_q||^2\in\mathbb{Q}_{>0}$, so $a\in\mathbb{Q}$ and,
 since $-1<a<0$, $r,q=2,3$. If  $\pi\setminus\dot{\Pi_0}$ contains three roots,
then $\pi=\Pi_0$ and $k=0$. 

We see that $L(\lambda), k\not=0$ can be $A_1^{(1)}$-integrable for any copy 
$A_1^{(1)}$ in $\Pi_0$,
but it is $A_1^{(1)}\times A_1^{(1)}$-integrable  only if $a\in\mathbb{Q}$
and the roots of $\pi$ have positive integral square length for some normalization of $(-,-)$.

Recall that $\pi=\{\alpha\in\Pi_0|\ ||\alpha||^2\in\mathbb{Q}_{>0}\}$
for some normalization of $(-,-)$. 
If $a\not\in\mathbb{Q}$, then $\pi$ can be any copy of  $A_1^{(1)}$.
If  $a\in\mathbb{Q}$, then either ${\pi}=A_1^{(1)}$, 
which corresponds
to the longest root (the absolute value of $||2\vareps_i||^2$ is maximal; this
is $2\vareps_1$ if $-1<a<0$),
or ${\pi}=A_1^{(1)}\times A_1^{(1)}$ (and then $\dot\pi=\{2\vareps_2,2\vareps_3\}$, by above).

Let us show that~(\ref{C111}) contradicts~(\ref{C331}).

Recall that there are $4$ sets of simple roots: 
$$\Pi_1=\{\delta-\vareps_1-\vareps_2-\vareps_3,
-\vareps_1+\vareps_2+\vareps_3, \vareps_1+\vareps_2-\vareps_3,
\vareps_1-\vareps_2+\vareps_3\},$$
and 
$\Pi_2:=\{\delta-2\vareps_1,\vareps_1-\vareps_2-\vareps_3,2\vareps_2,2\vareps_3\}$,
with similar $\Pi_3,\Pi_4$ ($\delta-2\vareps_i\in \Pi_{i+1}$).

Take $\Pi_s$. Write $\mu=j\delta-\sum_{i=1}^3 e_i\vareps_i$.
By~(\ref{C331}), $e_1,e_2,e_3\geq 0$.
By~(\ref{C111}), $\mu\in\mathbb{Z}_{\geq 0}\Pi_i$;
in all $4$ cases, using $e_1,e_2,e_3\geq 0$, we get $e_1,e_2,e_3\leq 2j$. Hence
\begin{equation}\label{ejj}
0\leq e_i\leq 2j\ \text{ for }i=1,2,3.
\end{equation}
In particular, $\mu=0$ if $j=0$. Since $\mu=0$ contradicts~(\ref{C111}), we assume that $j>0$.

Consider the case $\pi=\{2\vareps_1,\delta-2\vareps_1\}$.
Normalize $(-,-)$ by $||2\vareps_1||^2=2$ (then $||2\vareps_2||^2=2a,\ 
||2\vareps_3||^2=-2(a+1)$).
Recall  that $k\in\mathbb{Z}_{>0}$.
By~(\ref{C331}) one has $e_1\leq k$ and 
$||\mu||^2=2(\lambda+\rho,\mu)$, that is $e_1^2+ae_2^2-(a+1)e_3^2=4jk$.
If $a\not\in\mathbb{Q}$, then we get $e_2=e_3$ and
$e_1^2-e_3^2=4jk$; if $a\in\mathbb{Q}$, then, by our assumption, $-1<a<0$ and we obtain
$e_1^2\geq 4jk$. Combining with $e_1\leq k$ with~(\ref{ejj}), we get $j=0$, a contradiction.

Consider the case $\pi=\{\delta-2\vareps_2,\delta-2\vareps_3, 2\vareps_2, 2\vareps_3\}$
(with $-1<a<0$).
Normalize $(-,-)$ by $||2\vareps_2||^2=2$ (then
$||2\vareps_1||^2=2/a, \  ||2\vareps_3||^2=-2(a+1)/a$).
Recall that $k, -ka/(a+1)\in\mathbb{Z}_{>0}$.
By~(\ref{C331}) one has $e_2\leq k, e_3\leq -ka/(a+1)$ and
$||\mu||^2=2(\lambda+\rho,\mu)$, that is 
$$e_1^2/a+e_2^2-\frac{a+1}{a}e_3^2=4jk.$$
Recall that $j>0$.
By~(\ref{ejj}), $e_2,e_3\leq 2j$, so we obtain $e_1=0,e_2=2j=k, e_3=2j=-ka/(a+1)$.
If $a\not=-\frac{1}{2}$ this is impossible. For $a=-\frac{1}{2}$ we get
$\mu=j(\delta-2\vareps_2-2\vareps_3)$, which does not lie in 
$\mathbb{Z}_{\geq 0}\Pi_s$ for $s=1,3,4$. Hence~(\ref{C111}) contradicts~(\ref{C331})
for $\Pi_i$ with $i=1,3,4$.
\qed

\subsubsection{Proof of~\Thm{thmadm3}}
By~(\ref{C111}), $2(\ol{\lambda}_{\Pi}+\rho_L,\nu)=||\nu||^2$.
The condition (C) implies $\rho_L=\xi_1-\xi_2$, where
$(\xi_1,\nu)\geq 0$ (since $\nu \in \mathbb{Z}_{\geq 0}\Pi(L)$)
and $(\xi_2,\nu)\leq 0$  (by~(\ref{c333})).
Hence $(\rho_L,\nu)\geq 0$, that is
$$2(\ol{\lambda}_{\Pi},\nu)\leq ||\nu||^2.$$

Let $s\delta$ be the minimal imaginary root in $\Delta(L)$ 
($\delta$ is the minimal imaginary root in $\Delta$).
Write $\dot{\Pi}(L)_0=\dot{\pi}\cup\pi''$ and recall that $||\alpha||^2\not\in\mathbb{Q}_{>0}$
for $\alpha\in\pi''$.

Write $\nu=j(s\delta)-\nu'-\nu''$, where $\nu'\in \mathbb{Q}\dot{\pi},\ \nu''\in \mathbb{Q}\pi''$.
From  (Fin)  (see~\S~\ref{FinAff}) we deduce from
the condition~(\ref{c333}) that
$$\nu'\in\mathbb{Q}_{\geq 0}\dot{\pi},\ \nu''\in \mathbb{Q}_{\geq 0}\pi''.$$
Note that
$j(s\delta)-\nu'=\nu+\nu''\in \mathbb{Z}_{\geq 0}\Pi(L)$ because $\pi''\subset \mathbb{Z}_{\geq 0}\Pi(L)$. Therefore
$j(s\delta)-\nu'\in \mathbb{Q}{\pi}\cap \mathbb{Z}_{\geq 0}\Pi(L)$.

We claim that $j(s\delta)-\nu'\in \mathbb{Q}_{\geq 0}{\pi}$.
Indeed, write ${\pi}=\dot{\pi}\cup\{\alpha_0^{\sharp}\}$. One has
$$j(s\delta)-\nu'=j\alpha_0^{\sharp}+\sum_{\alpha\in\dot{\pi}} 
y_{\alpha}\alpha=j\sum_{\beta\in\supp(\alpha_0^{\sharp})}
x_{\alpha_0^{\sharp},\beta}+\sum_{\alpha\in\dot{\pi}} \sum_{\beta\in\supp(\alpha)}
y_{\alpha}x_{\alpha,\beta}\beta.$$
In the light of assumption (B), the coefficient of the root $b(\alpha)$ is equal to 
$y_{\alpha}x_{\alpha,b(\alpha)}$;
thus $j(s\delta)-\nu\in \mathbb{Z}_{\geq 0}\Pi(L)$ gives 
$y_{\alpha}x_{b(\alpha)}\geq 0$, so $y_{\alpha}\geq 0$.
Hence  $j(s\delta)-\nu'\in \mathbb{Q}_{\geq 0}{\pi}$, as required.

Set $k:=(\ol{\lambda},s\delta)$. Then $k=(\ol{\lambda},\alpha_0^{\sharp})$, so
$(\nu,\alpha_0^{\sharp})\leq k$ (by~\ref{C331}). Using~(\ref{c333}) we get

$$\begin{array}{l}
||\nu||^2=||j(s\delta)-\nu'||^2+||\nu''||^2\leq ||j(s\delta)-\nu'||^2\\
=j(j(s\delta)-\nu',\alpha_0^{\sharp})+
\sum_{\alpha\in\dot{\pi}}y_{\alpha} (j(s\delta)-\nu',\alpha)
=j(\nu,\alpha_0^{\sharp})+
\sum_{\alpha\in\dot{\pi}}\sum y_{\alpha} (\nu,\alpha)\leq jk
\end{array}$$

Since $2(\ol{\lambda},\nu)=2jk$ we obtain $j=0$, so $\nu=-\nu'-\nu''$. 
Since $\nu\in \mathbb{Z}_{\geq 0}\Pi(L)$ with
$\nu'\in\mathbb{Q}_{\geq 0}\dot{\pi}^{\sharp},\ \nu''\in \mathbb{Q}_{\geq 0}\dot{\pi}''$ 
we get $\nu=0$, a contradiction. 
\qed

\subsection{Proof of~\Thm{dimFL=1}}
Let $L$ be a non-critical $\fg$-module such that  for each connected component $\Delta^i$
of $\Delta(L)$ one has either $(F(L),\Delta^i)=0$ or $F(L)$ is $\Delta^i_{\ol{0}}$-integrable and
$\Delta^i$-typical (i.e., $(F(L),\beta)\not=0$ for each $\beta\in\Delta^i$).

We want to prove~(\ref{KRWconj}) holds.

By~\Lem{lemRRchL},   $Z'$ does not depend on the choice of $\Pi$ and 
is preserved by $\sim$ ($Z'(L)=Z'(L')$ if $L'\sim L$).

Decompose $\Delta(L)$ in the union of irreducible components $\Delta(L)=\coprod \Delta^j$.
By~(\ref{suppZ'}), we can decompose
$\ol{\lambda}_{\Pi}-\mu=\sum \mu_j$ with  $\mu_j\in\mathbb{Z}\Delta^j$.
This decomposition might be not unique (even for fixed $\Pi$): $\mu_s$ is uniquely defined if
 $\Delta^s$ is finite, but $\mu_s-\mu'_s\in\mathbb{Z}\delta$ for different decompositions of $\mu$, 
 if $\Delta^s$ is affine.  By~(\ref{suppZ'}), there exists a decomposition, where
 $\mu_j\in\mathbb{Z}(\Delta^j\cap\Delta^+)$.

 Note that $\ol{\lambda}_{r_{\beta}\Pi}=\ol{\lambda}$ if $\beta\not\in\Delta(L)$ or
$(\ol{\lambda}_{\Pi},\beta)=0$, and $\ol{\lambda}_{r_{\beta}\Pi}=\ol{\lambda}-\beta$ otherwise.
In particular, if $(\ol{\lambda}_{\Pi},\Delta^j)=0$ for some $\Pi$, then this holds for each $\Pi$.

 Assume that $(\ol{\lambda}_{\Pi},\Delta^j)=0$.
We claim that for any two decompositions
$\ol{\lambda}_{\Pi}-\mu=\sum \mu_s$ and $\ol{\lambda}_{\Pi'}-\mu=\sum \mu'_s$ 
one has $\mu_j-\mu'_j\in\mathbb{Z}\delta$ (in particular, $\mu_j=\mu'_j$ if 
$\Delta^j$ is finite). Indeed, it is enough to consider the case
 $\Pi'=r_{\beta}\Pi$ and $\ol{\lambda}_{\Pi}\not=\ol{\lambda}_{\Pi'}$. Then
$\ol{\lambda}_{\Pi'}=\ol{\lambda}-\beta$ and $(\ol{\lambda}_{\Pi},\beta)\not=0$, that is
$\beta\not\in\Delta^j$. Therefore there exists a decomposition
$\ol{\lambda}_{\Pi'}-\mu=\sum \mu"_s$ with $\mu"_j=\mu_j$; since
$\mu"_j-\mu_j\in\mathbb{Z}\delta$ the claim follows.

Let us show that  
\begin{equation}\label{olmabdas}
(\ol{\lambda}_{\Pi},\Delta^j)=0\ \Longrightarrow\ 
\left\{\begin{array}{ll}
\mu_j\in\mathbb{Z}\delta\ \text{ if }\Delta^j\not=A(m,n)^{(1)}, C(m)^{(1)},\\
\mu_j\in\mathbb{Z}\delta +\mathbb{Z}\xi_j, \text{ if }\Delta^j=A(m,n)^{(1)}, C(m)^{(1)},
\end{array}\right.
\end{equation}
where for $\Delta^j=A(m,n)^{(1)}, C(m)^{(1)}$ the element $\xi_j\in\mathbb{Z}\Delta^j$
is such that $(\xi_j,\Delta^j_{\ol{0}})=0$.

Take $j$ such that  $(\ol{\lambda}_{\Pi},\Delta^j)=0$.  
 Since $(\Delta^j,\Delta^s)=0$ for $s\not=j$, (\ref{C331})
 gives 
 \begin{equation}\label{C31}
(\mu_j,\alpha^{\vee})\leq 0\ \text{ for each }\alpha\in\Pi_0(L)\cap\Delta^s.
\end{equation}

\subsubsection{}
Assume that $\Delta^j$ is finite and
$\mathbb{Q}\Delta^j=\mathbb{Q}\Delta^j_{\ol{0}}$. Then, from (Fin) (see~\S~\ref{FinAff}),
$\mu_j\in -(\mathbb{Q} \Delta^j_{\ol{0}}\cap\Delta^+)$; since
$\mu_j\in\mathbb{Z}(\Delta^j\cap\Delta^+)$, we get $\mu_j=0$, as required.

Assume that $\Delta^j$ is affine and
$\mathbb{Q}\Delta^j=\mathbb{Q}\Delta^j_{\ol{0}}$. Then, from (Aff) (see~\S~\ref{FinAff}), $\mu_j\in\mathbb{Z}\delta$,
as required.

Assume that  $\Delta^j$ is $A(m,n)^{(1)}, C(n)^{(1)}$.
Since $L$ is non-critical, $(\rho_L,\delta)\not=0$, so the case $A(n,n)^{(1)}$ is excluded.
Let $\Delta^j$ be of type $A(m,n)^{(1)}$ with  $m\not=n$, or $C(n)^{(1)}$.
In this case $\mathbb{Q}\Delta^j$ lies in 
$\mathbb{Q}\Delta^j_{\ol{0}}+\mathbb{Z}\xi_j$, 
where $\xi_j\in\fh^*$ is orthogonal  $\Delta^j_{\ol{0}}$.
Combining (Aff)  and~(\ref{C31}) we obtain $\mu_j\in\mathbb{Z}\delta +\mathbb{Z}\xi_j$,
as required.

\subsubsection{}
Now consider the remaining case $\Delta^j=A(m,n), C(m)$.
By above, $\mu_j$ is uniquely defined, and, in particular,
 does not depend on $\Pi$.

One has $\mathbb{Q}\Delta^j=\mathbb{Q}\Delta^j_{\ol{0}}+\mathbb{Z}\xi$, 
where $\xi\in\fh^*$ is orthogonal  $\Delta^j_{\ol{0}}$.
Combining (Fin)  in~\S~\ref{FinAff} and~(\ref{C31}) we obtain $\mu_j=x\xi-\mu'$, where $\mu'\in \mathbb{Q}^+\Pi^j_0$.
Since $\mu_j\in \in\mathbb{Z}_{\geq 0}\Pi^j$ for each $\Pi$, we obtain
$x\xi\in \mathbb{Z}_{\geq 0}\Pi^j$ for each $\Pi$.
Thus for each set of simple roots $\Pi$ the corresponding set of simple roots $\Pi^j$
(the set of simple roots for $\Delta^j\cap \Delta^+(\Pi)$) is such that
$x\xi\in \mathbb{Z}_{\geq 0}\Pi^j$.

Consider the root $\vareps_1-\delta_1$ in $A(m,n)$ or $C(m)$. Assume that this root is simple in
$\Delta$, i.e. lies in a set of simple roots $\Pi$. If $\Pi^j$ is the 
set of simple roots in $\Delta^j\cap \Delta^+(\Pi)$, then $r_{\vareps_1-\delta_1}(\Pi^j)$
is the set of simple roots in $\Delta^j\cap \Delta^+(r_{\vareps_1-\delta_1}\Pi)$.
Write  $\xi=n\sum_{i=1}^m\vareps_i-m\sum_{i=1}^n\delta_i$ for $A(m,n)$ and $\xi=-\delta_1$
for $C(m)$.
Since $x\xi\in\mathbb{Z}_{\geq 0}\Pi^j$ and $\vareps_1-\delta_1\in\Pi^j$, 
 one has $x\geq 0$; similarly,  $x\xi\in\mathbb{Z}_{\geq 0}r_{\vareps_1-\delta_1}(\Pi^j)$
gives $x\leq 0$. Hence $x=0$. Since $x=0$, we have $\mu_j=-\mu'\in -\mathbb{Q}^+\Pi^j_0$.
Combining with $\mu_j\in  \mathbb{Z}_{\geq 0}\Pi^j$, we get $\mu_j=0$.

It remain to show that  $\vareps_1-\delta_1$ lies in some set of simple roots for $\Delta$.
 If $\Delta$ is of type $A(m',n')$, $C(n')$ or $A(m',n')^{(1)}$, $C(n')^{(1)}$,
then any odd root lies in a set of simple roots, so this holds.

Let us show that for other root systems $\Delta$ this can be achieved for some $L'\sim L$.
Denote by $\iota_L$ the embedding $A(m,n)\to\Delta$ with the image $\Delta^j$.
Since $m,n>1$, $\Delta$ is not exceptional or affinization of exceptional.
In the light of~\Lem{lemapp1} the root $\iota_L(\vareps_1-\delta_1)$ or 
$\iota_L(\delta_1-\vareps_1)\in X$ (see~\Lem{lemapp1} for notations). 
We may (and will) assume that $\iota_L(\vareps_1-\delta_1)\in X$.
Denote by $\Delta'$ the connected component of $\Delta_{\ol{0}}$  containing $\delta_1-\delta_2$
and by $\pi'$ its set of simple roots $(\pi'\subset\Pi_0$).
Take $\alpha\in\pi'$. If $\alpha\in\Delta(L)$ and $(\alpha,\beta)\not=0$, then $\alpha\in\Delta^j$,
and so $\alpha\in\Pi_0^j$ (since $\alpha\in\Pi_0$), which implies $(\beta,\alpha^{\vee})=-1$.
As a result, if $\alpha\in\pi'$ is such that $(\beta,\alpha^{\vee})>0$, then
$\alpha\not\in\Delta(L)$, and we can  apply the  Enright functor $T_{\alpha}(a)$ to $L$.
Set $L':=T_{\alpha}(a)(L)$.
Clearly, $\Delta(L')=r_{\alpha}(\Delta(L))$ and $\iota_{L'}=r_{\alpha}\iota_L$.
In particular,
$r_{\alpha}\Delta^j=A(m,n)$ and the $\iota_{L'}(\vareps_1-\delta_1)=
r_{\alpha}\beta<\beta$. By~\Lem{lemapp1}, repeating this procedure we
obtain $L"\sim L$, where  $\beta":=\iota_{L"}(\vareps_1-\delta_1)$ is an essentially simple
isotropic root, (see~\S~\ref{esse}), 
that is $\beta"\in \Pi$ for some $\Pi$
(and $\iota_{L"}(\delta_1-\vareps_1)\in \Pi'$ for $\Pi'=r_{\beta"}(\Pi)$).

\subsubsection{}
Now fix $\Pi$ and let $\ol{\lambda}_{\Pi}+\rho_L-\nu\in \supp(Z)$.
From~\Cor{corsuA2} it follows that   $\nu$
can be decomposed as a sum $\nu=\sum \nu_j$ with
\begin{equation}\label{C11}
2(\ol{\lambda}_{\Pi}+\rho_L,\nu_j)=
||\nu_j||^2,\ \nu_j\in \mathbb{Z}_{\geq 0}(\Delta^j\cap \Delta^+).
\end{equation}
and, moreover, that $\ol{\lambda}_{\Pi}+\rho_L-\nu_j\in W(\Delta^j) (\ol{\lambda}_{\Pi}+\rho_L)$
if the restriction of $\ol{\lambda}_{\Pi}$ to $\Delta^j$ is typical, i.e.
$(\ol{\lambda}_{\Pi}+\rho_{L,\Pi},\beta)\not=0$ for each $\beta\in\Delta^j$.

Consider the case when the restriction of $\ol{\lambda}_{\Pi}$ to $\Delta^j$ is typical, i.e.
$(\ol{\lambda}_{\Pi}+\rho_{L,\Pi},\beta)\not=0$ for each $\beta\in\Delta^j$.  By above,
$\ol{\lambda}_{\Pi}+\rho_L-\nu_j\in W(\Delta^j) (\ol{\lambda}_{\Pi}+\rho_L)$
and $\ol{\lambda}_{\Pi}+\rho_L$ is $\Pi_0(L)$-maximal 
in its  $W(\Delta^j)$-orbit by~\Thm{thmint12}. Since $(\nu-\nu_j,\Delta^j)=0$,
$\ol{\lambda}_{\Pi}+\rho_L-\nu\in W(\Delta^j) (\ol{\lambda}_{\Pi}+\rho_L-(\nu-\nu_j))$
and $\ol{\lambda}_{\Pi}+\rho_L-(\nu-\nu_j)$ is $\Pi_0(L)$-maximal in its $W(\Delta^j)$-orbit.
Therefore $Z\in \cR_{W(\Delta^j)}$. Since $Z$ is $W(\Delta^j)$-skew-invariant
(see~\Cor{propreint}), we conclude that
$\ol{\lambda}_{\Pi}+\rho_L-(\nu-\nu_j)\in \supp(Z)$. Note that
$\ol{\lambda}_{\Pi}+\rho_L-(\nu-\nu_j)>_{\Pi}\ol{\lambda}_{\Pi}+\rho_L-\nu$.

Now  let $\ol{\lambda}_{\Pi}+\rho_L-\nu$ 
be a maximal element in $\supp (Z)$ with respect to the order
$\nu'\geq \nu"$ if $\nu'-\nu"\in\mathbb{Z}_{\geq 0}\Delta^+$. 
Then $\ol{\lambda}_{\Pi}-\nu$  is a maximal element in $\supp (Z')$
with respect to the same order and so is a $\Pi_0(L)$-maximal element in its $W(L)$-orbit.
By above, $\nu_j=0$ for each $j$ such that 
the restriction of $\ol{\lambda}_{\Pi}$ to $\Delta^j$ is typical.

Let us show that $\nu_j=0$ for each $j$.
By the assumption $(\ol{\lambda}_{\Pi},\Delta^j)=0$. Then,
by~(\ref{olmabdas}), $\nu_j=0$ if $\Delta^j$ is finite and $\nu_j=k_j\delta$
if $\Delta^j\not=A(m,n)^{(1)}, C(m)^{(1)}$ is affine; in the latter case,
since $\Delta^j$ is not critical,
$(\rho_L,\delta)\not=0$, so~(\ref{C11}) forces $k_j=0$. 
If $\Delta^j=A(m,n)^{(1)}$ or $C(m)^{(1)}$, then $\nu_j=x_j\xi_j+k_j\delta$.
In the light of~\Lem{lemAmnxi},~(\ref{C11}) forces $\nu_j=0$ for $\Delta^j=A(m,n)^{(1)}$.
For $\Delta^j=C(n)^{(1)}$ we can (and will) normalize the form in such a way that 
$||\alpha||^2\geq 0$
for $\alpha\in\Delta^j$; then $||\xi_j||^2<0$ and $(\rho_L,\nu_j)\geq 0$
since $\nu_j\in \mathbb{Z}_{\geq 0}(\Delta^j\cap \Delta^+)$. Now (\ref{C11}) forces $x_j=0$, that is $2(\rho_L,k_j\delta)=0$; hence $\nu_j=0$, as required.

We conclude that $\nu=0$, a contradiction.
\qed

\section{Examples: $\fg^{\lambda}$ of type  $A(1,1)^{(1)}$ and $B(1,1)^{(1)}$}\label{sectexa}
In this section we establish the KW-formula in two more cases: $\ol{\fg}^{\lambda}$  is of types $A(1,1)^{(1)}$ or
$B(1,1)^{(1)}$.


\subsection{Case $\mathfrak{gl}(2,2)^{(1)},\ 
\mathfrak{sl}(2,2)^{(1)},\ \mathfrak{psl}(2,2)^{(1)}$}\label{A11}
Let $\fg:=\fgl(2,2)^{(1)}$ or $\fsl(2,2)^{(1)},\mathfrak{psl}(2,2)^{(1)}$.

Consider $\Pi=\{\alpha_0,\alpha_1,\alpha_2,\alpha_3\}$, where $||\alpha_i||^2=0$ and
 $(\alpha_i,\alpha_{i+1})\not=0$ (where $\alpha_4=\alpha_0$). Let
  $\Lambda_0,\Lambda_1,\Lambda_2,
\Lambda_3$ be the corresponding fundamental weights, i.e. $(\Lambda_i,\alpha_j)=\delta_{ij}$.
Let $L=L(\lambda,\Pi)$ be a non-critical module and $\Delta(L)\cong \Delta$.
We show that~(\ref{KRWconj}) holds if $L$ is a non-critical module such that
$F(L)\cong L(\lambda', \Pi)$, where $\lambda'=k_0\Lambda_0+k_2\Lambda_2$ or
$\lambda=k_0\Lambda_0+k_1\Lambda_1$, $k_0,k_1,k_2\in\mathbb{Z}_{\geq 0}$.

Note that we do not assume that $\ol{\lambda}$ is of above form (or that $\Pi(L)\cong \Pi$).

\subsection{Marked diagrams}\label{markdia}
Fix a irreducible highest weight module $L$. For each $\Pi$ take
$\lambda$ such that $L=L(\lambda,\Pi)$. Consider the Dynkin diagram of $\Pi(L)$ and
assign to each edge $\alpha-\alpha'$ the scalar product $(\alpha,\alpha')$ and to each node $\alpha$ the number
$x_{\alpha}:=(\lambda+\rho_{\Pi},\alpha)$ (which is integral).

 We call the diagram $\Pi(L)$ endowed by these numbers a {\em marked diagram} $D(L,\Pi)$ corresponding to  $(L,\Pi)$.

\subsubsection{}
If $\beta$ is an odd node of  a marked diagram $\Pi(L)$, we define
the action of $r_{\beta}$ on the marked diagram $D(L,\Pi)$ in such a way that $r_{\beta}*D(L,\Pi)=D(L,r_{\beta}\Pi)$
in the case when $\Pi(L)=\Pi$. This means that

the nodes connected to $\beta$ change their parity and other nodes (including $\beta$)
preserve the parity;

the scalar products between the node corresponding to $\beta$ and other nodes change
to the opposite; other scalar products do not change;

the mark $x_{\beta}$ of $\beta$ is changed to $-x_{\beta}$;
if the mark $x_{\beta}\not=0$, the new mark of the node $\beta'$
connected to $\beta$ is $x_{\beta}+x_{\beta'}$, and
if $x_{\beta}=0$, the new marks of the node $\beta'$
connected to $\beta$ is $x_{\beta'}+(\beta,\beta')$;
other marks do not change.

\subsubsection{Example}
The second diagram is obtained from the first one by the reflection with respect to the upper-right node:

$$\xymatrix{
 &^{k+1}\bigcirc \ar@{-}[r]^{-1}\ \ \ar@{-}[d]_{-1}&\bigotimes^{0}  \ar@{-}[d]^{1} &  &^{k}\bigotimes \ar@{-}[r]^{1}  \ar@{-}[d]_{-1}   &\bigotimes^{0} \ar@{-}[d]^{-1}\\
&\ ^{0}\bigotimes  \ar@{-}[r]^{1}   &\bigcirc^{-1} &  &^{0}\bigotimes \ar@{-}[r]^{1}  &\bigotimes^{0}
  }
$$

\subsubsection{}
By~\S~\ref{Pilambda}, $D(L,r_{\beta}\Pi)=r_{\beta}*D(L,\Pi)$
if $\beta\in\Pi$ and $D(L,r_{\beta}\Pi)=D(L,\Pi)$ otherwise.

We say that two marked diagrams $D,D'$ are {\em connected by an odd reflection} if $D'=r_{\beta}D$
for some odd node $\beta\in D$.

Denote by $DM(L)$ the set of marked diagrams $D(L',\Pi')$ for all $\Pi'$ (compatible with $\Pi_0$)
and all $L'$ such that $L\sim L'$. One readily sees that any two diagrams in $DM(L)$ are connected by
a chain of odd reflections $r_{\beta}*$. In~\Cor{cordia} below we show that if $D\in DM(L)$ and
$v\in D$ is an odd node, then $s_v* D\in DM(L)$. This implies that $DM(L)$ is the set of diagrams
obtained from $D$ by the action of chains of odd reflections; in particular, $DM(L)=DM(L')$
if $DM(L)\cap DM(L')$ is non-empty. Take a pair $(L,\Pi)$ and let $L'$ be such that $D(L,\Pi)=D(L',\Pi)$
with $\Delta(L')=\Delta$. Then $L'$ is partially integrable and $DM(L)=DM(L')$.

\subsubsection{}
\begin{lem}{lemoddroot22}
For each odd node of the marked diagram $D\in DM(L)$ there exists a pair $(L',\Pi')$ such that
$L\sim L'$, $D(L',\Pi')=D$ and the root in $\Pi'(L')$ which corresponds to this node is simple, i.e. lies in $\Pi'$.
\end{lem}
\begin{proof}
Let $D=D(L,\Pi)$ and let $\beta\in\Pi(L)$ be the root corresponding to the odd node in $D$.
We prove the assertion by induction on $\htt_{\Pi}(\beta)$.

If $\htt_{\Pi}\beta=1$, then $\beta\in\Pi$, as required.

If $\beta$ is of the form $\beta=j\delta+\beta'$ for $\beta'\in\Pi, j\in\mathbb{Z}_{>0}$, then $\beta'\not\in\Pi(L)$ and
for $\Pi':=r_{\beta'}\Pi$ one has  $\Pi'(L)=\Pi(L)$. Moreover, $\htt_{\beta}=4j+1$ and
$\htt_{\Pi'}\beta=4j-1$, because $\-\beta'\in\Pi'$, so $\htt_{\Pi'}(\delta+\beta')=3$.
Hence $\htt_{\Pi'}\beta<\htt_{\Pi}\beta$ and the assertion follows by induction.

Assume that $\beta\not=j\delta+\beta'$ for $\beta'\in\Pi, j\in\mathbb{Z}_{\geq 0}$.
Take $\alpha\in\Pi_0$ such that $||\alpha||^2=2$ and
$r_{\alpha}\beta<\beta$ (thus  $(\alpha,\beta)=1$).
Note that $\alpha\not\in\Pi(L)$, because for $\alpha,\beta\in\Pi(L)$ one has $r_{\alpha}\beta\geq \beta$.
Since $\alpha\in\Pi_0$, one has $1\leq \htt_{\Pi}\alpha\leq 3$.

If $\htt_{\Pi}\alpha=1$, i.e.
 $\alpha\in\Pi$, then $\alpha\not\in\Delta(L)$ (because $\alpha\not\in\Pi(L)$). Applying the Enright functor $T_{\alpha}$ (see~\S~\ref{enr})
we have
$$\Pi(T_{\alpha}(L))=r_{\alpha}\Pi(L),\ \ D(T_{\alpha}(L),\Pi)=D(L,\Pi)$$
and the node corresponding to $\beta$ is $r_{\alpha}\beta=\beta-t\alpha, t>0$. Thus 
$\htt_{\Pi} r_{\alpha}\beta=\htt_{\Pi}\beta-t<\htt_{\Pi}\beta$
and the assertion follows by induction.

Assume that $\htt_{\Pi}\alpha=2$. Then $\alpha=\alpha_1+\alpha_2$, where $\alpha_1,\alpha_2\in\Pi$ are odd roots  and $(\alpha_1,\alpha_2)=1$.
If $\alpha_1,\alpha_2\in\Pi(L)$, then $\Pi(L)$ contains three odd roots $\beta,\alpha_1,\alpha_2$ and since $(\alpha_1,\alpha_2)=1$ one has
$\{(\beta,\alpha_1),(\beta,\alpha_2)\}=\{0,-1\}$ that is $(\beta,\alpha_1+\alpha_2)=-1$, a contradiction. Thus at least one of the
roots $\alpha_1,\alpha_2$, say $\alpha_1$, is not in $\Pi(L)$. Since $\pi$ contains two non-orthogonal odd roots,
it contains only odd roots, so $\beta=j\delta\pm\beta'$ for some $\beta'\in\Pi, j\in\mathbb{Z}_{>0}$. From the above assumption
we obtain $\beta=j\delta-\beta'$. Since $\beta\in\Pi(L)$ one has $\beta'\not\in\Pi(L)$.  Moreover, $(\beta,\alpha_1+\alpha_2)=(\alpha_1,\alpha_2)$ forces
$\beta'\not\in\{\alpha_1,\alpha_2\}$.
For $\Pi':=r_{\alpha_1}\Pi$ one has $\alpha\in\Pi'$, $\Pi'(L)=\Pi(L)$.
Since $\alpha_1\not=\beta'$, one has $\htt_{\Pi'}(\delta-\beta')\leq \htt_{\Pi} (\delta-\beta')$
so $\htt_{\Pi'}\beta\leq \htt_{\Pi}\beta$. Since $\alpha\in\Pi'$, the assertion follows by induction from the above.

Now assume that $\htt_{\Pi}\alpha=3$, i.e. $\alpha=\alpha_1+\alpha_2+\alpha_3$, where $\alpha_i\in\Pi$ is odd for $i=1,3$ and even for $i=2$.
One readily sees that $||\alpha||^2=2$ forces  $||\alpha_2||^2=-2$. Since  $(\beta,\alpha)=1$,
$\beta$ is of the form $j\delta+\alpha_i$ or $j\delta+\alpha_i+\alpha_2$
for $i\in\{1,3\}$ and $j\in\mathbb{Z}_{\geq 0}$.  From the above assumption
we get $\beta=j\delta+\alpha_i+\alpha_2$. Since $\beta\in\Pi(L)$, one has
 $\alpha_i\not\in\Pi(L)$. Set $\Pi':=r_{\alpha_i}\Pi$.
One has $\Pi'(L)=\Pi(L)$ and $\htt_{\Pi'}\beta<\htt_{\Pi}\beta$ since $1=\htt_{\Pi'}(\alpha_i+\alpha_2)<\htt_{\Pi}(\alpha-i+\alpha_2)=2$.
The assertion follows by induction.
\end{proof}

\subsubsection{}
\begin{cor}{cordia}
If $D\in DM(L)$ and $v$ is an odd node of $D$, then $s_v*D\in DM(L)$.
\end{cor}

In the light of~\Cor{cordia} it is enough to verify formula~(\ref{KRWconj}) 
for one set of simple roots $\Pi(L)$ (if the formula holds for $L$ and
$F(L)\cong F(L')$, then the formula holds for $L'$).
We check the formula for the cases $\ol{\lambda}=k\Lambda_0+j\Lambda_2,
\ol{\lambda}=k\Lambda_0+j\Lambda_1$, $k,j\in\mathbb{Z}_{\geq 0}, k+j\not=0$,
where $\Pi(L)$ consists of odd roots (in particular, $\ol{\rho}=0$).

\subsubsection{Case $\ol{\lambda}=k_0\Lambda_0+k_2\Lambda_2, k_0,k_2\in\mathbb{Z}_{\geq 0}, k_0+k_2\not=0$}
In this case, $F(L)$ is integrable and $(\ol{\lambda}+\rho,\alpha_i)=0$ for $i=1,3$. 
Set $S:=\{\alpha_1,\alpha_3\}$ and $Z(\Pi)$ is the expansion of
\begin{equation}\label{Znew}
Z(\Pi):=Re^{\rho}\ch L-\cF_{ W(\pi)} \bigl(\frac{e^{(\lambda+\rho)}}{\prod_{\beta\in S}
(1+e^{-\beta})}\bigr)
\end{equation}
in $\cR(\Pi)$.
Arguing as in~Section~\ref{sectKW}, we obtain that the $\Pi$-maximal element
in $\supp Z(\Pi)$ is $\ol{\lambda}-\mu$, where $\mu\in \mathbb{Z}S$, that is  $\mu=a_1\alpha_1+a_3\alpha_3$.
where $a_1,a_3\geq 0, a_1+a_3\not=0$.
By~\Lem{lemoddroot22}, we can assume that $\alpha_1$ is a simple root. Then
$L(\lambda,\Pi)=L(\lambda,r_{\alpha_1}\Pi)$. Since $\lambda+\rho-(a_1-1)\alpha_1-a_3\alpha_3\not\in\supp Z(\Pi)$ and $\lambda+\rho-a_1\alpha_1-a_3\alpha_3\not\in\supp Z(\Pi)$,
\Lem{lemXX'} gives $\lambda+\rho-(a_1-1)\alpha_1-a_3\alpha_3\in \supp Z(\Pi')$
or $\lambda+\rho-a_1\alpha_1-a_3\alpha_3\in\supp Z(\Pi')$, where 
$\Pi':=r_{\alpha_1}\Pi$.
It is easy to see that $\supp Z(\Pi')\subset \lambda+\rho'-\mathbb{Z}_{\geq 0}\Pi'$,
so $-(a_1-1)\alpha_1-a_3\alpha_3$ or $-a_1\alpha_1-a_3\alpha_3$ lie in
$\rho'-\mathbb{Z}_{\geq 0}\Pi'=\alpha_1-\mathbb{Z}_{\geq 0}\Pi'$;
since $-\alpha_1,\alpha_3\in\Pi'$, we obtain $a_1\leq 0$, that is $a_1=0$.
Similarly, $a_3=0$, a contradiction. Hence $\supp Z$ is empty, that is
$Z=0$ and~(\ref{KRWconj}) holds.

\subsubsection{Case $\ol{\lambda}=k\Lambda_0+j\Lambda_1,\ k,j\in\mathbb{Z}_{\geq 1}$}
In this case $F(L)$ is $\pi\cup\{\alpha_0+\alpha_1\}$-integrable, where
$\pi=\{\alpha_0+\alpha_3,\alpha_1+\alpha_2\}\cong A_1^{(1)}$.
Set $S:=\{\alpha_2\}$ and define $Z:=Z(\Pi)$ as in~(\ref{Znew}).
Arguing as in~Section~\ref{sectKW}, we obtain that the $\Pi$-maximal element
in $\supp Z$ is of the form $\ol{\lambda}-\mu$, where
$\mu\in \mathbb{Z}_{\geq 0}\Pi(L), \mu\not=0$, $||\mu||^2=2(\ol{\lambda},\mu)$
and $(\ol{\lambda}-\mu,\alpha^{\vee})\geq 0$ for $\alpha\in \pi\cup\{\alpha_0+\alpha_1\}$.

Set 
$$\alpha_1=\vareps_1-\delta_1,
\alpha_2=\delta_1-\vareps_2,\alpha_3=\vareps_2-\delta_2,\alpha_0=s\delta-\vareps_1+\delta_2$$
($\delta$ is the minimal imaginary root in $\Delta$, $s\delta$  is the minimal imaginary root in $\Delta(L)$).
Note that $\xi:=\alpha_1+\alpha_3$
is orthogonal to $\Delta(L)$. Write 
$$\mu=a(s\delta)+b\xi+(d_1+d_2)\vareps_2-d_1\delta_1-d_2\delta_2=a\alpha_0+(a+b)\alpha_1+(a-d_1)\alpha_2+(a+b+d_2)\alpha_3.$$
The condition $\mu\in \mathbb{Z}_{\geq 0}\Pi(L)$ gives
$$a,b,d_1,d_2\in\mathbb{Z},\ \ a,a+b, a-d_1,a+b+d_2\geq 0.$$

The above conditions $(\ol{\lambda}-\mu,\alpha^{\vee})\geq 0$ give
$d_1\geq d_2, -j\leq d_1+d_2\leq k$. The condition
$||\mu||^2=2(\ol{\lambda},\mu)$
is equivalent to 
$$d_1d_2=ka+j(a+b).$$
Since $a,a+b\geq 0$ one has $d_1d_2\geq 0$. If $d_1+d_2\geq 0$, then the above  inequalities
 imply $0\leq d_1\leq a,d_1+d_2\leq k$, so $d_1d_2+d_1^2\leq ka$ and thus $d_1=d_2=0$.
 If $d_1+d_2<0$, then $d_1,d_2<0$ and the above  inequalities
 imply $-d_2\leq a+b, -d_1-d_2\leq j$, so $d_1d_2+d_2^2\leq j(a+b)$ and 
thus again $d_1=d_2=0$. Therefore
 $d_1=d_2=0$ and thus $a=a+b=0$ (since $j,k>0$).  Hence $\mu=0$, a contradiction, so~(\ref{KRWconj})
holds.

\subsection{Case $\fg_L=B(1,1)^{(1)}$}\label{B11}
Let $\fg$ be an affine Lie superalgebra and
$L$ be a $\fg$-module such that $\Delta(L)=B(1,1)^{(1)}$ 
and $F(L)$ satisfies the KW-condition. The case of typical $F(L)$
was considered in Section~\ref{sect8}. Below we establish  the KW-character formula
for atypical case when the KW-condition holds and $F(L)$ is $\pi$-integrable
for "sufficiently large" $\pi$ (in the standard notation $\pi=\{2\delta_1,\delta-2\delta_1\}$
or $\pi=\{\vareps_1,\delta-\vareps_1,2\delta_1\}$).

\subsubsection{}\label{betainPi}
Consider the embedding $\iota_L: B(1,1)^{(1)}\to\Delta$ given by the identification
$\Delta(L)\cong B(1,1)^{(1)}$. 

Recall (see~\S~\ref{esse}) that a root is called essentially simple if it lies
in some set of simple roots.
Let $\beta\in B(1,1)^{(1)}$ be an isotropic essentially simple root; let us show
that for some $L'\sim L$ the root $\iota_{L'}(\beta)$ is essentially simple.

Indeed, the non-isotropic roots of $B(1,1)^{(1)}$ are
$A_1^{(1)}\times B(0,1)^{(1)}$; 
each essentially simple root of $B(1,1)^{(1)}$
is of the form $\pm(\alpha_1-\alpha_2)$, were $\alpha_1$ (resp., $\alpha_2$) is a simple root of
$A_1^{(1)}$ (resp., of $B(0,1)^{(1)}$). Assume that $\iota_L(\alpha_1)\not\in\Pi_0$. Then there exists $\gamma\in\Pi_0$ such that 
$r_{\gamma}\iota(\alpha_1)<\alpha_1$ (see~\S\ref{ordering} for $<$);
note that $\gamma\not\in\iota(\Delta(L))$ (since $r_{\alpha'}\alpha\geq \alpha$
for $\alpha'\in\Pi_0(L)\setminus\{\alpha\}$) and that $r_{\gamma}\alpha_2=\alpha_2$.
For $L':=T_{\gamma}(L)$ we have $\iota_{L'}(\alpha_1)<\iota_L(\alpha_1)$ and 
$\iota_{L'}(\alpha_2)=\iota_{L'}(\alpha_2)$. 
A similar reasoning works if $\iota_L(\alpha_2)\not\in\Pi_0$.
Hence there exists $L'$ such that $\iota_{L'}(\alpha_1),\iota_{L'}(\alpha_2)\in\Pi_0$.
By~\S~\ref{ess24}, $\iota_{L'}(\alpha_1)-\iota_{L'}(\alpha_2)$ is an essentially simple root,
as required.

\subsubsection{}
Now assume that $L=L(\lambda,\Pi)$  is such that $F(L)$ satisfies KW-condition.
Let us show that  for some set of simple roots $\Pi'$ of $\Delta$ and
some $L'\sim L$ one has 
$L'=L(\lambda',\Pi')$, where $(\lambda'+\rho,\beta')=0$ for some $\beta'\in\Pi'$.

Since $F(L)$ satisfies KW-condition and
$\fg_L=B(1,1)^{(1)}$, one has $(\lambda+\rho,\iota_L(\beta))=0$,
where $\beta\in B(1,1)^{(1)}$ is essentially simple.

Indeed, if this does not hold, then for all $L'\sim L$ the value $\lambda'+\rho'$ does 
 not depend on $\Pi'$ (i.e., if $L'=L(\lambda',\Pi')=L(\lambda'',\Pi'')$, then
$\lambda'+\rho'=\lambda"+\rho"$). Then taking $L'=T_{\gamma_1}\ldots T_{\gamma_s}(L)$ as above, 
we obtain $L'=L(\lambda',\Pi')$, where $\lambda'+\rho'=w(\lambda+\rho)$,
and $\iota_{L'}(\beta)=w\iota_L(\beta)$, where
$w:=r_{\gamma_1}\ldots r_{\gamma_s}$. Hence $(\lambda'+\rho',\iota_{L'}(\beta))=0$.
Since $\beta':=\iota_{L'}(\beta)$, we can choose $\Pi'$ containing $\beta'$
and $L',\Pi'$ satisfies our requirements.

\subsubsection{}
By above, we can (and will) assume that $L=L(\lambda,\Pi)$ is such that
$(\lambda+\rho,\beta)=0$ for some $\beta\in\Pi$.

For $B(1,1)^{(1)}$ there are three sets of simple roots:
$$\Pi_1:=\{\delta-\delta_1-\vareps_1,\vareps_1-\delta_1,\delta_1\},\ \
\Pi_2:=\{\delta-2\delta_1,\delta_1-\vareps_1,\vareps_1\},\ \ 
\Pi_3:=\{\delta_1+\vareps_1-\delta,\delta-2\delta_1,\delta-\varepsilon_1\}$$
and the set $\Pi_1$ can be obtained from $\Pi_2$ (resp., from $\Pi_3$)
by an odd reflection with respect to the unique isotropic root in $\Pi_2$ 
(resp., in $\Pi_3$). Hence
we can (and will) assume that $\Pi(L)\cong\Pi_1$; we identify $\Pi(L)$
and $\Pi_1$.

If $F(L)$ is integrable,~(\ref{KRWconj}) follows from~\Thm{thmadm1} 
(since $S=\{\beta\}\subset \Pi$
and $\Pi_1$ satisfies the conditions of Section~\ref{sectKW} ).

\subsubsection{}
Now consider the case when $F(L)$ is non-critical $\pi$-integrable, where $\pi=\{\delta-\vareps_1,\vareps_1,2\delta_1\}$ (subprincipal case in~\cite{KW4}).
The formula for $\ol{\lambda}=0$ is proved in~\Thm{dimFL=1}, so we assume that 
$\ol{\lambda}\not=0$.
We will show that if 
$\ol{\lambda}\not=0$ and KW-condition holds, i.e. $(\lambda+\rho,\beta)=0$
for some $\beta\in\Pi(L)$, then
\begin{equation}\label{subprB11}
Re^{\rho}\ch L=\frac{1}{2}\cF_{ W(\pi)} \bigl(\frac{e^{\lambda+\rho}}
{1+e^{-\beta}}\bigr).
\end{equation}
This implies~(\ref{KRWconj}).

\subsubsection{}
Set $y_0:=(\ol{\lambda},\delta-\delta_1-\vareps_1), \ 
y_1:=(\ol{\lambda},\vareps_1-\delta_1),\ y_2:=(\ol{\lambda},\delta_1)$. The module $L(\ol{\lambda},\Pi_1)$ is $\pi$-integrable
if and only if 
either $\ol{\lambda}=0$ or it is one of the following cases, cf.~\cite{KW4}:

$y_1=y_2=0,\  -2(y_0+1)\in\mathbb{Z}_{\geq 0}$; 

$y_0=y_2=0,\  -2(y_1+1)\in\mathbb{Z}_{\geq 0}$;

$y_0,y_1\not=0,\  2y_2, -2(y_0+y_2+1),-2(y_1+y_2+1)\in\mathbb{Z}_{\geq 0}$.

KW-condition holds for first two cases and does not hold for atypical modules
in the third case (since $y_0,y_1\not=0$).
Note that $B(1,1)^{(1)}$
admits an automorphism given by $\vareps_1\mapsto \delta-\vareps_1,\ \delta_1\mapsto \delta_1$,
which interchanges the isotropic roots of $\Pi_1$; this automorphisms interchanges the first
and the second cases. For the third case KW-condition does not hold.

Therefore we may (and will) consider the first case (when $y_1=y_2=0$, i.e., 
 $F(L)$ is a vacuum module). In this case
$\beta=\vareps_1-\delta_1$.
Since $\ol{\lambda}$ is non-critical, $y_0\not=-1$.
Hence
$$y_1=y_2=0,\ \ 2y_0\in\mathbb{Z},\ y_0<-1.$$

Using the denominator identity for $B(1,1)$ (see~(\ref{denomfin})) we  rewrite~(\ref{subprB11})
in the form
$$Re^{\rho}\ch L=\cF_{W(\pi')} \bigl(\frac{e^{\lambda+\rho}}
{1+e^{-\beta}}\bigr),$$
where $\pi':=\{\delta-\vareps_1,\vareps_1\}\cong A_1^{(1)}$. 
Set $\alpha_1=\vareps_1,\alpha_0:=\delta-\vareps_1$ and
let $r_1,r_0\in W(\pi')$ be the corresponding reflections.

First, let us show that the support of the right-hand side is in 
$\lambda+\rho-\mathbb{Z}_{\geq 0}\Pi_1$ and
that the coefficient of $e^{\lambda+\rho}$ is equal to $1$.
 For $w\in W(\pi')$ let
$$Y_w:=\frac{e^{w(\lambda+\rho+\beta)}}
{1+e^{w\beta}}\in \cR(\Pi_1)$$
(i.e., $Y_w$ is the expansion in $\cR(\Pi_1)$ of the fraction in the right-hand side). 
It is enough to verify that for each $w\in W(\pi')$
\begin{equation}\label{eqb11}
\supp Y_w\subset \lambda+\rho-\mathbb{Z}_{\geq 0}\Pi_1\end{equation}
and for $w\not=Id$ 
\begin{equation}\label{eqb111}
\lambda+\rho\not\in\supp Y_w.\end{equation}
Our reasoning is based on the formula~(\ref{suppYw}).
One has $(\lambda+\rho,\alpha_1^{\vee})=-1,\ (\lambda+\rho,\alpha_0^{\vee})=-2y_0-1$
and so $(r_{1}(\lambda+\rho),\alpha_1^{\vee})=1$,
$(r_{1}(\lambda+\rho),\alpha_0^{\vee})=-2y_0-3\geq 0$.
Therefore $r_{1}(\lambda+\rho)$ is maximal in its $W(\pi')$-orbit;
this establishes~(\ref{eqb11}) for 
$w\not=Id, r_1, r_0r_1$ and~(\ref{eqb111}) for the same $w$ if
$y_0<-3/2$; for $y_0=-3/2$ this gives ~(\ref{eqb111}) for $w\not=r_1, r_0r_1, r_1r_0r_1$.
For $w=r_1,r_0r_1,r_1r_0r_1$ one has $w(-\beta)\in \Delta^+(\Pi_1)$ and,
moreover, $w(-\beta)\geq_1\delta_1+\vareps_1$, where $\geq_1$ stands for $\Pi_1$-partial order.
Thus for such $w$ one has
$$\supp Y_w\subset w(\lambda+\rho)+w\beta\leq_1 r_1(\lambda+\rho)+w\beta
\leq_1\lambda+\rho+\vareps_1-(\delta_1+\vareps_1)=
\lambda+\rho-\delta_1.$$
This establishes~(\ref{eqb11}),(\ref{eqb111}).

Now for $Z:=Re^{\rho}\ch L-\frac{1}{2}\cF_{W(\pi)} \bigl(\frac{e^{\lambda+\rho+\beta}}
{1+e^{\beta}}\bigr)$ we have $\supp Z\subset \lambda+\rho-\mathbb{Z}_{\geq 0}\Pi_1$
and $\lambda+\rho\not\in \supp Z$. 

Since $F(L)$ is $\pi$-integrable, $Z$ is $W(\pi)$-skew-invariant.
Arguing as in~\S~\ref{proofs1} we conclude that it is enough to verify that if
$\mu\in \mathbb{Z}_{\geq 0}\Pi_1$ satisfies $2(\lambda+\rho,\mu)=||\mu||^2$
and $(\lambda-\mu,\alpha^{\vee})\geq 0$ for $\alpha\in\pi$, then  $\mu=0$.
Write
$$\mu=s_0(\delta-\delta_1-\vareps_1)+s_1(\vareps_1-\delta_1)+
s_2\delta_1,\ \  s_0,s_1,s_2\geq 0.$$
Then $(\lambda-\mu,\alpha^{\vee})\geq 0$ for $\alpha\in\pi$ gives
$$s_0-s_1, s_0+s_1-s_2\geq 0,\ \ \  s_1-s_0\geq y_0$$
and $2(\lambda+\rho,\mu)=||\mu||^2$ gives
$$(s_0+s_1-s_2)^2-(s_0-s_1)^2=2y_0s_0+s_2.$$

The last formula can be rewritten as
$$(s_0+s_1-s_2)^2=(s_0-s_1)^2+y_0(2s_0-s_2)+(y_0+1)s_2.$$
By above, $y_0<-1, 0\leq s_0-s_1\leq -y_0, 2s_0-s_2$, so
the right-hand side is at most $(y_0+1)s_2<-s_2$. Hence
$s_2=s_0+s_1-s_2=0$ that is $\mu=0$, as required.
This completes the proof of~(\ref{subprB11}).

\section{Appendix}
Let $\Delta$ be a root system of a finite-dimensional 
basic Lie superalgebra or an associated (untwisted or twisted) affine Lie superalgebra.
As before, we fix $\Delta^+_{\ol{0}}$ and consider sets of simple roots $\Pi$
such that $\Delta(\Pi)^+$ contains  $\Delta^+_{\ol{0}}$.

\subsection{Essentially simple roots}\label{esse}
Let us call a  root $\beta$ {\em essentially simple} if 
there exists $\Pi$ which contains $\beta$.

The set of non-isotropic essentially simple roots coincides with the set of simple roots 
for the non-isotropic part of $\Delta$, see~\Prop{propS} (ii).
For $\Delta= A(m,n), C(n), A(m,n)^{(1)}$ and  
$C(n)^{(1)}$ any odd isotropic root
is essentially simple. 

\subsubsection{}\label{essen}
We say that $\Delta'$ is a root subsystem of $\Delta$ if $\Delta$ (resp., $\Delta'$)
has a subset of simple roots $\Pi$ (resp., $\Pi'$) such that $\Pi'\subset\Pi$.

Note that if $\Delta'$ is a root subsystem of $\Delta$, then the essentially simple
roots of $\Delta'$ are essentially simple for $\Delta$
(if $\beta$ is essential for $\Delta'$, then $\beta$ lies in a 
subset of simple roots for  ${\Delta}'$ which is obtained from
${\Pi}'$ by a chain of odd reflections; therefore $\beta$ lies 
in a subset of simple roots for  ${\Delta}$ which is obtained from
${\Pi}$ by the same  chain of odd reflections).

\subsubsection{}\label{ess24}
The following fact is useful. If $\alpha_1,\alpha_2$ are non-isotropic essentially simple roots
and $\alpha_1-\alpha_2$ is an isotropic root, then $\alpha_1-\alpha_2$ is
essentially simple (for instance, $\Delta=B(m,n)$ and 
$\alpha_1=\vareps_m,\alpha_2=\delta_n$).

Indeed, let $\Pi_i$ be a set of simple roots  containing $\alpha_i$.
Since $\alpha_1-\alpha_2\in\Delta$, $\Pi_1\not=\Pi_2$.
Since $\alpha_1,\alpha_2\in\Delta^+(\Pi_i)$ for $i=1,2$ we
have $\alpha_1-\alpha_2\in\Delta^+(\Pi_2), \alpha_2-\alpha_1\in\Delta^+(\Pi_1)$.
Recall that $\Pi_1$ can be obtained from $\Pi_2$ by a chain of odd reflections.
Since for an odd reflection $r_{\gamma}$ we have 
$$\Delta^+(r_{\gamma}\Pi)=(\Delta^+(\Pi)\setminus\{\gamma\})\cup\{-\gamma\},$$
$r_{\alpha_1-\alpha_2}$ is one of the reflections in this chain. Hence
$\alpha_1-\alpha_2$ is essentially simple.

\subsection{Finite parts}
\label{app3}
Let $X$ be an affine Dynkin diagram. We call its connected
subdiagram $\dot{X}$
its {\em finite part}, if $X\setminus\dot{X}$ contains exactly one root.
By~\Lem{lemdeltasum}, $\dot{X}$ is of finite type. We call a root subsystem $\dot{\Delta}$
a finite part of the affine root system $\Delta$ if $\dot{\Delta}$
has a subset of simple roots $\dot{\Pi}$ which is a finite part of a subset of simple 
roots for $\Delta$. 

In this section we will describe the finite parts of affine root systems.

Let $\Delta\not=F(4)^{(1)}, G(3)^{(1)}$ and  $\dot{\Delta}$ be its finite part.
We will show that each 
set of simple roots $\dot{\Pi}$ such that $\dot{\Delta}=\Delta(\dot{\Pi})$ is a finite part of 
some $\Pi$ (where $\Delta=\Delta(\Pi)$). Moreover, we will show that
each $\dot{\Pi}$ can be uniquely extended to $\Pi$:
if $\dot{\Delta}=\Delta(\dot{\Pi})$, then there exists 
 a unique subset of simple roots $\Pi$ for $\Delta$ containing $\dot{\Pi}$; moreover,
$\dot{\Pi}\cong \dot{\Pi}'$ forces $\Pi\cong\Pi'$: if there exists a $(-,-)$-preserving map
$\iota: \dot{\Delta}\to\dot{\Delta}'$
such that $\iota(\dot{\Pi})=\dot{\Pi'}$,  then $\iota$ can be extended to
a  $(-,-)$-preserving map $\tilde{\iota}:\Delta\to\Delta$ such that  $\iota({\Pi})={\Pi'}$.

Now let $X$ be a disjoint union of Dynkin diagrams of affine Lie algebras.
We call  $\dot{X}\subset X$ its finite part if  for each connected component
$X^j$ of $X$, $\dot{X}\cap X^j$  is a finite part of $X^j$. From the description of finite parts
given below it follows that for $\Delta\not=F(4)^{(1)}, G(3)^{(1)}$ the following holds:
if $\dot{\Delta}$ is a finite part of $\Delta$, then $\dot{\Delta}_{\ol{0}}$ 
is a finite part of $\Delta_{\ol{0}}$.

\subsubsection{}
The finite part of $X_l^{(1)}$ is $X_l$ if $X=A,C,D$.
The finite parts of other relevant to this paper 
affine Lie algebras are given by the following tables

\begin{tabular}{|l|l|l|l|l|}
\hline
$B_k^{(1)}$ & $A_{2k}^{(2)}$ & $A_{2k-1}^{(2)}$ & $D_{k+1}^{(2)}$ & $G_2^{(1)}$\\
\hline
$B_k, D_k$& $B_k, C_k$ & $C_k, D_k$ &  $B_k$ & $G_2$.\\
\hline
\end{tabular}

We claim that the finite parts of affine root systems for classical affine Lie superalgebras,
which are not Lie algebras, 
are given by the following tables:

\begin{tabular}{|l|l|l|l|l|}
\hline
$A(k,l)^{(1)} $& $B(0,l)^{(1)}$ & $B(k,l)^{(1)}, k>1 $& $C(k)^{(1)} $ &  $D(k,l)^{(1)}$ \\
\hline
$A(k,l)$  & $B(0,l), C(l)$& $B(k,l)$,$D(k,l)$ & $C(k)$ & $D(k,l)$ \\
\hline
\end{tabular}

\begin{tabular}{|l|l|l|l|l|}
\hline
$A(2k,2l-1)^{(2)}$ &   $A(2k-1,2l-1)^{(2)}$  & $A(2k,2l)^{(4)}$ &  $D(k+1,l)^{(2)}$ & $C(l+1)^{(2)}$\\
\hline
$B(k,l), D(l,k)$ & $D(k,l), D(l,k)$ & $B(k,l), B(l,k)$ & $B(k,l)$ &  $B(0,l)$ \\\hline
\end{tabular}

where we take $D(1,n):=C(n+1), D(1,1):=A(1,0)$.

The finite parts of $G(3)^{(1)}$ are $G(3), D(2,1,-\frac{3}{4}), A(2,0)$; the finite parts of
$F(4)^{(1)}$ are $F(4), A(3,0)$; the finite part of $D(2,1,a)^{(1)}$ is
$D(2,1,a)$.

\subsubsection{}
Let $\dot{\Delta}$ be a finite part of $\Delta$.
By definition, $\dot{\Delta}$ has a subset of simple roots $\dot{\Pi}$ 
which is a finite part
of a subset of simple roots for $\Delta$, which we denote by $\Pi$.
Since any other subset of simple roots for $\dot{\Delta}$ 
can be obtained via a chain of odd reflections,
any subset of simple roots $\dot{\Pi}'$ of $\dot{\Delta}$ 
is a finite part of some $\Pi'$. Let us show that $\Pi'$ is unique and that 
$\dot{\Pi}\cong\dot{\Pi'}$ implies $\Pi\cong \Pi'$: i.e., if there exists a $(-,-)$-preserving map
$\iota: \dot{\Delta}\to\dot{\Delta}'$
such that $\iota(\dot{\Pi})=\dot{\Pi'}$,  then $\iota$ can be extended to
a  $(-,-)$-preserving map $\tilde{\iota}:\Delta\to\Delta$ such that  $\iota({\Pi})={\Pi'}$.

For $\Delta=A(m,n)^{(1)}$ all Dynkin diagrams are cycles and the sum of all simple
roots is $\delta$; hence $\dot{\Pi}$ determines $\Pi$.

Since any set of simple roots for $\dot{\Delta}$ are connected by a chain of odd reflections,
it is enough to verify the assertion for one choice of $\dot{\Pi}$.
Let $\dot{\Pi}$ be a finite part of $\Pi$.
Denote the unique root in $\Pi\setminus\dot{\Pi}$
by $\alpha_0$.

Consider the case $\Delta\not=A(m,n)^{(1)},F(4)^{(1)}, G(3)^{(1)}$.
For each $\Pi$ any proper connected subdiagram 
of $\Pi$ is of finite type which is not $F(4), G(3)$.
Hence any proper connected subdiagram has at most one
branching node $\beta$, this node has three branches, 
two of these branches have length one (excluding the branching point)
and consist of the nodes $\gamma_1$, $\gamma_2$ respectively with
$||\gamma_1||^2=||\gamma_2||^2, (\gamma_1,\beta)=(\gamma_2,\beta)$,
which are connected if and only if $||\gamma_1||^2=0$.

It is not hard to show that $\dot{\Delta}\not=A(k,l)$ for $(k,l)\not=(1,1)$.
Therefore $\dot{\Delta}$ is $B(k,l)$ or $D(k,l)$.
Take $\dot{\Pi}=\{\vareps_1-\vareps_2,\ldots, \vareps_k-\delta_1,\ldots, a\delta_l\}$
($a=1$ for $B(k,l)$, $a=2$ for $D(k,l)$) and write $\alpha_1:=\vareps_1-\vareps_2,
\ldots, \alpha_{k+l}:=a\delta_l$.
 If in $\Pi$ the node $\alpha_0$ is connected to a node which is not 
$\alpha_1$,  then, by above, $\Pi\setminus\{\alpha_1\}$ is of type $A,B,C,D$ and
contains $\alpha_{k+l}=a\delta_l$; then 
$(\alpha_0,\alpha_i)=0$ for $i>2$ and $\alpha_2$ is a branching node;
since $\Pi\setminus\{\alpha_l\}$ is of finite type, 
$||\alpha_0||^2=||\alpha_1||^2, (\alpha_0,\alpha_2)=(\alpha_1,\alpha_2)$ and
$(\alpha_1,\alpha_2)=0$; hence $\Pi$ is uniquely defined (and it 
is of the type $B(k,l)^{(1)}$, $D(k,l)^{(1)}$ respectively).
Consider the remaining case when  $\alpha_0$ is connected only  to $\alpha_1$;
then the subdiagram $\alpha_0-\alpha_1$ can be one of the following:
$$\otimes-\bigcirc; \ \ \  \bigcirc-\bigcirc;\ \ \ 
\bigcirc\Longrightarrow\bigcirc; \ \ \ \
\bigcirc\Longleftarrow\bigcirc;\ \ \ \ \bullet-\bigcirc.$$
For $\dot{\Pi}=B(k,l)$ the above subdiagrams 
correspond to $\Pi$
of the types $B(k,l+1), B(k+1,l), A(2k,2l-1)^{(2)}, D(k+1,l)^{(2)}$ and $A(2k,2l)^{(4)}$
respectively. For $\dot{\Pi}=D(k,l)$ the above subdiagrams 
corresponds to $\Pi$
of the types $D(k,l+1), D(k+1,l), A(2k-1,2l-1)^{(2)}, B(k,l)^{(1)}$ and $A(2l,2k-1)^{(2)}$
respectively. We conclude that  in each type
of $\Delta$ the set of simple roots $\Pi$ containing $\dot{\Pi}$ is uniquely defined
(up to isomorphism).

\subsubsection{Root systems $F(4), F(4)^{(1)}$}
Recall that the non-isotropic roots  of $F(4)$ are $B_3\times A_1$.

The root system $F(4)$ has $6$ sets of simple roots, which are pairwise non-isomorphic;
the root system $F(4)^{(1)}$ has $7$ sets of simple roots, among them $4$ non-isomorphic.

The finite parts of $F(4)^{(1)}$ are $F(4)$ and $A(3,0)$.
Each $\Pi$ for $F(4)^{(1)}$ contains $\dot{\Pi}$ of type $F(4)$ and 
each $\dot{\Pi}$ of type $F(4)$
can be uniquely, up to isomorphism, extended to some $\Pi$ for $F(4)^{(1)}$.
Only $4$ (out of $7$)  sets of simple roots for $F(4)^{(1)}$
 contain $\dot{\Pi}$ of the type $A(3,0)$; three of these sets of simple roots are non-isomorphic
and each $\dot{\Pi}$ of the type $A(3,0)$ can be uniquely, up to isomorphism, extended to some $\Pi$. 

There are $4$ subsets $\dot{\Pi}_1,\ldots,\dot{\Pi}_4$ of simple roots for $F(4)$ satisfying
$||\alpha||^2\geq 0$ for each $\alpha\in\dot{\Pi}$;  there
are three subsets $\Pi_1,\Pi_2,\Pi_3$ of simple roots for $F(4)^{(1)}$ satisfying
$||\alpha||^2\geq 0$ for each $\alpha\in\Pi$ ($\dot{\Pi}_i\subset\Pi_i$);
one has $\Pi_1\cong \Pi_3$. If we consider the affine root $\alpha_0(i):=\Pi_i\setminus\dot{\Pi}_i$
($i=1,2,3$), then $\alpha_0(1),\alpha_0(2)$ are long roots in $B_3$ and $\alpha_0(3)$
is isotropic.

Each of the sets  $\Pi_1,\Pi_2,\Pi_3$  contains a finite part $A(3,0)$:
if we denote by $\alpha_0(i)'$ the corresponding affine root, then 
this root is isotropic for $\Pi_1,\Pi_3$ and is a short root for $\Pi_2$.

\subsubsection{Root systems $G(3), G(3)^{(1)}$}\label{Groots}
The non-isotropic roots of  $G(3)$ are $G_2\times B(0,1)$. We write $G_2$ 
in terms $\vareps_1,\vareps_2,\vareps_3$ subject to the relations
 $\vareps_1+\vareps_2+\vareps_3=0,\  ||\vareps_1||^2=||\vareps_2||^2=||\vareps_3||^2$;
 we choose the set of simple roots $\vareps_2-\vareps_3,\vareps_3$.
 For $B(0,1)$ we take the simple root $\delta_1$ with
 $||\delta_1||^2=-3||\vareps_i||^2$. Then the 
 isotropic roots are $\pm\delta_1\pm\vareps_i, i=1,2,3$.
 
 The sets of simple roots for $G(3)$ are the following:
 
 $$\begin{array}{l}
 \delta_1+\vareps_1,\vareps_2-\vareps_3,\vareps_3;\\
 -\delta_1-\vareps_1,\vareps_2-\vareps_3,\delta_1-\vareps_2;\\
 \vareps_3, \delta_1-\vareps_3 ,-\delta_1+\vareps_2;\\
  \delta_1,-\delta_1+\vareps_3, \vareps_2-\vareps_3.
 \end{array}$$
 
 Recall that we call an isotropic root essentially simple if it lies in some set of simple roots.
 The set of essentially simple roots for $G(3)$ is $\{\pm(\delta_1+\vareps_1);
 \pm(\delta_1-\vareps_2); \pm(\delta_1-\vareps_3)\}$.

The non-isotropic roots of  $G(3)^{(1)}$ are $G_2^{(1)}\times B(0,1)^{(1)}$.
Using the above notations, we write the set of simple roots for $G_3^{(1)}$ (resp., for $B(0,1)^{(1)}$)
as $\delta+\vareps_1-\vareps_2,\vareps_2-\vareps_3,\vareps_3$ (resp., $\delta-2\delta_1,\delta_1$).
Here is the  sets of simple roots for  $G(3)^{(1)}$ and their finite parts:
 
 $$\begin{array}{ll}
 \delta-2\delta_1, \delta_1+\vareps_1,\vareps_2-\vareps_3,\vareps_3;& G(3),D(2,1,-\frac{3}{4})\\
\delta -\delta_1+\vareps_1, -\delta_1-\vareps_1,\vareps_2-\vareps_3,\delta_1-\vareps_2;
&G(3),D(2,1,-\frac{3}{4}), A(2,0)\\
 \delta+\vareps_1-\vareps_2,\vareps_3, \delta_1-\vareps_3 ,-\delta_1+\vareps_2;
 &G(3),D(2,1,-\frac{3}{4}), A(2,0) \\
  \delta+\vareps_1-\vareps_2,   \delta_1,-\delta_1+\vareps_3, \vareps_2-\vareps_3;
 &G(3), A(2,0) \\
  -\delta +\delta_1-\vareps_1,\delta-2\delta_1,\vareps_2-\vareps_3,\delta+\vareps_1-\vareps_2;
  &D(2,1,-\frac{3}{4}), A(2,0).
 \end{array}$$
 
 The set of essentially simple roots is  the union
 of the corresponding set for $G(3)$ with $\{\pm (\delta -\delta_1+\vareps_1)\}$.
 
Note that $G(3)$ and $D(2,1,-\frac{3}{4})$ have $4$ sets of simple roots; each set occurs
exactly once as a finite part of a set of simple roots for  $G(3)^{(1)}$
(in other words, $\dot{\Pi}$ can be uniquely extended to $\Pi$).

\subsection{}
In this subsection we will prove the following statement.

\begin{lem}{lemAmnxi}
Let $\Delta=A(m,n)^{(1)}, m\not=n$. 
If $\nu\in \mathbb{Z}_{\geq 0}(\Pi)$ is such that $(\nu, \Delta_{\ol{0}})=0$
and $2(\rho,\nu)=(\nu,\nu)$, then $\nu=0$.
\end{lem}

\subsubsection{}\label{sequences}
We start with the following problem.
Let $X=(x_1,\ldots,x_{m+n})$ be a sequence of $m+n$ numbers, 
where $m$ numbers are equal to $1$ 
and $n$ numbers to $-1$. Let $f(X)$ be the "total
number of disorders":
$$f(X):=\sum_{i<j}\frac{1}{2}(x_i-x_j).$$
Clearly, $|f(X)|\leq mn$.
Let $X_1,\ldots, X_{m+n}$ be the sequences obtained from $X$ by cyclic permutations.
We claim that there exist $i,j$ such that $0\leq f(X_i)<2m$ and $-2n<f(X_j)\leq 0$.

Indeed, let $\sigma(X)$ be the sequence obtained from $X$ by moving last element to the first place: $\sigma(x)_i:=x_{i-1}$ if $1<i\leq m+n$, $\sigma(x)_1:=x_{m+n}$.
Set $f_k(X):=(f(\sigma^k(X))$. Since $\sigma^{m+n}=Id$, $f_k$ has period $m+n$.

We claim that $\sum_{k=0}^{m+n-1} f_k(X)=0$ for each $X$. 
Indeed,  let $X$ be any sequence and $s\in\{1,2,\ldots,m+n-1\}$ 
be such that  $x_s=1$ and $x_{s+1}=-1$; let $X'$ be the sequence obtained from $X$
by switching $x_s$ and $x_{s+1}$ ($x'_i=x_i$ if $i\not=s$,
$x'_s=-1, x'_{s+1}=1$). Then $f(X')=f(X)-2$. Note that $\sigma^j(X')$ is obtained from 
$\sigma^j(X)$ by the same operation (for different index $s$) if $j+s\not\equiv 0$ modulo $m+n$;  if $j+s\not\equiv 0$ modulo
$m+n$, then for $Y:=\sigma^j(X)$ we have $y_1=-1, y_{m+n}=1$
and $Y':=\sigma^j(X')$ is obtained from $Y$ by switching $y_1$ and $y_{m+n}$.
Clearly, $f(Y')=f(Y)+2(m+n-1)$; by above, if $j+s\not\equiv 0$ modulo
$m+n$, then $f(\sigma^j(X'))=f(\sigma^j(X'))-2$. Hence  
$\sum_{k=0}^{m+n-1} f_k(X)=\sum_{k=0}^{m+n-1} f_k(X')$. Since any sequence can be obtained from
the sequence $X_0=(1,\ldots,1,-1,\ldots,-1)$ by a chain of above operations,
we obtain $\sum_{k=0}^{m+n-1} f_k(X)=\sum_{k=0}^{m+n-1} f_k(X_0)$. One has
readily sees that $\sum_{k=0}^{m+n-1} f_k(X_0)=0$, as required.

Note that $f(\sigma(X))=f(X)+2n$ if $x_{m+n}=1$ and  $f(\sigma(X))=f(X)-2m$ if $x_{m+n}=-1$,
so $f_{i+1}-f_i$ is $2n$ or $-2m$. Since $\sum_{k=0}^{m+n-1} f_k(X)=0$, $f_k(X)$ contains
positive and negative elements. If $i$ is such that $f_{i+1}(X)<0\leq f_i(X)$,
then $f_i(X)<2m$, and if $j$ is such that $f_{j+1}(X)>0\geq f_j(X)$, then $f_j(X)>-2n$.
The claim follows.

\subsubsection{}
Recall that a set of simple roots  for $A(m,n)$ can be naturally
encoded  as a  sequence of $m$ dots and $n$ crosses, see~\S~\ref{dotcross}):  
for instance, $\Pi=\{\vareps_1-\vareps_2,\vareps_2-\delta_1,\delta_1-\vareps_3\}$ 
is encoded by the sequence $\cdot\ \cdot\ \times\ \cdot\ $;
putting $1$ instead of dots and 
$-1$ instead  of crosses, we obtain a sequence considered in~\S~\ref{sequences}.
Similarly, a set of simple roots for $A(m,n)^{(1)}$ can be encoded
by the same sequence viewed as a cycle. The inverse procedure can be described as follows:
to a sequence $X$ as in~\S~\ref{sequences} we assign the Dynkin diagram of $A(m,n)$-type with 
$||\alpha_i||^2=x_{i+1}+x_i$ for $i=1,\ldots, m+n-1$
and the Dynkin diagram of $A(m,n)^{(1)}$-type
with $||\alpha_i||^2=x_{i+1}+x_i$ for $i=1,\ldots, m+n-1$ and $||\alpha_0||^2=x_1+x_{m+n}$.
Clearly, $X$ and $\sigma(X)$ give the same Dynkin diagram of $A(m,n)^{(1)}$-type.

\subsubsection{}
Let $\Pi$ be a set of simple roots for $A(m,n)^{(1)}$
and $\tilde{X}$ be the corresponding cyclic sequence (which we view as a set 
containing $m+n$ usual sequences).
Take a sequence $X\in\tilde{X}$ such that $-2n<f(X)\leq 0$.
Let $\dot{\Pi}$ be the corresponding Dynkin diagram of $A(m,n)$-type.
Since $m\not=n$, we may (and will) assume
that $m>n$ and use the standard notations for $A(m,n)$.
We set
$$E:=\sum_{i=1}^m\vareps_i, \  D:=\sum_{i=1}^n\delta_i.$$
We  take the standard form  $(-,-)$ (i.e., $||\vareps_i||^2=-||\delta_j||^2=1$).
Then

$$2(\rho,E)=-(\sum_{\alpha\in\dot{\Delta}_{\ol{1}}^+}\alpha,E)=-f(X).$$

The condition $(\nu, \Delta_{\ol{0}})=0$ is equivalent to
$\nu=j\delta+u(nE-mD)$;
since $\nu\in \mathbb{Z}_{\geq 0}(\Pi)$, we have $j\in\mathbb{Z}_{\geq 0}$,
$u\in\mathbb{Z}\frac{1}{d}$, where $d:=GCD(m,n)$. One has
 $(\rho,\delta)=m-n$. 
Since  $(\Delta,E-D)=0$,  we have $(\rho,\nu)=j(m-n)-(m-n)u(\rho,E)$,
so  $2(\rho,\nu)=(\nu,\nu)$ gives
$$2j(m-n)+(m-n)uf(X)=-u^2mn(m-n),$$
that is $u^2mn+uf(X)+2j=0$. Writing $u=s/d$ with $s\in\mathbb{Z}$, 
we obtain
$$s^2mn+2jd^2=-df(X)s.$$
Since  $0\leq -f(X)<2n$, we get $s=0$ or $s>0$ and $s^2mn+2jd^2<2nds$.
One has $2d\leq m$ (because $n<m$), so the only solution is $s=j=0$.
Hence $\nu=0$, as required.

\subsection{}
Recall that an odd isotropic root $\beta$  is essentially simple if 
it belongs to a set of simple roots.

Let $\pi$ be a connected component of $\Pi_0$. For $\nu\in\fh^*$
write $\nu\succ_{\pi} r_{\alpha}\nu$ if $\alpha\in\pi$
and $\nu-r_{\alpha}\nu\in\mathbb{Z}_{>0}\alpha$, and
consider the order $\succ_{\pi}$ on $\fh^*$ generated by this
($\nu\succ_{\pi}\mu$ if $\mu=r_{\alpha_1}r_{\alpha_2}\ldots r_{\alpha_s}\nu$
with $(r_{\alpha_i}r_{\alpha_i+1}\ldots r_{\alpha_s}\nu,\alpha_{i-1}^{\vee})\in\mathbb{Z}$);
write $\nu\succeq_{\pi}\mu$ if $\nu\succ_{\pi}\mu$ or $\nu=\mu$.

\begin{lem}{lemapp1}
Let $\Delta$ be a finite or affine root system 
and $\Iso$ be the set of odd isotropic roots.
Let $X\subset \Iso$ be the set of odd isotropic roots $\beta$ with the following
property: for each connected component $\pi$ of $\Pi_0$
there exists an essentially simple root $\beta'$ such that
$\beta'\succeq_{\pi}\beta$. Then $\Iso\subset (X\cup (-X))$.
\end{lem}
\begin{proof}
Let $\pi$ be one of the root systems $B_m,C_m,D_m$ with the 
standard notations. 
Clearly, $\vareps_i\succeq_{\pi}\vareps_m\succeq_{\pi}-\vareps_{m-1}$
and $\pm\vareps_i\succeq_{\pi}-\vareps_1$ for each $i=1,\ldots,m$.
In particular, $\vareps_m$ is  the minimal element in $\{\vareps_i\}_{i=1}^m$ 
with respect
to the order $\succeq_{\pi}$.

Now let $\pi$ be one of the root systems $B_m^{(1)}, C_m^{(1)}, D_m^{(1)}, A_{2m-1}^{(2)}$
with the standard notations, or $\pi=A_{2m}^{2}=
 \{\delta-2\vareps_1,\vareps_1-\vareps_2,\ldots,\vareps_m\}$
and $Y:=\{\vareps_i, s\delta\pm\vareps_i: s\in\mathbb{Z}_{>0}\}_{i=1}^m$.
We claim that $\vareps_m$ is again the minimal element in $Y$ with respect
to the order $\succeq_{\pi}$. Indeed,
the finite part of $\pi$ is $B_m,C_m$ or $D_m$, so
 $s\delta\pm \vareps_i\succeq_{\pi} s\delta-\vareps_1$.
For $B_m^{(1)}, D_m^{(1)}, A_{2m-1}^{(2)}$ the affine root is $\delta-\vareps_1-\vareps_2$, so 
$s\delta-\vareps_1\succeq (s-1)\delta+\vareps_2$;
for $C_m^{(1)}, A_{2m}^{(2)}$ the affine root is $\delta-2\vareps_1$, so
$s\delta-\vareps_1\succeq (s-1)\delta+\vareps_1$. Hence, for $s>0$ one has
$s\delta\pm\vareps_i\succeq_{\pi}\vareps_m$, as required.

Now consider $\pi=D_{m}^{(2)}=
 \{\delta-\vareps_1,\vareps_1-\vareps_2,\ldots,\vareps_m\}$ with
 $Y:=\{\vareps_i, 2s\delta\pm\vareps_i, s\in\mathbb{Z}_{>0}\}_{i=1}^m$.
Since $2s\delta-\vareps_1\succeq 2(s-1)\delta+\vareps_1$, 
$\vareps_m$ is again the minimal element in $Y$ with respect
to the order $\succeq_{\pi}$.

Recall that for $\Delta= A(m,n), C(n), A(m,n)^{(1)}, C(n)^{(1)}$ any odd root
is essentially simple, so $X=\Delta_{\ol{1}}$.
Take $\Delta\not= A(m,n), C(n), A(m,n)^{(1)}, C(n)^{(1)}$.

If $\Delta$  is $B(m,n)$ or $D(m,n)$, then
$\Iso=\{\pm \vareps_i\pm\delta_j\}$. The roots
$\pm (\vareps_i-\delta_j)$ are essentially simple, so $X$ contains 
the roots $\pm (\vareps_i-\delta_j),\vareps_i+\delta_j$. Hence $\Iso\subset (X\cup (-X))$.

For $D(2,1,a)$, one has
 $\Iso=\{\pm\vareps_1\pm\vareps_2\pm\vareps_3\}$
 and all roots are essentially simple except for
 $\pm(\vareps_1+\vareps_2+\vareps_3)$. Thus $\vareps_1+\vareps_2+\vareps_3\in X$, so
 $\Iso\subset (X\cup (-X))$.

For $F(4)$ recall that $\Pi_0=A_1\times B_3$ and choose $\Pi_0=\{\delta_1;\vareps_1-\vareps_2,
\vareps_2-\vareps_3,\vareps_3\}$. In this case 
$\Iso=\{\pm\frac{1}{2}(\delta_1\pm\vareps_1\pm\vareps_2\pm\vareps_3)\}$. Take
 $\beta=\frac{1}{2}(\delta_1\pm\vareps_1\pm\vareps_2\pm\vareps_3)$; if 
at least two signs $\pm$ are $-$, then $\beta$ is essentially simple, so $\beta\in X$.
If $\beta$ is not essentially simple, 
then $-(\beta-\delta_1)$ is essentially simple, so $\beta-\delta_1$ is essentially simple,
 For $\pi=B_3$ we have
$$\vareps_1+\vareps_2+\vareps_3\succeq_{\pi} \vareps_1+\vareps_2-\vareps_3\succeq_{\pi} 
\vareps_1-\vareps_2+\vareps_3\succeq_{\pi} -\vareps_1+\vareps_2+\vareps_3\succeq_{\pi} 
-\vareps_1+\vareps_2-\vareps_3,$$
so $\beta \succeq_{\pi} \frac{1}{2}(\delta_1-\vareps_1+\vareps_2-\vareps_3)$;
the last root is essentially simple. Hence $\beta\in X$. 
Thus $X$ contains the roots of the form
$\frac{1}{2}(\delta_1\pm\vareps_1\pm\vareps_2\pm\vareps_3)$, so $\Iso\subset (X\cup (-X))$.

For $G(3)$ one has $\Iso=\{\pm \delta_1\pm\vareps_i\}_{i=1}^3$ and
and the essentially simple roots are $\pm(\delta_1+\vareps_1), \pm(\delta_1-\vareps_i)$, $i=2,3$.
It is easy to see that $X=\Iso\setminus\{-\delta_1-\vareps_i\}_{i=1}^3$.

In the remaining cases $\Delta$ is  affine.
Let $\dot{\Delta}$ be a finite part of $\Delta$. Clearly,
$\dot{\Iso}=Iso\cap\dot{\Delta}$ is the set of isotropic odd roots
in $\dot{\Delta}$; let $\dot{X}\subset \dot{\Iso}$ be the corresponding set
for $\dot{\Delta}$. Recall that any essentially simple root for $\dot{\Delta}$
is essentially simple for $\Delta$, so $\dot{X}\subset X$.

Note that $\Iso=\dot{\Iso}+\mathbb{Z}\delta$, 
except for $C(m)^{(2)}, D(m,n)^{(2)}$ with
$\Iso=\dot{\Iso}+\mathbb{Z}\delta$.
Let us show that $X$ contains $\dot{\Iso}+j\delta$ for $j\in\mathbb{Z}_{>0}$, where
$j$ is even for $C(m)^{(2)}, D(m,n)^{(2)}$.

First, consider the case when $\dot{\Delta}\not=F(4), G(3)$.
Then $\dot{\Delta}$ is $B(m,n)$ or $D(m,n)$ and
$\pi$ is one of the root systems $A_1^{(1)},
B_m^{(1)}, C_m^{(1)}, D_m^{(1)}, A_{2m-1}^{(2)},
A_{2m}^{2}, D_m^{(2)}$. If $\pi$ lies in the span
of $\vareps_i$s, then, by above, for $s>0$ (and $s$ even for $C(m)^{(2)}, D(m,n)^{(2)}$)
one has
$s\delta\pm\vareps_i\succeq_{\pi} \vareps_m$, so $s\delta\pm\vareps_i\pm\delta_j
\succeq_{\pi} \vareps_m\pm\delta_j$; $\vareps_m-\delta_j$
is essentially simple and $\vareps_m+\delta_j
\succeq_{\pi} -\vareps_1\pm\delta_j$, where $-\vareps_1\pm\delta_j$ is also
essentially simple. Hence $X$ contains the roots $s\delta\pm\vareps_i\pm\delta_j$ for $s>0$ (and $s$ even for $C(m)^{(2)}, D(m,n)^{(2)}$), as required.
 The similar reasoning shows that $X$ contains $\dot{\Iso}+\mathbb{Z}_{>0}\delta$
for $F(4)^{(1)}$ and $G(3)^{(1)}$.

Combining $\Iso=\dot{\Iso}+\mathbb{Z}\delta$ (resp., $\Iso=\dot{\Iso}+\mathbb{Z}\delta$
for $C(m)^{(2)}, D(m,n)^{(2)}$), the fact that
$X$ contains $\dot{X}$ and $\dot{\Iso}+\mathbb{Z}_{>0}\delta$
 (resp., $\dot{\Iso}+2\mathbb{Z}_{>0}\delta$
for $C(m)^{(2)}, D(m,n)^{(2)}$), and  the inclusion $\dot{Iso}\subset (\dot{X}\cup
(-\dot{X}))$,  we obtain $\Iso\subset (X\cup (-X))$ as required.
\end{proof}

\subsection{}
\label{appendix2}
Let $\Pi\not=A(m,n)^{(1)}, C(n)^{(1)}$ be a set of simple roots of affine type, $\dot{\Pi}$
be its finite part and $\pi\subset\Pi_0$ be as in~\S~\ref{defintegr}:
$\pi:=\{\alpha\in\Pi_0|\ (\alpha,\alpha)\in\mathbb{Q}_{>0}\}$ if the form $(-,-)$ is such that
$(\rho,\delta)\in\mathbb{Q}_{\geq 0}$.
Recall the conditions (A)--(C) from~\Thm{thmadm3}:

(A)  $||\alpha_0||^2\geq 0$;

(B) for each $\alpha\in\pi$ there exists $\beta\in\supp(\alpha)$
such that $\beta\not\in\supp(\alpha')$ for each $\alpha'\in\pi$; 
this $\beta$ is denoted by $b(\alpha)$;

(C) $\rho\in X_1-X_2$, where
$$X_1:=\{\mu\in\fh^*|
(\mu,\alpha)\in \mathbb{Q}_{\geq 0} \ \text{ for all }\alpha\in \Pi \},\ \ 
X_2:=\sum_{\alpha\in\dot{\Pi}_0} \mathbb{Q}_{\geq 0}\alpha^{\vee}.$$

Let us give some examples when these conditions hold. 

If $\Delta\not=F(4)$ is finite and $\pi$ is a connected component of $\Pi_0$, then
$\supp(\alpha)\cap\supp(\alpha')=\emptyset$ for each $\alpha,\alpha'\in\pi$, except for the pair
$\alpha,\alpha'=\vareps_{m-1}\pm\vareps_m$, for $\pi=D_m$.
Since for affine $\Delta$, $\dot{\Delta}$ is finite, this implies that
(B) holds for all affine sets of simple roots $\Pi$ which are not of
type $F(4)^{(1)}$.

\subsubsection{Case $B(m,n)^{(1)}, D(m,n)^{(1)}$}
\label{BC}
By above, the condition (B) always hold.
Condition (C) holds if $||\alpha||^2\geq 0$ for each $\alpha\in\Pi(L)$
(since $\rho\in X_1$). Let us describe other cases when (C) holds.

We write $\Pi=\{\alpha_0,\alpha_1,\ldots,\alpha_{m+n}\}$
with a standard enumeration: this means $(\alpha_i,\alpha_{i+1})\not=0$ for each $i$,
 except for the case $D(m,n)^{(1)}$,
where sometimes $(\alpha_{m+n-1},\alpha_{m+n})=0$, but 
$(\alpha_{m+n-2},\alpha_{m+n})\not=0$ (i.e., $\alpha_{m+n-2}$ 
is a branching point of the Dynkin diagram).
Note that this convention determines the enumeration, except for
$D(m,n)^{(1)}$ with the branching point $\alpha_{m+n-2}$, where
we can interchange $\alpha_{m+n-1}$ and $\alpha_{m+n}$;
note that in this case $||\alpha_{m+n-1}||^2=||\alpha_{m+n}||^2$.

Introduce the numbers $d_{m+n},\ldots,d_1$ by the following rule:
$$d_{m+n}:=\max(-||\alpha_{m+n}||^2,0),\ \ 
d_i:=\max(d_{i+1}-||\alpha_{i}||^2,0)\ \text{ for }i=1,\ldots,m+n-1,$$
if $\alpha_{m-n-2}$ is not a branching point and
$$d_{m+n}=d_{m+n-1}:=\max(-||\alpha_{m+n}||^2,0),\ \ 
d_i:=\max(d_{i+1}-||\alpha_{i}||^2,0)\ \text{ for }i=1,\ldots,m+n-2,$$
if $\alpha_{m-n-2}$ is a branching point (in this case 
$||\alpha_{m+n}||^2=||\alpha_{m+n-1}||^2$).

The property (C) holds if and only if the sum of $d_i$ with 
$(\alpha_0,\alpha_i)\not=0$
is not greater than $2$. If $m>n$, then $\alpha_1$ is  a branching point  
(i.e., $(\alpha_0,\alpha_2)\not=0$) and
(C) is equivalent to $d_1+d_2\leq 2$; if $m\leq n$, (C) is equivalent to  
$d_1\leq 2$.
For instance,  for $B(m,n)^{(1)}$ with $m\leq n$, if
$\Pi$  has $j$ roots of negative square length and there are $j$ roots 
of positive square length which precede  (counting from $\alpha_0$)
the roots of negative square length, then $\Pi$ satisfies (C).

\subsubsection{Conditions (B) and (C) for exceptional Lie superalgebras}\label{FG}
For $D(2,1,a)^{(1)}$ we have two sets of simple roots satisfying (A): 
one consists of isotropic roots
and another one with $||\alpha_0||^2>0$ 
(for $a\in\mathbb{Q},\ 0<a<1$, it takes the form $\Pi=\{\delta-2\vareps_1,
\vareps_1-\vareps_2-\vareps_3,2\vareps_2,2\vareps_3\}$). 
Both of them satisfy  conditions (B) and (C).

For  $G(3)^{(1)}$ there are three sets of simple roots, 
which satisfy (A), all of them satisfy (B) and (C);
these are the second, the third and the forth sets in~\S~\ref{Groots}.

For $F(4)^{(1)}$ there are two sets of simple roots which satisfy (A)-(C);
they correspond to the third and the forth
sets of simple roots in~\cite{K1}, 2.5.4.

\subsubsection{}
For $D(n+1,n)^{(1)}, A(2n-1,2n-1)^{(2)}$  for $\Pi$ as in~\S~\ref{A2n-12n-1}
we have $||\alpha_0||^2=0$, 
(so (A) holds) and $\rho=0$ (so (C) holds); it is easy to verify that (B)
 holds for both choices of $\pi$, so (A)-(C) hold for this $\Pi$
  (for both choices of $\pi$).

For $D(n+1,n)^{(2)}, A(2n,2n)^{(4)}$ with fixed $\pi$, we choose a presentation
of $\Delta$ via  $\vareps_i,\delta_j$, where $\pi$ lies in the span of $\vareps_i$s
(as it was done in~\S~\ref{A2n4}). The set of simple root 
$\{\delta-\delta_1,\delta_1-\vareps_1,\ldots, \vareps_n\}$
chosen in~\S~\ref{A2n4} does not satisfy (A). However, 
$\Pi=\{\delta-\vareps_1,\vareps_1-\delta_1,\ldots, \vareps_n-\delta_n,\delta_n\}$
 satisfies (A)-(C). Indeed, we can fix $(-,-)$ with $||\vareps_i||^2=1$. Then
 $||\delta-\vareps_1||^2=1$, so (A) holds.
 One has $\pi=\{a'(\delta-\vareps_1),\vareps_1-\vareps_2,\ldots,a\vareps_n\}$
 with $a,a'\in\{1,2\}$; thus $\supp(\alpha)\supp(\alpha')=\emptyset$
 for $\alpha\not=\alpha'\in\pi$ and (B) holds. 
 Note that $\vareps_i$ or $2\vareps_i$ (resp., $\delta_i$ or $2\delta_i$) lies in
 $\dot{\Delta}^+_{\ol{0}}$, so $\vareps_i,-\delta_i\in X_2$.
 One has $2\rho=\sum_{i=1}^n (\delta_i-\vareps_i)$, so $\rho\in- X_2$;
 hence (C) holds.



\begin{thebibliography}{MMM}

\bibitem[BL]{BL}
I.N.~Bernstein, D.A.~Leites, 
{\em A formula for the characters of the irreducible finite-dimensional representations
of Lie superalgebras of series $\fgl$ and $\fsl$}, C. R. Acad. Bulgare Sci. {\bf 33} 
(1980), 1049--1051.


\bibitem[B]{B} J.~Brundan, {\em Kazhdan-Lusztig polynomials and 
character formulas for the Lie superalgebra $\fgl(m|n)$}, JAMS {\bf 16}, (2003), No. 1,
185--231.

\bibitem[CMW]{CMW} S.-J.~Cheng, V.~Mazorchuk, W.~Wang, 
{\em Equivalence of blocks for the general linear Lie superalgebra}, 
arXiv:1301.1204.

\bibitem[CHR]{CHR} M.~Chmutov, C.~Hoyt, S.~Reif, {\em Kac-Wakimoto character
formula for general linear Lie superalgebra}, arXiv: 1310.3798.




\bibitem[F]{F} P.~Fiebig, {\em The combinatorics of category $\mathcal{O}$ over symmetrizable Kac-Moody algebras}, Transformation Groups 
{\bf 11} (2006), no.1, 29--49.

\bibitem[G1]{Gfin} M.~Gorelik, {\em Weyl denominator identity for finite-dimensional Lie superalgebras}, Highlights in Lie algebraic methods, 167--188, Progr. Math. {\bf 295},
Birkha\"user/Springer, New York, 2012.

\bibitem[G2]{Gaff} M.~Gorelik, {\em Weyl denominator identity for affine Lie superalgebras with non-zero dual Coxeter number}, J. Algebra {\bf 337} (2011), 50--62.

\bibitem[GK]{GK} M.~Gorelik, V.~G.~Kac, 
{\em On simplicity of vacuum modules}, Adv. Math. 
{\bf 211} (2007), no.2, 621--677.

\bibitem[GKMP]{GKMP} M.~Gorelik, V.~G.~Kac,  P.~M\"o�seneder Frajria, P.~ Papi, 
{\em Denominator identities for finite-dimensional Lie superalgebras and Howe duality for compact dual pairs}, Japan. J. Math. {\bf 7}, (2012), no.1, 41--134.

\bibitem[GR]{GR} M.~Gorelik, S.~Reif, {\em Weyl denominator identity for affine Lie superalgebras with zero dual Coxeter number}, Algebra \& Number Theory {\bf 6} (2012), no. 5.
1043--1059.


\bibitem[IK]{IK}  K.~Iohara, Y.~Koga, {\em Enright functors for Kac-Moody superalgebras}, 
Abh. Math. Semin. Univ. Hambg. {\bf 82}, (2012), no. 2, 205--226.



\bibitem[J1]{VJ1} J.~Van der Jeugt, {\em Irreducible representations of the exceptional Lie superalgebras $D(2,1:\alpha)$}, J. Math. Phys. {\bf 26} (1985), 913--924.

\bibitem[J2]{VJ2}
J.~Van der Jeugt, {\em Character formulas for the Lie superalgebra $C(n)$},
Comm. Algebra {\bf 19} (1991), no.~1, 199--222.


\bibitem[JHKT]{VJHKT} J.~Van der Jeugt, J.~W.~B.~Hughes, R.~C.~King, J.~Thierry-Mieg, {\em 
Character formulas for irreducible modules of
the Lie superalgebras $\fsl(m,n)$}, J. Math. Phys. {\bf 31} (1990) 2278--2304.

\bibitem[K1]{K1} V.~G.~Kac, {\em Lie superalgebras},  Adv.
    in Math., \textbf{26}, no.~1 (1977), 8--96.

\bibitem[K2]{K1.5} V.~G.~Kac, {\em Representations of classical Lie superalgebras}, in Lecture Notes Math. {\bf 676},
Springer-Verlag, 1978,  597--626.

\bibitem[K3]{K2} V.~G.~Kac, {\em Infinite-dimensional Lie algebras},
Third edition, Cambridge University Press, 1990.




\bibitem[KK]{KK} V.~G.~Kac, D.~A.~Kazhdan, {\em Structure of representations
 with highest weight of 
infinite-dimensional Lie algebras}, Adv. in Math. {\bf 34} 
(1979), no. 1, 97-108.

\bibitem[KRW]{KRW} V.~G.~Kac, S.-S.~Roan, M.~Wakimoto, 
{\em Quantum reduction for affine Lie superalgebras},
Comm. Math. Pfhys. {\bf 241} (2003), No. 2-3, 307-342.

\bibitem[KW1]{KW1} V.~G.~Kac, M.~Wakimoto, {\em Modular invariant 
representations of
  infinite-dimensional Lie algebras and superalgebras},  
\textsl{Proc. Nat'l. Acad. Sci. USA} \textbf{85}  (1988), 4956--4960.

\bibitem[KW2]{KW2} V.~G.~Kac, M.~Wakimoto, {\em Classification of 
modular invariant  representation of affine algebras}, Advanced Ser. Math. Phys. 7, World Sci.,
1989,  138-177.

\bibitem[KW3]{KW3} V.~G.~Kac, M.~Wakimoto, {\em Integrable highest
  weight modules over affine superalgebras and number theory},
  Progress in Math. {\bf 123} (1994), 415-456.

\bibitem[KW4]{KW4} V.~G.~Kac, M.~Wakimoto, {\em Integrable highest
  weight modules over affine superalgebras and Appell's function},
  Commun. Math. Phys. {\bf 215} (2001), 631-682.

\bibitem[KW5]{KW5} V.~G.~Kac, M.~Wakimoto, {\em Representations
of affine superalgebras and mock theta functions}, to appear in Transf. Groups.
{\bf 19} (2014), 383--455.

\bibitem[KW6]{KW6} V.~G.~Kac, M.~Wakimoto, {\em Representations
of affine superalgebras and mock theta functions II}, arXiv: 1402.0727.

\bibitem[KT1]{KT1}  M.~Kashiwara, T.~Tanisaki, {\em Kazhdan-Lusztig conjecture
for symmetrizable Kac-Moody algebras III. Positive rational case}, 
in Mikio Sato: a great Japanese mathematician of the twentieth century, Asian J. Math. {\bf 2} (1998), no. 4,
779-832.


\bibitem[KT2]{KT2}  M.~Kashiwara, T.~Tanisaki, {\em Characters of the irreducible modules with non-critical
highest weights over affine Lie algebras}, in  Representations and quantizations (Shanghai, 1998),  275-296,
China High. Educ. Press, Beijing, 2000.


\bibitem[R]{R}{Sh.~Reif}, {\em Denominator Identity for twisted affine 
Lie superalgebras}, IMRN (2013), imrn/rnt053.

\bibitem[S1]{S0} V.~Serganova, {\em Kazhdan-Lusztig polynomials and 
character formula for the Lie superalgebra $\fgl(m|n)$},  Selecta
Math. (N.S.) {\bf 2} (1996), no. 4, 607--651.


\bibitem[S2]{S05} V.~Serganova, {\em  
Characters of irreducible representations of simple Lie superalgebras}, 
Proceedings of the International Congress of Mathematicians, vol. II, 1998, 
Berlin, Doc. Math., J. Deutsch. Math.-Verein (1998), 583--593.

\bibitem[S3]{VGRS} V.~Serganova, {\em On generalizations of root systems}, 
Comm. in Algebra {\bf 24 (13)}
(1996) 4281--4299.


\bibitem[S4]{S} V.~Serganova, {\em Kac-Moody superalgebras and integrability}, in
Developments and trends in infinite-dimensional Lie theory, 169-218, Progr. Math., 288, Birkh\"auser Boston, Inc., Boston, MA, 2011.


\bibitem[SZ]{SZ} Y.~Su, R.~B.~Zhang,
{\em Character and dimension formulae for general linear superalgebra},
Advances in Mathematics {\bf 211} (2007), 1--33.

\bibitem[Z]{Z} S.~Zwegers, {\em Mock  theta functions}, arXiv:0807.4834.

\end{thebibliography}
\end{document}